\documentclass[12pt]{acta_2011xz}
\usepackage{harvard_acta}
\usepackage[section]{placeins}
\makeatletter

\def\input@path{{./0AMG/}{./}}

\makeatother

\usepackage{hyperref}

\newcommand{\triplenorm}[1]{
  \left\vert\kern-0.9pt\left\vert\kern-0.9pt\left\vert #1
  \right\vert\kern-0.9pt\right\vert\kern-0.9pt\right\vert}  
\newcommand{\Tscalar}[2]{\ensuremath{(#1 , #2)_{\bar R^{-1}}}} 
\newcommand{\Tnorm}[1]{\ensuremath{\|#1\|_{\bar R^{-1}}}}
\newcommand{\Dscalar}[2]{\ensuremath{( #1 , #2 )_D}}
\newcommand{\Dnorm}[1]{\|#1\|_{D}}
\newcommand{\Tproj}{\ensuremath{Q_c}}
\newcommand{\Dproj}{\ensuremath{Q_D}}

\newcommand{\trace}{\ensuremath{\operatorname{trace}}}
\newcommand{\Span}{\ensuremath{\operatorname{span}}}

\newcommand{\eqqsim}{\mathbin{\rotatebox[origin=c]{180}{\ensuremath{\cong}}}}

\newcommand{\rs}{\ensuremath{R_s}}

\newcommand{\sparse}{\ensuremath{\operatorname{S}}}
\newcommand{\dof}{\ensuremath{N}}
\newcommand{\co}{\ensuremath{C_{o}}}

\usepackage{amsmath}
\usepackage{amsbsy}
\usepackage{amsfonts}
\usepackage{amssymb}
\usepackage{amscd}
\usepackage{mathrsfs}
\usepackage{bm}

\usepackage{algpseudocode}
\usepackage[Algorithm]{algorithm}
\algnewcommand{\IIf}[1]{\State\algorithmicif\ #1\ \algorithmicthen}
\algnewcommand{\EElse}{\unskip\ \algorithmicelse\ }
\algnewcommand{\EndIIf}{\unskip\ \algorithmicend\ \algorithmicif}
\algnewcommand{\FFor}[1]{\State\algorithmicfor\ #1\ }
\algnewcommand{\EndFFor}{\unskip\ \algorithmicend\ \algorithmicfor}

\usepackage{txfonts}
\usepackage[latin1]{inputenc}

\usepackage{graphicx}
\usepackage{subfigure} 
\usepackage{float}

\usepackage{tikz}   
\usetikzlibrary{shapes,arrows,decorations.pathmorphing,backgrounds,positioning,fit,matrix,calc}  

\tikzstyle{decision} = [diamond, draw, fill=blue!20, 
    text width=4.5em, text badly centered, node distance=3cm, inner sep=0pt]
\tikzstyle{block} = [rectangle, draw, fill=blue!20, 
    text width=5em, text centered, rounded corners, minimum height=4em]
\tikzstyle{line} = [draw, -latex']
\tikzstyle{cloud} = [draw, ellipse,fill=red!20, node distance=3cm,
    minimum height=2em]

\usepackage{type1cm}       
\usepackage{soul}
\usepackage{url}
\usepackage{undertilde}
\usepackage{rotating} 
\usepackage{makeidx}
\makeindex             
\usepackage{multicol}   
\usepackage{enumerate}
\usepackage{xspace}
 \usepackage{comment}  
 
\usepackage{etoolbox}
\usepackage{fancyhdr}
\ifcsmacro{theorem}{}{
\newtheorem{theorem}{Theorem}[section]
}
\ifcsmacro{lemma}{}{
\newtheorem{lemma}[theorem]{Lemma}
}
\ifcsmacro{corollary}{}{
\newtheorem{corollary}[theorem]{Corollary}
}
\ifcsmacro{proposition}{}{
\newtheorem{proposition}[theorem]{Proposition}
}
\ifcsmacro{algorithm}{}{
\newtheorem{algorithm}[equation]{Algorithm}
}

\newtheorem{assumption}[equation]{Assumption}
\newtheorem{observation}[equation]{Observation}

\let\oldchapter\chapter
\def\chapter{
  \setcounter{exercise}{0}
  \oldchapter
}

\newcommand{\Label}{\label}
\newcommand{\Rf}[1]{\mbox{$(\ref{#1})$}}

\DeclareMathOperator*{\argmin}{arg\,min}

\DeclareMathOperator*{\range}{Range}

\newcommand{\beas}{\begin{eqnarray*}}
\newcommand{\eeas}{\end{eqnarray*}}
\newcommand{\bary}{\begin{array}}
\newcommand{\eary}{\end{array}}

\def\ec{\mathrel{\hbox{$\copy\Ea\kern-\wd\Ea\raise-3.5pt\hbox{$\sim$}$}}}
\newcommand{\lc}{\mathrel{\raise2pt\hbox{${\mathop<\limits_{\raise1pt\hbox{\mbox{$\sim$}}}}$}}}
\newcommand{\gc}{\mathrel{\raise2pt\hbox{${\mathop>\limits_{\raise1pt\hbox{\mbox{$\sim$}}}}$}}}

\newbox\Ea
\setbox\Ea=\hbox{\raise0.9pt\hbox{$=$}}

\newcommand{\bproof}{\begin{proof}}
\newcommand{\eproof}{\end{proof}}

\newcommand{\cth}{{\cal T}_h}

\newcommand{\curl}{{\rm curl }}

\newcommand{\innerA}{(\cdot,\cdot)_A}

\ifcsmacro{R}{}{
 
}

\newcommand{\diam}{\mbox{\rm diam\,}}

\newtheorem{remark}[theorem]{Remark}
\newtheorem{definition}[theorem]{Definition}
\newtheorem{exercise}[theorem]{Exercise}

\graphicspath{{figures/}}

\begin{document}

\title{Algebraic Multigrid Methods}

\author{Jinchao Xu and Ludmil Zikatanov\thanks{Center for
  Computational Mathematics and Applications, Department of
  Mathematics, The Pennsylvania State University, University Park,
  PA 16802, USA. \newline {Email:} jinchao@psu.edu \& ludmil@psu.edu}
}
\date{\today}

\maketitle

\begin{abstract}
  This paper is to give an overview of AMG methods for solving large 
  scale systems of equations such as those from the discretization of 
  partial differential equations.  AMG is often understood as the 
  acronym of ``Algebraic Multi-Grid'', but it can also be understood 
  as ``Abstract Muti-Grid''.  Indeed, as it demonstrates in this
  paper, how and why an algebraic multigrid method can be better
  understood in a more abstract level.  In the literature, there are a
   variety of different algebraic multigrid methods that have
  been developed from different perspectives.  In this paper, we try
  to develop a unified framework and theory that can be used to derive
  and analyze different algebraic multigrid methods in a coherent
  manner.  Given a smoother $R$ for a matrix $A$, such as Gauss-Seidel
  or Jacobi, we prove that the optimal coarse space of dimension $n_c$
  is the span of the eigen-vectors corresponding to the first $n_c$
  eigen-values of $\bar RA$ (with $\bar R=R+R^T-R^TAR$).  We also prove
  that this optimal coarse space can be obtained by a constrained
  trace-minimization problem for a matrix associated with $\bar RA$
  and demonstrate that coarse spaces of most of existing AMG methods
  can be viewed some approximate solution of this trace-minimization
  problem.  Furthermore, we provide a general approach to the
  construction of a quasi-optimal coarse space and we prove that under
  appropriate assumptions the resulting two-level AMG method for the
  underlying linear system converges uniformly with respect to the
  size of the problem, the coefficient variation, and the anisotropy.
  Our theory applies to most existing multigrid methods, including the
  standard geometric multigrid method, the classic AMG,
  energy-minimization AMG, unsmoothed and smoothed aggregation AMG,
  and spectral AMGe.
\end{abstract}
\newpage 
\tableofcontents 
\newpage 

\section{Introduction}
Multigrid methods are among the most efficient numerical methods for
solving large scale systems of equations, linear and nonlinear alike,
arising from the discretization of partial differential equations.
This type of methods can be viewed as an acceleration of traditional
iterative methods based on local relaxation such as Gauss-Seidel and
Jacobi methods.  For linear systems arising from finite element or
finite difference discretization of elliptic boundary value problems,
local relaxation method were observed to converge very fast on the
high frequency part of the solution.  The low frequency part of the
solution, although slow to converge, corresponds to a relatively
smoother part of the function that can be well-approximated on a
coarser grid.  The main idea behind such multigrid methods is to
project the error obtained after applying a few iterations of local
relaxation methods onto a coarser grid.  The projected error equations
have two characteristics.  Firstly, the resulting system has a smaller
size.  Secondly, part of the slow-to-converge low frequency error on a
finer grid becomes relatively high frequency on the coarser grid and
these frequencies can be further corrected by a local relaxation
method, but this time on the coarse grid.  By repeating such a process
and going to further coarser and coarser grids, a multilevel iterative
process is obtained. Such algorithms have been proven to have uniform
convergence with nearly optimal complexity for a large class of linear
algebraic systems arising from the discretization of partial
differential equations, especially elliptic boundary problems of 2nd
and 4th order.  One main component of this type of multilevel
algorithm is a hierarchy of geometric grids, typically a sequence of
nested grids obtained by successive refinement.  The resulting
algorithms are known as geometric multigrid (GMG) methods.

Despite of their extraordinary efficiency, however, the GMG methods have their
limitations. They depend on a hierarchy of geometric grids
which is often not readily available and it can be argued
that the range of applicability of the GMG methods is, therefore, limited.  

The Algebraic MultiGrid (AMG) methods were designed in an attempt to
address such limitations. They were proposed as means to
generalize geometric multigrid methods for systems of equations that
share properties with discretized PDEs, such as the Laplacian equation,
but potentially have unstructured grids in the underlying
discretization.  The first AMG algorithm
in~\cite{Brandt.A;McCormick.S;Ruge.J.1982a} was a method developed
under the assumption that such a problem was being solved.  Later, the
AMG algorithm was generalized using many heuristics to
extend its applicability to more general problems and matrices.  As a
result, a variety of AMG methods have been developed in the last
three decades and they have been applied to many practical problems
with success.

In this paper, we give an overview of AMG methods from a
theoretical viewpoint.  AMG methods have been developed through a
combination of certain theoretical consideration and heuristic
arguments, and many AMG methods work, with various degrees of
efficiency, for different applications.  We find it very hard to give a
coherent picture on the state of the art of AMG methods if we choose
to simply make a comprehensive list of existing algorithms without
digging into their theoretical foundation.  But, unfortunately, a
good theoretical understanding of why and how these methods work is
still seriously lacking.

In preparing this article, we have undertaken the task of making a
thorough investigation on the design and analysis of AMG from a
theoretical point of view. While there are many bits and pieces of
ideas spreading out in the literature, we managed to re-examine most
of the existing results and ``re-invent the wheel'' trying to deliver
a coherent theoretical description. To do this, we have developed
several tools for the design and analysis of AMG.

With very few exceptions, the AMG algorithms have been mostly
targeting the solution of symmetric positive definite (SPD) systems.
In this paper, we choose to present our studies for a slightly larger
class of problems, namely symmetric semi-positive definite (SSPD)
systems.  This approach is not only more inclusive, but more
importantly, the SSPD class of linear systems can be viewed as more
intrinsic to the AMG ideas.  For example, while the standard
discretizations of the Laplacian operator with homogeneous Dirichlet
boundary condition results in an SPD system, the design of AMG may be
better understood by using local problem (defined on subdomains) with
homogeneous Neumann boundary condition, which would amount to an SSPD
sub-systems.

In short, in this paper we consider AMG techniques for solving a linear
system of equations: 
\begin{equation}
  \label{Au=f}
	Au =f,
\end{equation}
where $A$ is a given SSPD operator or sparse matrix, and the problem
is posed in a vector space of a large dimension.  The starting point
of an AMG procedure is to first choose a smoother, which is often
taken to be some local relaxation iterative methods such as the point
Jacobi, Gauss-Seidel method, or more generally, overlapping Schwarz
methods.  The use of pointwise smoothers seems to encompass most of
the efforts in the literature in constructing and implementing most
(if not all) of the AMG methods.  More general smoothers based on
overlapping Schwarz methods are necessary for some problems, but we
shall not study them in detail in this paper.  In any event, any
chosen smoother is expected to only converge well on certain components
of solution, which will be known as {\it algebraic high frequencies}
with respect to the given smoother.  With the smoother fixed, the
main task of an AMG method then is to identify a sequence of coarse
spaces that would complement well this smoother.  Roughly speaking,
an ideal sequence of coarse spaces is such that any vector (namely
solution to \eqref{Au=f} for any $f$) in the finest space can be well
represented by a linear combination of all the algebraic high
frequencies on all coarse spaces.  As a result, the AMG method would
converge well for the problem \eqref{Au=f}.

It is hard to make a theoretically concise statement for what is said
above in the multilevel setting.  Instead we will focus on first
answering such a question for a two-level setting.  For a two-level
setting, the first theoretical question is as follows:
``{\it 
  Given a smoother, say $R$, what is the optimal coarse space of given
  dimension so that the resulting AMG has the best convergence rate?}''
This question will be thoroughly addressed in \S\ref{s:2Level}.
As it turns out, the optimal coarse space will consist of the
eigenvectors corresponding to the lower-end of the spectrum of the a
matrix such as $RA$.  While our two level theory is theoretically
pleasing, it does not offer a practical solution as finding these
eigenvectors will be too expensive.  Thus, the task of the AMG design
is to find good but inexpensive approximation of this \emph{algebraic
  low frequency} eigen-space which will still result in an AMG
algorithm with desirable convergence properties.  We call such an
approximation of optimal coarse space as ``quasi-optimal'' coarse
space.

In the design of all these AMG algorithms, one key component is the
coarsening of spaces or the graph associated with the matrix $A$.  Two
main strategies are: independent-set based and the aggregation-based.
We will present several approaches on the construction of
quasi-optimal coarse spaces.  We would especially advocate two
approaches that have sound theoretical foundation.  The first approach
is outlined in \S\ref{s:2Level} and later on in
\S\ref{sec:energy-min}.  We will first prove that the optimal coarse
space from the theory in \S\ref{s:2Level} can be characterized, in a
mathematically equivalent manner, by solutions to a trace minimization
problem.  In a functional setting, such a trace-minimization can be
interpreted as minimization of the sum of the energy norm of a set of
coarse basis functions.  These precise equivalents give a very clear
guidance how some AMG methods can be constructed.  One practical
approach based on such a theorem is to look for energy-minimization
basis functions among locally supported function classes.  The
resulting algorithms are known as energy-min AMG.  Other AMG methods,
such as classical AMG and aggregation-based AMG, can be viewed as
approximations to the energy-min AMG.  The second approach is outlined
in \S\ref{sec:unifiedAMG}.  The main idea is to construct
quasi-optimal coarse space by piecing together the low-end eigenspaces
of some appropriately defined local operators or matrices.  This
approach can be used to provide quasi-optimal coarse spaces for
various AMG methods, including the standard geometric multigrid
method, the classic AMG, energy-minimization AMG, unsmoothed and
smoothed aggregation AMG, and spectral AMGe.  As a simple example,
this method relies on the fact that $n_c=Jm$ dimensional low-end
eigenspace of an operator can be well approximated by gluing together
$J$-pieces of $m$-dimensional low-end eigenspaces for some carefully
chosen local operators.  Here $m$ is a very small integer.  For
example, $m=1$ for the for the Laplacian operator and $m=3$
(resp. $m=6$) for $2-$dimensional (resp. $3-$dimensional) linear
elasticity operator.  This important property of eigenspaces is
closely related to the Weyl's Lemma on the asymptotic behavior of
eigenvalues for elliptic boundary value problems, discussed in
\S\ref{s:algebraic-spectral}.

Most AMG methods are designed in terms of the adjacency graph of the
coefficient matrix of a given linear algebraic system.   In
\S\ref{sec:graphs}, we give a brief description of graph theory and
the adjacency graph of a sparse matrix.   One highlight in
\S\ref{sec:graphs} is the concept of M-matrix relative.   This simple
tool is instrumental in the design and analysis of the classical AMG
method. 

One important step in the design of most AMG methods is to zero out
some entries of the coefficient matrix $A$ by using the concept of
strength of connections to get a filtered matrix $\tilde A$, which is
equivalent to dropping out the weakly connected edges in the adjacency
graph $\mathcal G(A)$ to get $\mathcal G(\tilde A)$.  Several
definitions of strength functions are introduced in
\S\ref{sc:strength} to describe the strength of connection.  In
\S\ref{sc:connection}, the graph $\mathcal G(\tilde A)$ is then
coarsened by either keeping a maximal independent set (MIS) as a
coarse vertex set $\mathcal C$ and the dropping the rest of grid, or
use aggregation/agglomeration.  In \S\ref{sc:connection}, some technical details
are also given on the construction of coarse space by using degrees of
freedom. 

By using the aforementioned general approaches and theoretical
techniques, we then motivate and present a number of AMG methods. Some
of the highlights in the paper are outlined below.

We first give an overview of GMG and its relationship with AMG in
\S\ref{s:GMG}.  After describing some details in a typical GMG method
for linear finite element matrix, we argue that the geometric
information used in defining a GMG is essentially the graph
information of the underlying finite element grid (without using other
geometric information such as coordinates of the grid points).  This
is a strong indication that at least some GMG method can be realized
by a pure algebraic fashion: only using the stiffness matrix, an
algebraic smoother, and the adjacency graph of the stiffness matrix.
On the other hand, we prove that a GMG method can also be formally
obtained by our general AMG approach presented in
\S\ref{sec:unifiedAMG}.  Furthermore, we use the example of AMGe in
\S\ref{s:AMGe} to demonstrate that geometric information on the grid
can be effectively used to construct a geometry-based AMG.

In \S\ref{sec:energy-min}, we give a detailed account on AMG methods
based on energy-minimization.  We first present our new theory that
the optimal coarse space shown in the 2-level theory in
\S\ref{s:2Level} can be actually obtained through trace-minimization
(\S\ref{s:energy-trace}).   After proving that the trace-minimization can be
equivalently formulated to an energy-minimization problem in
\S\ref{s:energy-trace},  we then derive energy-min AMG method by
seeking a set of locally supported coarse basis functions for
energy-minimization. 

Classical AMG, as the first class of AMG algorithms studied in the
literature, will be presented in \S\ref{s:classical-amg}. We derive
and analyze this type of methods using the framework in
\S\ref{sec:unifiedAMG} and also the notion of M-matrix relatives
introduced in \S\ref{sec:m-matrix}.  We further discuss how a
classical AMG method can be viewed as an approximation of energy-min
AMG method.

Aggregation-based AMG will be presented in \S\ref{s:agmg}.  Again we
derive and analyze this method using the framework in
\S\ref{sec:unifiedAMG}.  One remarkable feature of the aggregation AMG
methods is their ease to preserve multi-dimensional near-null
space, such as the rigid-body modes in linear elasticity.

To demonstrate how an AMG method addresses possible heterogeneous
properties in a given problem, we devote \S\ref{s:AnisoJump} to show
how classical AMG is designed to address the difficulties arising from
the discretized elliptic problems with strong discontinuous jumps or
anisotropy in the coefficients of the underlying PDE.  Finally we make
some concluding remarks in \S\ref{s:conclusion}.
 
In \S\ref{s:adaptiveAMG}, we outline a class of AMG methods that
attempt to choose the coarse spaces in a bootstrap and adaptive
fashion.  This line of AMG algorithms do not fall into the theoretical
frameworks presented in this paper, but they provide a practical
approach to generalize many existing AMG techniques to a more general class
of problems.

We conclude these introductory remarks by a brief summary of the
acronyms used in different AMG algorithms reviewed in this paper.

\begin{enumerate}
\item Aggregation-based AMG 
\begin{itemize}
\item Unsmoothed aggregation \hfill UA-AMG
\item Smoothed aggregation \hfill SA-AMG
\end{itemize}
\item Bootstrap \& Adaptive AMG 
\begin{itemize}
\item Classical \hfill $\alpha$AMG
\item Smoothed aggregation\hfill $\alpha$SA-AMG
\item Bootstrap AMG \hfill BAMG
\end{itemize}
\item Element-based AMG \hfill AMGe
\item Spectral AMGe\hfill $\rho$AMGe
\end{enumerate}

\section{Model problems and discretization}
While AMG has found applications to a wide range of linear algebraic
systems, its development has been mainly motivated by the solution of
systems arising from the discretization of partial differential
equations by finite element, finite difference or other numerical methods.  In this section, we will discuss a model of second order
elliptic boundary problem,  their finite difference and finite element
discretization and relevant properties of the relevant underlying
differential operators and their discretization. 

\subsection{Model elliptic PDE operators}
We consider the following boundary value problems
\begin{equation}
  \label{Model0}
{\mathcal L}u=-\nabla\cdot \alpha(x)\nabla u=f, \quad x\in \Omega
\end{equation}
where $\alpha: \Omega\mapsto \mathbb R^{d\times d}$ is an SPD matrix function
satisfying
\begin{equation}
  \label{alpha}
\alpha_0\|\xi\|^2\le 
\xi^T\alpha (x)\xi \le
\alpha_1\|\xi\|^2,\quad \xi\in \mathbb R^d.
\end{equation}
for some positive constants $\alpha_0$ and $\alpha_1$.  Here $d=1,2,3$
and $\Omega\subset\mathbb R^d$ is a bounded domain with boundary
$\Gamma=\partial \Omega$.

A variational formulation for \eqref{Model0} is as follows: Find
$u\in V$ such that
\begin{equation}
  \label{Vari}
a(u,v)=(f, v), \quad\forall v\in V.   
\end{equation}
Here
$$
a(u,v)=\int_\Omega (\alpha(x)\nabla u)\cdot \nabla v, \quad 
(f,v)= \int_\Omega fv.
$$ 
and $V$ is a Sobolev space that can be chosen to address different
boundary conditions accompanying the equation~\eqref{Model0}. 
One case is the mixed boundary conditions: 
\begin{equation}
  \label{MixedBoundary}
  \begin{array}{rcl}
u=&0, &x\in \Gamma_D,\\
(\alpha\nabla u)\cdot n=&0,&x \in\Gamma_N,
\end{array}
\end{equation}
where
$\Gamma=\Gamma_D\cup\Gamma_N$.  The pure Dirichlet problem is when
$\Gamma_D = \Gamma$ while the pure Neumann problem is when $\Gamma_N
=\Gamma$. We thus have $V$ as
\begin{equation}
  \label{3V}
V=
\left\{
  \begin{array}{l}
    H^1(\Omega) = \{v\in L^2(\Omega): \partial_iv\in    L^2(\Omega), i=1:d\};\\
    H^1_D(\Omega) = \{v\in H^1(\Omega): v|_{\Gamma_D}=0\}. 
  \end{array}
\right.  
\end{equation}
When we consider a pure Dirichlet problem, $\Gamma_D = \Gamma$, we
denote the space by $V=H^1_0(\Omega)$. In addition, for pure Neumann
boundary conditions, the following condition is  added to
assure the existence of the solution to~\eqref{Vari}:
\begin{equation}
  \label{consistent-f}  \int_\Omega f =0. 
\end{equation}

One most commonly used model problem is when 
\begin{equation}
  \label{iso}
\alpha(x)=1, \quad x\in \Omega,
\end{equation}
which corresponds to the Poisson equation
\begin{equation}
  \label{Poisson}
-\Delta u=f.  
\end{equation}
This simple problem provides a good representative model for isotropic problems.

There are other two cases that are of special interests.  The first case is
when $\alpha$ is a scalar and it has discontinuous jumps such as
\begin{equation}\label{coeff12}
\alpha(x) = \begin{cases}
\epsilon , \quad x\in \Omega_1,\\
1 , \quad x\in \Omega_2.
\end{cases}
\end{equation}

The second case is when $\alpha$ is a diagonal matrix such as (for
$d=2$):
\begin{equation}
  \label{aniso-a}
\alpha(x)=
\begin{pmatrix}
  1 &0\\
0 &\epsilon
\end{pmatrix},
\end{equation}
which corresponds to the following operator
\begin{equation}\label{eq:anisotropic}
        -u_{xx}-\epsilon u_{yy} = f.
\end{equation}
In both cases above, we assume that $\epsilon$ is sufficiently small
to investigate the robustness of algorithms with respect to
discontinuous jumps and an-isotropy. 

\subsection{Examples of finite difference and finite element
  discretizations}\label{sec:fem}
As an illustrative example, we consider a finite difference
discretization of the Poisson equation \eqref{Poisson} with pure
Dirichlet boundary conditions on the unit square $\Omega = (0,1)
\times (0,1)$.  We consider a uniform triangulation of $\Omega$ (see
the two left figures in Fig.\ref{fig:2dpartition}) and we set
\[
(x_i, y_j)=\left(\frac{i}{n+1},\frac{j}{n+1}\right),\quad u_{i,j}
\approx u(x_i,y_j), \quad (i, j=0,\cdots,n+1).
\]
\begin{figure}[H]
\begin{center}
\includegraphics[height=1.5in]{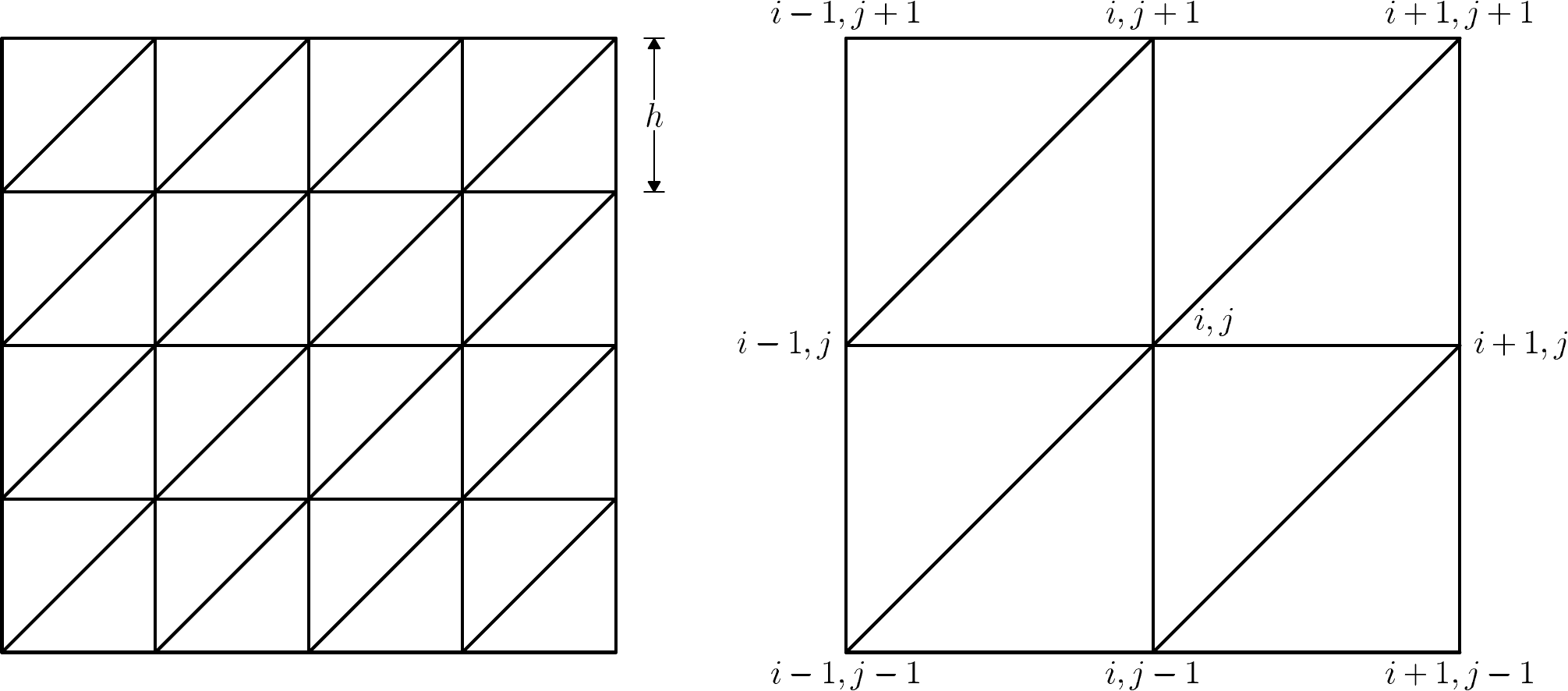}\hfill   
\includegraphics[height=1.5in]{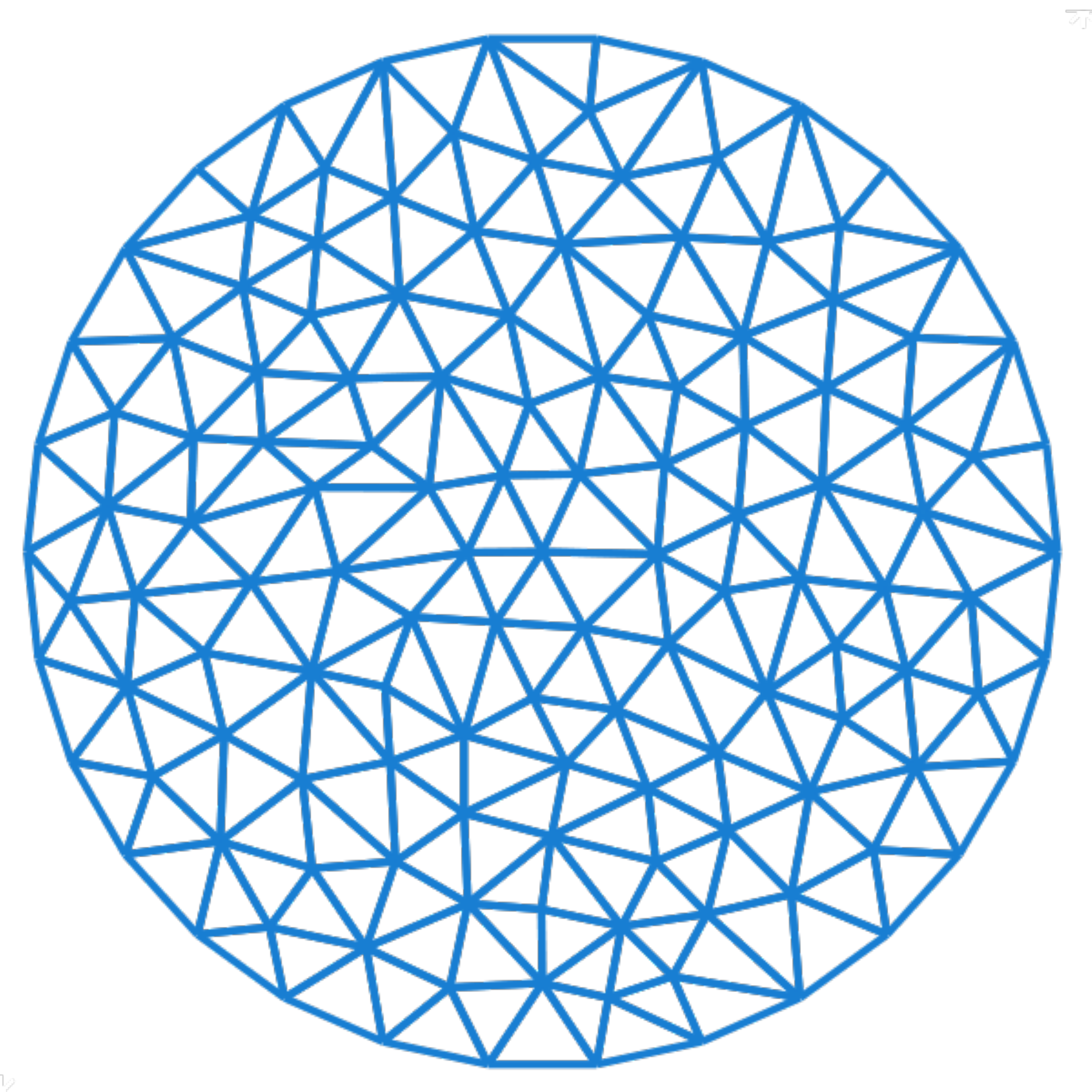}
\caption{Regular (uniform) triangulations for the unit square (left
  and center) and unstructured mesh approximating the unit disk
  (right).\label{fig:2dpartition}}
\end{center}
\end{figure}

We use the standard center difference approximation to the Laplacian
operator
\[
(-\Delta u)(x_i,y_j) \approx   
\frac{4u_{i,j}-u_{i+1,j}-u_{i-1,j}-u_{i,j+1}-u_{i,j-1}}{h^2}.
\]
The finite difference scheme is then given by
\begin{equation}
  \label{2d-fd0}
4u_{i,j}-(u_{i+1,j}+u_{i-1,j}+u_{i,j+1}+u_{i,j-1})=h^2f_{i,j},
\end{equation}
where 
\begin{equation}
  \label{fij-fd}
f_{i,j} = f(x_i, y_j) 
\end{equation}
and $u_{i,j} \approx u(x_i,y_j)$. The approximations $u_{i,j}$ are found by solving a linear
system. We order the points $(x_i,y_j)$ lexicographically and we have for $k=1, \dots, n^2$,
\begin{equation}
  \label{lexi}
k=(j-1)n+i,  \quad x^h_{k}=(x_i,y_j), \quad \mu_k=u_{i,j},\quad 1\le i,j\le n, 
\end{equation}

We can then write \eqref{2d-fd0} as
\begin{equation}\label{2d-fd}
A\mu=b,
\end{equation}
where
\begin{equation}
  \label{iso-A}
A=\operatorname{tridiag} (-I, B, -I), \mbox{ and }
B=\operatorname{tridiag} (-1, 4, -1). 
\end{equation}
A slightly different scheme is obtained using more of the neighboring
points of $(x_i, y_j)$. We can build an approximation using 8
points $(x_{i\pm 1}, y_{j\pm 1})$ together with the ``center''
point $(x_i, y_j)$. As a result we have the 9-point finite difference scheme as follows:
\begin{eqnarray}
 &&   8\mu_{i,j} - \mu_{i-1, j}-\mu_{i+1, j}-\mu_{i, j-1}-\mu_{i, j+1}
    \label{2d-fd-9p} \\ \label{9-point}
&&\phantom{8\mu_{i,j}} - \mu_{i-1, j-1} -\mu_{i+1, j-1}-\mu_{i-1, j+1}-\mu_{i+1,j+1}=2h^2f_{i,j}. 
\nonumber
\end{eqnarray}
Again, if we order $(x_i, y_j)$ lexicographically,
then~\eqref{2d-fd-9p} is the linear system~\eqref{2d-fd} corresponding
to the 9-point finite difference discretization of the Laplace
equation with 
\begin{equation}
  \label{9point}
A=\operatorname{tridiag}(-C, B, -C) \mbox{ with }
B=\operatorname{tridiag}(-1, 8,-1), C=\operatorname{tridiag}(1, 1,1).
\end{equation}

We now give an example of finite element discretization. 
Given a triangulation ${\mathcal T}_h$ for $\Omega$, such as that
given in Figure~\ref{fig:2dpartition}, let $V_h\subset V$ be a finite
element space consisting of piecewise linear (or higher order)
polynomials with respect to the triangulation ${\mathcal T}_h$.  The
finite element approximation of the variational problem \eqref{Vari}
is: Find $u_h\in V_h$ such that
\begin{equation}
  \label{vph}
a(u_h,  v_h)=(f,v_h), \quad\forall\,v_h\in V_h.
\end{equation}
Assume $\{\phi_i\}_{i=1}^{N}$ is the nodal basis of $V_h$,  namely,
$\phi_i(x_j)=\delta_{ij}$ for any nodes $x_j$.
We write
\(
        u_h(x)=\sum_{j=1}^{N}\mu_j\phi_j(x)
\)
the equation \eqref{vph}  is then equivalent to
\[
        \sum_{j=1}^{N}\mu_ja(\phi_j,\phi_i)=(f,\phi_i),\quad
        j=1,2,\cdots, N, 
\]
which is a linear system of equations:
\begin{equation}\label{axb}
        A\mu=b, \quad (A)_{ij} = a(\phi_j,\phi_i), \quad 
 \mbox{and}\quad (b)_i=(f,\phi_i).
\end{equation}
Here, the matrix $A$ is known as the stiffness matrix of the  nodal basis
$\{\phi_i  \}_{i=1}^N$.

For $d=2$ and the special uniform triangulation as shown on the in
Figure~\ref{fig:2dpartition}, this stiffness matrix for the
Laplacian operator turns out to be exactly the one given by
\eqref{2d-fd0}.  This special case is an example of the close
relationship between finite difference and finite element methods.

We note that the finite element method is based on the variational
formulation \eqref{Vari}, whereas the finite difference method is not.
In the development of AMG method, however, variational method is also
used to derive coarse level equations for finite difference methods.

    For any $T\in \mathcal T_h$, we define
    \begin{equation}\label{hT}
        \overline h_T=\diam(T),\quad h_T=|T|^{\frac{1}{d}}, \quad \underline h_T=2\sup\{r>0: B(x, r)\subset T \text{ for } x\in T\}.
    \end{equation}
    We say that the mesh $\mathcal T_h$ is \emph{shape regular} if there exists a uniformly bounded constant $\sigma \ge 1$ such that     
    \begin{equation}
        \underline h_T\le h_T \le \overline h_T\le \sigma \underline h_T, \quad \forall T\in \mathcal T_h.
    \end{equation}
    And we call  $\sigma$ the \emph{shape regularity constant}. 

    Let $h=\max_{T\in \mathcal T_h} \overline h_T$, with $\overline
    h_T$ defined in \eqref{hT}. We say that the mesh $\mathcal T_h$ is
    \emph{quasi-uniform} if there exists a uniformly bounded constant
    $C > 0$ such that
    \begin{equation}
        \frac{h}{\underline h_T} \le C.
    \end{equation}

\subsection{Spectral   properties}\label{s:algebraic-spectral}
We now discuss the spectral properties of the
partial differential operator ${\cal L}$ given in \eqref{Model0}. 

We recall the well-known Courant-Fischer min-max
principle~\cite{courant1924methoden} for eigenvalues of symmetric
matrices.
\begin{theorem}\label{thm:minmax}
    Let $T$ be a $n\times n$ symmetric matrix with respect to $(\cdot, \cdot)_*$, and $\{\lambda_j, \zeta_j\}$ are its eigenpairs with $\lambda_1\le \lambda_2\le \cdots \le \lambda_n$, then 
    \begin{equation}\label{minmax}
        \lambda_k=\min_{\dim W=k}\max_{x\in W, x\ne 0}\frac{(Tx, x)_*}{(x, x)_*},
    \end{equation}
    where the minimum is achieved if  
    \begin{equation}\label{minmax-opt}
        W=\operatorname{span}\{\zeta_j: j=1:k\},
    \end{equation}
    and 
    \begin{equation}\label{maxmin}
        \lambda_k=\max_{\dim W=n-k+1}\min_{x\in W, x\ne 0}\frac{(Tx, x)_*}{(x, x)_*},
    \end{equation}
    where the maximum is achieved if 
    \begin{equation}\label{maxmin-opt}
        W=\operatorname{span}\{\zeta_j: j=k:n\}.
    \end{equation}
\end{theorem}
Next, we recall Theorem 1 in \cite{fan1949theorem} which is known as
Ky-Fan trace minimization principle.
\begin{theorem}\label{thm:fan}
  We suppose $T$ is symmetric with respect to $(\cdot, \cdot)_*$, and
  $\{\lambda_j, \zeta_j\}$ are its eigenpairs with
  $\lambda_1\le \lambda_2\le \cdots \le \lambda_n$, then
    \begin{equation*}
        \min_{P\in \mathbb{R}^{n\times k}, P^*P=I}\operatorname{trace}(P^*TP) = \sum_{j=1}^{k}\lambda_j.
    \end{equation*}
    Furthermore, the minimum is achieved when 
\begin{equation*}
    \operatorname{range}(P)=\operatorname{span}\{\zeta_j\}_{j=1}^{k} \text{ and } P^*P=I. 
\end{equation*}
    Here $P^*\in \mathbb{R}^{k\times n}$ is the adjoint of $P$ corresponds to $(\cdot, \cdot)_*$ inner product, namely
    \begin{equation*}
        (P^* u, v)= (u, Pv)_*, \quad \text{for all } u\in \mathbb{R}^{n}, v\in \mathbb{R}^{k}.
    \end{equation*}
\end{theorem}

Finally, following \cite{1992XuJ-aa}, we use the notation
$a\lesssim b$ to represent the existence of a generic positive
constant $C$, which is independent of important parameters, such as
problem size, anisotropic ratio, or other, and such that $a\le Cb$.
Furthermore, we write $a\eqqsim b$ iff $a\lesssim b$ and
$b\lesssim a$.

\begin{theorem}   
The PDE operator ${\mathcal L}$ has a complete set of eigenfunctions
$(\varphi_k)$ and nonnegative eigenvalues 
$$
0\le\lambda_1\le\lambda_2\le \ldots 
$$
such that 
$$
{\mathcal L}\varphi_k=\lambda_k\varphi_k, \quad k=1, 2, 3\ldots. 
$$
\begin{enumerate}
\item $\lim_{k\to\infty}\lambda_k=\infty.$
\item $(\varphi_i)$ forms an orthonomal basis of $V$ as well as for
  $L^2(\Omega)$. 
\end{enumerate}
Furthermore
\begin{enumerate}
\item For pure Neumann problem, $\lambda_1=0$ and $\varphi_1$ is
the  constant function. 
\item For pure Dirichlet problem, $\lambda_1>0$ is simple and 
  $\varphi_1$ does not change sign.
\end{enumerate}
\end{theorem}  

We have the well-known Weyl's estimate on the asymptotic behavior of
the Laplacian operator
\cite{weyl1911asymptotische,wbyii1912asymptotische,reed1978iv}.
\begin{lemma}[Weyl's law]\label{lem:Weyl}
  Assume that $\Omega$ is contented (which means $\Omega$ can be
  approximated by unions of cubes in $\mathbb{R}^d$, see \cite[page
  271]{reed1978iv} for the exact definition).  For homogeneous Dirichlet
  boundary condition,  the eigenvalues of the pure Laplacian operator satisfy:
\begin{equation}\label{Weyl0}
    \lim_{k\rightarrow\infty}\frac{\lambda_k}{k^{\frac{2}{d}}}=w_{\Omega},
    \mbox{ with }
    \quad w_{\Omega}=\frac{(2\pi)^2}{[\omega_d \operatorname{Vol}(\Omega)]^{\frac{2}{d}}},
\end{equation}
where $\omega_d$ is a volume of the unit ball in $\mathbb{R}^d$, and
the eigenvalues of the operator $\mathcal L$ given in \eqref{Model0}
satisfy:
\begin{equation}\label{Weyl}
    \lambda_k\eqqsim k^{\frac{2}{d}}, \quad \forall k\ge 1.
\end{equation}
\end{lemma}

Next, we extend the above Weyl's law to discretized PDE operators.   
The following theorem gives a discrete version of the Weyl's law for
the finite element discretization. Further details on such a result
are found in~\cite{Weyl}.

\begin{theorem}\label{thm:WeylFE}
  Let $V_h\subset H_0^1({\Omega})$ be a family of finite element
  spaces on a quasi-uniform mesh with $\dim V_h = N$.  Consider the discretized operator of \eqref{Model0}
\begin{equation*}
    \mathcal L_h: V_h\mapsto V_h, \quad (\mathcal L_h u, v) = a(u, v), \quad \forall u, v\in V_h,
\end{equation*}
and its eigenvalues:
\[
\lambda_{h,1}\le \lambda_{h,2}\le \cdots \lambda_{h,N}.
\]
Then, for all
  $1\le k\le N$, there exists a constant $C_w>0$ independent of $k$
  such that we have the following estimates:
    \begin{equation}\label{e:discrete-continuous}
        \lambda_k\le\lambda_{h,k}\le C_w\lambda_k.
    \end{equation}
and 
    \begin{equation}\label{discrete-Weyl}
        \lambda_{h,k}\eqqsim k^{2/d}.
    \end{equation}
\end{theorem}

\subsection{Properties of finite element matrices}
The main algebraic property for the stiffness matrices given by
\eqref{axb} is that: it is sparse with ${\cal O}(N)$ nonzeros,
symmetric-postive definite (for both Dirichlet and mixed boundary
conditions) and semi-definite for pure Neumann boundary
conditions; Its eigenvalues satisfy the discrete Weyl's law.

For simplicity, we will only consider the pure Dirichlet
boundary conditions in the rest of this section.
\begin{lemma}
\label{lm:stiffness}
    The stiffness matrix $A$ given by \eqref{axb} has the following properties:
\begin{enumerate}
\item The condition number of $A$, defined by the ratio of the extreme
  eigenvalues of $A$, 
$$
\kappa(A)=\frac{\lambda_{\max}(A)}{\lambda_{\min}(A)},
$$
satisfies
$$
\kappa(A) \eqqsim h^{-2}.
$$
Furthermore,
$$
\lambda_{\min}(A) \eqqsim h^2 \mbox{ and } \lambda_{\max}(A)\eqqsim 1.
$$
\item The discrete version of the Weyl's law holds:
$$
        \lambda_k(A)\eqqsim \left(\frac{k}{N}\right)^{2/d}.
$$
\end{enumerate}
\end{lemma}

We next discuss some more refined spectral properties of finite
element stiffness matrices from uniform grids for the unit square
domain $\Omega=(0, 1)\times (0, 1)$  for $d=2$. 
We begin with the Poisson equation.  It is easy to derive a
closed-form solution of the eigenpairs of $A$ given by \eqref{2d-fd}
and we have:
\begin{equation}
  \label{2d-stiffness-eigenvalue}
\lambda_{kl}(A)
=4\left(\sin^2{{k\pi}\over {2(n+1)}}+\sin^2{{l\pi}\over {2(n+1)}}\right),
\end{equation}
and 
\begin{equation}
\label{2d-stiffness-eigenvector}
\phi^{kl}_{ij}=\sin\frac{ki\pi}{n+1}\sin\frac{lj\pi}{n+1}, \quad 1\leq i\leq n, \quad 1\leq j\leq n.
\end{equation}

Consider now the important case of anisotropic problem \eqref{eq:anisotropic}. 
We order the vertices of the triangulation lexicographically,  and, as before,
denote them by $\{(ih,jh)\}_{i,j=0}^n$.  The stiffness matrix then is
\begin{equation}
  \label{aniso-A}
	A=\operatorname{tridiag} (-I, B, -I)\quad\mbox{with}\quad
 B=\operatorname{tridiag} (-\epsilon, 2(1+\epsilon), -\epsilon).
\end{equation}
Obviously,
$$
	A=I\otimes B+C\otimes I\quad\mbox{with}\quad C=\operatorname{tridiag}(-1,0,-1),
$$
and it is easily verified that
$$
\lambda_i(B)=2(1+\epsilon)-2\epsilon\cos{{i\pi}\over {(n+1)}}, \quad
\lambda_j(C)=-2\cos{{j\pi}\over {(n+1)}},\quad 1\leq i,j\leq N.
$$
which leads to the following expression for the eigenvalues
$$
\lambda_{ij}(A)
=4\epsilon\sin^2{{i\pi}\over {2(n+1)}}+4\sin^2{{j\pi}\over {2(n+1)}}.
$$
and the corresponding eigen-vectors
$$
\phi_{ij}^{k\ell}=\sin{{ki\pi}\over {n+1}}\sin{{\ell j\pi}\over {n+1}}.
$$

\section{Linear vector spaces and duals}\label{s:duals}
In this paper, we will mainly consider linear system of equation of
the following form:
\begin{equation}
  \label{auf}
Au=f.  
\end{equation}
Here 
\begin{equation}
A: V\mapsto V',  
\end{equation}
$$
 \quad f\in V',
$$
$V$ is a finite-dimensional linear vector space and $V'$ is the
dual of $V$.  If we use the notation $\langle\cdot, \cdot\rangle$
to denote the pairing between $V'$ and $V$, we can write \eqref{auf}
in a variational form:  Find $u\in V$ such that
\begin{equation}
  \label{VariAu=f}
a(u,v)=\langle f, v\rangle, \quad\forall v\in V 
\end{equation}
where
\begin{equation}
a(u,v)=\langle Au, v\rangle.  
\end{equation}

\subsection{Dual and inner product}
For convenience of exposition, we will assume that $V$ is equipped with an inner
product $(\cdot, \cdot)$.  By Riesz  representation theorem, for any $f\in V'$,
there is a unique $u\in V$ such that
\begin{equation}
  \label{L2dual}
(u,v) = \langle f, v\rangle, \quad \forall v\in V.   
\end{equation}
It is through this representation, we will take $V'=V$.  In the
rest of this paper, for convenience, we will always assume that $V'=V$
for any finite dimensional vector space $V$.    As a result, we have 
\begin{equation}
  \label{dualdual}
V''=(V')'=V'=V.  
\end{equation}
Thanks to the identification $V'=V$ via \eqref{L2dual}, the identities
in \eqref{dualdual} are clear without any ambiguity. 

We would like to point out that, in an abstract discussion of all
iterative methods for problem \eqref{auf}, it suffice to use the
abstract dual pairing $\langle\cdot, \cdot \rangle$ without having to
introducing an inner product $(\cdot, \cdot)$ on $V$.  But we find
that the use an inner product is convenient for exposition as we shall
see later.  
We further point out that we will not use an inner product
to identify $V'=V$ for any infinite dimensional vector space in this
paper.

If $\{\phi_i\}_{i=1}^N$ is a basis of $V$, we will always choose a
basis,  $\{\psi_i\}_{i=1}^N$, of $V'$ that is dual to the basis of
$V$.  Namely
\begin{equation}
  \label{dual-basis}
(\psi_j, \phi_i)=\delta_{ij}, \quad 1\le i, j\le N.   
\end{equation}
Such a dual basis will only used for theoretical consideration and it
will not be used in actual implementation of any algorithms. 

We will only consider two kinds of linear vector spaces:
The first kind is $V=\mathbb R^n$ and the inner product is just the
dot product
$$
(u, v)_{\ell^2}=\sum_{i=1}^nu_iv_i, \quad \forall u=(u_i), v=(v_i)\in \mathbb R^n, 
$$ 
A canonical basis of $\mathbb R^N$ is formed by the column vectors of the
identity matrix, $\{e_i\}_{i=1}^N$, it is easy to see that its dual
basis of $(\mathbb R^N)'=\mathbb R^N$ is just the original basis
$\{e_i\}_{i=1}^N$ itself.

The second kind is a finite dimensional functional subspace of $L^2(\Omega)$ for
a given domain $\Omega\subset\mathbb R^d$ ($1\le d\le 3$), equipped
with the $L^2$ inner product:
$$
(u,v)=\int_\Omega u(x)v(x).
$$
One commonly used linear vector space is a finite element space
$V_h$ and oftentimes the nodal basis functions $\{\phi_i\}_{i=1}^N$ are
used as a basis.   In this case, the dual basis,  $\{\psi_i\}_{i=1}^N$, of $V'$
are not the original nodal basis functions anymore, but rather, a set of
functions (which are usually globally supported) that satisfy
\eqref{dual-basis}.   This set of dual basis functions is usually needed
in deriving matrix representation of operators between
various spaces and their duals, but they are not needed in the actual implementation of
relevant algorithms. 

For a linear operator 
\begin{equation}
  \label{LVV}
L: V\mapsto V,
\end{equation}
 its adjoint:
\begin{equation}
  \label{Ldual}
L': V\mapsto V,
\end{equation}
is defined as follows
\begin{equation}
  \label{LdualL}
(L'u,v)=(u, Lv),\quad u,v\in V.   
\end{equation}
Since $V$ plays the role of both $V$ and its dual $V'$, the notion
\eqref{LVV} and \eqref{Ldual} can have four different meanings:
\begin{enumerate}[1.]
\item If $L: V\mapsto V$, then $L': V'\mapsto V'$;
\item If $L: V\mapsto V'$,  then $L': V\mapsto V'$;
\item If $L: V'\mapsto V$,  then $L': V'\mapsto V$;
\item If $L: V'\mapsto V'$, then $L': V\mapsto V$.
\end{enumerate}
Thanks to the identification we made between $V'$ and $V$ through
\eqref{L2dual}, the definition \eqref{LdualL} is applicable to all the
above four different cases. 

If $V=\mathbb R^n$ and $(u,v)=(u,v)_{\ell^2}$, $L' = L^T$, namely the
matrix transpose.
We say that an operator $A:V\mapsto V'$ is symmetric positive definite
(SPD) if 
$$
A'=A, \quad (Av,v)>0\quad \forall v\in V\setminus\{0\}.
$$

When $A$ is SPD, it defines another inner product $(\cdot, \cdot)_A$
on $V$:
$$
(u,v)_A=(Au,v),\quad u,v\in V
$$
and a corresponding norm 
$$
\|v\|_A=(v,v)_A^{1/2}, \quad v\in V. 
$$
We use the superscript ``*'' for the adjoint
operator with respect to $\innerA$, i.e. 
$$
(Bu, v)_A = (u, B^*v)_A. 
$$
It is easy to see that 
\begin{equation}
  \label{BA}
 (BA)^* = B'A,
\end{equation}
and $(BA)^* = BA$ if and only if $B'=B$.

\subsection{Matrix representation}
Let $V_c\subset V$ be a subspace and consider the inclusion operator $\imath_c:
V_c\mapsto V$.  Assume that $\{\phi_i^c\}_{i=1}^{n_c}$ and $\{\phi_i\}_{i=1}^n$ are basis functions of $V_c$ and $V$ respectively, the matrix
representation of $\imath_c$ is a matrix 
\begin{equation}\label{defP}
P:\mathbb R^{n_c}\mapsto \mathbb R^n \mbox{ satisfying }
(\phi_1^c,\ldots, \phi_{n_c}^c)=(\phi_1,\ldots, \phi_n)P. 
\end{equation}
The identity written above is a shorthand for the expansion of the basis in $V_c$ via the basis in $V$:
\begin{equation}\label{Pexpand}
\phi_k^c =\sum_{j=1}^n p_{jk} \phi_j, \quad \quad P=(p_{jk}),\quad k=1,\ldots, n_c, \quad j=1,\ldots,n.
\end{equation}
What is the matrix representation of $\imath_c': V'\mapsto V_c'$?
Although we have $V'=V$ and $V_c'=V_c$, we need to use dual bases 
$\{\psi_i^c\}_{i=1}^{n_c}\subset V_c'$ and $\{\psi_i\}_{i=1}^n\subset V'$ respectively.   With respect to these dual bases,
the matrix representation of $\imath_c'$ is simply $P^T$ (the
transpose of $P$) since it is easy to verify that
$$
(\imath_c'\psi_1,\ldots, \imath_c'\psi_n)=(\psi_1^c,\ldots, \psi_{n_c}^c)P^T.
$$

Consider now a linear operator 
\begin{equation}
  \label{AVV}
A:V\mapsto V.
\end{equation}
There are two different ways to get
a matrix representation of $A$ because $V$ plays two roles here.
First $V$ is $V$ itself, and secondly $V=V'$.   For the first case, we
use the same basis $\{\phi_i\}$ for $V$ as the domain of $A$ and $V$
as the range of $A$.  In this case, the matrix representation of $A$
is the matrix 
\begin{equation}
  \label{matrixrep0}
\hat A\in \mathbb R^{n\times n}\mbox{ satisfying }
(A\phi_1, \ldots, A\phi_n)  =(\phi_1, \ldots, \phi_n)\hat A. 
\end{equation}
In the second case, we use the basis $\{\phi_i\}$ for $V$ as the
domain space  of $A$, but use the dual basis $\{\psi_i\}$ for $V'=V$ as the
range space  of $A$. 
In this case, the matrix representation of $A$
is the matrix 
\begin{equation}
\label{matrixrep}
\tilde A\in \mathbb R^{n\times n}\mbox{ satisfying }
(A\phi_1, \ldots, A\phi_n)  =(\psi_1, \ldots, \psi_n)\tilde A. 
\end{equation}
It is easy to see that
\begin{equation}
  \label{tildeA}
\tilde A=\bigg((A\phi_j,\phi_i)  \bigg)
\end{equation}
and
\begin{equation}
  \label{hatA}
\tilde A=M\hat A, \quad M=\bigg((\phi_j,\phi_i)  \bigg).
\end{equation}
The matrix $\tilde A$ in \eqref{tildeA} is often called the stiffness
matrix of $A$ and the matrix $M$ in \eqref{hatA} is called the mass
matrix. 

In the early multigrid literature, a discrete inner product equivalent to the $L^2$ inner
product was often introduced for finite element spaces so that the corresponding
mass matrix becomes diagonal.  But if we view the underlying finite
element operator as in \eqref{AVV} in a slightly different way:
\begin{equation}
  \label{AVVdual}
A: V\mapsto V',  
\end{equation}
and we will then see easily that the introduction of the discrete
$L^2$ inner product is not necessary.

If $V=\mathbb R^n$ and we choose the canonical basis $\{e_i\}$ for
$V$, we would not encounter the mass matrix problem as in the
functional space case since in this case $\{e_i\}$ is also the dual
basis of $V'$.    This is certainly convenient, but such a
convenience tends to hide some subtle but important difference between
various vectors and matrices in a given problem and the objects (functions)
that they represent. 

Given a matrix $A\in \mathbb R^{n\times n}$, we can either view it as
\begin{equation}
  \label{A0}
A: \mathbb R^{n}\mapsto \mathbb R^{n},
\end{equation}
or
\begin{equation}
  \label{A1}
A: \mathbb R^{n}\mapsto (\mathbb R^{n})'.  
\end{equation}
As it turns out, when $A$ is obtained from the discretization of
partial differential equations, \eqref{A1} is more informative than
\eqref{A0}.  Hence we write a matrix equation
\begin{equation}
  \label{Axb}
Ax=b  
\end{equation}
it is sometimes helpful to view that $x$ and $b$ live in two ``different''
spaces:
\begin{equation}
  \label{xb}
x\in   \mathbb R^{n} \mbox{ and } b\in (\mathbb R^{n})'.  
\end{equation}

\subsection{Eigenvalues and eigenvectors}  Let us discuss briefly on
eigenvalues and eigenvectors for symmetric operator $T: V\mapsto V$.
If $(\lambda, \phi)$ is an eigen-pair of $T$, 
$$
T\phi=\lambda \phi.
$$
Then it is easy to see that $(\lambda, \tilde\phi)$ is an eigenpair of
the matrix representation $\tilde T$ of $T$:
$$
\tilde T\tilde\phi=\lambda\tilde \phi. 
$$
Here $\tilde\phi=\in \mathbb R^n$ is the vector
representation of $\phi$:
$$
\phi=(\phi_1, \ldots,\phi_n)\tilde\phi.
$$
We note that for an operator $A$ defined in \eqref{AVV}, we need to be
cautious when we talk about eigenvalues of $A$.  Although we identity
$V'=V$ through \eqref{L2dual}, $V'$ and $V$ play two different roles
and thus $A$ is essentially a mapping between two ``different'' spaces
$V$ and $V'$ and the spectrum of $A$ should be defined carefully.
But if we consider a symmetric operator
\begin{equation}
  \label{RVV}
R: V'\mapsto V.  
\end{equation}
Then $RA: V\mapsto V$ is an operator that is symmetric with respect to
$A$-inner product.  In this case if $(\lambda, \phi)$ is an eigenpair
of $RA$, then $(\lambda, \tilde \phi)$ is an eigenpair of $\tilde
R\tilde A$ (that is equal to the matrix representation of $RA$)

If we consider a trivial identification operator 
$$
\jmath:  V'\mapsto V \mbox{ such that } \jmath
\psi_i=\psi_i,\quad\forall i.
$$
Namely $\jmath v=v$ for all $v\in V'=V$.  It is easy to see that the
matrix representation of $\jmath$ is the inverse of the mass matrix
$M=((\phi_j,\phi_i))$, namely
$$
\tilde \jmath=M^{-1}.
$$
Using this identification operator, the operator $\jmath A: V\mapsto
V$ is a symmetric operator from $V$ to $V$.  We can then
talk about its spectrum.  For example, if $(\lambda, \phi)$ is an
eigen-pair of $\jmath A$, then $(\lambda, \tilde\phi)$ satisfies
$$
\tilde A\tilde \phi=\lambda M\tilde \phi.
$$
This is often the generalized eigenvalue problem appearing in finite
element analysis. 

Although for all $v\in V$, $\jmath Av=Av$ because of the
identification introduced above, $\jmath A$ and $A$ are, strictly
speaking, two different operators: they have two different ranges
and their matrix representations are different.

The discussions above, although simple, may sound a little bit
confusing at first glance, but an unambiguous understanding and
clarification of these concepts and the underlying subtleties will
be helpful for the presentation of algebraic multigrid methods in the
rest of this article. For more detailed discussions on relevant
topics, we refer to \cite{XuMSC-Notes}.

\subsection{Bibliographical notes}
For a general reading on the basic linear algebra materials used here,
we refer to~\cite{1974HalmosP-aa,1992XuJ-aa,XuMSC-Notes}.  In
particular, for a more detailed discussion related to dual spaces and matrix
representations, we refer to~\cite{1992XuJ-aa,XuMSC-Notes}.

\section{Basic iterative methods}
\label{sec:iterate}
We now consider linear iterative methods for solving \eqref{axb}.  We
will focus on two most commonly used algorithms, namely Jacobi and
Gauss-Seidel methods.  Let us first give a brief introduction of
linear iterative methods in a more general setting.  Recall the basic problem under consideration: Given a finite
dimensional vector space $V$ equipped with an inner product
$(\cdot,\cdot)$, we consider
\begin{equation}
\label{Auf}
Au=f,   
\end{equation}
where $ A: V\mapsto V' $ is symmetric positive definite (SPD)
and $V'$ is the dual of $V$.   As mentioned in \S\ref{s:duals}, we will identify $V'=V$ through an inner product $(\cdot, \cdot)$.

\subsection{Basic iterative methods}
A general linear iterative method for solving \eqref{Auf} can be
written as follows: given $u^0\in V$,
\begin{equation}
\label{eq:iterate}
u^{m}= u^{m-1}+B(f-Au^{m-1}), \quad m=1,2,\cdots ,
\end{equation}
where  $B: V'\mapsto V$ is a linear operator which can be thought of
as an approximate inverse of $A$.

Sometimes it is more desirable that the iterator $B$ is symmetric.  If
$B$ is not symmetric, there is a natural way to symmetrize it.
Consider the following iteration
\begin{equation}
\label{symmetrized-iteration}  
\left\{
\begin{array}{lcl}
u^{m-1/2} &=& u^{m-1}+ B(f-Au^{m-1}),\\
u^{m} &=& u^{m-1/2}+ B'(f-Au^{m-1/2}).
\end{array}
\right.
\end{equation}
The symmetrized iteration \eqref{symmetrized-iteration} can be written as
\begin{equation}
\label{sym-iterate}
u^{m}= u^{m-1}+\bar B(f-Au^{m-1}), \quad m=1,2,\cdots 
\end{equation}
where
\begin{equation}\label{symm}
\bar B=B'+B-B'AB,
\end{equation}
which satisfies
\begin{equation}\label{symm1}
I-\bar BA=(I-BA)^*(I-BA).
\end{equation}
Obviously, $\rho(I-\bar BA)<0$ $\Longleftrightarrow \bar B>0$
$\Longleftrightarrow$ $G\equiv (B')^{-1}+B^{-1}-A>0$.

\begin{theorem}\label{thm:iterate-converge} The following results hold
\begin{enumerate}
\item \eqref{symmetrized-iteration} converges $\Longleftrightarrow$
$G>0 \Longrightarrow $ \eqref{eq:iterate} converges.   Furthermore
\begin{equation}
  \label{errA}
\|I-BA\|_A^2
=\lambda_{\max}(I-\bar BA)
=1-\left(\sup_{\|v\|_A=1}(\bar B^{-1}v,v)\right)^{-1}
\end{equation}
\item If $B'=B$, $G>0 \Longleftrightarrow$ \eqref{eq:iterate}
converges and, with $\eta=\lambda_{\min}(G)$, 
\begin{equation}
  \label{BB}
  \frac{2\eta}{\eta+1}(Bv,v)\le (\bar Bv,v)\le
  2(Bv,v), \quad v\in V. 
\end{equation}
\end{enumerate}
\end{theorem}

\exercise Generalize Theorem \ref{thm:iterate-converge} to the case
that $A$ is SSPD.

\subsection{Jacobi and Gauss-Seidel methods}  For
$A=(a_{ij})\in\mathbb R^{n\times n}$, we write 
$$
A=D+L+U,
$$
where $D$ is the diagonal of $A$, $L$ and $U$ are the strict lower and
upper triangular parts of $A$ respectively.

Given $\omega>0$, 
the (modified) Jacobi method can be written as
\eqref{eq:iterate} with 
$$
B=\omega D^{-1}=(\omega^{-1} D)^{-1},
$$
and the resulting algorithm is as follows:

\begin{algorithm}\caption{Modified Jacobi}
$$
    \mbox{For } i=1:n,\quad
x_i^{m}=x_i^{m-1}+\omega a_{ii}^{-1}\left(b_i-\sum_{j=1}^{n}a_{ij}x_j^{m-1}\right).
$$
\end{algorithm}

The (modified) Gauss-Seidel method can be written as
\eqref{eq:iterate} with 
$$
B=(\omega^{-1} D+L)^{-1}
$$
and the resulting algorithm is as follows:

\begin{algorithm}\caption{Modified Gauss-Seidel Method}
\label{alg:modifiedGS}
    $$\mbox{For } i=1:n,\quad 
    x_i^{m}= x_i^{m-1}+\omega a_{ii}^{-1}\left(b_i-\sum_{j=1}^{i-1}a_{ij}x_j^{m}
-\sum_{j=i}^{n}a_{ij}x_j^{m-1}\right).$$
\end{algorithm}

The following result, which follows easily from Theorem
\ref{thm:iterate-converge},  is well-known.
\begin{theorem}
The modified Jacobi method converges if and only 
\begin{equation}
  \label{omega-J}
0<\omega<\frac{2}{\rho(D^{-1}A)},   
\end{equation}
and the modified Gauss-Seidel method converges if and only if 
\begin{equation}
\label{omega-GS}
0<\omega<2.
\end{equation}
\end{theorem}
In practice, it is often easy to properly choose $\omega$ to satisfy
\eqref{omega-J} so that the modified Jacobi method is guaranteed to
converge.  In the rest of this paper, we may always assume that such
a choice of $\omega$ is made.  For Gauss-Seidel method, we
will always choose $\omega=1$ (optimal SOR is not usually used in multigrid method).
The Jacobi and Gauss-Seidel methods together with their convergence
theory can be extended to block-matrices in a straightforward
fashion.

\subsection{The method of subspace corrections}\label{sec:msc}
We consider a sequence of spaces $V_{1},\ldots,V_{J}$.  
These spaces, which will be known as {\it auxiliary spaces}, are not
necessarily subspaces of $V$, but each of them is related to the
original space $V$ by a linear operator
\begin{equation}
  \label{Pai-k}
\Pi_k: V_k\mapsto V.   
\end{equation}
Our very basic assumption is that the following decomposition holds:
\begin{equation}
  \label{aux-decomp}
V=\sum_{i=1}^J\Pi_iV_i. 
\end{equation}
This means that for any $v\in V$, there exists $v_i\in V_i$ (which may
not be unique) such that 
\begin{equation}
  \label{aux-decomp0}
v=\sum_{i=1}^J\Pi_iv_i.   
\end{equation}
Furthermore, we assume that each $V_i$ is equipped with an energy inner product
$a_i(\cdot,\cdot)$. 
We define
$$
A_i:  V_i\mapsto V_i',
$$ 
by
$$
(A_iu_i,v_i)=a_i(u_i,v_i), \quad u_i,v_i\in V_i. 
$$
Let $\Pi_i':  V'\mapsto V_i'$ be the adjoint of $\Pi_i$:
$$
(\Pi_i'f, v_i)=(f,\Pi_i v_i), \quad f\in V',  v_i\in V_i.
$$
Let $P_i=\Pi_i^*: V\mapsto V_i$ be the adjoint of $\Pi_i$ with respect
to the A-inner products:
$$
(P_iu,v_i)_{A_i}=(u,\Pi_iv_i)_{A}, u\in V, v_i\in V_i.
$$

The following identity holds
\begin{equation}
  \label{PiAAP}
  \Pi_i'A=A_iP_i.
\end{equation}

If $u$ is the solution of \Rf{Auf}, by \eqref{PiAAP}, we have
\begin{equation}
\label{3.3a}A_iu_i=f_i,
\end{equation}
where 
$$
u_i=P_iu, \quad f_i=\Pi_i'f.
$$
This equation may be regarded as the restriction of \Rf{Auf} to
$V_i$. 
We assume that each such $A_i$ has an approximate inverse or preconditioner:
\begin{equation}
  \label{Ri}
R_i: V_i'\mapsto V_i.   
\end{equation}

The {\it parallel subspace correction} (PSC in short) method is
\eqref{eq:iterate} with $B=B_{psc}$ given by 
\begin{equation}
  \label{PSC}
B_{psc}=\sum_{i=1}^J\Pi_iR_i\Pi_i'.
\end{equation}

The {\it successive subspace correction} (SSC in short) method is
defined as:

\begin{algorithm}\caption{Successive Subspace Correction Method}\label{alg:MSC}  Given $u^0\in V$, for any $m=1,2,\ldots$, 
    \begin{enumerate}[1.]
  \item $v\leftarrow u^{m-1}$
\item $v\leftarrow v+\Pi_iR_i\Pi_i'(f-Av)$,   $i=1,2,\ldots J$, 
  \item $u^m\leftarrow v$
  \end{enumerate}
\end{algorithm} 

The Algorithm \ref{alg:MSC} is equivalent to \eqref{eq:iterate} with
$B=B_{ssc}$ given by 
\begin{equation}
  \label{ssc}
I-B_{ssc}A= (I-T_J)(I-T_{J-1})\ldots(I-T_1),
\end{equation}
where
\begin{equation}
  \label{Ti}
T_i  =\Pi_iR_i\Pi_i'A =\Pi_iR_iA_iP_i.
\end{equation}

\begin{theorem} \label{thm:add}
Assume that all $R_k$ are SPD.  Then
\begin{equation}\label{eq:add}
(B_{psc}^{-1} v,v)  = \min_{\sum_i \Pi_iv_i=v}\sum_{k=1}^J(R_k^{-1}v_k, v_k)
\end{equation}
with the unique minimizer given by 
\begin{equation}
  v_k^*=R_k\Pi_k'B_{psc}^{-1}v.
\end{equation}
\end{theorem}

\begin{theorem}\label{thm:c0}
Under the assumptions given above the following identity holds:
\begin{eqnarray}
  \label{identity}
\|I-B_{ssc}A\|_A^2 
&= &\|(I-T_J)(I-T_{J-1})\ldots(I-T_1)\|_A^2 \nonumber\\
&= &1-\frac{1}{1+c_0}\label{identity0}\\
&= &1-\frac{1}{c_1}\label{identity1}.
\end{eqnarray}
Here 
$$
c_0=\sup_{\|v\|_A=1}c_0(v),  \quad
c_1=\sup_{\|v\|_A=1}c_1(v) =1+c_0,
$$
and, with $w_i= (I-T_i^{-1})\Pi_iv_i + \sum_{j=i+1}^J \Pi_jv_j$ 
\begin{equation}\label{defc0}
c_0(v)=\inf_{\sum_i \Pi_iv_i = v}
\sum_{i=1}^J (T_i\overline{T}_i^{-1}T_i^*w_i,w_i)_A,
\end{equation}
and 
\begin{equation}\label{defc1}
    c_1(v)=(\overline{B}_{ssc}^{-1} v,v) =\inf_{\sum_i \Pi_iv_i = v} (\overline{T}_i^{-1}(\overline{T}_iT_i^{-1}\Pi_iv_i+T_i^*w_i), (\overline{T}_iT_i^{-1}\Pi_iv_i+T_i^*w_i))_A.
\end{equation}
In particular, if $R_i=A_i^{-1}$, then 
\begin{equation}\label{c0v} 
c_0(v)
    = \inf_{\sum_i \Pi_iv_i = v}\sum_{i=1}^J\|P_i\sum_{j=i+1}^J \Pi_jv_j\|_{A_i}^2,
\end{equation}
and
\begin{equation}\label{eq:mult} 
c_1(v)
= \inf_{\sum_i \Pi_iv_i = v}\sum_{i=1}^J\|P_i\sum_{j=i}^J \Pi_jv_j\|_{A_i}^2.
\end{equation}
\end{theorem}
\begin{lemma}\label{lem:psc-ssc}
If $R_k=A_k^{-1}$ for all $k$, and $V_k$ are subspaces of $V$, then 
\begin{equation}
  \label{eq:3}
\frac14 
({B_{psc}^{-1}}v,v)\le ({\bar B_{ssc}^{-1}}v, v)\le c^*
(B_{psc}^{-1}v,v), \quad
v\in V.  
\end{equation} 
where
\[
c^*=\max_{1\le k \le M} [N(k)]^2
\mbox{ with } 
N(k) = \left\{j\in \{1,\ldots,J\} \;\big|\; 
\;\mbox{and}\;V_j\cap V_k\neq\{0\} \right\}.
\]
\end{lemma}
\begin{proof} 
Given $v=\sum_{i=1}^J v_i$, with $v_i\in V_i$. 
It follows that
$$
\|v\|_A^2
=\sum_{k,j=1}^J(v_k,v_j)_A
=\sum_{k=1}^J(v_k,v_k)_A+2\sum_{j>k}^J(v_k,v_j)_A
=-\sum_{k=1}^J(v_k,v_k)_A+2\sum_{j\ge k}^J(v_k,v_j)_A.
$$
Thus
\begin{eqnarray*}
\sum_{k=1}^{J}\|v_k\|_A^2
&\le &
2\sum_{k=1}^{J}(v_k,\sum_{j = k}^J v_j)_A
  =    2\sum_{k=1}^{J}(v_k, P_k\sum_{j = k}^J v_j)_A\\
&\le&
2(\sum_{k=1}^{J}\|P_k\sum_{j = k}^J v_j\|_A^2)^{1/2}
(\sum_{k=1}^{J}\|v_k\|_A^2)^{1/2}.
  \end{eqnarray*}
Consequently
$$
\sum_{k=1}^{J}\|v_k\|_A^2
\le 4\sum_{k=1}^{J}\|P_k\sum_{j = k}^J v_j\|_A^2.
$$
By \eqref{eq:add}, \eqref{eq:mult} and \eqref{defc1}, we have
$$
(B_{psc}^{-1}v,v)\le 4c_1(v) =4(\bar B_{ssc}^{-1}v,v).
$$
The upper bound also follows easily. From $\|P_k\|_A= 1$ and 
the Schwarz inequality, we obtain
\begin{eqnarray*}
\sum_{k=1}^{J}\|P_k\sum_{j = k}^J v_j\|_A^2
&=& 
\sum_{k=1}^{J}\|P_k\sum_{j \in N(k); j\ge k} v_j\|_A^2\\
& \le & 
\sum_{k=1}^{J}\|\sum_{j \in N(k); j\ge k} v_j\|_A^2\le 
\sum_{k=1}^{J}N(k)\sum_{j \in N(k); j\ge k} \|v_j\|_A^2 \\
&\le& 
\sqrt{c^*} \sum_{k=1}^{J}\sum_{j \in N(k); j\ge k} \|v_j\|_A^2 
\le 
c^* \sum_{k=1}^{J}\|v_k\|_A^2. 
\end{eqnarray*}
The proof is concluded by taking the infimum over all decompositions
on both sides and applying~\eqref{eq:mult}. 
\end{proof}

\begin{remark}
  We would like to point that the estimate in \eqref{eq:3} holds for
  anisotropic and jump coefficient problems, and the constant $c^*$ only depends on the
  ``topology'' of the overlaps between the subspaces and does not
  depend on other ingredients and properties.
\end{remark}

The Jacobi and Gauss-Seidel method can be interpreted as PSC and SSC
based on the decomposition
$$
\mathbb R^n=\sum_{i=1}^n{\rm span} \{e_i\} 
$$
with exact subspace solves
such that
$$
B_{psc}=D^{-1}, \quad B_{ssc}=(D+L)^{-1}, \quad
\bar B_{ssc}=(D+U)^{-1}D (D+L)^{-1}. 
$$
By Theorem \ref{thm:c0}
$$
c_0=\sup_{\|v\|_A=1}(D^{-1}Uv, Uv), \quad
c_1=\sup_{\|v\|_A=1}(D^{-1}(D+U)v, (D+U)v). 
$$
By Lemma \ref{lem:psc-ssc}
\begin{equation}
  \label{J-GS}
{1\over 4} (Dv,v) \le (D^{-1}(D+U)v, (D+U)v)\le c_*(Dv,v),\quad v\in \mathbb R^n   
\end{equation}

We note that
$$
c_1=
\sup_{v\in V}\frac{(\bar B_{ssc}^{-1}v, v)}{\|v\|_A^2}
\le \sigma \sup_{v\in V}\frac{\|v\|_D^2}{\|v\|_A^2}
$$
where
\begin{equation}
  \label{sigma}
    \sigma = \sup_{v\in V} \frac{\|v\|_{\bar{B}_{ssc}^{-1}}^2}{(D v,v)},  
\end{equation}
In the above presentation, most results are for SPD problems.  We
would like to point out that all of these results can be extended a
more general class of problems: namely symmetric, semi-positive
definite problems.  When $A$ is a matrix, we further assume that all
the diagonals of $A$ are non-zero and hence positive.  When the method
of subspace correction is used for a more general symmetric,
semi-positive definite operator $A$, we further assume that each $A_i$
is SPD.  We should use the acronym SSPD to denote matrices or
operators that satisfy the aforementioned properties.

\subsection{Bibliographical notes} 
The general notion of subspace corrections by means of space
decompositions was described in Xu \cite{1992XuJ-aa} based on
\cite{1991BrambleJ_PasciakJ_WangJ_XuJ-ac,1991BrambleJ_PasciakJ_WangJ_XuJ-aa}.
It is an abstract point of view encompassing the theory and practice
of a large class of iterative algorithms such as multigrid and domain
decomposition methods. In the last two decades a great deal of effort has
been put into the investigation of the theoretical and practical
issues related to these methods.  General results, applicable in many
cases, in the theory of additive and multiplicative methods in Hilbert
space is found in~\cite{Xu.J;Zikatanov.L.2002a}.  For a literature
review and basic results we refer the reader to some monographs and
survey articles: \cite{1985HackbuschW-aa}, \cite{1993BrambleJ-aa},
\cite{2008VassilevskiP-aa}, \cite{1989XuJ-aa,1997XuJ-aa},
\cite{xu1998some},\cite{1993YserentantH-aa},
\cite{2005ToselliA_WidlundO-aa}, \cite{1995GriebelM_OswaldP-aa},
\cite{1996SmithB_BjorstadP_GroppW-aa}.  For detailed studies of
classical iterative methods, we refer to the
monographs~\cite{1971YoungD-aa}, \cite{1994HackbuschW-aa},
\cite{2000VargaR-aa}, \cite{2003SaadY-aa}.

We note that in this section we have considered SSPD matrices, and
according to
\cite{2008LeeY_WuJ_XuJ_ZikatanovL-aa,2007LeeY_WuJ_XuJ_ZikatanovL-aa,2014Ayuso-de-DiosB_BrezziF_MariniL_XuJ_ZikatanovL-aa}
all the results in this section are valid for SSPD problems with the
semi-norm. Relations between auxiliary space method and the subspace
correction methods is drawn in~\cite{2011ChenL-aa}.  In the classical
multigrid
literature~\cite{1stAMG,Brandt.A;McCormick.S;Ruge.J.1985a,Ruge.J;Stuben.K.1987a,Trottenberg.U;Oosterlee.C;Schuller.A.2001a}
the notions of algebraically smooth (low) frequencies and algebraic
high frequencies play an important role. They are also instrumental in
the design of new AMG methods.  As indicated by the convergence
estimates, for a given a smoother, the desirable coarse spaces should
capture or approximate well the lower end of the spectrum of the
relaxed matrix $\bar R A$ or $D^{-1}A$.  This is usually referred to
as \emph{near-null space}
\cite{2015TreisterE_YavnehI-aa,2011LaiJ_OlsonL-aa,2009XuJ-aa},
\cite{2006BrezinaM_FalgoutR_MacLachlanS_ManteuffelT_McCormickS_RugeJ-aa}.

\section{Abstract multigrid methods and 2-level theory}\label{s:2Level}
In this section, we will present algebraic multigrid methods in an
abstract setting.  The acronym ``AMG'' for Algebraic Multi-Grid can
also be used to stand for Abstract Multi-Grid.

Our focus will be on two level methods.  In view of
algorithmic design, the extension of two-level to multi-level is
straightforward: a general multilevel V-cycle algorithm can be
obtained by recursively applying a two-level algorithm.  But the
extension of a two-level convergence theory to multi-level case can be
highly nontrivial.

We will only consider SSPD problems as described in \S\ref{sec:iterate}.
As is done in most literature, the designing principle of an AMG is to
optimize the choice of coarse space with a given a smoother.  The most
commonly used smoothers are the Gauss-Seidel method and (modified or
scaled) Jacobi method.  As these smoothers are qualitatively
convergent as an iterative method itself, the resulting AMG method is
always qualitatively convergent.  The task of our AMG convergence 
theory is to make sure such a convergence is also quantitatively
fast.  In particular, for system arising from the discretization of
partial differential equations, we hope that our AMG method converges
uniformly with respect to the size of the problem and/or some crucial
parameters from the underlying PDEs.   We sometimes call such a
convergence ``uniformly convergent'' or ``uniform convergence''.

As it turns out, we are often able to establish such a uniform
convergence for two-level AMG, but very rarely we can extend such a
uniform convergence result to multi-level case.  For second order
elliptic boundary value problems, multilevel convergence are very
well-understood for geometric multigrid methods.  But a rigorous
multilevel convergence theory for an AMG without using geometric
information is still a widely open problem.

We will mainly focus on two-level convergence theory on
AMG methods in this section and also in the rest of this paper. 

\subsection{A two-level method}

\newcommand{\pr}{P} 

A two-level method typically consists of the following components:
\begin{enumerate}[1.]
\item A smoother $R: V'\mapsto V$;
\item A coarse space $V_c$,  which may or may not be a subspace of
  $V$ but linked with $V$ via a prolongation operator:
$$
\pr:  V_c\mapsto V. 
$$
\item A coarse space solver $B_c: V_c'\mapsto V_c$.
\end{enumerate}

In the discussion below we need the following inner product
\begin{equation}\label{eq:norm-star}
\Tscalar{u}{v} = (\overline{T}^{-1} u,v)_A=( \bar R^{-1}u,v), \quad \overline{T} = \overline{R} A,
\end{equation}
and the accompanying norm $\Tnorm{\cdot}$.  Here we recall that the
definition of $\bar R$ is analogous to that in \eqref{symm}.

We always assume that $\bar R$ is SPD and hence the smoother $R$ is always convergent. Further more, 
\begin{equation}
    \|v\|_A^2\le \|v\|_{\bar R^{-1}}^2.
\end{equation}

The
restriction of \eqref{Auf} is then
\begin{equation}
  \label{coarse:Au=f}
A_cu_c=f_c,   
\end{equation}
where 
$$
A_c=\pr'A\pr, \quad f_c=\pr'f. 
$$

The coarse space solver $B_c$ is often chosen to be the exact solver,
namely $B_c=A_c^{-1}$, for analysis, but in a multilevel setting,
$B_c$ is recursively defined and it is an approximation for
$A_c^{-1}$. We distinguish these two different cases in choosing
$B_c$:
\begin{eqnarray}
&& \mbox{\emph{exact} two level method if $B_c=A_c^{-1}$.}\label{eq:exact}\\ 
&& \mbox{\emph{inexact} two level method if $B_c\neq A^{-1}_c$}.  \label{eq:inexact}
\end{eqnarray}
In the case that $A$ is semi-definite, we use $N$ to denote the kernel of $A$ and we always assume that $N\subset V_c$. Let 
\begin{equation}
    W:=N^{\perp}, \text{ and } W_c:=V_c\cap W,
\end{equation}
where the orthogonality is understood with respect to the
$(\cdot,\cdot)_{\bar{R}^{-1}}$ inner product. Let $Q_1: V \mapsto W$ be the orthogonal projection respect to $(\cdot,\cdot)_{\bar{R}^{-1}}$ inner product
\begin{equation}\label{e:q0}
(Q_1v, w)_{\bar{R}^{-1}} = (v,w)_{\bar{R}^{-1}}, \quad\mbox{for
  all}\quad v\in V, w\in W. 
\end{equation}
$A_c$ is 
semi-definite on $V_c$ but invertible on $W_c$. We denote the restriction of $A_c$ on $W_c$ by $\hat A_c$, and define the psudo-inverse of $A_c$
\begin{equation}
    A_c^{\dag}: = Q_1'\hat A_c^{-1} Q_1.
\end{equation}
With a slight abuse of
notation, we will still use $A_c^{-1}$ to denote the psudo-inverse of
$A_c$, namely
$$
A_c^{-1}=A_c^\dag.
$$
We will use similar notation for psudo-inverse of other relevant singular operators and
matrices in the rest of the paper. 

We choose to define an AMG algorithm in terms of an operator $B:
V'\mapsto V$, which can be considered as an approximate inverse or a
preconditioner of $A$.   A typical two level MG method is as
follows. 
\begin{algorithm}[H]
\caption{A two level MG method}\label{alg:two-level}
Given $g\in V'$ the action  $Bg$  is defined via the following three steps
\begin{enumerate}
\item Coarse grid correction: $w=\pr B_c \pr' g$.
\item Post-smoothing: $Bg:= w +R(g-Aw)$.
\end{enumerate}
\end{algorithm}

In the rest of this section, we take $B_c=A_c^{-1}$.

There are usually two different (and mathematically equivalent) ways
to choose $V_c$, $V$, $P$ and $R$.    The first one, known as operator
version, is such that 
$$
V_c\subset V. 
$$
In this case, $P=\imath_c$ where 
\begin{equation}
  \label{inclusion}
\imath_c: V_c\to V,
\end{equation}
is the natural inclusion of $V_c$ into $V$.  In the application to
finite element discretization for 2nd order elliptic boundary value
problems, $V_c$ and $V$ are just the finite element subspaces of
$H^1(\Omega)$.
This type of notation is convenient for analysis. But this is not the
algorithm that can be directly used for implementation. 

The second one, known as matrix version, is such that
$$
V_c=\mathbb R^{n_c} \mbox{ and } V=\mathbb R^n,
$$
and 
\begin{equation}
  \label{prolong}
P: \mathbb R^{n_c} \mapsto \mathbb R^n,
\end{equation}
is the prolongation matrix.   

These two different set of notations are related through the use of
basis functions of 
$\{\phi_i^c\}_{i=1}^{n_c}\subset  V_c$ and $\{\phi_i:\}_{i=1}^{n}\subset  V$.   As noted earlier, the prolongation matrix $P$ given in
\eqref{defP}--\eqref{Pexpand} is simply the matrix representation of
$\imath_c$ given in \eqref{inclusion}, and we have
\begin{equation}
  \label{icP}
(\phi_i^c, \ldots, \phi_{n_c}^c)=(\phi_i, \ldots, \phi_n)P.
\end{equation}
The following observation is clear.
\begin{observation}
  Finding a coarse space $V_c\subset V$ is equivalent to finding a
  prolongation matrix $P$ in \eqref{prolong}.
\end{observation}

\begin{lemma}
    The error propagation operator for two-level AMG operator $E=I-BA$ is
    \begin{equation}\label{E_op}
        E=(I-RA)(I-\Pi_c), 
    \end{equation}
    where $\Pi_c= \imath_cA_c^{-1}\imath_c'A$, which is the $(\cdot,\cdot)_A$ orthogonal projection on $V_c$,  in matrix notation $\Pi_c=PA_c^{-1}P^TA$.
\end{lemma}

\subsection{An optimal two-level AMG theory}\label{2-level-theory}
The design of AMG method is to balance the interplay between smoother
$R$ and the coarse space $V_c$.  The design of most existing AMG is to first
fix a smoother, which is often given by Jacobi or Gauss-Seidel method (or
their combinations and variations), and then to optimize the choice of
coarse space.  This is the approach that we discuss mostly in
this paper.  But we also comment on a different approach by
first fixing the coarse space and then trying to optimize the choice
of the smoother.  It is also possible to try to make optimal choice of
smoother and coarse space simultaneously, but we will not address this
approach in this paper.

Let $Q_c:V\mapsto V_c $ 
be orthogonal
projection respect to $(\cdot,\cdot)_{\bar{R}^{-1}}$ inner product
\begin{equation}\label{e:qc}
(Q_cu,v_c)_{\bar{R}^{-1}} = (u,v_c)_{\bar{R}^{-1}}, \quad\mbox{for
  all}\quad v_c\in V_c. 
\end{equation}

By the definition of $W$ and $W_c$ we have that $(\cdot, \cdot)_A$ is an inner product on $W$,
$\|\cdot\|_A$ is a norm on $W$, and the projection
$\Pi_c: V\mapsto W_c$ is well defined:
\begin{equation}\label{e:77}
(\Pi_cu,v_c)_A = (u,v_c)_A, \quad \forall u\in V, v_c\in W_c.
\end{equation}

The two level convergence rate is obtained in the following theorem.

\begin{theorem}\label{thm:two-level-convergence}  
Assume that $N\subset V_c$. The convergence rate of an exact two level
AMG is given by 
\begin{equation}\label{eq:two-level-convergence}
\|E\|_A^2 = 1- \frac1{K(V_c)}, 
\end{equation}
where
\begin{equation}
  \label{KVc}
K(V_c)=\max_{v\in W}\frac{\Tnorm{(I-\Tproj)v}^2}{\|v\|^2_A}
=\max_{v\in W}\min_{v_c\in W_c}\frac{\Tnorm{v-v_c}^2}{\|v\|^2_A}.
\end{equation}
\end{theorem}

\begin{proof}
We notice that.
$$
\|(I-T)v\|_A^2=((I-\bar T)v,v)_A, \forall v\in V. 
$$
Then we have
\begin{eqnarray*}
    \|E\|_A^2  &=&  \max\limits_{w\in W}\frac{\|(I-T)(I-\Pi_c)w\|_A^2}{\|w\|_A^2}\\
    &= &\max\limits_{w\in W}\frac{((I-\overline{T})(I-\Pi_c)w,(I-\Pi_c) w)_A}{\|w\|_A^2}\\
    &= &1-\min\limits_{w\in W}\frac{(\overline{T}(I-\Pi_c)w, (I-\Pi_c) w)_A}{\|w\|_A^2}\\
    &= &1-\min\limits_{w\in W}\frac{(Q_1\overline{T}(I-\Pi_c)w, (I-\Pi_c)w)_A}{\|(I-\Pi_c)w\|_A^2+ \|\Pi_cw\|_A^2}\\
    & = & 1-\min\limits_{v\in W_c^{\perp_A}}\frac{(Q_1\overline{T}v,v)_A}{\|v\|_A^2} \\
    & = & 1-\min\limits_{v\in W_c^{\perp_A}}\frac{((I-\Pi_c)Q_1\overline{T}v,v)_A}{\|v\|_A^2} \\
    & = &1- \lambda_{\min{}}(X),
\end{eqnarray*}
where 
$$
X=(I-\Pi_c)Q_1\overline{T}: W_c^{\perp_A}\mapsto W_c^{\perp_A},
$$
and it is easy to see that $X$ is self-adjoint with respect to
$(\cdot,\cdot)_A$.

One key observation is that the inverse of $X$ can be explicitly
written as
$$
    Z = (Q_1\overline{T})^{-1}(I-\Tproj),
$$
    since by definition, we have, for any $u, v\in V$
    \begin{eqnarray*}
        (\Pi_cZ u, v)_A & =& ((Q_1\overline{T})^{-1}(I-\Tproj)u, \Pi_cv)_A = (\overline T(Q_1\overline T)^{-1}(I-\Tproj)u, \Pi_c v)_{\bar R^{-1}}\\
        &=& (Q_1\overline T(Q_1\overline T)^{-1}(I-\Tproj)u, \Pi_c v)_{\bar R^{-1}} =((I-\Tproj)u, \Pi_cv)_{\bar R^{-1}} = 0,
    \end{eqnarray*}
    which implies $\Pi_cZ=0$. Thus we have
$$
Z: W_c^{\perp_A}\mapsto W_c^{\perp_A},
$$
and furthermore
$$
XZ= (I-\Pi_c)(I-\Tproj) = I-\Pi_c=I \mbox{ on } W_c^{\perp_A}.
$$
Consequently $\lambda_{\min{}}(X) = \frac{1}{\lambda_{\max{}}(Z)}$.
Finally,
\begin{eqnarray*}
    \lambda_{\max{}}(Z) &=& \max_{v\in W_c^{\perp_A}} \frac{((Q_1\overline{T})^{-1}(I-\Tproj)v,v)_A}{(v,v)_A} = \max_{v\in W_c^{\perp_A}} \frac{(\overline T(Q_1\overline{T})^{-1}(I-\Tproj)v,v)_{\bar R^{-1}}}{(v,v)_A} \\
    &=& \max_{v\in W_c^{\perp_A}} \frac{(Q_1\overline T(Q_1\overline{T})^{-1}(I-\Tproj)v,v)_{\bar R^{-1}}}{(v,v)_A}= \max_{v\in W_c^{\perp_A}}
\frac{\Tscalar{(I-\Tproj)v}{v}}{(v,v)_A}\\
& = & \max_{v\in W_c^{\perp_A}}\frac{\Tnorm{(I-\Tproj)v}^2}{(v,v)_A} =  K(V_c). 
\end{eqnarray*}
The last identity holds because $I-\Tproj = (I-\Tproj)(I-\Pi_c)$ and we
can then take the maximum over all $v\in W$. 
This completes the proof. 
\end{proof}

Theorem~\ref{thm:two-level-convergence} can be stated as follows using the matrix representation introduced in \S\ref{s:duals}.
\begin{theorem}
    Assume that $P\in \mathbb{R}^{n\times n_c}$ and $N(\tilde A)\subset \operatorname{Range}(P)$. The convergence rate of an exact two level AMG is given by
    \begin{equation}
        \|E\|_A^2 = 1-\frac{1}{\tilde K(P)},
    \end{equation}
    where 
    \begin{equation}
        \tilde K(P)=\max_{v\in \mathbb{R}^n}\min_{v_c\in \mathbb{R}^{n_c}}\frac{\|v-Pv_c\|_{\tilde{\bar R}^{-1}}^2}{\|v\|_{\tilde A}^2}.
    \end{equation}
\end{theorem}

\begin{remark}
  The result in the theorem above can be viewed as follows. We note
  that $\overline{T}^{-1}(I-\Tproj)$ is a selfadjoint operator in $A$
  inner product. Hence, we immediately have
$$ 
K(V_c) = \|\overline{T}^{-1} (I-\Tproj)\|_A= \|(\overline{R}A)^{-1}
(I-\Tproj)\|_A.
$$
\end{remark}

In case of two subspaces, we have the following theorem giving the
precise convergence rate of the corresponding SSC method.
\begin{theorem}\label{thm:two-level-optimal} 
  Let $\{\mu_j, \zeta_j\}_{j=1}^{n}$ be the eigenpairs of
    $\bar T = \bar {R} A$. And let assume that $\{\zeta_j\}$ are orthogonal with respect to $(\cdot, \cdot)_{\bar R^{-1}}$. The convergence rate $\|E(V_c)\|_A$ is minimal
  for coarse space
\begin{equation}\label{another-vcopt}
    V_c^{\rm opt}=\operatorname{span}\{\zeta_j\}_{j=1}^{n_c}\in \argmin_{\dim V_c=n_c, N\subset V_c}K(V_c). 
\end{equation}
In this case, 
\begin{equation}
  \label{optimalrate}
    \|E\|_A^2=1-\mu_{n_c+1}.
\end{equation}
\end{theorem}

\begin{proof}
    By Theorem~\ref{thm:two-level-convergence}, we just need to maximize $\frac{1}{K(V_c)}$.
For any $v\in V_c^{\perp}$, where $\perp$ is with respect to the $\bar{R}^{-1}$ inner product, we have
\begin{equation}
\min_{v_c\in V_c}\|v-v_c\|^2_{\bar R^{-1}} = \|v\|^2_{\bar{R}^{-1}}.
\end{equation}
Then it follows that
\begin{equation*}
    \frac{1}{K(V_c)}=\min_{v\in V}\max_{v_c\in V_c}\frac{\|v\|^2_A}{\|v-v_c\|_{\bar R^{-1}}^2}\le \min_{v\in V_c^{\perp}}\max_{v_c\in V_c}\frac{\|v\|_A^2}{\|v-v_c\|_{\bar R^{-1}}^2}=\min_{v\in V_c^{\perp}}\frac{\|v\|_A^2}{\|v\|_{\bar R^{-1}}^2}.
\end{equation*}
    By the min-max principle (Theorem~\ref{thm:minmax}), we have
\begin{equation*}
    \max_{\dim V_c=n_c}\frac{1}{K(V_c)}\le\max_{\dim V_c=n_c}\min_{v\in V_c^{\perp}}\frac{\|v\|_A^2}{\|v\|_{\bar R^{-1}}^2}=\mu_{n_c+1}.
\end{equation*}
    On the other hand, if we choose $V_c^{\rm opt}=\operatorname{span}\{\zeta_j\}_{j=1}^{n_c}$, it is easy to compute that  $K(V_c^{\rm opt})=\frac{1}{\mu_{n_c+1}}$. So we have
\begin{equation*}
    \max_{\dim V_c=n_c}\frac{1}{K(V_c)}=\mu_{n_c+1},
\end{equation*}
with optimal coarse space 
\begin{equation*}
    V_c^{\rm opt}= \operatorname{span}\{\zeta_j\}_{j=1}^{n_c}.
\end{equation*}
\end{proof}

Using the matrix representation introduced in \S\ref{s:duals}, we state the matrix version of Theorem~\ref{thm:two-level-optimal} below. For simplicity, with an abuse of notation, we still use $A$ to denote the matrix representation of operator $A$.
\begin{theorem} 
  Let $\{\mu_j, \zeta_j\}_{j=1}^{n}$ be the eigenpairs of
    $\bar T = \bar{R} {A}$. And let assume that $\{\zeta_j\}$ are orthogonal with respect to $(\cdot, \cdot)_{\bar R^{-1}}$. The convergence rate $\|E(P)\|_{A}$ is minimal
  for $P$ such that  
\begin{equation}
    \operatorname{Range}(P)=    \operatorname{Range}(P^{\rm opt})
    \end{equation}
where
    \begin{equation}
      \label{optP}
P^{\rm opt}=(\zeta_1,\ldots\zeta_{n_c})
\end{equation}
In this case, 
\begin{equation}
\|E\|_A^2=1-\mu_{n_c+1}  
\end{equation}
\end{theorem}

The following theorem is important in motivating most AMG algorithms.
\begin{theorem}\label{t:trace-min}
Given $\eta>0$, let $\mathcal X_{\eta}$ be defined as 
\begin{equation}\label{chi}
\mathcal X_\eta=\bigg\{
  P\in \mathbb R^{n\times n_c}: (Pv,Pv)_{\bar R^{-1}}\ge\eta (v,v),\;\;
    v\in \mathbb R^{n_c}\bigg\},
\end{equation}
Then,  with $P^{\rm opt}$ given by \eqref{optP}, we have $P\in \argmin_{Q\in \mathcal{X}_\eta}\operatorname{trace}(Q^TAQ)$ if 
$$
    P\in \mathcal X_{\eta} \text{ and } \range(P)=\range(P^{\rm opt})).
$$
\end{theorem}

Since the eigenvalues of $\bar RA$ are expensive to compute, the
practical value of Theorem \ref{thm:two-level-optimal} is limited.
But it provides useful guidance in the design practical AMG method.

For finite element discretizations we can use the Weyl's Law combined with Theorem~\ref{thm:two-level-optimal} to prove an estimate 
on the convergence rate of a two-grid method with optimal coarse space. 
\begin{corollary} Let the assumptions of the discrete Weyl's law
  (Theorem~\ref{thm:WeylFE}) hold and the smoother $\bar{R}$ be 
  spectrally equivalent to the diagonal of the stiffness matrix 
  $A$. Let $\gamma > 0 $ be such that $\gamma n\le n_c<n$. Then, for
  the optimal coarse space, we have the estimate,
\[
\mu_{n_c+1}\ge \delta_0,\quad \mbox{and},\quad \|E\|_A^2\le 1-\delta_0.
\]
where $\delta_0\in(0,1)$ only depends on $\gamma$ and
the constants $\gamma_0$ and $\gamma_1$ in~\eqref{discrete-Weyl}.
\end{corollary}
\begin{proof}
  By the assumptions in Theorem~\ref{thm:WeylFE} and the fact that
  $\bar{R}$ is spectrally equivalent to the diagonal of $A$ we have
\begin{equation}
    |w|_1^2\eqqsim \|w\|_A^2, \quad \|w\|_0^2\eqqsim h^2\|w\|_{\bar R^{-1}}^2.
\end{equation}
Further, Lemma~\ref{lem:Weyl} and Theorem~\ref{thm:WeylFE} then show that
\begin{equation}
    \mu_{n_c+1}(\bar RA)=\min_{\substack{W\subset V_h\\ \dim{W}=k}}\max_{\substack{w\in W\\ \|w\|_{\bar R^{-1}}\ne 0}}\frac{\|w\|_A^2}{\|w\|_{\bar R^{-1}}^2}\eqqsim h^2 \lambda_{n_c+1} = O\left(\left(\frac{n_ch^d}{\operatorname{Vol}(\Omega)}\right)^{2/d}\right).
\end{equation}
The desired result follows immediately from
Theorem~\ref{thm:two-level-optimal} because
$\operatorname{Vol}(\Omega)\eqqsim h^dn$ and $\gamma n\le n_c<n$, which gives $\mu_{n_c+1} \eqqsim 1$. 
\end{proof}

\begin{remark}
Since the coarse space which minimizes the convergence rate is the
coarse space which minimizes also $K(V_c)$ and 
as a corollary we have the following equality
\[
K(V_c) = \frac{1}{1-\|E\|_A^2} \ge \frac{1}{\mu_{n_c+1}},  
\]
or
$$
\|E\|_A^2\ge 1-\mu_{n_c+1}. 
$$
\end{remark}

Theorem~\ref{thm:two-level-convergence} provides an explicit estimate
on the convergence of a two level method in terms of $K(V_c)$.  For a
given method, a smaller bound on $K(V_c)$ means a faster convergence
rate.  In particular, the two-level AMG method is uniformly convergent
if $K(V_c)$ is uniformly bounded with respect to mesh parameters.

\subsection{Quasi-optimal theories}\label{s:semidefinite-theory}
We now look at the necessary and sufficient condition for uniform
convergence of a two level method as proved in~\S\ref{2-level-theory} 
(see Theorem~\ref{thm:two-level-convergence}):
\begin{equation}\label{e:ns}
\min_{v_c\in V_c}\Tnorm{v-v_c}\le K(V_c)\|v\|_A^2. 
\end{equation}
Here $K(V_c)$ is the smallest constant for which~\eqref{e:ns} holds
for all $v\in V$.  The space $V_c$ which minimizes $K(V_c)$ is
$V_c^{\rm opt}$. Similar argument was used in the proof of
Theorem~\ref{thm:two-level-optimal} in \S\ref{2-level-theory}. We
generalize here the result to semidefinite $A$.

For a given smoother $R$, one basic strategy in the design of AMG is
to find a coarse space such that $K(V_c)$ is made as practically small
as possible.  There are many cases, however, in which the operator
$\bar{R}^{-1}$ in the definition of $K(V_c)$ is difficult to work
with.  

One commonly used approach is to replace $\bar{R}^{-1}$ by a
simpler but spectrally equivalent SPD operator.  More specifically, we
assume that $D: V\mapsto V'$ is an SPD operator such that
\begin{equation}\label{star-equiv-Y} c_D(Dv,v) \le
(\bar{R}^{-1} v,v) \le c^D( D v,v ), \quad
\forall v\in V,
\end{equation} Namely
\begin{equation}\label{star-equiv-norms} c_D \Dnorm{v}^2\le
\Tnorm{v}^2 \le c^D\Dnorm{v}^2, \quad \forall v\in V,
\end{equation} where
\[ \Dscalar{u}{v}=( Du,v), \quad \Dnorm{v}^2 =
\Dscalar{v}{v}.
\]

Example of such equivalent norms for Schwarz smoothers are given
in~\eqref{eq:add} and~\eqref{eq:mult}. As a rule, the norm defined by
$\bar{R}$ corresponding to the symmetric Gauss-Seidel method, i.e. $R$
defined by pointwise Gauss-Seidel method can be replaced by the norm
defined by the diagonal of $A$ (i.e. by Jacobi method, which, while
not always convergent as a relaxation provides an equivalent norm).

In terms of this operator $D$, we introduce the following quantity
\begin{equation}
  \label{KVcD} K(V_c,D)= \max_{v}\frac{\Dnorm{v-\Dproj
v}^2}{\|v\|_A^2} =\max_{v}\min_{v_c\in
V_c}\frac{\Dnorm{v-v_c}^2}{\|v\|_A^2},
\end{equation} 
where $\Dproj: V\mapsto V_c$ is the $\Dscalar{u}{v}$-orthogonal projection.

By \eqref{KVc}, \eqref{KVcD} and \eqref{star-equiv-norms}, we have
\begin{equation}
  \label{KK} c_D K(V_c,D)\le K(V_c)\le c^D K(V_c,D).
\end{equation}

\begin{theorem}\label{thm:two-level-theorem-period} The two level
algorithm satisfies
\begin{equation}\label{eq:y-two-level-estimate}
1-\frac{1}{c_DK(V_c,D)}\le \|E\|_A^2 \le 1-\frac{1}{c^DK(V_c,D)} \le
1-\frac{1}{c^DC}.
\end{equation} where $C$ is any upper bound of $K(V_c,D)$, namely
\begin{equation}\label{eq:approximation-z} 
    \min_{w\in V_c}\Dnorm{v-w}^2 \le C\|v\|_A^2, \quad\mbox{for all}\quad v\in V.
\end{equation}
\end{theorem} The proof of the above theorem is straightforward and indicates that, if $c_D$ and $c^D$ are ``uniform'' constants,
the convergence rate of the two-level method is ``uniformly'' dictated
by the quantity $K(V_c,D)$.

We say that $V_c$ is quasi-optimal if the following inequality holds 
\begin{equation}\label{quasi-opt}
    \min_{w\in V_c}\Dnorm{v-w}^2 \le \gamma\mu_{n_c+1}^{-1}\|v\|_A^2, \quad\mbox{for all } v\in V,
\end{equation}
with a constant $\gamma>0$ independent of the size of the problem.

The construction of an approximation to the optimal coarse space
$V_c^{\rm opt}$ which is used in most AMG algorithms relies on two
operators $A_M$ and $D_M$ which satisfy:
\begin{equation}\label{e:m-and-a}
c_1\|v\|_D^2 \le \|v\|_{D_M}^2, \quad 
\|v\|_{A_M}^2\le c_2\|v\|_A^2, \quad\mbox{for all}\quad v\in V , 
\end{equation}
with constants $c_1$ and $c_2$ independent of the problem size.
Here, on the right side, we have a seminorm $\|\cdot\|_{A_M}$, because
sometimes $A_M$ is only semi-definite. We point out that here $A_M$
and $D_M$ are analogues to $\utilde A_W$ and $\utilde D$
defined in \eqref{utildeAW} and \eqref{utildeD}, respectively, in the
general framework in \S\ref{sec:unifiedAMG}. And the assumptions in
\eqref{e:m-and-a} are analogous to the
Assumptions~\ref{a:2level} which we made in the general
AMG framework in~\S\ref{sec:unifiedAMG}.

\begin{theorem}\label{thm:quasi-opt}
    If $D_M$ and $A_M$ satisfy \eqref{e:m-and-a}, and $V_c$ is a coarse space such that 
    \begin{equation}
        \min_{w\in V_c}\|v-w\|_{D_M}^2 \le \gamma\mu_{n_c+1}^{-1}\|v\|_{A_M}^2, \quad\mbox{for all } v\in V,
    \end{equation}
Then
\begin{enumerate}
\item The following estimate holds
    \begin{equation}
        \min_{w\in V_c}\Tnorm{v-w}^2 \le \frac{c^D}{c_D}\frac{c_2}{c_1}\gamma \mu_{n_c+1}^{-1}\|v\|_A^2, \quad\mbox{for all } v\in V.
    \end{equation}
\item The corresponding two-level AMG algorithm satisfies
  \begin{equation}
    \label{2level-converge}
\|I-BA\|_A^2\le 1-    \frac{c_D}{c^D}\frac{c_1}{c_2}\frac{1}{\gamma}\;\mu_{n_c+1}
  \end{equation}
\end{enumerate}
\end{theorem}

\subsection{Algebraically high and low frequencies}\label{s:hf-lf}

In geometric MG, algebraically smooth error is also smooth in the usual geometric
sense. However, in AMG settings, smooth error can be geometrically
non-smooth. In order to make this distinction, we use the term
\emph{algebraically smooth error} when we refer to the error in the
AMG setting that is not damped (eliminated) by the smoother $R$.  
In general, a good interpretation of the algebraically smooth error
leads to an efficient and robust AMG algorithm. A careful
characterization of the algebraically smooth error is needed, since in
such case we can try to construct a coarser level which captures these
error components well.  

Here is a more formal definition of a algebraically smooth error.
\begin{definition} \label{def:smoothing} Let $R:V\mapsto V$ be a
smoothing operator such that its symmetrization
$\overline{R}=R+R^T-R^TAR$ is positive definite.  Given $\varepsilon
\in (0,1)$, we say that the vector $v$ is algebraically
$\varepsilon$-smooth (or $v$ is an \emph{$\varepsilon$-algebraic low frequency})
with respect to $A$ if
\begin{equation}\label{smoothing-1-norm} 
\|v\|_A^2\le \varepsilon\|v\|_{\bar{R}^{-1}}^2.
\end{equation}
\end{definition} The set of algebraically smooth vectors will be
denoted by:
\begin{equation} 
    \mathcal L_{\varepsilon} =\{ v: \|v\|_A^2\le \varepsilon\|v\|_{\bar{R}^{-1}}^2\}.
\end{equation} 
We point out that this is a set of vectors (a ball, or,
rather, an ellipsoid) and not a linear vector space in general.  It is
then clear that the elements of this set need to be approximated well
by elements from the coarse space.

The rationale of the above definition can be seen from the following
simple result (We notice it is always true that $(\bar RAv, v)\le \|v\|_A^2$).
\begin{lemma}\label{lem:smoothing-1-norm} Any vector $v\in V$ that
satisfies
\begin{equation}\label{eq:smoothing-2-norm} (\overline{R}
A\bm{v},\bm{v})_A\le \varepsilon\|v\|^2_A.
\end{equation} is $\varepsilon$-algebraically smooth.
\end{lemma}
\begin{proof} By the Schwarz inequality for the inner product defined
by $\overline{R}^{-1}$ and~\eqref{eq:smoothing-2-norm} we have
\[ \|\bm{v}\|^2_A = (\overline{R}A \bm{v}, \overline{R}^{-1}\bm{v})\le
(\overline{R}A \bm{v}, A\bm{v})^{1/2} (\overline{R}^{-1} \bm{v},
    \bm{v})^{1/2} \le \sqrt{\varepsilon} (\overline{R}^{-1} \bm{v},
\bm{v})^{1/2}\|\bm{v}\|_A.
\]
\end{proof}

We can easily show that this definition is equivalent to saying that
the algebraically smooth error components are the components for which
the smoother converges slowly. Indeed, the
inequality~\eqref{eq:smoothing-2-norm} is clearly equivalent to
$((I-\overline{R}A) v,v )_A\ge (1-\varepsilon) (v,v)_A$, namely, 
\begin{equation}\label{eq:smoothing-property}
\frac{\|S\bm{v}\|_A^2}{\|\bm{v}\|_A^2} \ge 1-\varepsilon, \quad
S=I-T,\quad \mbox{and}\quad T=RA.
\end{equation} 
The property~\eqref{eq:smoothing-property} is often
referred to as \emph{smoothing property}.

\begin{remark} 
 In the classical multigrid literature
algebraically smooth error is defined as  ${\bm e}\in V$ such that
\begin{equation}\label{eq:ase} 
\|{\bm e}\|_{AD^{-1}A}^2 \leq \varepsilon \|{\bm e}\|_A^2,
 \end{equation} for a small and positive parameter $\varepsilon$ which
implies
\begin{equation}\label{eq:normequiv10} \|{\bm e}\|_A^2 \leq \|{\bm
    e}\|_D \|{\bm e}\|_ {AD^{-1}A}\leq \sqrt{\varepsilon} \|{\bm e}\|_D \|{\bm
e}\|_A .
 \end{equation} Namely,
 \begin{equation}
   \label{eAeD}
\|{\bm e}\|_A^2\le \varepsilon \|{\bm e}\|_D^2.   
 \end{equation}
As it is clearly seen from Definition~\ref{def:smoothing},
Lemma~\ref{lem:smoothing-1-norm} implies~\eqref{eq:normequiv10} with
$\overline{R}\approx D^{-1}$, where $D$ is the diagonal of $A$.
\end{remark}
Thanks to \eqref{MMrel-0}, we have
\begin{lemma} \label{lem:eAeD}
If $e$ is algebraically smooth, namely $e$ satisfies
\eqref{eAeD}.  Then  
\begin{equation}
\|{\bm e}\|_{\tilde A}^2\lesssim\varepsilon \|{\bm e}\|_{\tilde D}^2.   
\end{equation}
Namely $e$ is also algebraically smooth with respect $\tilde A$, the
M-matrix relative of $A$. 
\end{lemma}

On the other hand, note that by the definition of $\Tnorm{\cdot}$, we
always have $\Tnorm{v}\ge C\|v\|_A$ with constant $C$ independent of
the parameters of interest.  Drawing from analogy with geometric
multigrid method we introduce the notion of algebraic high frequency as follows:

\begin{definition}\label{def:high-frequencies} Given $\delta\in
(0,1]$, we call $v\in V$ an $\delta$-\emph{algebraic high frequency}
if,
$$
\|v\|_A^2\ge \delta\Tnorm{v}^2.
$$ 
\end{definition} The set of algebraically high frequency vectors will
be denoted by:
\begin{equation}\label{eq:smootherror} \mathcal{H}_{\delta} =\bigg\{
        v: \|v\|_A^2\ge\delta\|v\|_{\bar{R}^{-1}}^2\bigg\}.
\end{equation}

The concept of algebraic high-and low-frequencies will be used in a
2-level AMG theory \S\ref{s:semidefinite-theory} and also be used in the
design of classical AMG \S\ref{s:idealHL}.

\begin{lemma} Let $(\phi_i, \mu_i)$ are all the eigen-pairs for $\bar
  RA$, namely $\bar RA\phi_i=\mu_i\phi_i$. Then
$$
\Span\{\phi_i: \mu_i\le \varepsilon\}\subset \mathcal L_\varepsilon
$$
and 
$$
\Span\{\phi_i: \mu_i\ge \delta\}\subset \mathcal H_\delta
$$
\end{lemma}

We now  introduce the notion of {\it near-null space} as follows.
\begin{definition}[Near-null space]  
For sufficiently small
  $\varepsilon\in(0,1)$, we call $\Span\{\phi_i: \mu_i\le
  \varepsilon\}$ an $\varepsilon$-near-null space of $\bar RA$. 
\end{definition}

\subsection{Smoothing properties of Jacobi and Gauss-Seidel methods}
The essence of multigrid methods is that simple iterative methods such
as Jacobi and Gauss-Seidel methods that have a special property,
known-as the \emph{smoothing property}. 
As an illustration, we apply the Gauss-Seidel method to 
$$
A\mu=b,
$$ 
with $A$ given by \eqref{iso-A} for isotropic problem and 
\eqref{aniso-A} for anisotropic problem respectively.  We first choose
$\mu$ randomly (as shown in Fig. \ref{random-exact}) and then compute
$A\mu$ for both \eqref{iso-A} and \eqref{aniso-A} to compute right
hand sides $b=A\mu$ respectively.  We then apply Gauss-Seidel method to
both equations with initial guess $\mu^0=0$.

\begin{figure}[!htb]
    \centering
    \includegraphics[width=0.5\textwidth]{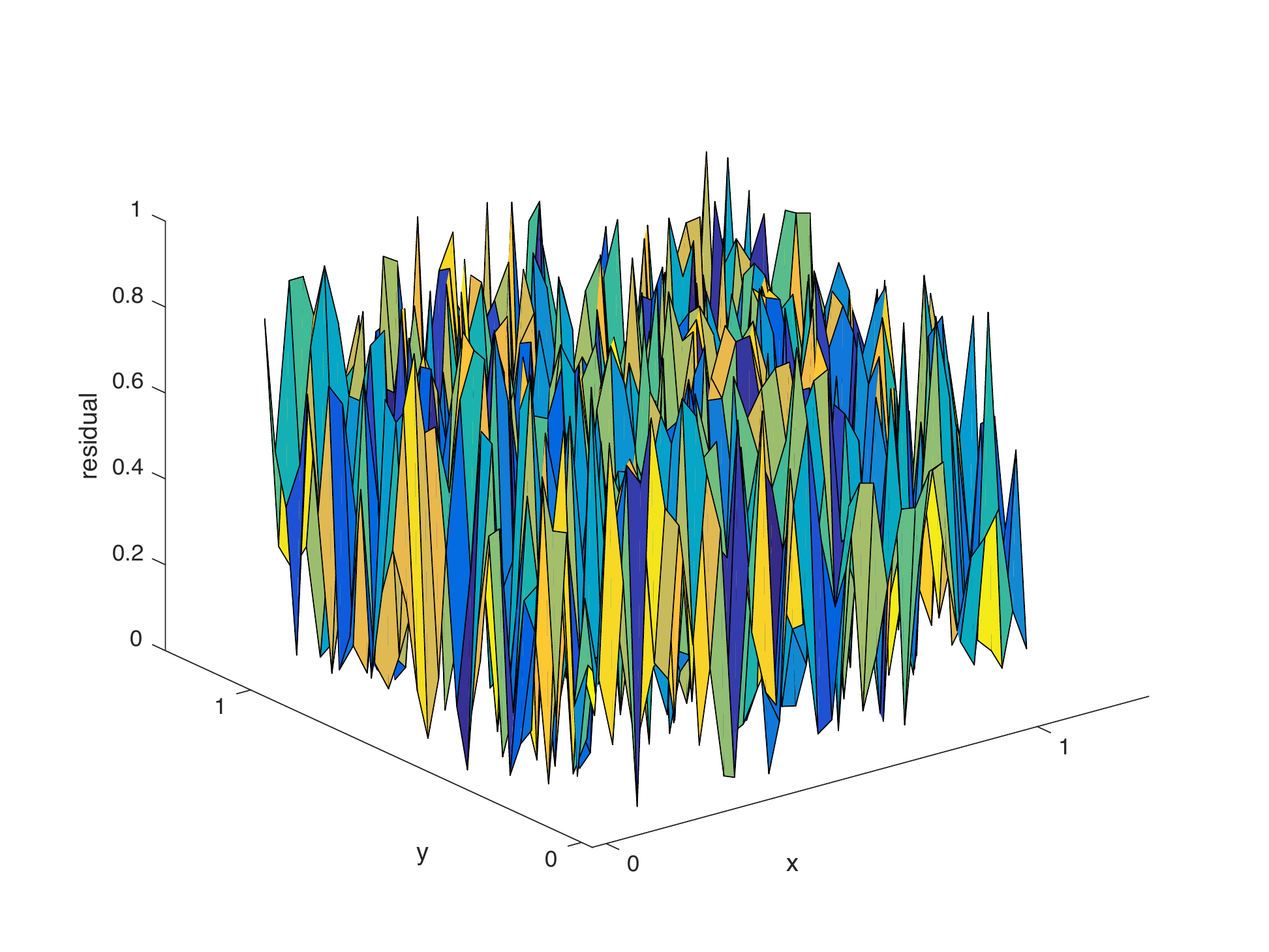}
    \caption{Initial error for both \eqref{iso-A} and \eqref{aniso-A}.}
    \label{random-exact}
\end{figure}

\begin{figure}[!htb]
    \centering
    \includegraphics[width=0.3\textwidth]{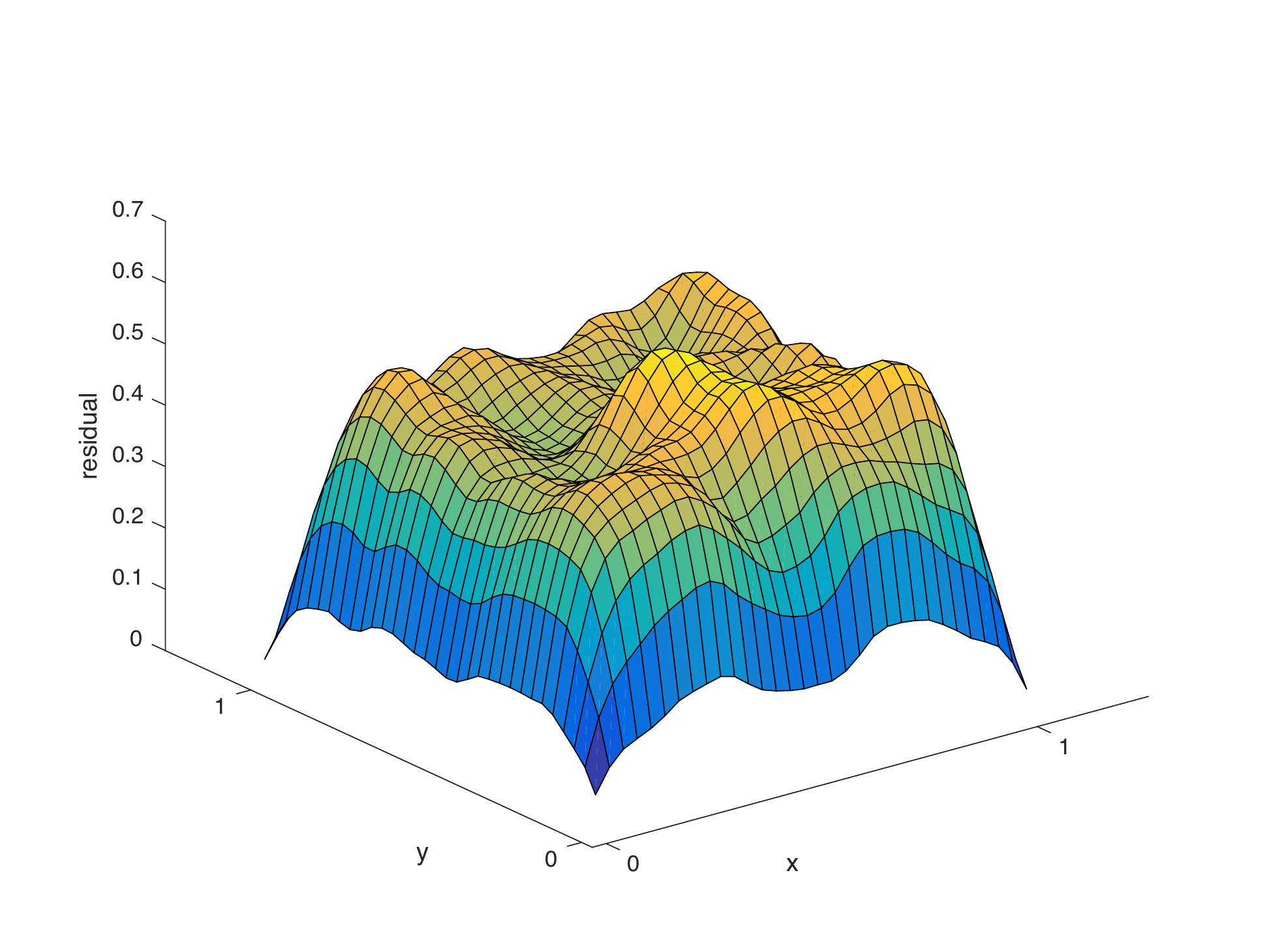}
    \includegraphics[width=0.3\textwidth]{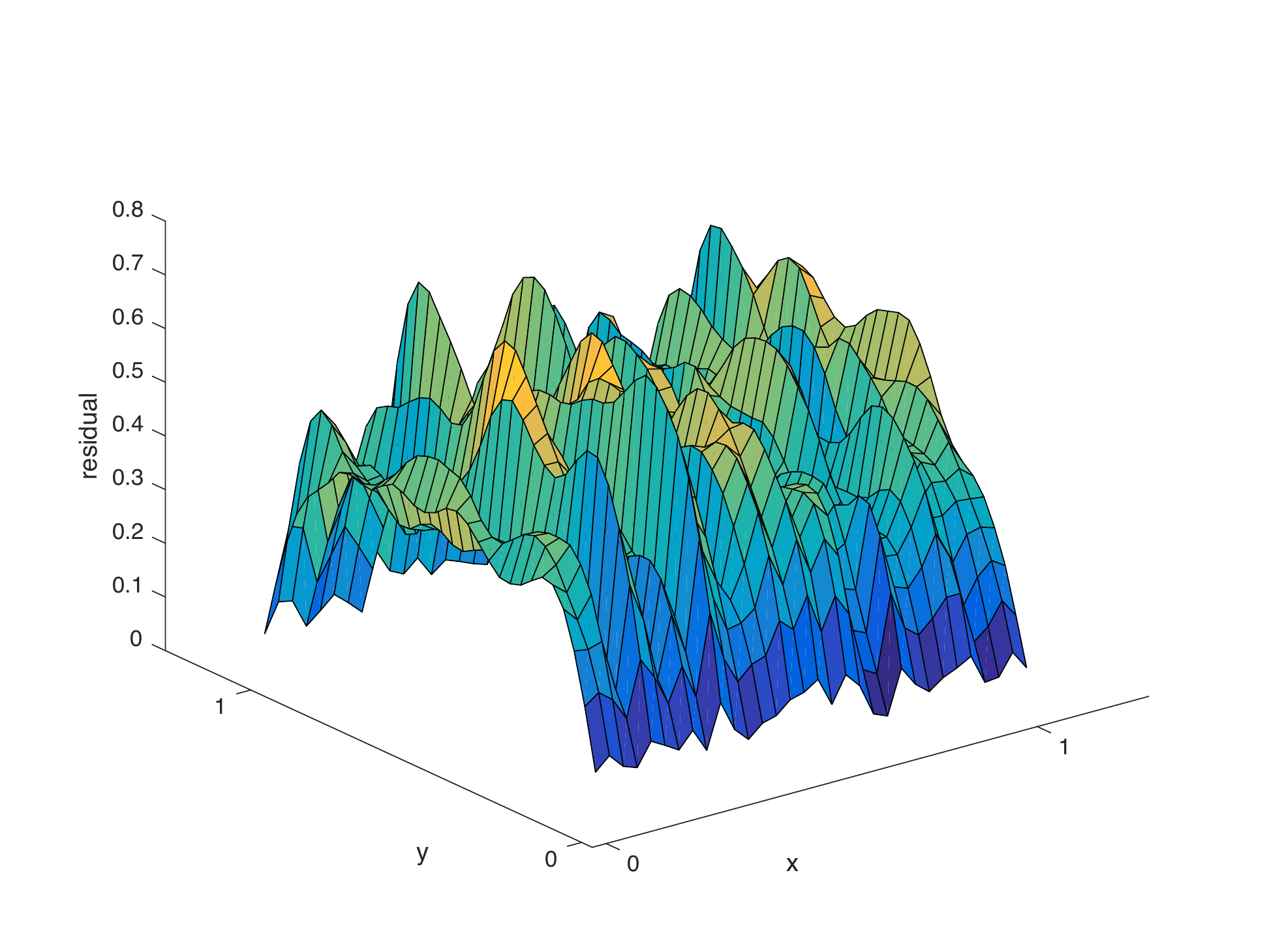}
    \includegraphics[width=0.3\textwidth]{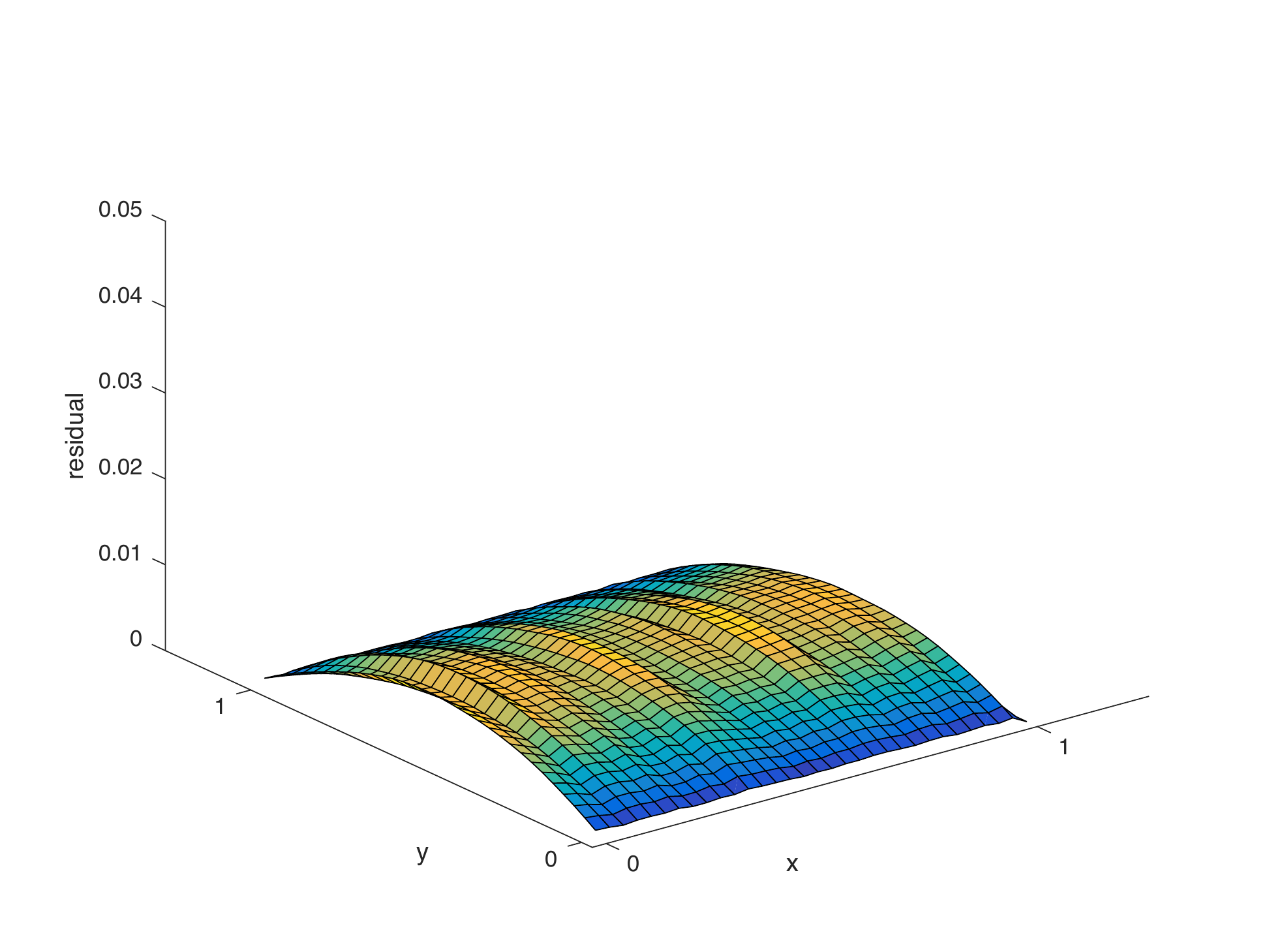}
    \caption{Left: error after applying $5$ Gauss-Seidel iterations to \eqref{iso-A}. Middle: error after applying $5$ Gauss-Seidel iterations to \eqref{aniso-A}. Right: error after applying $1$ block Gauss-Seidel iterations to \eqref{aniso-A}.}
    \label{f:anisotropic}
\end{figure}

We note that, for $A$ given by \eqref{aniso-A} when $\epsilon\ll 1$, we have 
$$
\lambda_{11}<\lambda_{21}< \ldots <\lambda_{N1}<\lambda_{ij}, \quad i\ge
1, j\ge 2.
$$
The corresponding eigen functions, which can be viewed as ``algebraic
low-frequencies'' can be highly oscillatory in $x$-direction. 

As an illustration of the difference between algebraic high/low frequencies and geometric high/low frequencies, we consider the linear system given by \eqref{aniso-A} for anisotropic problem. Clearly, $A$ can be written as
\begin{equation*}
    A= \epsilon I\otimes M + M\otimes I \text{ with } M=\operatorname{tridiag}(-1, 2, -1).
\end{equation*}

We define the vector $\mu \in R^N$ as 
\begin{equation*}
    \mu=x\otimes y, \text{ with } x=\bm{1}_n, \text{ and } y= (1~0~1~ 0\cdots 1~0~1)^T\in R^n.
\end{equation*}

Then it is easy to compute that 
\begin{equation*}
    Mx= \begin{pmatrix}
        1\\
        0\\
        0\\
        \vdots\\
        0\\
        1
    \end{pmatrix}, \text{ and }
    My= \begin{pmatrix}
        2\\
        -2\\
        2\\
        \vdots\\
        -2\\
        2\\ 
    \end{pmatrix}.
\end{equation*}
We have 
\begin{equation*}
  A\mu = \epsilon (I\otimes M)(x\otimes y)+(M\otimes I)(x\otimes y) = \epsilon (x\otimes My)+(Mx\otimes y),
\end{equation*}
and 
\begin{eqnarray*}
    \|\mu\|_A^2 &= &\mu^TA\mu= \epsilon (x\otimes y)^T(x\otimes My)+(x\otimes y)^T(Mx\otimes y)\\
    &=& \epsilon (x^Tx)\otimes (y^TMy)+(x^TMx)\otimes (y^Ty)=\epsilon n(n+1)+n+1.
\end{eqnarray*}
Letting $D$ be the diagonal of $A$, we then have
\begin{equation*}
    \|\mu\|_D^2 = 2(1+\epsilon)\mu^T\mu= (1+\epsilon)n(n+1).
\end{equation*}
This shows that
\begin{equation*}
    \frac{\|\mu\|_A^2}{\|\mu\|_D^2}= \frac{\epsilon}{1+\epsilon}+\frac{1}{(1+\epsilon)(n+1)},
\end{equation*}
which implies that $\mu$ is an algebraic low frequency if $\epsilon$ is sufficiently small.

On the other hand, if we denote the nodal basis functions corresponding to the uniform finite element mesh by $\{\phi_{ij}: 1\le i, j\le n\}$, namely, $\phi_{ij}$ is a piecewise linear function such that
\begin{equation*}
    \phi_{ij}(kh, lh)=\delta_{ik}\delta_{jl}.
\end{equation*}
Then we define
\begin{equation*}
    \bm{\Phi}=(\phi_{11}, \phi_{12}, \cdots, \phi_{1n}, \phi_{21}, \cdots, \phi_{nn}).
\end{equation*}
if we consider the finite element function corresponding to $\mu$, namely, the function defined by
\begin{equation}\label{fx}
    u= \bm{\Phi}\mu = \sum_{i=1}^{\frac{n+1}{2}}\sum_{j=1}^{n}\phi_{2i-1, j}.
\end{equation}
Then, in the geometric point of view, this function is highly oscillatory on the $x$ direction, which is a geometric high frequency (see Figure~\ref{f:fx})

\begin{figure}[h]
    \centering
    \includegraphics[width=0.6\textwidth]{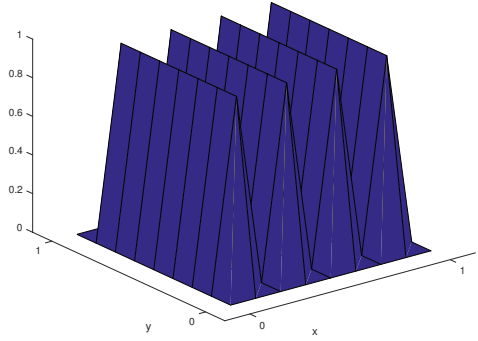}
    \caption{Graph of function defined in ~\eqref{fx}. This function is highly oscillatory on the $x$ direction. }
    \label{f:fx}
\end{figure}

\subsection{Convergence theory in view of algebraic high- and
  low-frequencies}

We next present a convergence theory based on algebraic high- and
  low-frequencies.  We first prove the following lemma.
\begin{lemma}\label{l:VhfVc} If $V_c\subset V$ is such that the
  following ``stable decomposition''  holds:
$$
V=V_c+V_{hf}
$$
for some $V_{hf}\subset V$ which consists of $\delta$-algebraic
high frequencies (see Definition~\ref{def:high-frequencies}). 
Namely,  for any $v\in V$, there exists $v_c\in V_c$ and
$v_{hf}\in V_{hf}$ such that
$$
v=v_c+v_{hf}, \quad \|v_{hf}\|_A^2\le c_1\|v\|_A^2.
$$
Then the corresponding 2-level AMG satisfies
$$
\|E\|_A\le 1-\frac{\delta}{c_1} .
$$
\end{lemma}
\begin{proof} 
It follows that
$$
\inf_{w_c\in V_c}\Tnorm{v-w_c}^2
\le \Tnorm{v_{hf}}^2\le \delta^{-1}\|v_{hf}\|_A^2\le 
\frac{c_1}{\delta}\|v\|_A^2.
$$
As a result, 
$$
K(V_c)\le\frac{c_1}{\delta},
$$
and, finally, we have,
$$
\|E\|_A=1-\frac{1}{K(V_c)}\le 1-\frac{\delta}{c_1}.
$$
\end{proof}

\begin{corollary}\label{c:AHF} 
If $V_{hf}$ consists of $\delta$-algebraic
high frequencies, then for the coarse space $V_c$ given by
\begin{equation}
  \label{quasiVc} V_c= \operatorname{Range}(I-P_{hf}),
\end{equation}
where $P_{hf}: V\mapsto V_{hf}$ is the A-orthogonal projection.  Then
$$
\|E\|_A\le 1-\delta .
$$
\end{corollary}

\subsection{Bibliographical notes}
One of the first results on two level convergence of AMG methods are
found in earlier papers~\cite{1stAMG,Ruge.J;Stuben.K.1987a}.  There
have been a lot of research on reflecting the MG theory through algebraic settings:
\cite{1983MaitreJ_MusyF-aa,1985BankR_DouglasC-aa,Mandel.J.1988a}; algebraic
variational approach to the two level MG
theory~\cite{1985McCormickS-aa,1984McCormickS-aa,1982McCormickS_RugeJ-aa}. 

For the two grid convergence, sharper results, including two sided
bounds are given in~ \cite{2008ZikatanovL-aa} and also considered in
\cite{Falgout.R;Vassilevski.P.2004a}
and~\cite{Falgout.R;Vasilevski.P;Zikatanov.L.2005a}. These two-level
results are more or less a direct consequences of the abstract theory
provided
in~\cite{1991BrambleJ_PasciakJ_WangJ_XuJ-ac,1992XuJ-aa,Xu.J;Zikatanov.L.2002a}. A
survey of these and other related results is found in a recent
article~\cite{2014MacLachlanS_OlsonL-aa}.  The two approaches for
analyzing that are included in this section were recently developed in
\cite{UnifiedAMG,EnergyMin}.

Theorem~\ref{thm:two-level-convergence} can be found in
\cite{2008ZikatanovL-aa} 
and can be viewed as a consequence of the XZ identity
\cite{Xu.J;Zikatanov.L.2002a} in the special case of two subspaces
from the general framework of the method of subspace corrections.  The
original proof of this theorem in \cite{2008ZikatanovL-aa} was based
on the XZ identity.  The proof here is new and is more direct.
 
Multilevel results are difficult to establish in general algebraic
settings, and most of them are based on either not realistic
assumptions or they use geometrical grids to prove convergence. We
refer
to~\cite{1996VanekP_MandelJ_BrezinaM-aa,2011BrezinaM_VassilevskiP-aa}
for results in this direction. Rigorous multilevel results for finite element
equations can be derived using the auxiliary space framework, which is
developed in~\cite{1996XuJ-aa} for quasi-uniform meshes. More recently
multilevel convergence results for adaptively refined grids were shown
to be optimal in~\cite{2012ChenL_NochettoR_XuJ-aa}.  A multilevel
convergence result on shape regular grids using AMG based on quad-tree
(in 2D) and oct-tree (in 3D) coarsening is shown
in~\cite{2015GrasedyckL_WangL_XuJ-aa}.

Finally, we point out that the notation used in parts of this section
originates
in~\cite{1980BankR_DupontT-aa,bramble1987new,bramble1990parallel} and
is convenient for the analysis, especially when finite element
equations are considered.

\section{A general approach to the construction of coarse
  space}\label{sec:unifiedAMG}
In this section, we describe an abstract framework for constructing
coarse spaces by using the notion of space decomposition and subspace
corrections. 

Let us first introduce some technical results that are used as
analytic tools later. 
\begin{lemma}
\label{lm:auxiliary} 
Let $\utilde{V}$ and $V$ be two vector spaces and let
$\Pi:\utilde{V}\mapsto V$ be a surjective map. Let
$\utilde{B}:\utilde{V'}\mapsto \utilde{V}$ be an SPD operator. Then
$B:=\Pi\utilde{B}\Pi'$ is also
SPD. Furthermore 
\begin{equation}
  \label{auxi-identity}
    (B^{-1}v,v) =\min_{\Pi\utilde{v}=v}\langle \utilde{B}^{-1}\utilde{v},\utilde{v}\rangle,
\end{equation}
with the unique minimizer given by 
\begin{equation}
  \label{minimizer}
  \utilde{v}^*=\utilde{B}\Pi'B^{-1}v.
\end{equation}
\end{lemma}
\begin{lemma}\label{l:fictitiousB}
Assume that following two conditions are satisfied for $\Pi$:
    \begin{enumerate}[1.]
    \item For all $\utilde v\in \utilde V$, 
        \begin{equation}
          \label{tBA1}
        \|\Pi \utilde v\|_A\le \tilde\mu_1\|\utilde v\|_{\utilde B^{-1}}.
        \end{equation}
    \item For any $v\in V$, there exists $\utilde v \in\utilde V$ such
        that $\Pi\utilde v=v$ and 
        \begin{equation}
          \label{tBA0}
        \|\utilde v\|_{\utilde B^{-1}}\le \tilde\mu_0\|v\|_A.   
        \end{equation}
\end{enumerate}
Then
$$
\kappa (BA)\le \left(\frac{\tilde\mu_1}{\tilde\mu_0}\right)^2.
$$
\end{lemma}

A direct consequence of the above Lemma \ref{l:fictitiousB} is the
following result. 
\begin{theorem}[Fictitious Space Lemma]\label{theorem:fictitiousA}
Assume that following two conditions are satisfied for $\Pi$.  First
$$
\|\Pi \utilde v\|_A\le \mu_1\|\utilde v\|_{\utilde A}, \quad\forall
\utilde v\in \utilde V
$$
Secondly, for any $v\in V$, there exists $\utilde v \in\utilde V$ such
that $\Pi\utilde v=v$ and 
$$
\|\utilde v\|_{\utilde A}\le \mu_0\|v\|_A. 
$$
Then $\kappa(\Pi)\le \mu_1/\mu_0$ and, under the assumptions of Lemma \ref{lm:auxiliary}
$$
\kappa (BA)\le \left(\frac{\mu_1}{\mu_0}\right)^2\kappa(\utilde B\utilde A). 
$$
\end{theorem}

We assume there exist a sequence of spaces
    $V_1, V_2, \dots, V_J$, which are not necessarily subspaces of
    $V$, but each of them is related to the original space $V$ by a
    linear operator 
\begin{equation} \Pi_j: V_j\mapsto
      V.  \end{equation} 
We assume that $V$ can be written as a sum of
    subspaces and \eqref{aux-decomp} and \eqref{aux-decomp0} hold.

Denote
$$\utilde{W} = V_1\times V_2\times...\times V_J,
$$ 
with the inner product 
$$(\utilde u,\utilde v) = \sum_{i=1}^J(u_i,v_i),
$$ 
where $\utilde u=(u_1,...,u_J)^T$ and $\utilde
v=(v_1,...,v_J)^T$.  Or more generally, for $\utilde f=(f_1, \ldots, f_J)^T\in
\utilde V'$ with $f_i\in V_i'$, we can define
$$
( \utilde f, \utilde v) =\sum_{i=1}^J( f_i, v_i).
$$
We now define $\Pi_W:\utilde W\mapsto V$ by
$$
\Pi_W\utilde u = \sum_{i=1}^J \Pi_i u_i, \quad\forall \utilde
u=(u_1,...,u_J)^T\in \utilde W. 
$$ 
Formally, we can write 
$$
\Pi_W=(\Pi_1,\ldots,\Pi_J) \mbox{ and }
\Pi_W'=
\begin{pmatrix}
\Pi_1'\\
\vdots\\
\Pi_J'
\end{pmatrix}.
$$

We assume there is an operator $A_j: V_j\mapsto V_j'$ which is symmetric, positive semi-definite for each $j$ and define $\utilde A_W: \utilde W\mapsto \utilde W'$ as follows
  \begin{equation}\label{utildeAW}
      \utilde A_W: = \operatorname{diag}(A_1, A_2, \dots, A_J).
  \end{equation}

    For each $j$, we assume there is a symmetric positive
    definite operator $D_j: V_j\mapsto V_j'$, and define $\utilde D:
    \utilde W\mapsto \utilde W'$ as follows
  \begin{equation}\label{utildeD}
      \utilde D: = \operatorname{diag}(D_1, D_2, \dots, D_J).
  \end{equation}

    We associate a coarse space $V_j^c$, $V_j^c \subset V_j$, with
    each of the spaces $V_j$, and consider the corresponding
    orthogonal projection $Q_j: V_j\mapsto V_j^c$ with respect to
    $(\cdot, \cdot)_{D_j}$. We define $\utilde Q: \utilde W\mapsto \utilde W'$ by
    \begin{equation}
        \utilde Q: =\operatorname{diag}(Q_1, Q_2, \dots, Q_J).
    \end{equation}

\begin{assumption}\label{a:2level}\quad
  \begin{enumerate}
\item 
        The following inequality holds for all $\utilde w\in \utilde W$:
        \begin{equation}\label{assm:D_j}
            \|\Pi_W \utilde w\|_D^2\le C_{p,2}\|\utilde w\|_{\utilde D}^2, 
        \end{equation}
        for some positive constant $C_{p, 2}$.
\item For each $w\in V$, there exists a
  $\utilde w\in \utilde W$
  such that $w=\Pi_W\utilde w$ and the following inequality holds
    \begin{equation}\label{sum_Aj}
        \|\utilde w\|_{\utilde A_W}^2\le C_{p,1}\|w\|_A^2
\end{equation}
with a positive constant $C_{p, 1}$ independent of $w$. 
\item For all $j$, 
  \begin{equation}
    \label{NAjVjc}
N(A_j)\subset V_j^c.    
  \end{equation}

 \end{enumerate}
\end{assumption}

\begin{remark}
The above assumption implies that 
$$
w\in N(A) \Rightarrow \utilde w\in N(A_1)\times\ldots \times N(A_J).
$$
\end{remark}

We define the global coarse space $V_c$ by 
\begin{equation}\label{V_c}
    V_c:=\sum_{j=1}^J\Pi_j V_j^c.
\end{equation}

Further, for each coarse space $V_j^c$, we define
\begin{equation}\label{muj}
    \mu^{-1}_j(V_j^c):=\max_{v_j\in V_j}\min_{v_j^c\in V_j^c}\frac{\|v_j-v_j^c\|_{D_j}^2}{\|v_j\|_{A_j}^2},
\end{equation}
and 
\begin{equation}
  \label{muc}
    \mu_c=\min_{1\le j\le J}\mu_j(V_j^c),
\end{equation}
which is finite, thanks to Assumption \ref{a:2level}.3 (namely, \eqref{NAjVjc}).

By the two level convergence theory, if $D_j$ provides a convergent smoother, then
$(1-\mu_j(V_j^c))$ is the convergence rate for
two-level AMG method for $V_j$ with coarse space $V_j^c$. Next theorem
 gives an estimate on the convergence of the two level method in
terms of the constants from Assumptions~\ref{a:2level}
and $\mu_c$.
\begin{theorem}\label{thm:two-level}
If Assumption~\ref{a:2level} holds, then for each $v\in V$,  
    we have the following error estimate
    \begin{equation}
        \min_{v_c\in V_c}\|v-v_c\|_D^2 \le C_{p,1}C_{p,2}\mu_c^{-1}\|v\|_A^2.
    \end{equation}

\end{theorem}

    \begin{proof}
        By Assumption~\ref{a:2level}, for each $v\in V$, there exists $\utilde v\in \utilde V$ such that 
        \begin{equation}
            v=\Pi_W\utilde v
        \end{equation}
        and \eqref{sum_Aj} is satisfied.

        By the definition of $\mu_c$, we have
        \begin{equation}
            \|\utilde v-\utilde Q\utilde v\|_{\utilde D}^2\le \mu_c^{-1} \|\utilde v\|_{\utilde A_W}^2.
        \end{equation}
        We let $v_c=\Pi_W\utilde Q\utilde v$. Then $v_c\in V_c$ and by
        Assumption~\ref{a:2level}, we
        have
        \begin{equation*}
            \|v-v_c\|_D^2 =  \|\Pi_W(\utilde v- \utilde Q\utilde v)\|_D^2 \le  C_{p,2}\|\utilde v-\utilde Q\utilde v\|_{\utilde D}^2 \le  C_{p,2}\mu_c^{-1}\|\utilde v\|_{\utilde A_W}^2\le  C_{p,1}C_{p,2}\mu_c^{-1}\|v\|_A^2.
        \end{equation*}
        
      \end{proof}

We define another product space
\begin{equation}
    \utilde{V}:= V_c\times V_1\times V_2\times \cdots \times V_J,
\end{equation}
and we set $\Pi_c: V_c\mapsto V$ to be the natural inclusion from $V_c$ to $V$. Then we define $\Pi: \utilde V\mapsto V$ by
\begin{equation}
    \Pi:=(\Pi_c~\Pi_1~\Pi_2~\cdots~\Pi_J),
\end{equation}
and $\utilde A: \utilde V\mapsto \utilde V'$ by
\begin{equation}
  \utilde A_: = \begin{pmatrix}
      A_c & & & \\
      & A_1 & & \\
      & & \ddots & \\
      & & & A_J \\
      \end{pmatrix},
\end{equation}
where $A_c: V_c\mapsto V_c'$ is given as
\begin{equation}
    A_c: =\Pi_c'APi_c.
\end{equation}
And $\utilde B: \utilde V\mapsto \utilde V'$ is given as
\begin{equation}
    \utilde B: = \begin{pmatrix}
      A_c^{-1} & & & &\\
      & D_1^{-1} & & &\\
      & & D_2^{-1} & &\\
      & & & \ddots &\\
      & & & & D_J^{-1}\\
      \end{pmatrix},
\end{equation}

We introduce the additive preconditioner $\widehat B$ 
\begin{equation}\label{additive_B}
    \widehat B: = \Pi \utilde B\Pi' = \Pi_cA_{c}^{-1}\Pi_c' + \sum_{j=1}^J \Pi_j D_j^{-1}\Pi_j',
\end{equation}
and we have the following results.

\begin{lemma}\label{lem:tBA0}
    If Assumption~\ref{a:2level} holds, then for any $v\in V$, there exists $\utilde v\in \utilde V$ such that \eqref{tBA0} holds, namely 
    $$\|\utilde v\|_{\utilde B^{-1}}\le \tilde \mu_0\|v\|_A$$
    with $\tilde \mu_0$ being a constant depending on $C_{p,1}$, $C_{p,2}$, $\mu_c$ and $c^D$. 
\end{lemma}

\begin{lemma}\label{lem:tBA1}
    If Assumption~\ref{assm:D_j} holds, then \eqref{tBA1} holds with constant $\tilde \mu_1$ depends on $C_{p,2}$ and $c^D$.
\end{lemma}

By directly applying Lemma~\ref{l:fictitiousB}, we immediately have
\begin{theorem}\label{thm:additive-converge}
    If Assumption~\ref{a:2level} holds, then 
    \begin{equation}
        \kappa (\widehat B A) \le \left(\frac{\tilde \mu_1}{\tilde \mu_0}\right)^2.
    \end{equation}
\end{theorem}

The following two-level convergence result is an application of 
the convergence theorem (Theorem~\ref{thm:two-level-convergence}) with
the error estimate in Theorem~\ref{thm:two-level}.

\begin{theorem}\label{thm:two-level-conv}
  If Assumption~\ref{a:2level} holds.
  Then the two-level AMG method with coarse space defined in
    \eqref{V_c} converges with a rate
    \[
        \|E\|_A^2\le 1-\frac{\mu_c}{C_{p,1}C_{p,2}c^D}.
    \]
\end{theorem}

\subsection{Bibliographical notes}
The fictitious space Lemma was first proved
in~\cite{1985MatsokinA_NepomnyashchikhS-aa}. Related is also the work
on auxiliary space method~\cite{1996XuJ-aa}.  Additive version of
Lemma~\ref{lm:auxiliary} is found in~\cite{Xu.J;Zikatanov.L.2002a},
and the most general case (including multiplicative preconditioners)
is in~\cite{XuMSC-Notes}.

In later sections we show how the general theory used here can be applied to various AMG algorithms, e.g. classical AMG~\cite{1stAMG,Ruge.J;Stuben.K.1987a}, 
smoothed aggregation AMG~\cite{Mika.S;Vanek.P.1992a,Mika.S;Vanek.P.1992b},
spectral AMGe~\cite{Chartier.T;Falgout.R;Henson.V;Jones.J;Manteuffel.T;McCormick.S;Ruge.J;Vassilevski.P.2003b,2011EfendievY_GalvisJ_VassilevskiP-aa}, and other algorithms. 

Many of the works in the Domain Decomposition (DD) literature also use
techniques for defining coarse spaces, which, to a large extend, have
similar aims as the coarse space constructions for AMG outlined in this
section. We refer
to~\cite{2005ToselliA_WidlundO-aa,2009WidlundO-aa,1994WidlundO-aa,2008DohrmannC_KlawonnA_WidlundO-aa,2014SpillaneN_DoleanV_HauretP_NatafF_PechsteinC_ScheichlR-aa}
and the references therein for more details on using local eigenspaces
for constructing coarse space in DD methods.

\section{Graphs and sparse matrices}\label{sec:graphs}
In this section, we give a brief introduction to some basic
notion of graph theory that is often used for sparse matrices and also
for the study of AMG. 

\subsection{Sparse matrix and its adjacency graph}\label{s:ugraph}
An \emph{undirected graph} (or simply a \emph{graph}) $\mathcal{G}$ is
a pair $(\mathcal{V},\mathcal{E})$, where $\mathcal{V}$ is a finite
set of points called \emph{vertices} and $\mathcal{E}$ is a finite set
of pairs of vertices, known as \emph{edges}.  We often write
$\mathcal V=\{1,\ldots,n\}$ for some fixed $n$.  We will not consider
directed graphs in this article because the graphs corresponding to the
symmetric sparse matrices are undirected.

An edge $e\in \mathcal{E}$ is an unordered pair $(j,\,k)$, where
$j,\,k \in \mathcal{V}$.  The vertices $j$ and $k$ are said to be 
\emph{adjacent} if $(j,k)\in \mathcal{E}$.
A \emph{path} from a vertex
$j$ to a vertex $k$ is a sequence $(j_0,\,j_1,\,j_2,...,\,j_l)$ of
vertices where $j_0=j$, $j_l = k$, and $(j_i,\,j_{i+1}) \in
\mathcal{E}$ for all $i = 0,\,1, ...,\,l-1$.  A vertex $j$ is
\emph{connected} to a vertex $k$ if there is a path from $j$ to
$k$. $\mathcal{G}=(\mathcal{V}, \,\mathcal{E})$ is
\emph{connected} if every pair of vertices is connected by a path,
otherwise it is said to be \emph{disconnected}. A graph
$\mathcal{G}_0=(\mathcal{V}_0,\mathcal{E}_0)$ is called a
\emph{subgraph} of $\mathcal{G}=(\mathcal{V},\mathcal{E})$ if
$\mathcal{V}_0\subset\mathcal{V}$ and
$\mathcal{E}_0\subset\mathcal{E}$.

The \emph{neighborhood} $N(i)$ are the vertices adjacent to the
vertex $i$. The \emph{degree} or \emph{valency} of a vertex is the number of edges
that connect to it. These are defined as:
\begin{equation}\label{ni-di}
    N(i)=\{j: (i, j)\in \mathcal E\},\quad    d_i=|\{j: (i, j)\in \mathcal E\}|. 
\end{equation}
A \emph{path} connecting two vertices $i$ and $j$ is a sequence of
edges $(k_0, k_1)$, $(k_1, k_2)$, \ldots , $(k_{m-1}, k_m)$ in
$\mathcal E$ such that $k_0=i$ and $k_m=j$. The length of the path is
the number of edges in it. The \emph{distance} between two vertices
$i$ and $j$ is the length of the shortest path connecting $i$ and $j$,
and we denote it by $\operatorname{dist}(i, j)$. If $i$, $j$ are not
connected, then $\operatorname{dist}(i, j)=\infty$. The diameter of a
graph is the largest distance between two vertices, i.e.
$\operatorname{diam}(\mathcal G) = \max\limits_{(i, j)\in \mathcal
  E}\operatorname{dist}(i, j)$.  An \emph{independent set} is a set
of vertices in which no two of which are adjacent. A \emph{maximal
  independent set} is an independent set such that adding any other
vertex to the set forces the set to contain an edge.

Given a symmetric matrix $A \in \mathbb{R}^{n\times n}$, the
\emph{adjacency graph} of $A$ is an undirected graph, denoted by
$\mathcal{G}(A)$, $\mathcal{G}= (\mathcal{V}, \,\mathcal{E})$ with
$\mathcal{V}=\{1, \,2, ..., \,n\}$. The edges $\mathcal{E}$ are
defined as
\[
\mathcal{E} = \{(j,\,k)\;\big|\; a_{jk} \ne 0\}. 
\]
A matrix $A$ is called \emph{irreducible} if its adjacency graph
$\mathcal G(A) = (\mathcal V, \mathcal E)$ is connected.  Otherwise,
$A$ is called \emph{reducible}.

An example of a symmetric matrix is shown in
Figure~\ref{fig:graph-and-matrix} (left) and a drawing of the
corresponding graph is in Figure~\ref{fig:graph-and-matrix} (right).
The pictorial representation of a graph is often not available and a
graph can be drawn in different ways with different coordinates of the
vertices.  As a general rule, sparse matrices do not provide any
geometrical information for the underlying graph and only the
combinatorial/topological properties of $\mathcal{G}(A)$.  

\begin{figure}[!htbp]
\begin{center}
\parbox{0.35\textwidth}{$
A=\begin{pmatrix}
      * & * & * & * & * & * \\
      * & * & * & 0 & 0 & 0 \\
      * & * & * & * & * &  0 \\
      * & 0 & * & * & * & 0\\
      * & 0 & * & * & * & 0\\
      * & 0 & 0 & 0 & 0 & *   
\end{pmatrix}$}
\parbox{0.35\textwidth}{\includegraphics*[width=0.3\textwidth]{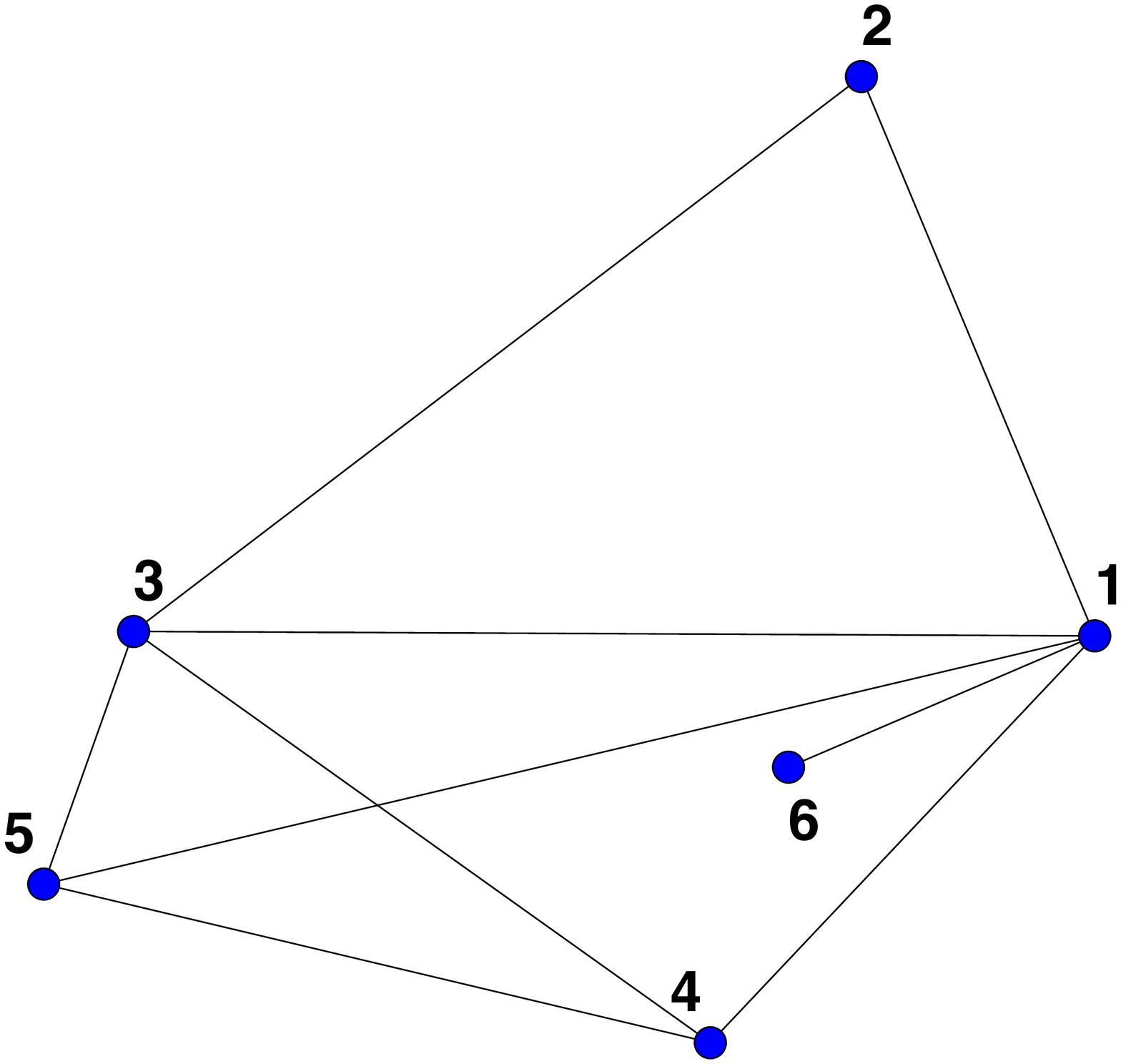}}
\end{center}
\caption{A sparse symmetric matrix (left) and its associated graph (right).\label{fig:graph-and-matrix}}
\end{figure}

Given 
$$
S\subset \{1,\ldots,n\}\times\{1,\ldots,m\},  
$$
we define 
\begin{equation}
  \label{ZS}
\mathbb R_S^{n\times m}=\bigg\{X=(x_{ij})\in \mathbb R^{n\times m}:  x_{ij}=0 \mbox{ if } (i,j)\notin S\bigg\}
\end{equation}
We say that $X$ has sparsity pattern given by $S$ if and only if $X\in
\mathbb R_S^{n\times m}$. 

Often, the sparsity pattern of a matrix is determined in advance and
the set $S$ is determined by a given matrix. For $Y\in
\mathbb{R}^{n\times m}$ we denote
\[
\sparse(Y) = \big\{(i,j)\;\big|\; y_{ij}\neq 0\big\}.
\]

We now consider the graphs associated with finite element or finite
difference stiffness matrices.  
In the case of finite elements, this is the space of FE functions
$V_h$, and, for finite difference discretizations this is the space of
\emph{mesh-functions} $V_h$, which can be identified with
$\mathbb{R}^N$.

We assume that we have a finite dimensional space $V_h^N$ which we use
to discretize the Neumann problem.  We also have the FE space for the
Dirichlet (or mixed boundary conditions) problem and we assume that
the following inclusions hold $V_h=V_h^D\subset V_h^N$. Equivalently,
we have that a subspace of the degrees of freedom vanishes on $V_h$:
for example, the values of the finite element or finite difference
solution at the nodes on the Dirichlet boundary vanish.

Let $A^{\mathcal N}$ be the matrix corresponding to the finite element
or finite difference discretization of the model second order elliptic
equation with Neumann boundary conditions.  
Clearly we have the following identity,
\begin{equation}\label{neu}
(A^{\mathcal N}u,v) = \sum_{e\in \mathcal{E}} \omega_e \delta_e u \delta_e v, \quad 
\end{equation}
Here, the sum is over all edges $\mathcal{E}$ of the graph
$\mathcal{G}(A^{\mathcal N}) = (\mathcal{V},\mathcal{E})$, $\delta_e v
= v_i - v_j$ if $\mathcal{E}\ni e=(i,j)$, $i<j$.  Also, $\omega_e= -
(A^{\mathcal{N}})_{ij}$ are the off-diagonal entries of
$A^{\mathcal{N}}$. Note that, since we consider the Neumann problem,
the bilinear form defined by $A^{\mathcal{N}}$ vanishes for $u$
(resp. $v$) such that $u_i=1$ (resp. $v_i=1$) for all $i$. For both
5-point and 9-point stencils we have that $\omega_e=1$ for all $e$.

We consider the stiffness matrix $A^{\mathcal N}$, corresponding to
the model problem~\eqref{Model0} with Neumann boundary conditions on a
bounded domain $\Omega\subset \mathbb{R}^d$, namely we have boundary
condition:
\begin{equation}\label{like-laplace}
\alpha\nabla u\cdot\bm{n} = 0, \quad \mbox{on}\quad\partial\Omega,
\end{equation}
where $\bm{n}$ is the unit normal vector to $\partial \Omega$ pointing
outward.  It is easy to derive the stiffness matrices corresponding to
the Dirichlet or mixed boundary condition problem: we just restrict
the bilinear form defined by $A^{\mathcal N}$ to a subspace:
\begin{equation}\label{bilinear1}
(Au,v) = \sum_{e\in \mathcal{E}} \omega_e \delta_e u \delta_e v, \quad
u_j=v_j=0\quad x_j\in \Gamma_D.  
\end{equation}
\begin{remark}
  Similar relations between differential problems with natural
  (Neumann) and essential (Dirichlet) boundary conditions are seen not
  only for the model problem considered here, but also for problems on
  $H(\curl)$, $H(div)$, linear elasticity and other.
\end{remark}

\subsection{$M$-matrix relatives of finite element stiffness matrices}\label{sec:m-matrix}
A symmetric matrix $A \in \mathbb R^{n\times n}$ is called an \emph{$M$-matrix} if it satisfies
the following three properties:
\begin{align}
\label{eq:sign1} &a_{ii} > 0 \;\;\text{for}\;\; i = 1,...,n,\\
\label{eq:sign2} &a_{ij} \le 0 \;\;
                  \text{for}\;\; i \ne j, \;\;i, \;j = 1,...,n,\\
\label{eq:sign3} &A \;\;\text{is semi-definite}.
\end{align}

As first step in creating space hierarchy the majority of the AMG
algorithms for $Au = f$ with positive semidefinite $A$ uses a simple
filtering of the entries of $A$ and construct an $M$-matrix which is
then used to define crucial AMG components.
\begin{definition}[$M$-matrix relative] We call a matrix
  $\widetilde A$ an \emph{$M$-matrix relative} of $A$ if
  $\widetilde A$  is an M-matrix and satisfies the inequalities
\begin{equation}\label{MMrel-0}
(v,v)_{\widetilde{A}} \lesssim (v,v)_{A}, \quad\mbox{and}\quad
(v,v)_{D}\lesssim (v,v)_{\widetilde{D}}, \quad \mbox{for all} \quad v\in V,
\end{equation}
where $\widetilde D$ and $D$ are the diagonals of $\widetilde{A}$ and $A$ respectively. 
\end{definition}

A few remarks are in order: (1) We have used the term $M$-matrix to
denote semidefinite matrices, and we are aware that this is not the
precise definition. It is however much more convenient to use
reference to $M$-matrices and we decided to relax a bit the definition here with
the hope that such inaccuracy pays off by better appeal to the reader;
(2) We point out that the restricted $M$-matrix relatives
are instrumental in the definition of coarse spaces and also in the
convergence rate estimates. This is clearly seen later
in~\S\ref{2-level-theory} where we present the unified two level
theory for AMG.  (3) Often, the case is that the one sided inequality
in~\eqref{MMrel-0} is in fact a spectral equivalence.

By definition, we have the following simple but important result. 
\begin{lemma}\label{lemma-equiv}
Let $A_+$ be an
  $M$-matrix relative of $A$ and let $D$ and $D_+$ be the diagonal matrices of
  $A$ and $A_+$, respectively. If $V_c\subset V$ is a subspace, then the estimate
  \begin{equation}
    \label{VcA}
\|u - u_c\|_D^2\lesssim \|u\|_A^2    
  \end{equation}
holds for some $u_c\in V_c$, if the estimate 
\begin{equation}
  \label{VcA+}
\| u - u_c \|_{D_+}^2\lesssim \|u\|_{A_+}^2  
\end{equation}
holds. 
\end{lemma}
This result means that we only need to work on the M-matrix relative
of $A$ in order to get the estimate \eqref{VcA}.

In this section we show how to construct $M$-matrix relative to the
matrix resulting from a finite element discretization of the model
problem~\eqref{Model0} with linear elements.  We consider first an
isotropic problem with Neumann boundary condition~\eqref{like-laplace}
and isotropic $\alpha=a(x) I$.  Construction of M-matrix relatives in
the case of anisotropic tensor $\alpha(x)$ in~\eqref{Model0} is
postponed to~\S\ref{sec:aniso}.

In the rest of this section, we make the following
assumptions on the coefficient and the geometry of $\Omega$:
\begin{itemize}
\item The domain $\Omega\subset \mathbb{R}^d$ is partitioned into simplices
$\Omega=\cup_{T\in\mathcal{T}_h} T$.  

\item The coefficient $a(x)$ is a scalar valued function and 
its discontinuities are aligned with the partition $\mathcal{T}_h$. 
\item We consider the Neumann problem, and, therefore, the bilinear
  form~\eqref{Vari} is
\begin{equation}\label{auv}
\int_{\Omega}a(x) \nabla v \cdot \nabla u  = 
\sum_{(i,j)\in \mathcal{E}} (-a_{ij})\delta_eu\delta_e v 
=\sum_{e\in \mathcal{E}} \omega_e\delta_eu\delta_e v.
\end{equation}
\item 
It is well known that the off-diagonal entries of the stiffness matrix
$A$ are given by
\begin{eqnarray*}
&& \omega_e  =  -(\phi_j,\phi_i)_A = \sum_{T\supset e}\omega_{e,T}\\ 
&& \omega_{e,T} = \frac{1}{d(d-1)}\overline{a}_T |\kappa_{e,T}|\cot\alpha_{e,T}, \quad
\overline a_T= \frac{1}{|T|}\int_{T} a(x) \; dx.
\end{eqnarray*}
Here, $e=(i,j)$ is a fixed edge with end points $x_i$ and $x_j$;
$T\supset e$ is the set of all elements containing $e$;
$|\kappa_{e,T}|$ is the volume of $(d-2)$-dimensional simplex opposite to $e$ in
$T$; $\alpha_{e,T}$ is the dihedral angle between the two faces in $T$ not containing $e$.
\item Let $\mathcal{E}$ denote the set of edges in the graph defined by the
  triangulation and let $\mathcal{E}^{-}$ be the set of edges where
  $a_{ij}\ge 0$, $i\neq j$. The set complementary to $\mathcal{E}^-$ is
  $\mathcal{E}^+=\mathcal{E}\setminus\mathcal{E}^-$.  Then, with
  $\omega_e = -a_{ij}$, and, $\delta_e u=(u_i-u_j)$, $e=(i,j)$ we
  have
\begin{equation}\label{auv-do}
\int_{\Omega}a(x) \nabla v \cdot \nabla u  = 
\sum_{e\in \mathcal{E}^+} \omega_e\delta_eu\delta_ev-\sum_{e\in \mathcal{E}^-} |\omega_e|\delta_eu\delta_ev. 
\end{equation}

\item We also assume that the partitioning is such that the constant
  function is the only function in the null space of the bilinear
  form~\eqref{auv}.  This is, of course, the case when $\Omega$ is
  connected (which is true, as $\Omega$ is a domain).

\end{itemize}

The non-zero off-diagonal entries of $A$ may have either positive or
negative sign, and, usually $\mathcal{E}^-\neq \emptyset$. The next
theorem shows that the stiffness matrix $A$ defined via the bilinear
form~\eqref{auv} is spectrally equivalent to the matrix $A_+$ defined
as
\begin{equation}\label{diag-compensate}
(A_+ u,v) = \sum_{e\in \mathcal{E}^+} \omega_e (u_i-u_j)(v_i-v_j). 
\end{equation}
Thus, we can ignore any positive off-diagonal entries in $A$, or
equivalently, we may drop all $\omega_e$ for $e\in \mathcal{E}^-$.
Indeed, $A_+$ is obtained from $A$ by adding to
the diagonal all positive off diagonal elements and setting the
corresponding off-diagonal elements to zero.  This is a stronger result that we need later, 
because it gives spectral equivalence with the $M$-matrix relative $A_+$.
\begin{theorem}\label{m-matrix-plus}
  If $A$ is the stiffness matrix corresponding to linear finite
  element discretization of \eqref{Model0} with boundary conditions
  given by~\eqref{like-laplace}. Then $A_+$ is an $M$-matrix relative
  of $A$ which is spectrally equivalent to $A$. The constants of equivalence depend only on the shape regularity of the mesh. Moreover, the graph
  corresponding to $A_+$ is connected.
\end{theorem}

A simple corollary which we use later in proving estimates on the convergence rate 
is as follows.
\begin{corollary}\label{corollary-diag}
  Assume that $A$ is the stiffness matrix for piece-wise linear
  discretization of equation~\eqref{auv} and $A_+$ is the $M$-matrix
  relative defined in Theorem~\ref{m-matrix-plus}. Then the diagonal
  $D$ of $A$ and the diagonal $D_+$ of $A_+$ are spectrally
  equivalent.
\end{corollary}
\begin{proof}
For the
  diagonal elements of $A$ and $A_+$ we have
\[
[D]_j = (\phi_j,\phi_j)_A \eqqsim (\phi_j,\phi_j)_{A_+} =[D_+]_j.
\]
The equivalences written above follow directly from Lemma~\ref{m-matrix-plus}. 
\end{proof}

Corollary \ref{corollary-diag} together with Lemma~\ref{lemma-equiv}
provide a theoretical foundation for using M-matrix relative to design
AMG for finite element matrices. 

\subsection{Bibliographical notes}
We have introduced some standard notions from graph theory.  For the
reader interested in more details descriptions, we refer to classical
textbooks~\cite{2010DiestelR-aa,1985GibbonsA-aa} as general
introduction on graph theory; and to \cite{2003SaadY-aa},
\cite{2000VargaR-aa} for considerations linking graphs, sparse
matrices and iterative methods.

Our results on $M$-matrix relatives are related to the some of the
works on preconditioning by $Z$-matrices and
$L$-matrices~\cite{kraus2006algebraic,kraus2008algebraic}. They are
implicitly used in most of the AMG
literature~\cite{Ruge.J;Stuben.K.1987a} where the classical strength
of connection definition gives an $M$-matrix.  We point out that the
$M$-matrix property and the existence of $M$-matrix relative 
is often not sufficient to achieve even a two-level
uniform convergence of AMG. A typical example is a matrix which has
been re-scaled and the constant is not in the kernel of the discrete
operator anymore. In such case, the standard AMG application may fail,
and the near kernel needs to be recovered by different means, such as
the adaptive AMG processes considered in~\S\ref{s:adaptiveAMG} and the
references given there-in.

\section{Strength of connections}\label{sc:strength}
\newcommand{\cs}{s_c}
A central task in AMG is to obtain an appropriate coarse space or
prolongation.  This process is known as a coarsening process.  In a
geometric grid or, more generally, the adjacency graph of the
stiffness matrix, we need to identify vertices to be deleted from the
graph.  We need the coarsened graph still provide a good approximation
for algebraically low frequencies.  

\subsection{Basic idea and strength function}
If a subset of vertices on which an algebraically smooth vector,
say $v$, change very slowly, we only need to keep one degree of
freedom to represent $v$ in this subset.  In other words, we can
either aggregate this subset together or keep one of vertex and delete
the rest of vertices in this subset.  We say vertices in this subset
are strongly connected to each other.  The strength of connection is a
concept introduced to identify strongly connected pairs of vertices.
Roughly speaking, we say $i$ and $j$ is strongly connected if
$v_i\approx v_j$.

We imagine to coarsen the graph with two different steps: first step
is to remove some edges and the second step is to remove some
vertices.  The second step is the goal.  Let us exam how the second
step is carried out: we either (1) aggregate some neighboring vertices
together or (2) pick an MIS, denoted as $\mathcal C$, in the filtered
graph and then remove all the remaining vertices.  Using the argument
above, (1) each aggregate should only consists of strongly connected,
or (2) each one of all the deleted vertices should be strongly
connected to some point in $\mathcal C$.  To guarantee either of these
two situations, we then have to remove all the weakly connected edges
in the first step.

Let us further use some heuristic arguments to motivate how the
strength of connection should be defined.  Let $v$ be an algebraically
smooth \eqref{smoothing-1-norm}, namely
$$
\|v\|_A^2\le\epsilon \|v\|_{\bar R^{-1}}^2.
$$
Let $u=v/\|v\|_{\bar R^{-1}}$, we then have
\begin{equation}
  \label{slow}
(Au,u) \le\epsilon\Rightarrow \sum_{e=(i,j)\in \mathcal{E}}
(-a_{ij})(u_i-u_j)^2\le \epsilon.
\end{equation}
Thanks to Lemma \ref{lem:eAeD}, we can assume that $A$ is an
M-matrix, namely $-a_{ij}=|a_{ij}|$.   It follows from \eqref{slow},
we have the following observations:
\begin{enumerate}
\item A larger $|a_{ij}|$ means a smaller $(u_i-u_j)^2$;
\item An algebraically smooth error varies more slowly in the direction
  where $|a_{ij}|$ is larger;
\end{enumerate}
The observation leads to the following definition of the strength of
connection: Given a threshold $\theta>0$, we say that the vertex $j$
of the adjacency graph is $\theta$-\emph{strongly connected} to vertex
$i$ if
\begin{equation}\label{classical-strength}
-a_{ij}\geq\theta\max_{k\neq i}-a_{ik}
\end{equation}
We note that, by the definition above we may have $j$ strongly
connected to $i$, while $i$ is not strongly connected to $j$. As a
result, the adjacency graph corresponding to the matrix with the
strong connections may not be symmetric.

But our theoretical framework is given in terms of the symmetrized
operator $\bar R$, regardless if the original smoother $R$ is
symmetric or not. This is due to the fact that we used the energy
norm, namely $A$-norm, to measure the convergence rate and the
resulting convergence rate is given in terms of $\bar R$. This is the
best convergence theory we have, and we will use this theory to study
AMG algorithms. As a result, we will only consider the strength
functions that are symmetric. Associated with an SSPD matrix, a
strength function is such that 

\begin{equation}
  \label{sc-vertices}
\cs : \mathcal V\times\mathcal V\mapsto \mathbb R_{+},
\end{equation}
which is symmetric, namely $\cs(i,j)=\cs(j,i)$. 

Given a threshold $\theta>0$, we say $i$ and $j$ are $\theta$-strongly connected if 
$$
\cs(i,j)
\ge \theta.
$$
We then define the strength matrix:
\begin{equation}\label{defS}
S=\sum_{\cs(i,j)\ge\theta} e_ie_j^T.
\end{equation}

Notice that $S$ is a boolean matrix with entries equal to 0 or 1 depending on the strength of connection.

Consider non-overlapping decomposition
\begin{equation}
  \label{Vagg}
    \mathcal V=\bigcup_{i=1}^m\mathcal A_i=\bigcup_{\tilde{\mathcal A}\in \mathcal V_{\mathcal A}}\tilde{\mathcal A}, \quad 
\mathcal V_{\mathcal A} =(\mathcal A_1, \ldots, \mathcal A_m).
\end{equation}
We extend the definition of strength function
\begin{equation}
  \label{sc-patches}
    \cs: \mathcal V_{\mathcal A} \times\mathcal V_{\mathcal A} \mapsto \mathbb R_{+}.
\end{equation}
and we assume that $\cs$ is symmetric. 

Given a threshold $\theta>0$, we say $\mathcal A_i$ and
$\mathcal A_j$ are $\theta$-strongly connected
to each other if 
$$
\cs(\mathcal A_i,\mathcal A_j)
\ge \theta.
$$
An example of such a strength function is given in~\eqref{sc-agg}.

In the AMG literature, a number of heuristics have been proposed for
identifying strong connections, particularly when considering
discretizations of anisotropic equations.  In general, the strength of
connection is a notion that is difficult to address theoretically or
to relate it with the convergence rate of an algorithm.  We refer to
the classical papers and monographs mentioned
in~\S\ref{s:coarsening-biblio} for further discussions on related
issues.  Current trends in AMG development aim to re-evaluate the role
of the classical definition of the strength of connection.

We would finally like to comment that the strength of connection is
used to define the sparsity of $P$ and it is crucial, for example, in
the proof of the convergence of the two-level method for
discretizations of elliptic equation with jump coefficients.  Choosing
the ``right'' sparsity of $P$ is crucial as a denser $P$ would lead to
a better approximation from a coarser space; and a sparser $P$ would
lead to a less expensive algorithm.

In the rest of section, we discuss different definitions of strength of
connections:
\begin{enumerate}
\item classical AMG;
\item lean AMG;
\item local-optimization based.
\end{enumerate}

\subsection{Classical AMG} \label{s:strength}
With the above motivation, we define the strength function as follows:
\begin{eqnarray}
\cs(i,j)& =&\frac{-a_{ij}}{\min\bigg(\max_{k\ne i}(-a_{ik}),
  \max_{k\ne j}(-a_{jk})\bigg)} \nonumber
\\
&=&\frac{a_{ij}}{\max\bigg(\min_{k\neq i}a_{ik}, \min_{k\neq j}a_{jk}\bigg)}. 
\label{sc-cAMG}
\end{eqnarray}
The definition \eqref{sc-cAMG} is symmetrized version of strength
function used in the classical AMG literature (see \eqref{classical-strength}).

The following definition is also commonly used in classical AMG algorithms
\begin{equation}\label{def:strong-connection-2}
\cs(i,j)=\frac{|a_{ij}|}{\min\bigg(\frac{1}{|N(i)|}\sum_{k\neq i}|a_{ik}|, \frac{1}{|N(j)|}\sum_{k\neq i}
|a_{jk}|\bigg)}.
\end{equation}
Again, this is a symmetrized version of strength functions used in the
AMG literature. 

Finally we may also have the following two definitions which are based on
Cauchy-Schwarz for SSPD matrices:
\begin{equation}\label{strong-connection-1}
s_1(i,j)=\frac{|a_{ij}|}{\sqrt{a_{ii}a_{jj}}} 
\end{equation}
 and 
\begin{equation}\label{strong-connection-2}
s_2(i,j)=\frac{-2a_{ij}}{a_{ii}+a_{jj}} 
\end{equation}

Note that the definition \eqref{sc-cAMG} (which is mostly associated
with Classical AMG) ignores all the non-negative entries of the
stiffness matrix $A=(a_{ij})$.

\subsection{Local-optimization strength function}
Suppose that each pair of the index $\{i, j\}\subset \{1, \dots, n\}$ is associated with a space $V_{ij}$, which is not necessarily a subspace of $V$. Assume now we have two operators: $A_{ij}: V_{ij}\mapsto V_{ij}'$ which is symmetric positive, semi-definite and $D_{ij}: V_{ij}\mapsto V_{ij}'$ which is SPD.

Given a number $k_{ij}< \dim V_{ij}$, we define a coarse space $V_{ij}^c\subset V_{ij}$ as follows
\begin{equation*}
    V_{ij}^c:=\operatorname{span}\{\zeta_{ij}^{(k)}, k=1: k_{ij}\},
\end{equation*}
where $\zeta_{ij}^{(k)}$ is the eigenvector corresponds to the $k$-th smallest eigenvalue of $D_{ij}^{-1}A_{ij}$.

We denote $Q_{ij}: V_{ij}\mapsto V_{ij}^c$ to be the orthogonal projection with respect to $(\cdot, \cdot)_{D_{ij}}$.
Motivated by \eqref{muj}, we define the strength function $\cs$ as follows
\begin{equation}\label{local-opt-sc}
    \cs(i, j):=\left(\sup_{v\in V_{ij}}\frac{\|(I-Q_{ij})v\|_{D_{ij}}^2}{\|v\|_{A_{ij}}^2}\right)^{-1}.
\end{equation}

A special case of the above definition is introduced in AGMG. Suppose now we have  a set of aggregates $\{\mathcal A_1, \dots, \mathcal A_J\}$. We fix a pair $\{i, j\}\subset \{1, \dots, J\}$, and define 
$$G=\mathcal A_i\bigcup \mathcal A_j.$$
We then define by $V_{ij}$ the restriction of $V$ on $G$ as follows 
\begin{equation}
    V_{ij}:=\{\left.v\right|_{G}: v\in V\},
\end{equation}
where 
\begin{equation*}
    \left. v\right|_{G}(x)= \begin{cases}
        v(x), & \text{ if } x\in G,\\
        0, & \text{ if } x\notin G.
    \end{cases}
\end{equation*}
We denote the restriction of $A$ and $D$ on $G$ by $A_{ij}$ and $D_{ij}$ respectively.
We then choose $k_{ij}=1$ and define the local coarse space $V_{ij}^c$
\begin{equation*}
    V_{ij}^C= \Span\{\zeta_G\}, \quad \zeta_G=\zeta_{ij}^{(1)}.
\end{equation*}
and the orthogonal projection $Q_{ij}: V_{ij}\mapsto V_{ij}^C$ with respect to $(\cdot, \cdot)_{D_{ij}}$, namely,
\begin{equation*}
    Q_{ij}v = \frac{(v, \zeta_G)_{D_{ij}}}{\|\zeta_G\|_{D_{ij}}^2}\zeta_G.
\end{equation*}
The strength function based on aggregation is defined as follows:
\begin{equation}\label{sc-agg}
    \cs(i, j):=\left(\sup_{v\in V_{ij}}\frac{\|(I-Q_{ij})v\|_{D_{ij}}^2}{\|v\|_{A_{ij}}^2}\right)^{-1}.
\end{equation}

Another example is choosing $V_{ij}=\mathbb{R}^2$ and 
\begin{equation*}
    A_{ij}=  \begin{pmatrix}
        a_{ii} & a_{ij}\\
        a_{ij} & a_{jj}
    \end{pmatrix}, \quad
        D_{ij}= \begin{pmatrix}
        a_{ii} & 0\\
        0 & a_{jj}
    \end{pmatrix}.
\end{equation*}
    We choose $k_{ij}=1$ and define the coarse space  $V_{ij}^c\subset V_{ij}$ as follows
\begin{equation*}
    V_{ij}^c=\operatorname{span}\left\{ \begin{pmatrix}
        1\\
        1
    \end{pmatrix}  \right\}.
\end{equation*}
    By a direct computation, we have the strength function defined in \eqref{local-opt-sc} is 
\begin{equation}\label{scij}
\cs(i, j)= \frac{1-s_1^2}{1-s_2}, \quad s_1=\frac{|a_{ij}|}{\sqrt{a_{ii}a_{jj}}}, \quad s_2=-\frac{2a_{ij}}{a_{ii}+a_{jj}}.
\end{equation}
We notice that $s_1\ge s_2$ and hence
    \begin{equation}\label{e:s1s2}
        \cs(i, j) \le 1+s_2 \le 1+s_1.
    \end{equation}

We would like to point out that the strength function given by
\eqref{scij} is obtained by using the theory in
\S\ref{sec:unifiedAMG}, while the other strength functions such as
\eqref{sc-cAMG} are obtained by heuristic considerations. 

\subsection{Lean AMG}
Instead of using the absolute value of matrix entries as the criteria
to determine if two points are strongly coupled, Lean AMG use the
\emph{affinity} to measure the strength of connections, which is based
on the following heuristic observation:

\emph{Given a vector $v$, if $(i, j)$ is a strong connected pair of vertices, after several relaxation on $v$, namely
\begin{equation*}
    v\leftarrow (I-RA)^{\nu}v,
\end{equation*}
the values of $v_i$ and $v_j$ should be close. 
}

In Lean AMG, we generate $K$ Test Vectors (TVs). Each TV is the result of applying $\nu$ GS relaxation sweeps to $Ax=0$, staring from randomly generated vectors $x^{(1)}, \dots, x^{(K)}\in \mathbb R^n$. And we denote 
\begin{equation}
    X_{n\times K}:= \begin{pmatrix}
        X_1^T\\
        \vdots\\
        X_n^T
    \end{pmatrix}= (I-RA)^{\nu}
    \left(x^{(1)}~\dots~x^{(K)}\right).
\end{equation}
Here $X_i^T$ is the $i$-th row of $X_{n\times K}$. The strength
function for the Lean AMG is then defined as follows
\begin{equation}\label{sc-Lean}
    \cs(i,j):=\frac{|(X_i, X_j)|^2}{(X_i, X_i)(X_j, X_j)}
\end{equation}

\subsection{Bibliographical notes}
Classical algorithms for determining strength of connection are found
in~\cite{1stAMG,Stuben.K.1983a,Brandt.A;McCormick.S;Ruge.J.1985a,Ruge.J;Stuben.K.1987a,Briggs.W;Henson.V;McCormick.S.2000a}.
The original measure of strength of connection given
in~\cite{ruge5110final,mccormick1989algebraic,Brandt.A;McCormick.S;Ruge.J.1985a,ruge1983algebraic}
is nonsymmetric, but for theoretical considerations, which only depend
on the symmetrized smoothers, it suffices to use the slightly more
restrictive, but symmetric versions of the strength of connection.

Some extension of these classical algorithms for defining strong
connections are based on different measures for connectivity and
distance such as \emph{measure of importance}, and \emph{algebraic
  distance}. Details are found in \cite{Ruge.J;Stuben.K.1987a} and
\cite[Appendix~A]{Trottenberg.U;Oosterlee.C;Schuller.A.2001a}.

The strength of connection measures has had little theoretical backing
in the past. The results developed in this section, such as local
Poincar\'{e} inequalities and especially the strength
functions~\eqref{e:s1s2} shows that such heuristics are reasonable and
their choice can be motivated by theoretical results. 

For aggregation AMG, typically, a symmetric strength of connection
function as defined in~\cite{1996VanekP_MandelJ_BrezinaM-aa} is
used. Some recent aggregation algorithms also define strength of
connection based on sharp theoretical results and use the local
two-level convergence rate~\eqref{local-opt-sc} as a
measure~\cite{Notay.Y.2010b,Napov.A;Notay.Y.2012a}.

For aggregations based on matching (aggregates of size $2$) the the
``heavy edge'' matching algorithms in~\cite{Karypis.G;Kumar.V.1998a}
corresponds to a strength of connection function selecting aggregates
depending on the edge weight in the adjacency graph. Some recent
works~\cite{Livne.O;Brandt.A.2012a} use strength of connection
functions based on the size of the entries in the Gramm matrix formed
by a set of smoothed test vectors.

\section{Coarsening strategies}\label{sc:connection}
Once the smoother is identified, the central task of an AMG method is
to identify a sequence of coarse spaces, in functional terminology, or
equivalently, in an algebraic setting, to identify a sequence of
prolongation matrices.  This procedure is known as a
``coarsening'' procedure.   In this section, we will discuss how this
coarsening procedure. 

Roughly
speaking, given the equation \eqref{Au=f}, namely $Au=f$, on a vector
space $V$, the goal is to find a subspace $V_c\subset V$ such that the
solution $u_c\in V_c$ of the following ``coarsened'' problem:
\begin{equation}
  \label{Ac}
A_cu_c=f_c, \quad A_c=\imath_c'A\imath_c, f_c=\imath_c'f  
\end{equation}
would provide a good approximation to the original solution $u\in V$.
More specifically, the solution of $u_c$ of \eqref{Ac} would provide
good approximation to those ``algebraically smooth'' components of the
error for which the given smoother does not converge well. 

\subsection{Motivations}
In some sense, ``coarsening'' is done almost everywhere in numerical
analysis.  For example, the finite element equation \eqref{vph} can be
viewed as a coarsened equation of the original equation
\eqref{Vari}. In this case the finite element space $V_h$ is a
coarsened subspace of $V=H_0^1(\Omega)$.  

It is therefore informative for us to exam how a finite element space
is constructed in general.  While there are many different ways to
construct finite element spaces, mathematically speaking, the most
convenient approach is through the use of ``degrees of freedom'' which
refers to a basis of a dual space.  More specifically, in a finite
element discretization, the finite element space $V_h$ is obtained by
specifying the dual space $V_h'$ (the so called space of degrees of
freedom) first.  For linear finite elements, $V_h'=\{\psi_i:i=1:n_h\}$
is such that
$$
\psi_i(v)=v(x_i^h).
$$
With such degrees of freedom (nodal evaluation), we then find a dual
basis $\{\phi_i: i=1:n_h\}$ which are piecewise linear functions such
that
$$
\psi_i(\phi_j)=\delta_{ij}, \quad 1\le i, j\le n. 
$$
The finite element space $V_h$ is then defined by 
$$
V_h=\Span \{\phi_i: i=1:n_h\}.
$$
In fact, mathematically speaking, all existing finite element spaces
$V_h$ can be obtained by first constructing $V_h$.  This is the
approach taken in the classical literature on finite element
methods~, c.f. \cite{2002CiarletP-aa}.

Similar to finite element method, we will therefore focus on
techniques for constructing a coarse space $V_c\subset V$ by first
identifying its dual basis $V_c'$.  Such an approach is rather
abstract, but it turns out to be more intrinsic, more general, and, in
fact, more commonly used in AMG literature (implicitly).

It is interesting to note that we rarely use the word ``coarsening''
in the design of a geometric multigrid method.  Instead, we use a
``refinement'' procedure to define a sequence of nested spaces.  As an
example, Figure~\ref{fig-refinement} shows a uniformly refined
triangular grid used for discretization of Poisson equation with with
linear finite elements.

In AMG, we do not have the luxury of this hierarchy of spaces given by
a geometric refinement.  Instead, we carry out a reverse-engineering
process of the refinement process, namely {\it coarsening.}

\begin{figure}[!htb]
\centering
\includegraphics*[width=0.3\textwidth]{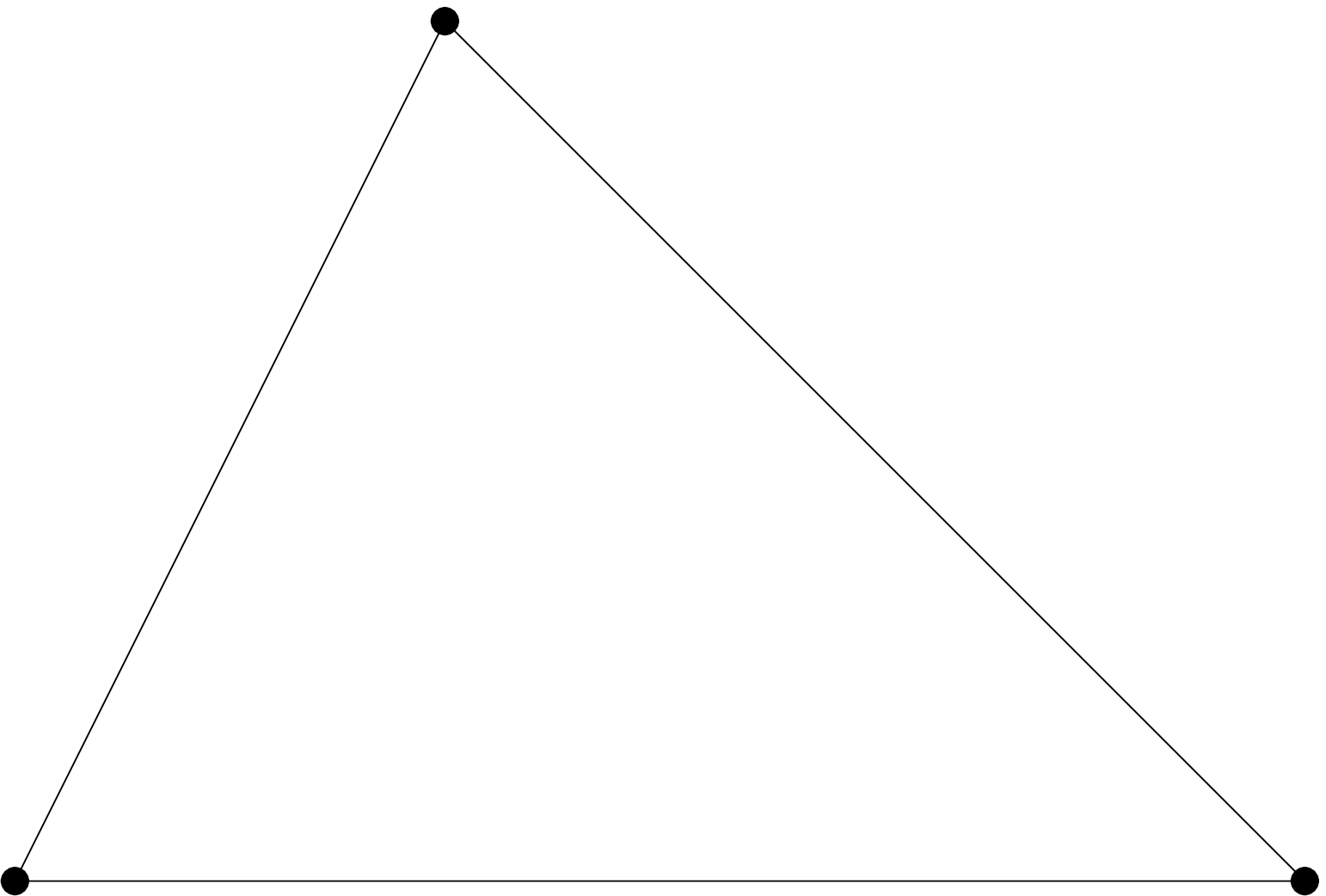}
~~\includegraphics*[width=0.3\textwidth]{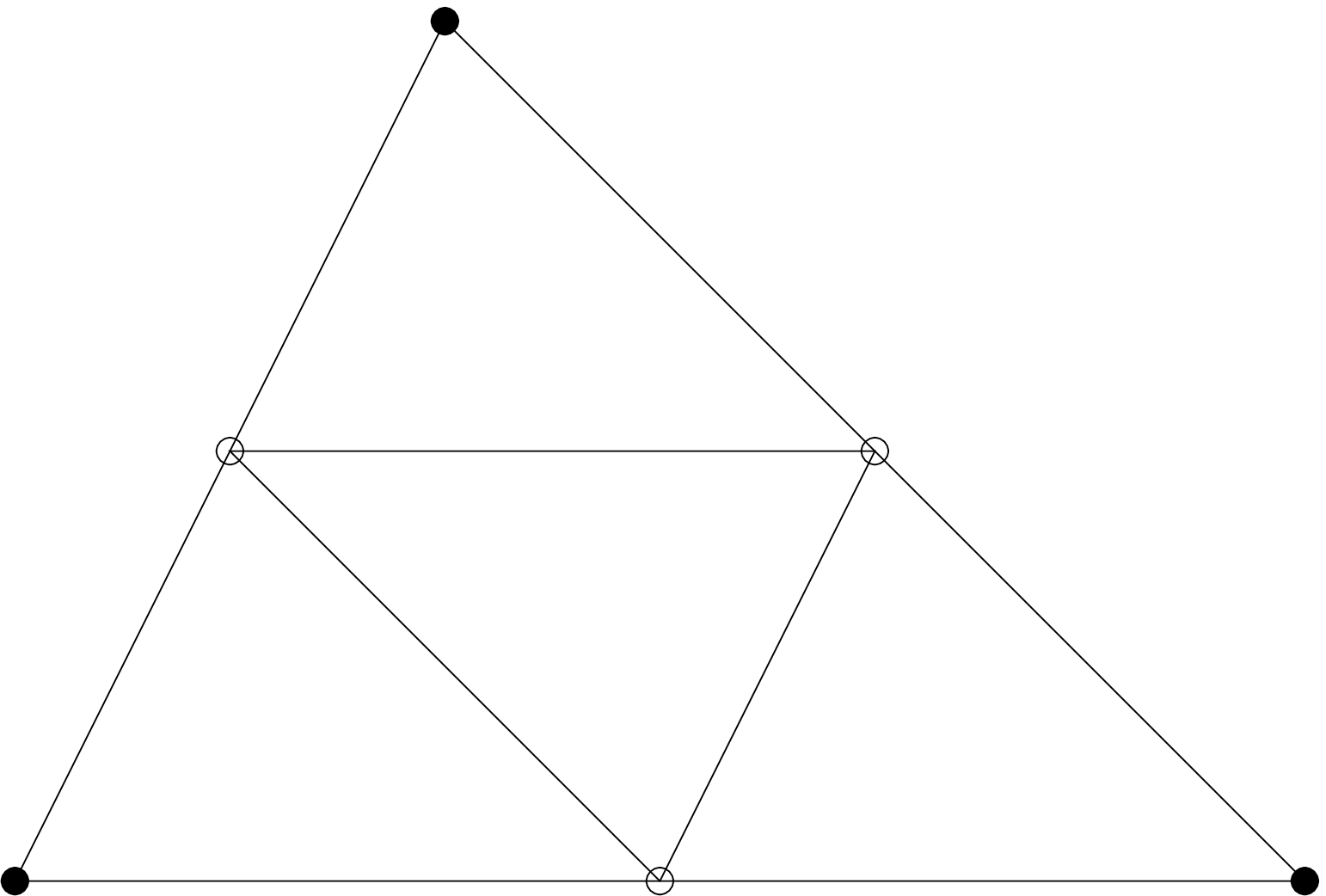}
~~\includegraphics*[width=0.3\textwidth]{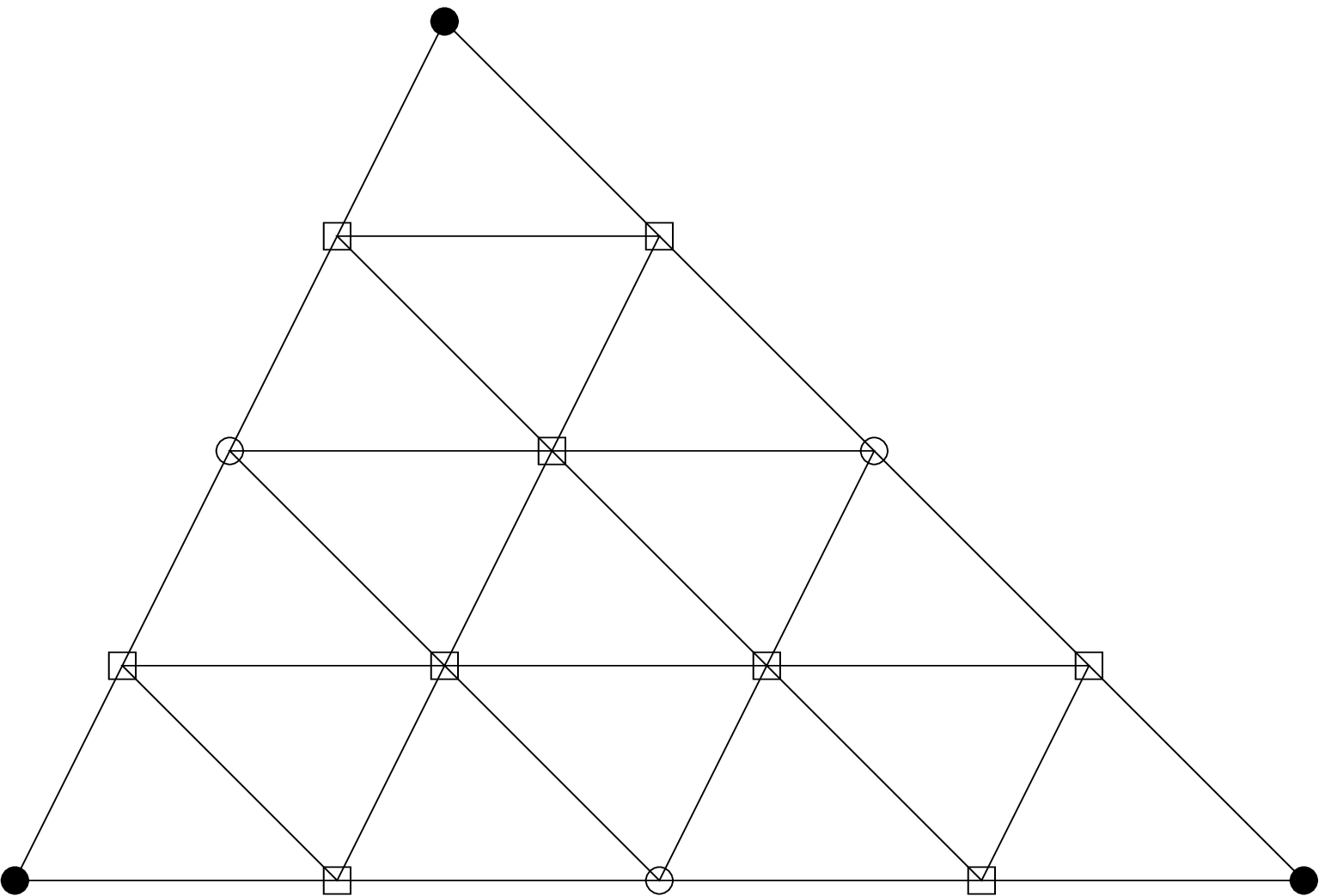}
\caption{Regular refinement of a coarse grid
  element. $\bullet$-coarsest level vertices, $\circ$-first level of
  refinement, $\Box$-second level of
  refinement.\label{fig-refinement}}
\end{figure}

Conceivably, we can also use such an inverse-engineering process to
recover the geometric multigrid method by starting from the finest
geometric grid, at least for some special case.  For example, if all
triangles in the triangulation as shown in Figure \ref{fig-refinement}
are of acute type and then the graph corresponding to the mesh is also
an adjacency graph of the stiffness matrix.  It is clear that the
coarse grid vertices, i.e. the set $\mathcal{C}$ (known from the
refinement), is a maximal independent set (MIS) of vertices in the
graph corresponding to the refined mesh. In the simplest case, the
degrees of freedom associated with the coarse grid (the dual basis for
$V_c$) precisely corresponds to the MIS.  Such an observation is
explored later in constructing algorithms in the framework of
Classical AMG for a selection of coarse grid vertices via the MIS
algorithm in \S\ref{s:MIS}. For further examples on relations between
geometric AMG and GMG coarsening we refer to~\S\ref{s:amg-from-gmg}.

The above reverse-engineering process for GMG gives some hint how a
coarsening process needs to be done in AMG, but we need to study the
process in a broader framework and more importantly we will use the
degrees of freedom, namely basis of dual basis, to obtain coarse
spaces.  In the above GMG example, each grid point in the geometric
grid corresponds exactly to one degree of freedom.  But this is not
always the case in applications.

\subsection{Basic approach} 
By mimicking the construction of finite element space as described
above, given a linear algebraic system of equation $Au=f$, we adapt a
coarsening strategy consisting of the following steps:

\begin{enumerate}
\item We consider the adjacency graph $\mathcal G(A)$ of the
  coefficient matrix $A$.  Based on certain strength function $s_c$ as
  described in \S\ref{sc:strength}, we remove the weakly connected
  edge in $\mathcal G(A)$, namely drop certain entries in $A$ to form
  a filtered matrix $\tilde A$.
\item We carry out one of the following two:
  \begin{enumerate}
  \item Classical AMG: Find an MIS of $\mathcal G(\tilde A)$ to form
    the set of coarse vertices $\mathcal C$.  Then remove the rest,
    namely $\mathcal V\setminus\mathcal C\equiv\mathcal F$.
\item Aggregation AMG: Agglomerate using some greedy algorithm: pick a point and
  agglomerate its neighbors and go from there. 
  \end{enumerate}
\item We obtain a subset of  d.o.f. obtained from the above steps,
  namely $(V')_c$
\item We use $(V')_c$ to define a high frequency space $V_{hf}=(V')_c^{(0)}$ by 
\begin{equation}
  \label{polarVhf}
V_{hf}=(V')_c^{(0)}:=\{v\in V: \langle g, v\rangle=0, \quad \forall
g\in (V')_c \}.
\end{equation}
\item Find a tentative coarse space $W_c$ such that 
  \begin{equation}
    \label{VhfWc}
V=V_{hf}\oplus W_c    
  \end{equation}
\item Apply certain postprocess (such as smoothing) to $W_c$ to obtain $V_c$:
$$
V_c=SW_c
$$
\item Using $V_c$, or equivalently the prolongation $P$, we form
  coarse matrix $A_c=P^TAP$. 
\item We then repeat the above steps to $A_c$ in place of $A$ until
  a desirable coarsest level is reached. 
\end{enumerate}

\subsubsection{Construction of $(V')_c$}
Given $A\in \mathbb R^{n\times n}$ and the associated graph $\mathcal G=(\mathcal V, \mathcal E)$, we proceed as follows:
\begin{enumerate}[1.]
\item Form the following two non-overlapping decompositions
$$
\mathcal V=\mathcal C\cup \mathcal F,\quad
\mathcal C = \bigcup_{i=1}^{n_c}\mathcal A_i.
$$
\item Identify $ (V')_c=\Span\{\dof_i:i=1:n_c\}\subset V'. $
\end{enumerate}

Here are three examples which will be discussed in detail in later
sections (c.f. \S\ref{sec:energy-min}, \S\ref{s:classical-amg}, and
\S\ref{s:agmg}).
\begin{description}
\item[Aggregation AMG]: $\mathcal F=\emptyset$ and 
  \begin{equation}
    \label{AggNi}
\dof_i(v)=\langle v\rangle_{\mathcal A_i}\equiv \frac{1}{|\mathcal
  A_i|}\sum_{j\in \mathcal A_i} \psi_{j}(v) = \frac{1}{|\mathcal
  A_i|}\sum_{j\in \mathcal A_i}v_j , \quad i=1:n_c. 
  \end{equation}

\item[Classical AMG]: $\mathcal F\neq\emptyset$ and 
  \begin{equation}
    \label{ClassicalNi}
      \mathcal A_i=\{k_i\}, \quad \dof_i(v)=\psi_{k_i}(v)=v_{k_i}.    
  \end{equation}
In this case, $\mathcal C$ usually consists of disconnected vertices. 
\item[Energy-min AMG]: $\mathcal F\ne \emptyset $ and $\mathcal A_i$ are aggregates
    \begin{equation}
\dof_i(v)=\langle v\rangle_{\mathcal A_i}\equiv \frac{1}{|\mathcal
  A_i|}\sum_{j\in \mathcal A_i} \psi_{j}(v) = \frac{1}{|\mathcal
  A_i|}\sum_{j\in \mathcal A_i}v_j , \quad i=1:n_c. 
    \end{equation}
\end{description}

\subsubsection{Construction of $V_c$}

Given coarse grid degrees of freedoms $(V')_c\subset V'$,
we define $V_{hf}$ as in \eqref{polarVhf}.
The following lemma shows how to find a subspace $W_c$, a ``pre-coarse space", such that
$V=V_{hf}\oplus W_c$.
\begin{lemma}\label{lemma:direct-sum} If $\phi_{k,c}$, $k=1,\ldots,
n_c$ are elements of $V$ such that the ``Gramm'' matrix
    \[ G=(G_{km})=(\dof_m(\phi_{k,c})),
\] 
    is nonsingular, then we have
\[ V=V_{hf} \oplus W_c, \quad W_c =
\operatorname{span}\{\phi_{k,c}\}_{k=1}^{n_c}.
\]
\end{lemma}

\begin{proof} We first show $W_c\cap V_{hf}=\{0\}$. In fact, if
$v:=\sum_{k=1}^{n_c}(\tilde v)_k\phi _{k,c}\in W_c\cap V_{hf}$, then
we have
$$
    0=\dof_m(v)=\sum_{k=1}^{n_c}\dof_m(\phi_{k,c})=\sum_{k=1}^{n_c}(\tilde v)_kG_{km}, \quad m=1, \dots,
n_c.
$$
Hence, 
$$
G\tilde v=0,
$$
and since by assumption, $G$ is nonsingular, we must have $v=0$. Therefore, 
$W_c\cap V_{hf}=\{0\}$.

Next, for any $v\in V$, we define $w_c\in W_c$ as
\[ w_c = \sum_{k=1}^{n_c} (\widetilde{w}_c)_k\phi_{k,c}, \quad
\widetilde{w}_c= G^{-1}
    \begin{pmatrix} \dof_1(v) \\ \vdots\\
        \dof_{n_c}(v).
\end{pmatrix}.
\] It is immediate to check that
\[ \dof_m (v-w_c) = 0, \quad m=1,\ldots
n_c.
\] This means that $(v-w_c)\in V_{hf}$. This proves that $v=w_c +
(\underbrace{v-w_c}_{\in V_{hf}})$ and completes the proof.
\end{proof}

The above lemma gives us a way to construct a subspace $W_c$ such that
$V=V_{hf}\oplus W_c$.

\begin{lemma}
If the coarse grid degrees of freedom are defined as
\begin{equation*}
    \dof_k: = \sum_{j\in \mathcal A_k}\alpha_j\psi_j, \quad k=1, \dots, n_c,
\end{equation*}
    where $\sum_{j\in \mathcal A_k}\alpha_j=1$ and $\{\psi_j\}$ is dual basis of $\{\phi_j\}$.
    If $\{\phi_{k,c}: k=1,\ldots, n_c\}$ are elements of $V$ such that the ``Gramm'' matrix
\[ 
    G=(\dof_l(\phi_{k, c}))=I,
\]
then, we have 
    
\begin{equation}\label{Wcbasis}
    \phi_{k, c}=\sum_{j\in \mathcal A_k}\phi_j+v_{hf}, \quad v_{hf}\in V_{hf}.
\end{equation}

\end{lemma}
\begin{proof}
    We fix a $k$ and consider the subset $W_k\subset V$ such that  
    \begin{equation*}
        W_k:=\{v\in V: N_k(v)=1 \text{ and } N_l(v)=0, \forall l\neq k\}
    \end{equation*}
    Pick any $v_1, v_2\in W_k$, we have
    \begin{equation*}
        N_l(v_1-v_2)=N_l(v_1)-N_l(v_2)=0, \quad \forall l=1, \dots, n_c,
    \end{equation*}
    and this shows that $(v_1-v_2)\in V_{hf}$.

    Furthermore, if we define $v_k^c=\sum_{j\in \mathcal A_k}\phi_j$, then
    \begin{equation*}
        N_k(v_k^c)=\sum_{j\in\mathcal A_k}\sum_{i\in \mathcal A_k}\alpha_j (\psi_j, \phi_i)=\sum_{j\in \mathcal A_k}\alpha_j=1.
    \end{equation*}
    and $N_l(v_k^c)=0$ for all $l\neq k$ since $\mathcal A_k$ and $\mathcal A_l$ have no overlap.
    We then have $v_k^c\in W_k$ and hence, 
    \begin{equation*}
        W_k=\{v_k^c+ v_{hf}: v_{hf}\in V_{hf}\}.
    \end{equation*}
\end{proof}

Next, we describe how one can construct the
coarse space $V_c$ using $W_c$.

\begin{lemma}\label{lem:basis-in-vc} Assume that $V=V_{hf}\oplus
W_c$ and $\varphi_{1,c},\ldots \varphi_{n_c,c}$ is a basis in $W_c$. Then
$\phi_{k,c} = S\varphi_{k,c}$, $k=1,\ldots,n_c$ are linearly
independent if $S: V\mapsto V$ satisfies one of the following
conditions
\begin{enumerate}[1.]
\item  $S$ maps a linear independent set in $W_c$ into a linear independent
  set in $V$. 
\item $S$ is invertible;
\item $S=I-Q_{hf}$ where $Q_{hf}: V\mapsto V_{hf}$.
\end{enumerate}
As a result, we have 
\begin{equation*}
    V=V_{hf}\oplus V_c, \quad V_c=\operatorname{span}\{S\varphi_{k,c}: k=1, \dots, n_c\}.
\end{equation*}
\end{lemma}
\begin{proof}   We only need to prove the third case as the first two  cases
are trivial. If $\phi_{k,c}$ were linearly dependent, 
there would be a linear combination of $\{\phi_{k,c}\}_{k=1}^{n_c}$ which
vanishes. Equivalently, this means that there exists $w_c\in W_c$ such
that $(I-Q_{hf}) w_c=0$, which implies that $w_c\in V_{hf}$. Since
$V_{f}\cap W_c=\{0\}$ we have that $w_c=0$ and the only vanishing
linear combination in $\operatorname{span}\{\phi_{k,c}\}_{k=1}^{n_c}$
is the trivial one and this completes the proof.
\end{proof}
\begin{remark}~

\begin{enumerate}[1.]
\item For smoothed aggregation, $S=I-\omega D^{-1}A$ for some properly
  chosen $\omega$ so that $S$ is nonsingular, or maps a special
  linearly independent set of vectors to linearly independent set of
  vectors (see~\S\ref{s:agmg}). 
\item For classical AMG with ideal interpolation, $S=I-Q_{hf}$ and
  $Q_{hf}$ is $A$-othogonal projection (see~\S\ref{s:classical-amg}).
\item For classical AMG with standard interpolation, $S=I-Q_{hf}$ and
  $Q_{hf}$ is an approximation to the ideal interpolation
  (see~\S\ref{s:classical-amg}).
\end{enumerate}
\end{remark}

To give a summary of above discussions, the coarsening algorithms in
AMG are methods for determining the coarse grid degrees of freedom (or
coarse grid variables).  Such algorithms are based on selecting
degrees of freedom associated with subsets of vertices in the
adjacency graph that corresponds to the matrix $A$ or to the strength
matrix $S$ as is done in geometric coarsening, when the hierarchy of
meshes or adjacency graphs is known.

\subsection{Two basic coarsening algorithms}\label{s:basic-coarsen}
In the next two sections we present typical algorithms for finding the
coarse grid degrees of freedom. Each such degree of freedom is
associated with a vertex or a subset of a graph. Two types of
algorithms are distinguished: Classical AMG algorithm pick coarse grid
degrees of freedom that correspond to a maximal independent set of
vertices in adjacency graph of the strength matrix; and aggregation based AMG algorithms which use
splitting of adjacency graph of the strength matrix in connected subgraphs.

\subsubsection{A maximal independent set (MIS) algorithm} \label{s:MIS} We
present here a simple ``greedy'' maximal independent set (MIS)
algorithm, which has been used in the Classical AMG algorithms to
identify coarse grid degrees of freedom. Given adjacency graph of the strength matrix, the simple greedy MIS algorithm
is as follows.

\begin{algorithm}\caption{MIS\label{a:MIS}}
\begin{enumerate}[1.]
\item \textbf{Set} $C=\emptyset$, $i\leftarrow 1$.
\item \textbf{If} $i$ and all its neighbors are not visited, \textbf{then} set $C=C\cup\{i\}$ and mark $i$ and all vertices in $N(i)$ as visited. 
\item \textbf{If} all vertices are visited, \textbf{then} output $C$ and stop; \textbf{else} 
  set $i\leftarrow i+1$ and go to 2.
\end{enumerate}
\end{algorithm}
\begin{remark}
  We note that the MIS Algorithm~\ref{a:MIS} recovers the geometric
  coarsening if the vertices are visited in an order so that coarse
  grid vertices are ordered first and all the connections in
  Figure~\ref{fig-refinement} are strong. This is obvious, but,
  nevertheless, shows that the geometric coarsening can sometimes be
  recovered by an algebraic algorithm.
\end{remark}

Let us point out that for finite element stiffness matrices obtained
via adaptive refinement algorithms, the hierarchy of vertices is
naturally included in the refinement procedure.  For regular
refinement such choice of MIS is illustrated in
Figure~\ref{fig-refinement}.

\subsubsection{An aggregation algorithm} 
The class of algorithms, known as aggregation algorithms refer to
splitting of adjacency graph of the strength matrix as a union of connected subgraphs.  Let
$\{\mathcal{V}_k\}_{k=1}^{n_c}$ be a non-overlapping splitting of the
set of vertices
\[
\mathcal{V}=\cup_{k=1}^{n_c}\mathcal{V}_k, \quad
\mathcal{V}_j\cap\mathcal{V}_k=\emptyset, \quad \mbox{for}\quad j\neq
k. 
\]
We then define 
\begin{equation}\label{eq:edgesk}
\mathcal{E}_k = \{(l,m)\in\mathcal{E}~\big|~ l\in
\mathcal{V}_k~ \mbox{and}~ m\in\mathcal{V}_k\},
\end{equation}
to be the set of edges associated with $\mathcal{V}_k$. 

An aggregation can be done in many different and some very
sophisticated ways. In general all combinatorial graph partitioning
algorithms can be used for aggregation. We, however, will not consider
in details such algorithms, but rather, we provide here (Algorithm~\ref{alg:aggregation}) the basic and
most important example of greedy aggregation algorithm.

\begin{algorithm}\caption{Greedy aggregation algorithm\label{alg:aggregation}}
\textbf{Input:} Graph $\mathcal{G}$ with $n$ vertices;
\textbf{Output:} $\mathcal{V}=\cup_{k=1}^{n_c}\mathcal{V}_k$, and
$\mathcal{V}_k\cap \mathcal{V}_j = \emptyset$ when $k\neq j$.
\begin{enumerate}[1.]
\item Set $n_c=0$ and for $k=1:n$ do:
    \begin{enumerate}[a.]
\item If $k$ and all its neighbors have not been visited, then: (a) we
set $n_c=n_c + 1$; (b) label with $n_c$ the subgraph whose vertices
are $k$ and the neighbors of $k$; and (c) mark $k$ and all its
neighbors as visited.
\item If at least one neighbor of $k$ has been visited, we continue
the loop over the vertices.
\end{enumerate}
\item Since after this procedure there might be vertices which do not
belong to any aggregate (but definitely have a neighboring aggregate),
we add each such vertex to a neighboring aggregate and we pick the one
which has minimal number of vertices in it.
\item The algorithm ends when all vertices are in a subset.
\end{enumerate}
\end{algorithm} Such algorithm can be recursively applied to provide a
multilevel hierarchy of aggregates.

\subsubsection{Aggressive coarsening}
The extended strong connections and the corresponding strength
operator are used to construct coarse spaces of
smaller dimension. This procedure is also known as aggressive
coarsening.  We recall the definition of a path in the graph given
in~\S\ref{s:ugraph} and all our considerations are on the adjacency
graph $\mathcal{G}(S)$ of the strength matrix
$S\in \mathbb{R}^{n\times n}$ defined in~\eqref{defS}.  Aggressive
coarsening refers to selection of coarse grid vertices as independent
set in the adjacency graph corresponding to the strength operator are
at distances larger than $2$. 

\begin{definition}[Strong connection along a path]
\label{def:extend-strong-connection-1}  
A vertex $i$ is said to strongly
connect to a vertex $j$ along a path of length $l$ if there exits a
    path $(k_0, k_1, \dots, k_l)$  in $\mathcal{G}(S)$, such that 
    $k_0=i$, $k_l=j$, and $\cs(k_m, k_{m+1})\ge \theta$, $m=0, 1, \dots, l-1$.
\end{definition}
Next definition is related to the number of strongly-path-connected vertices. 
\begin{definition}[$(m,l)$-strong connection]
\label{def:extend-strong-connection-2}  
For given integers $m>0$ and $l>0$, a vertex $i$ is $(m,l)$-strongly
connected to a vertex $j$ if and only if $i$ strongly connect to $j$
along at least $m$ paths of length $l$ (per
Definition~\ref{def:extend-strong-connection-1}).
\end{definition}

An aggressive coarsening algorithm generates a MIS using
Algorithm~\ref{a:MIS} for the graph
$\mathcal{G}_{m,l}=(\mathcal{V},\mathcal{E}_{m,l})$ with a set of
vertices $\mathcal{V}=\{1,\ldots,n$ and set of edges
$\mathcal{E}_{m,l}$ defined as
\begin{equation}\label{e:eml}
\mathcal{E}_{m,l} := \{(i,j)\;\big|\;\;\mbox{$i$ is $(m,l)$ strongly connected to $j$} \}
\end{equation}

As is well known~\cite{2010DiestelR-aa} 
$(S^{l})_{ij}$ is nonzero if and only if there is a path of length
$\le l$ between $i$ and $j$.  An aggressive coarsening exploits this
property and is an algorithm which selects a set of coarse grid
degrees of freedom corresponding to vertices in the graph which are at
graph distance larger than $l$. It uses the adjacency graph
$\mathcal{G}(S^l)$ of $S^{l}$ in place of $\mathcal{G}(S)$ in an
aggregation or a MIS algorithm.

As an example, let us consider aggressive coarsening with $l=2$ and
$m=1$. The set of coarse grid degrees of freedom is obtained by
applying the standard MIS Algorithm~\ref{a:MIS} twice: (1) we find a
MIS in $\mathcal{G}(S)$ and then obtain a set of coarse grid degrees
of freedom $C$ (these are at graph distance at least $2$); (2) Apply
the MIS algorithm for a second time on the graph with vertices the
$C$-points and edges between them given by the strength operator corresponding to
$S^{2}$. 

Similarly, for aggregation, an aggressive coarsening corresponds to
applying the aggregation Algorithm~\ref{alg:aggregation} recursively,
or applying it directly to the graph corresponding to $S^l$ for a given $l$. 

\subsection{Adaptive coarsening for classical AMG}
An adaptive coarsening algorithm is an algorithm which adaptively
chooses the coarse grid degrees of freedom based on a given definition
of strength function based on the smoother in a two grid
algorithm. One example of adaptive coarsening follows from the
classical compatible relaxation introduced by A.~Brandt. The algorithm
takes as input a smoother which leaves the coarse grid variables
invariant and only smooths the components in the algebraic
high-frequency space $V_{\text{hf}}$.

A typical adaptive coarsening algorithm follows the steps given below:
\begin{description}
\item[Step 0] Set $k=0$ 
    and choose $(V')_{k,c}\subset V'$, for example, using the MIS or aggregation method introduced in \S\ref{s:basic-coarsen}. 
\item[Step 1] Define 
$V_{k,f}$ as the subspace of $V$ which is annihilated by the 
functionals in $(V')_{k,c}$, namely
\begin{equation*}
    V_{k, f}=\{v\in V: (g, v) = 0, \forall g\in (V')_{k, c}\}.
\end{equation*}
\item[Step 2] Let $\imath_{k, f}: V_{k,f} \mapsto V$ be the natural
  inclusion operator and compute an estimate $\rho_{k,f}$ of the norm
  of the smoother on $V_{k,f}$. Below, $R_{k,f}$ could be the
  restriction of the smoother $R$ on $V_{k,f}$, or more generally, any
  relaxation on $V_{k,f}$.
    \begin{equation*}
        \rho_{k, f}\approx \sup_{v\in V_{k,f}}
        \frac{\|(I-\imath_{k,f}R_{k,f} \imath_{k,f}'A )v\|_A^2}{\|v\|_A^2}.
    \end{equation*}
\item[Step 3]
Given a threshold $\delta_f>0$, if 
$\rho_{k,f} > \delta_f$, we set $k=k+1$, add more functionals to $(V')_c$ and go to
{\bf Step 1}. 
Otherwise, 
we set $(V')_{c}  = (V')_{k,c}$, and accordingly $V_f = V_{k,f}$ and stop the iteration. 
\end{description}

In {\bf Step 3}, if the stopping criteria is not satisfied, we need to
enrich the space $(V')_c$.
by extending the set . One
example of doing so is introduced in compatible relaxation method
by extending the set $\mathcal C$ using the following procedure:

First we randomly choose a vector $v^0\in V_{k,f}$, and form
$$
v=(I-\imath_{k,f}R_{k,f} \imath_{k,f}'A)^{\nu}v^0\mbox{ for some } \nu\ge 1.
$$
Then, with a given threshold $\theta \in (0, 1)$, we let 
\begin{equation*}
    \mathcal C_0^1=\{i\in \mathcal F: |v_i|> \theta \max_{k}|v_k|\}.
\end{equation*}
and 
$$
\mathcal C_1=\mathcal C_0\cup \mathcal {\rm MIS}(\mathcal C_0^{1}). 
$$
Finally, we update $\mathcal C\leftarrow \mathcal C_1$ and we proceed with the
next CR iteration.

As is clear from the algorithm outlined above, we can use any of the
definitions of strength of connections to obtain  $(V')_{0, c}$ at {\bf Step 0}. 
When $\rho_{k,f}>\delta_f$, it means that
$\mathcal C$ obtained from the strength function $s_0$ is not satisfactory,
namely too coarse.  This either means that the threshold for $s_0$ is
too small, or $s_0$ itself is not satisfactory.  We could still use
$s_0$ but with a smaller threshold to obtain a $\mathcal C_0'$ which
is bigger than $\mathcal C_0$, but not necessarily contain $\mathcal
C_0$. For example, if we initially use the $(m, l)$-strong connection defined in
Definition~\ref{def:extend-strong-connection-2}, $\mathcal C$ can be
extended by increasing the value of $m$ or decreasing the value of
$l$.
This approach may not be computationally efficient.   A more
effective approach, as it is used in compatible relaxation, is to find candidates for additional $C$-points by
examining the following filtered matrix:
$$
A^{(1)}=\bigg\{a_{ij}\;:\; i, j\in \mathcal F_0\bigg\},
$$
where $\mathcal F_0=\Omega \setminus \mathcal C_0$,  to get
the set of coarse grid degrees of freedom $\mathcal C_0^{1}$ and add them to $\mathcal C_0$ 
to extend the size of $\mathcal C$.

\subsection{AGMG coarsening: a pairwise aggregation}\label{s:agmg-not}

AGMG uses the strength function defined in \eqref{sc-agg} to form aggregations such that the local convergence rate (c.f \eqref{muj} in \S\ref{sec:unifiedAMG}) on each aggregate is bounded by a given threshold. The main idea of the algorithm can be explained as follows:

It first splits the index set $\Omega$ into aggregates each of which has at most
two elements, namely
\begin{equation}
  \Omega =\bigcup_{j}\mathcal A_j^{(0)}, \quad \mathcal A_i^{(0)}\bigcap \mathcal A_j^{(0)}=\emptyset, \quad \text{and } |\mathcal A_j^{(0)}|\le 2.
\end{equation}
This process is done by a greedy algorithm. 
At each step, the algorithm finds the pair $G=\{i, j\}$ for which the strength function defined in \eqref{scij} is maximal.

Using these pairs as aggregates $\{\mathcal A_j^{(0)}\}$, we form the
unsmoothed aggregation prolongation $P$, which is piecewise constant
with respect to the aggregates and $P$ has orthogonal columns.

We denote $A^{(0)}:=A$ and $P^{(0)}:=P$. Then
$A^{(1)}:=(P^{(0)})^TA^{(0)}P^{(0)}$. We then apply pairwise
aggregation algorithm on $A^{(1)}$ and find larger aggregates
$\{\mathcal A_j^{(1)}\}$. Each $\mathcal A_j^{(1)}$ is union of two pairs in
$\{\mathcal A_j^{(0)}\}$ which minimize the strength function defined in \eqref{sc-agg}. Then we obtain $P^{(1)}$ and $A^{(2)}$. Applying
this procedure recursively, we obtain the final aggregates $\mathcal A_j$
each of which is a union of several pairs in $\{\mathcal A_j^{(0)}\}$.

The pairwise aggregation strategy aims to find the aggregates on which the Poncar\'e constant  $\mu_j(V_j^c)^{-1}$, which is introduced in our
abstract framework by \eqref{muj}, is bounded . And as it is stated in the
abstract convergence theorem, bounding $\mu_j(V_j^c)^{-1}$ will bound the convergence rate of the AMG method.

\subsection{Bibliographical notes}\label{s:coarsening-biblio}
The coarsening strategies used in AMG are basically two types: the
first one uses strength of connection to define a ``strength'' graph,
and then performs greedy aggregation or MIS algorithms and the other
is a based on algorithms such as Compatible Relaxation coarsening
which uses the smoother to detect slow to converge components.  These
are heuristic approaches, which work well on a certain class of
problems, but rarely have a theoretical justification of their
efficiency as AMG splitting algorithms.

Classical algorithms for selection of the coarse grid degrees of
freedom are found in
\cite{1stAMG,Stuben.K.1983a,Brandt.A;McCormick.S;Ruge.J.1985a,Ruge.J;Stuben.K.1987a},
and the MG the tutorial \cite{Briggs.W;Henson.V;McCormick.S.2000a}.

Parallel coarse-grid selection algorithms are found in
\cite{Sterck.H;Yang.U;Heys.J.2006a} and, in combination with scalable
interpolation algorithms in
\cite{De-Sterck.H;Falgout.R;Nolting.J;Yang.U.2008a}.  Coarsening using
information about discretization, i.e. AMGe, are given in~
\cite{Jones.J;Vassilevski.P.2001a,Brezina.M;Cleary.A;Falgout.R;Henson.V;Jones.J;Manteuffel.T;McCormick.S;Ruge.J.2001a}.
Spectral AMGe coarsening considered in detail in
\cite{Chartier.T;Falgout.R;Henson.V;Jones.J;Manteuffel.T;McCormick.S;Ruge.J;Vassilevski.P.2003b}. Many of the ``upscaling'' and related techniques in homogenization
(c.f. \cite{efendiev2000convergence,hou1999convergence}), resemble the
coarsening procedures introduced in the classical and modern AMG
literature.

More sophisticated maximal independent set (MIS) algorithm for
selection of coarse grid degrees of freedom, using different measures
for connectivity and distance in the graph corresponding to $A$ such
are found in \cite{Ruge.J;Stuben.K.1987a} and
\cite[Appendix~A]{Trottenberg.U;Oosterlee.C;Schuller.A.2001a}.  Most
of these algorithms are refinements of the greedy algorithm given in
this section. For parallel versions we refer to
\cite{1986LubyM-aa}, \cite{1998ClearyA_FalgoutR_HensonV_JonesJ-aa},
\cite{Sterck.H;Yang.U;Heys.J.2006a} for specific details on parallel
and parallel randomized MIS algorithms.  Other coarsening schemes that
are also suitable for parallel implementation are the coupled and
decoupled coarsening schemes
\cite{Yang.U.2006a,Henson.V;Yang.U.2002a}.

Regarding the aggregation coarsening methods we refer
to~\cite{Vakhutinsky.I;Dudkin.L;Ryvkin.A.1979a},
\cite{Blaheta.R.1986a}, and Marek~\cite{Marek.I.1991a} a for earlier
work on such methods. The greedy aggregation algorithm presented here
is found in~\cite{1996VanekP_MandelJ_BrezinaM-aa}.  A special class of
aggregation coarsening method based on matching were first employed by
Kumar and Karypis for fast graph
partitioning~\cite{Karypis.G;Kumar.V.1998a} and later used in several
of the AMG methods. One example is the AGMG algorithm described in
\S\ref{s:agmg-not} and found
in~\cite{Napov.A;Notay.Y.2012a,Notay.Y.2012b}.  The algorithm given
in~\cite{Kim.H;Xu.J;Zikatanov.L.2003b} also uses such coarsening
approach.Special matching techniques which optimize matrix invariants
are used
in~\cite{DAmbra.P;Vassilevski.P.2014a,DAmbra.P;Vassilevski.P.2013a}. 
a
The Compatible Relaxation (CR) algorithm, first introduced
in~\cite{2000BrandtA-aa} and further investigated
in~\cite{Livne.O.2004a}, \cite{Falgout.R;Vassilevski.P.2004a}, and
\cite{CR}, is a device that reduces the role of the strength of
connection to only define initial set of coarse grid degrees of
freedom and then use the smoother to select additional degrees of
freedom.  Other coarse-fine degrees of freedom partitioning algorithms
are considered in~\cite{MacLachlan.S;Saad.Y.2007a} from both classical as well as
compatible relaxation point of view.  A somewhat different
adaptive coarsening algorithms are the aggregation algorithms which
aggregates vertices together based on a local measure for two-level
convergence~\cite{Notay.Y.2010b,Napov.A;Notay.Y.2012a,Livne.O;Brandt.A.2012a}.

\section{GMG, AMG and a geometry-based AMG}\label{s:GMG}
Historically, the algebraic multigrid, AMG, method was motivated by
the geometric multigrid, GMG, method.   In this section, we will give
exploration on the relationship between these two types of methods. 
 
\subsection{Geometric multigrid method}
We begin out discussion for a simple 1D model problem, namely
\eqref{Model0} for $d=1, \Omega=(0,1)$ and $\alpha\equiv 1$ with zero
Dirichlet boundary condition.  For any integer $N$, we consider a
uniform grid, denoted by ${\mathcal T}_h$, of the interval $[0,1]$ as
follows:
\begin{equation}\Label{1dpartition}
        0=x_0<x_1<\cdots<x_{N+1}=1, \quad x_j=\frac{j}{N+1} (j=0:N+1).
\end{equation}
This partition consists of uniform subintervals with the length
$h=\frac{1}{N+1}$, i.e.,  $\cth=\bigcup_{i}\{\tau_i\} $ where $\tau_i=(x_{i-1}, x_i)$ for $d=1$.  Such a uniform partition is shown in Figure \ref{fig:1dpartition}.

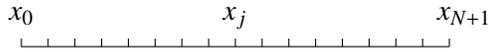
\begin{figure}[h]
\setlength{\unitlength}{0.14in} 
\begin{center} 
\begin{picture}(32,1) 
\put(8,0){\line(1,0){16}}
\put(8,0){\line(0,1){0.3}}
\put(7.5,1){$x_0$}
\put(9,0){\line(0,1){0.3}}
\put(10,0){\line(0,1){0.3}}
\put(11,0){\line(0,1){0.3}}
\put(12,0){\line(0,1){0.3}}
\put(13,0){\line(0,1){0.3}}
\put(14,0){\line(0,1){0.3}}
\put(15,0){\line(0,1){0.3}}
\put(16,0){\line(0,1){0.3}}
\put(15.5,1){$x_j$}
\put(17,0){\line(0,1){0.3}}
\put(18,0){\line(0,1){0.3}}
\put(19,0){\line(0,1){0.3}}
\put(20,0){\line(0,1){0.3}}
\put(21,0){\line(0,1){0.3}}
\put(22,0){\line(0,1){0.3}}
\put(23,0){\line(0,1){0.3}}
\put(24,0){\line(0,1){0.3}}
\put(23.5,1){$x_{N+1}$}
\end{picture}
\end{center}
\caption{One-dimensional uniform grid\label{fig:1dpartition}} 
\end{figure}

We define a linear finite element space associated with the partition
${\mathcal T} _h$
\begin{equation}
  \label{FEMh}
V_h=\{v:\; \mbox{$v$ is continuous and piecewise linear
w.r.t. ${\mathcal T} _h$, } v(0)=v(1)=0\}. 
\end{equation}

Let $V_h=V_h$. Recall from previous section, the finite element approximation of our model
problem is then $u_h\in V_h$ satisfying \eqref{vph}. We introduce the
operator:  $A_h: V_h\mapsto V_h$ such that
$$
(A_hv_h,w_h)=a(v_h,w_h),\quad v_h,w_h\in V_h.
$$
Then the finite element solution $u_h$ satisfies
\begin{equation}
  \label{Auf-h}
A_hu_h=f_h  
\end{equation}
where $f_h\in V_h$ is the $L^2$ projection of $f$: $
(f_h,v_h)=(f,v_h),\quad v_h\in V_h.  $

To describe a geometric multigrid algorithm, we need to have a
multiple level of grids, say ${\mathcal T}_k$ with $k=1:J$ and
${\mathcal T}_J={\mathcal T}_h$ being the finest mesh. One simple
definition of the grid points in ${\mathcal T}_k$ is as follows:
$$
        x_i^k=\frac{i}{2^k},\quad i=0,1,2,\cdots, N_k+1, k=1,2,\cdots,J,
$$
where $N_k=2^k-1$.  Note that ${\mathcal T}_k$ can be viewed as being obtained
by adding midpoints of the subintervals in ${\mathcal T}_{k-1}$.  For each $k$
the set of above nodes will be denoted by ${\mathcal N}_k$.  

\begin{figure}[htp]
\begin{center}
\includegraphics[width=0.5\textwidth]{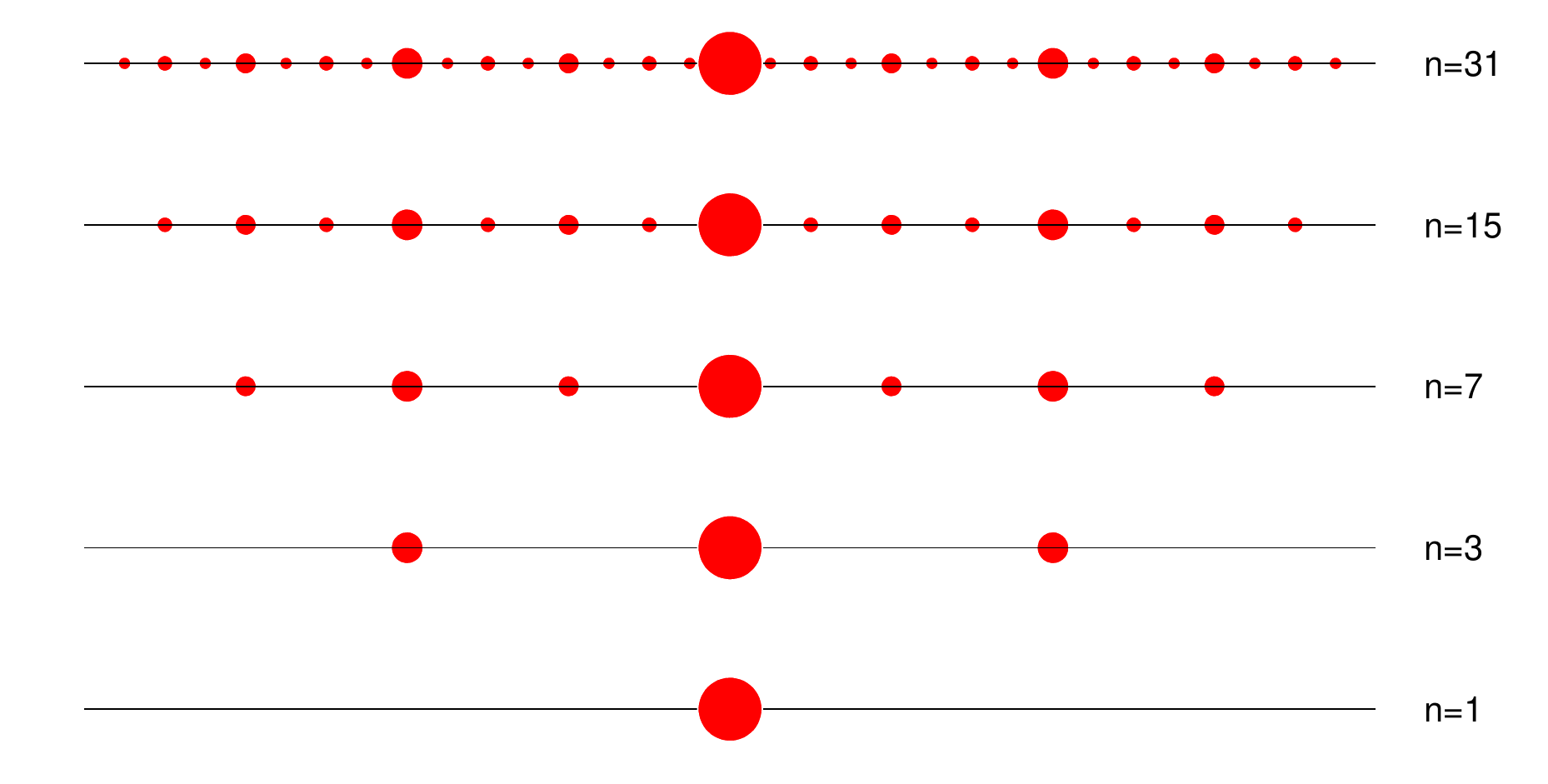}
\end{center}
\caption{Multiple grids in one dimension
\label{fig:manygrids}}
\end{figure}

For $k=1:J$, similar to finite element space $V_h$ defined as
in \eqref{FEMh}, we define finite element space $V_k$
associated with the grid $\mathcal T_k$ to obtain a nested
sequence of finite element spaces as follows:
\begin{equation}\label{nested_M}
V_1\subset\ldots\subset V_k\subset\ldots\subset V_J.
\end{equation}

The classic V-cycle geometric multigrid method simply applies
Algorithm~\ref{alg:two-level} recursively with the following setting:
\begin{enumerate}
\item $V_k=V_k$;
\item $A_k: V_k\mapsto V_k$ defined by
$$
(A_ku_k,v_k)=a(u_k,v_k), \quad u_k,v_k\in V_k;
$$
\item $\imath_{k-1}^k: V_{k-1}\mapsto V_k$ is the inclusion operator;
\item $R_k: V_k\mapsto V_k$ corresponding a smoother such as
  Gauss-Seidel method.
\end{enumerate}

\subsubsection*{Algebraic setting}

The equation \ref{Auf-h} may be called the operator form of the finite
element equation.  To get an equation in terms of vectors and matrix,
we use  the nodal basis functions for $V_h$
\begin{equation}
  \label{Mh-basis}
\phi_i(x)=\left\{\begin{array}{cl}
\frac{x-x_{i-1}}{h}, & x\in[x_{i-1},x_i];\\
\frac{x_{i+1}-x}{h}, & x\in[x_{i},x_{i+1}];\\
0 &\mbox{elsewhere}.
\end{array}\right.
\end{equation}

On each level, similarly as \eqref{Mh-basis}, we can introduce a set of nodal basis functions, denoted by $\{\phi_i^{(k)}: i=1:N_k\}$, for finite element space $V_k$.
\begin{figure}[!htb]
\begin{center}
\includegraphics[width=\textwidth]{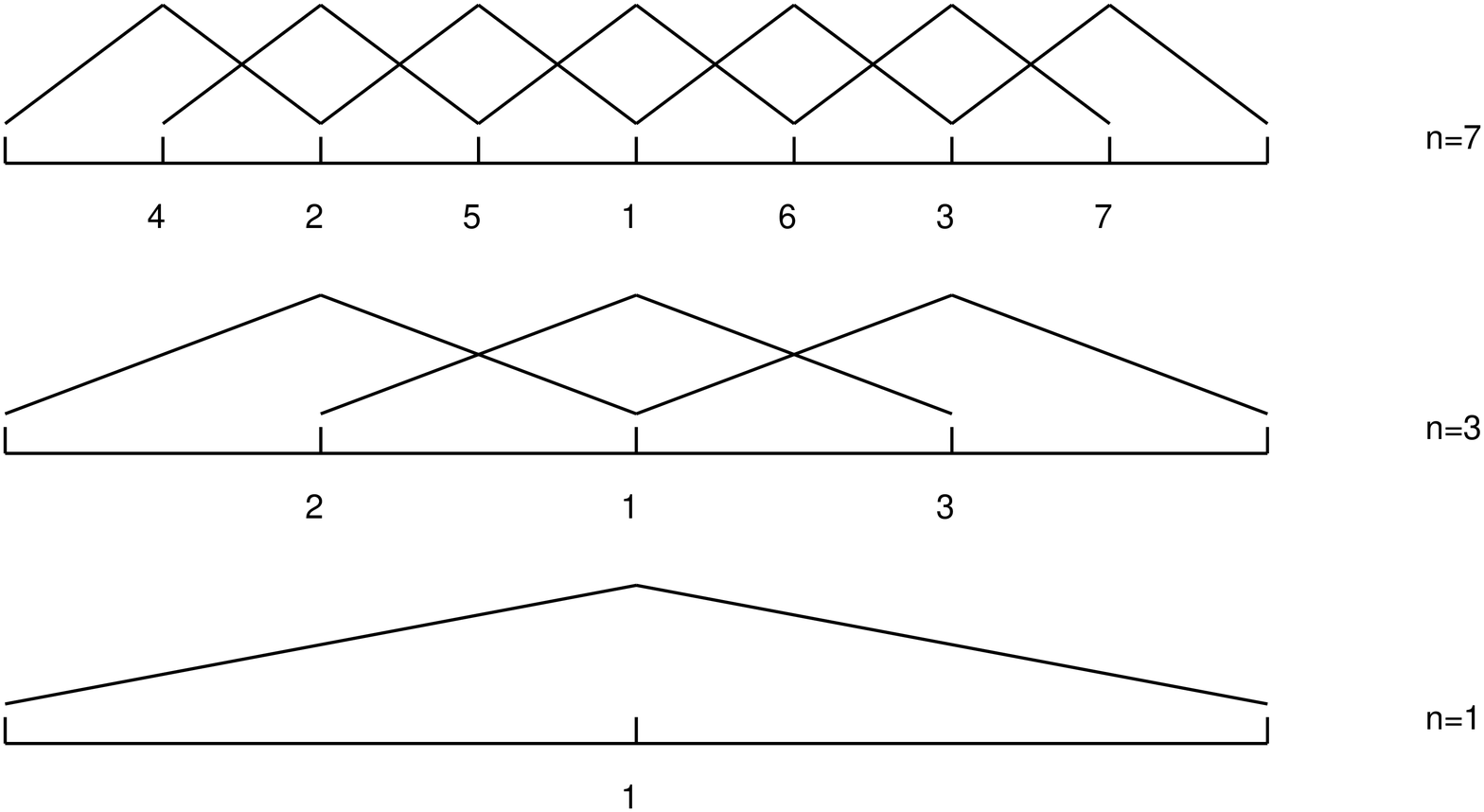}
\end{center}
\caption{1D nodal basis functions on each level
\label{20150914_add1}}
\end{figure}

Each $v\in V_k$ can be uniquely written as the following linear combination of the basis functions
\begin{equation}
    v=\xi_1\phi_1^{(k)}+\xi_2\phi_2^{(k)}+\cdots +\xi_{N_k}\phi_{N_k}^{(k)}.
\end{equation}
This gives an isomorphism from $V_k$ to $\mathbb R^{N_k}$, which maps $v\in V_k$ to $\mu \in \mathbb R^{N_k}$ as following
\begin{equation}
    v=\xi_1\phi_1^{(k)}+\xi_2\phi_2^{(k)}+\cdots +
\xi_{N_k}\phi_{N_k}^{(k)} \longrightarrow \mu=\begin{pmatrix}
        \xi_1\\
        \xi_2\\
        \vdots\\
        \xi_{N_k}\\
    \end{pmatrix}.
\end{equation}
$\mu$ is called \emph{matrix representation} of $v$. Recall from the discussion earlier in \S\ref{sec:fem}, giving a basis of $V_h$ defined in \eqref{Mh-basis},  the \eqref{vph} is equivalent to the linear system equations in \eqref{axb}.

We introduce auxiliary space $V_k:=\mathbb R^{N_k}$. 

The  transition operator $P_{k}^{k+1}$ from $V_k$ to $V_{k+1}$ is a matrix in $\mathbb R^{N_{k+1}\times N_k}$ and satisfies 
\begin{equation}\label{basis-representation}
    (\phi_1^{(k)} \cdots \phi_{N_k}^{(k)})=(\phi_1^{(k+1)} \cdots \phi_{N_{k+1}}^{(k+1)})\imath_{k}^{k+1}.
\end{equation}
For the special 1D problem we are now considering, we have for $k=1, 2, \dots, J-1$
\begin{equation}\label{gmg-basis-1d}
    \phi_{j}^{(k)}=\frac12\phi_{2j-1}^{(k+1)}+\phi_{2j}^{(k+1)}+\frac12\phi_{2j+1}^{(k+1)},
\end{equation}
The matrix which encodes this relation is
\begin{equation}\label{p-gmg-1d}
    P_{k}^{k+1} =\begin{pmatrix}
        \frac{1}{2} & & & &\\
        1 & & &\\
        \frac{1}{2} & \frac{1}{2} & & &\\
        & 1 & & \\
        & \frac{1}{2} & \frac{1}{2} &\\
        & & 1 & &\\ 
        & & \frac{1}{2} & &\\ 
        & & & \ddots &\frac{1}{2}\\
        & & & &  1\\
        & & & &  \frac{1}{2}\\
    \end{pmatrix}.
\end{equation}
The classical $V$-cycle AMG method following from simply applying
Algorithm~\ref{alg:two-level} recursively with the following setting:
\begin{enumerate}[1.]
    \item $V_k=\mathbb R^{N_k}$;
    \item $A_k\in \mathbb R^{N_k\times N_k}: V_k\mapsto V_k$ defined by
        $$
            (A_k)_{ij}= a(\phi_{i}^{(k)}, \phi_{j}^{(k)}), \quad 1\le i, j\le N_k;
        $$
    \item $P_{k-1}^k: V_{k-1}\mapsto V_k$ a matrix in $\mathbb R^{N_{k+1}\times N_{k}}$ defined by \eqref{basis-representation};
    \item $R_k: V_k\mapsto V_k$ corresponding a smoother such as Gauss-Seidel method.
\end{enumerate}
A similar multigrid algorithm can be obtained for problems in 2D and
3D as long as we have a multiple level of grids and the corresponding
finite element spaces on each level.

\subsection{Obtaining AMG from GMG}\label{s:amg-from-gmg}

The first barrier of extending GMG to AMG is the geometric information used in GMG.  But a close
inspection of the GMG reveals that a GMG method only depends on the
following 2 major ingredients:
\begin{enumerate}[1.]
\item The stiffness matrix corresponding to the finest grid;
\item The prolongation matrix on each level.
\end{enumerate}
Once all the prolongation matrices are given, the stiffness matrices
on all coarser levels are given by
\begin{equation}\label{galerkin-cgm}
A_{k}=P_{k}^TAP_{k},\quad  k=J-1, J-2, \ldots 1, \quad P_k =\prod_{j=k}^{J-1} P_{j}^{j+1}.
\end{equation}
Of course, a smoother is also needed on each level, but its definition
can be considered, for the moment purely algebraic.

As the stiffness matrix on the finest grid is always
available in any given application, the only thing left is the
prolongation matrices.   We will now use the example of linear finite
element method to discuss about the relationship between the
prolongation matrix and geometric information.   Two observations are
most relevant:
\begin{description}
\item[Observation 1] The prolongation matrix only depends on the
  natural graph associated with the underlying grid, but not on the
  coordinates of grid points.
\item[Observation 2] The graph of the underlying grid is very close to
  the adjacency graph of the stiffness matrix.
\end{description}

Based on the above discussions, roughly speaking, we can essentially
recover a geometric multigrid method for the stiffness matrix
corresponding to the continuous linear finite element discretization
of Laplace equation by only using the algebraic and graph information
provided by the stiffness matrix.
\begin{enumerate}[1.]
\item Form the adjacency graph $\mathcal G(A)$ of the stiffness matrix $A$;
\item Coarsen $\mathcal G(A)$.
\end{enumerate}

As an illustrative example, let us consider the stiffness matrices
corresponding to a discretization of the Laplace equation on a square
domain with bilinear elements. It is well known,
\cite{2002CiarletP-aa}, that the stiffness matrix in this case is the
same as the scaled matrix for the 9-point finite difference
stencil \eqref{9-point}. The corresponding adjacency graph, shown on the right in
Figure~\ref{gmg} is denser (has more edges) than the mesh graph shown
on the left in~Figure~\ref{gmg}. The set of its edges includes the
diagonals of each of the squares forming the mesh.
\begin{figure}[!htb]
\centering
    \includegraphics[width=1.0\textwidth]{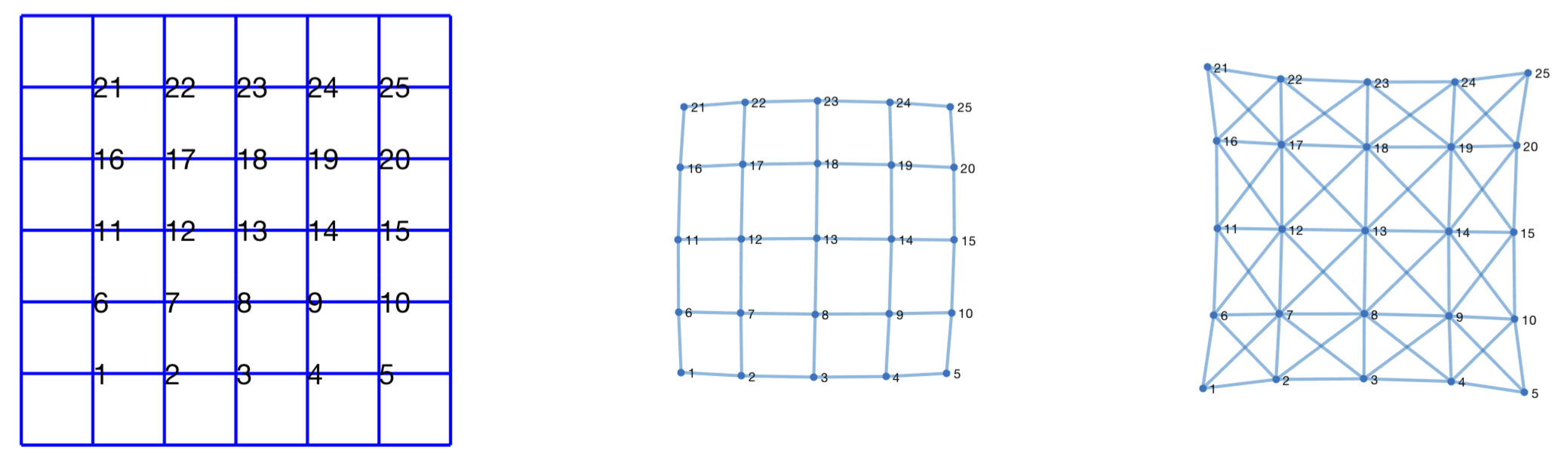}
\caption{A $6\times 6$ uniform grid (left), graph of the matrix
  corresponding to the 5-point finite difference stencil (middle) and graph of the matrix
  corresponding to the 9-point finite difference stencil (right)}
\label{gmg}
\end{figure}
For the construction of the prolongation/interpolation matrix we
recall that the prolongation matrix gives the coefficients of the
expansion of a coarse grid basis function on a grid of size $2h$ in
terms of the finer grid basis on a grid of size $h$. Locally this matrix looks as follows:
\begin{eqnarray*}
  [(P_{2h}^h)^T]_i=
\begin{pmatrix}
  \frac{1}{4} &\frac{1}{2} &\frac{1}{4}\\
  \frac{1}{2} & 1 & \frac{1}{2}\\
 \frac{1}{4} & \frac{1}{2} &\frac{1}{4}
\end{pmatrix}.
\end{eqnarray*}
This matrix is often dubbed as ``prolongation stencil'' and it shows
in a compact form the coefficients in the expansion of a coarse grid
basis function. In the center we have the coefficient $1$, in front of the
fine grid basis function associated with a coarse grid vertex. The
rest of the entries correspond to the coefficients in front of the
fine grid basis functions in the expansion.  

On regularly refined, triangular grids we have an analogous
situation. We refer to~\S\ref{sc:connection} for details on the
selection of coarse grid degrees of freedom in this case. Prolongation
and restriction matrices only depend on the topological structure of
this graph.  Similar observations led to the development of the AMG:
if the geometric coordinates are unknown, different avenues for
constructing coarse spaces are needed, leading to different variants of AMG algorithms.

\subsection{Obtaining GMG from AMG}\label{s:gmg-from-amg}
In this section, we use the unified theory in \S\ref{sec:unifiedAMG} to obtain GMG from AMG.
The main ingredients needed are spaces $V_j$, operators $\Pi_j$, $A_j$, $D_j$ and coarse spaces $V_j^c$.

We now consider constructing a two-level geometric multigrid method
for \eqref{Model0} (or the variabtional formulation
\eqref{Vari}). Suppose we have two grids: a fine grid $\mathcal T_h$
and a coarse grid $\mathcal T_H$. On each grid, we define a linear
finite element space $V_h$ and $V_H$ with nodal basis functions
$\{\phi_j^h\}$ and $\{\phi_j^H\}$ respectively and 
 we consider the following partition of the domain $\Omega$:
\begin{equation}
  \overline{\Omega}=\bigcup_{j=1}^J\overline{\Omega}_j, \text{ with } \overline{\Omega}_j=\operatorname{supp}(\phi_j^H).
\end{equation}

We then define $V_j$ as 
\begin{equation}
    V_j := \{\chi_jv: v\in V_h\},
\end{equation}
where $\chi_j$ is the characteristic function of $\Omega_j$.  We note
that $V_j$ is not a subspace of $H^1(\Omega)$.
The operator $\Pi_j: V_j\mapsto V_h$ is defined in the following way:
\begin{equation}\label{pi-gamg}
    \Pi_j v_j: = I_h(\phi_j^H v_j), \quad \forall v_j\in V_j,\quad j=1,\ldots,J,
\end{equation}
where $I_h$ is the nodal interpolation operator on the fine grid.
Here we note that $\phi_j^H v_j$ is continuous in $\Omega$. 
We
notice that, by definition, $\phi^{H}_j \chi_j = \phi^{H}_j$ on
$\Omega$ which implies the identities
$$
\sum_{j=1}^J \Pi_j \chi_j v = I_h\left(\sum_{j=1}^J\phi_j^H\chi_j
  v\right)
= I_h\left(\sum_{j=1}^J\phi_j^Hv\right)
=I_h(v)=v,
$$
and hence
$$
\sum_{j=1}^J \Pi_j \chi_j =\mathrm{Id}. 
$$

The operator $A_j: V_j\mapsto V_j'$ is the local restriction of the
bilinear form $a(\cdot, \cdot)$, namely,
\begin{equation}\label{gmg-Aj}
    (A_j u_j, v_j): = a(u_j, v_j)_{\Omega_j} = \int_{\Omega_j}\alpha(x)\nabla u_j\cdot \nabla v_j, \quad u_j, v_j\in V_j
\end{equation}

The following estimate tells us that $A_j$ satisfies
\eqref{sum_Aj} with decomposition $v=\sum_{j=1}^J\Pi_jv_j$
where $v_j=\chi_jv$:
$$
    \sum_{j=1}^{m_c}\|v_j\|_{A_j}^2  =  \sum_{j=1}^{m_c} a(v, v)_{\Omega_j}\le C_1\|v\|_A^2,
$$
where $C_1$ depends on the number of overlaps of $\Omega_j$.

We choose $D: V \mapsto V'$ using the fine grid basis functions as follows 
\begin{equation}
    (D \phi_i^h, \phi_j^h): = a(\phi_i^h, \phi_j^h)\delta_{ij}, \quad 1\le i, j\le n,
\end{equation}
and
\begin{equation}\label{duv}
    (D u, v): = \sum_{j=1}^na(\phi_j^h, \phi_j^h)u_jv_j, \quad u, v\in V. 
\end{equation}
Note that $(D\cdot, \cdot)$ is an inner product on $V$ which induces a norm $\|\cdot\|_D$.

In the above definition, the matrix representation of $D$ is the diagonal of the matrix representation of $A$. In a similar way, we can define $D_j :V_j\mapsto V_j'$ using the basis functions for $V_j$. 

By a simple scaling argument, we obtain 
\begin{equation*}
    \|v\|_0^2 \eqqsim h^d\|v\|_{\ell^2}^2.
\end{equation*}
From \eqref{duv}, using inverse inequality, we then have
\begin{equation*}
    \|v\|_D^2\lesssim h^{-2}\sum_{j=1}^n\|\phi_j^h\|_0^2v_j^2\eqqsim h^{d-2}\|v\|_{\ell^2}^2\eqqsim h^{-2}\|v\|_0^2.
\end{equation*}
We also have
\begin{equation*}
    \|I_hv\|_0^2\eqqsim  h^d\|I_hv\|_{\ell^2}^2 = h^d\|v\|_{l^2}^2\eqqsim \|v\|_0^2.
\end{equation*}
Assumption~\ref{assm:D_j} then can then be verified by the following
\begin{eqnarray*}
    \|\sum_{j=1}^{J}\Pi_jv_j\|_D^2 &\lesssim & h^{-2}\|\sum_{j=1}^J\Pi_jv_j\|_0^2 =h^{-2}\|I_h\left(\sum_{j=1}^J\phi_j^Hv_j\right)\|_0^2 \\
    &\eqqsim&h^{-2}\|\sum_{j=1}^J\phi_j^Hv_j\|_0^2 =h^{-2}\sum_{i,j=1}^{J}\int_{\Omega}\phi_i^Hv_i\phi_j^Hv_j\\
    &\le & h^{-2}\sum_{i,j=1}^{J}\|\phi_i^H\|_{\infty}\|\phi_j^H\|_{\infty}\int_{\Omega}v_iv_j\le h^{-2}\sum_{i, j=1}^{J}\int_{\Omega_i\bigcap\Omega_j}v_iv_j\\
    &\le &\co\sum_{j=1}^{J}h^{-2}\int_{\Omega_j}|v_j|^2\lesssim \co\sum_{j=1}^{J}\|v_j\|_{D_j}^2.
\end{eqnarray*}
By the definition of $A_j$, the kernel of $A_j$ consists of all
constant functions in $V_j$ and we choose $V_j^c$ to be the one
dimensional space of constant functions on $\Omega_j$. Then
\begin{equation}
    \mu_j(V_j^c)=\lambda_j^{(2)},
\end{equation}
where $\lambda_j^{(2)}$ is the second smallest eigenvalue of the operator $D_j^{-1}A_j$.

The global coarse space $V_c$ is defined as in \eqref{V_c}. Note that
in this case, it is easy to show by the definition that the coarse
space $V_c$ constructed by \eqref{V_c} is in fact identical to $V_H$,
namely
\begin{equation}
    V_c=\operatorname{span}\{\phi_j^H: j=1, \dots, J\}.
\end{equation}
By Theorem~\ref{thm:two-level-conv}, the converges rate of this two-level
geometric multigrid method depends on the
$\min_j(\lambda_j^{(2)})$. If the Poincar{\'e } inequality is true for
each $V_j$, namely,
\begin{equation}
    \inf_{v_c\in V_j^c} \|v-v_c\|_{D_j}^2\le c_j\|v\|_{A_j}^2, \quad \forall v\in V_j,
\end{equation}
with $c_j$ to be a constant, then the two-level geometry multigrid
method converges uniformly.

\subsection{Spectral AMG$e$: a geometry-based AMG}\label{s:AMGe}
We now consider the \emph{element} based AMG approaches, and also define
the ingredients needed to fit these methods in the general results in 
\S\ref{sec:unifiedAMG}. The AMG$e$ methods are less
algebraic as they assume an underlying grid and use element local stiffness
matrices to define interpolation operators.  
In the AMG$e$ setting, it is assumed that we know a decomposition of the $n\times n$
matrix $A$,
\begin{equation}
    A=\sum_{\tau\in \mathcal T}\tilde A_{\tau},
\end{equation}
where, $\mathcal T$ is the set of finite elements used to discretize
the problem and for each element $\tau\in \mathcal T$, and $\tilde
A_\tau$ is the zero-extension of the local stiffness matrix $A_{\tau}$
on $\tau$ (which is symmetric positive semi-definite).

 To define $V_j$, $D_j$, $A_j$ and $\Pi_j$ in \S\ref{sec:unifiedAMG}, corresponding to AMG$e$, we partition the domain $\Omega$ into disjoint subdomains, $\Omega_1, \dots, \Omega_J$. Each subdomain is an agglomerate of elements, and $\bar \Omega =\bigcup_{j=1}^J\bar \Omega_j$. For each subdomain $\Omega_j$, we introduce the cutoff operator $\chi_j: V\mapsto V_j$ whose action on $v\in V$ is defined by 
\begin{equation}\label{cutoff}
    (\chi_jv)(x):= \begin{cases}
        v(x), & \text{ if } x\in \bar \Omega_j,\\
        0, & \text{ if } x\notin \bar \Omega_j.
    \end{cases}
\end{equation}

Then we define the space $V_j$ by
\begin{equation}
    V_j: = \chi_j V.
\end{equation}
$A_j$ is defined by summing up all the associated stiffness matrices for elements in $\Omega_j$, namely,
\begin{equation}
    A_j:=\sum_{\tau\subset\Omega_j}\tilde A_{\tau}.
\end{equation}
Clearly, $A_j$ is symmetric positive semi-definite. It is easy to verify that \eqref{sum_Aj} holds. In fact, we have the following equations 
\begin{equation}
    \sum_{j=1}^J\|\chi_jv\|_{A_j}^2 = \sum_{j=1}^J a(v, v)_{\Omega_j} = \sum_{\tau\in \mathcal T} a(v, v)_{\tau} =\|v\|_A^2.
\end{equation}

If we denote the diagonal of $A$ by $D$, then $D_j$ is a diagonal matrix defined following
\begin{equation}
    [D_j]_{ii}: = \begin{cases}
        D_{ii}, & \text{ if } i\in \Omega_j,\\
        0, & \text{ if } i\notin \Omega_j.
    \end{cases}
\end{equation}
The operator $\Pi_j$ is provided by diagonal matrices defined as
\begin{equation}
    [\Pi_j]_{ii}: = \begin{cases}
        [A_j]_{ii}/[A]_{ii}, & \text{ if } i\in \Omega_j,\\
        0, & \text{ if } i\notin \Omega_j.
    \end{cases}
\end{equation}
Note that $[\Pi_j]_{ii}=1$ if $i$ is an inner point of $\Omega_j$. \eqref{assm:D_j} is verified in Lemma~\ref{lem:Dj}.

On each local space $V_j$, spectral AMG$e$ chooses locally the ``best"
coarse space $V_j^c$ which is the subspace spanned by eigenvectors of
$D_j^{-1}A_j$ belonging to its $m_j^c$ smallest eigenvalues.
 The
global coarse space is then defined by \eqref{V_c}.

By the abstract convergence theorem
(Theorem~\ref{thm:two-level-convergence}), the convergence rate of two
level spectral AMG$e$ depends on minimum of each $m_j^c+1$ smallest
eigenvalue of $D_j^{-1}A_j$ on $V_j$, more precisely,
\begin{equation*}
    \|E\|_A\le 1-\frac{\mu_c}{C_{p,1}C_{p,2}},
\end{equation*}
with $\mu_c=\min_{1\le j\le J} \mu_{m_j^c+1}^{(j)}$.

\subsection{Bibliographical notes}
  
The main ideas of GMG method were first demonstrated by pioneering
works of Fedorenko~\cite{1961FedorenkoR-aa,1964FedorenkoR-aa}, and
Bahvalov~\cite{1966BahvalovN-aa}. Similar ideas, using group
relaxation methods, can be traced back to the works of Southwell in
the 1940s~\cite{1940SouthwellR-aa,1946SouthwellR-aa}.  The first
description of truly multi-grid method is found in the seminal work by
Brandt \cite{1973BrandtA-aa}. Further developments in multilevel
methods are by Brandt~\cite{1977BrandtA-aa} and also by Hackbusch
\cite{1977HackbuschW-aa,1978HackbuschW-aa}. These works have drawn a
lot of attention from computational mathematics and engineering
community.  Advances of convergence analysis of multigrid methods have
been made by Nicolaides~\cite{1975NicolaidesR-aa,1977NicolaidesR-aa},
Bank and Dupont~\cite{1980BankR_DupontT-aa} Braess and
Hackbush~\cite{1983BraessD_HackbuschW-aa}, Bramble and
Pasciak~\cite{bramble1987new}, Bramble, Pasciak and
Xu~\cite{bramble1990parallel,1991BrambleJ_PasciakJ_XuJ-aa}, Bramble,
Pasciak, Wang and Xu\cite{1991BrambleJ_PasciakJ_WangJ_XuJ-ac}, Xu\cite{1992XuJ-aa}. 

BoxMG~\cite{1982DendyJ-ab,1983DendyJ-aa} is a method that uses
geometrically refined grids and defines interpolation using algebraic
techniques. We refer to~\cite{1982DendyJ-ab,1983DendyJ-aa},
\cite{1990ZeeuwP-aa} for details and aslo to~\cite{2012MacLachlanS_MoultonJ_ChartierT-aa}
 for results on the equivalence between BoxMG and Classical AMG.

The \emph{element} based AMG approaches, which are less algebraic as
they assume an underlying grid and use element stiffness matrices to
define interpolation operators. Such methods include \emph{plain}
AMGe, \emph{element-free} AMGe, \emph{spectral} AMGe, and
\emph{spectral agglomerate} and are developed to improve AMG
robustness for finite element problems element based AMG. We refer to results and discussions on different flavors of AMGe to~\cite{Jones.J;Vassilevski.P.2001a,Brezina.M;Cleary.A;Falgout.R;Henson.V;Jones.J;Manteuffel.T;McCormick.S;Ruge.J.2001a,Henson.V;Vassilevski.P.2001a,Brezina.M;Falgout.R;MacLachlan.S;Manteuffel.T;McCormick.S;Ruge.J.2006b}.

\section{Energy-min AMG}\label{sec:energy-min}
 Here we consider the energy minimization
algorithms for construction of coarse spaces. While this is not
historically the first AMG approach to coarsening, we focus on this
technique first, as it can be used to motivate most other AMG algorithms.  

\subsection{Energy-minimization versus
  trace-minimization}\label{s:energy-trace}
In the next Theorem we add a constraint to Theorem~\ref{t:trace-min}
and give a relation between the optimal coarse space $V_c^{\rm opt}$
and the energy minimization. We refer to \S\ref{2-level-theory} for
the definition of $P^{\rm opt}$ and $\mathcal X_{\eta}$.
\begin{theorem}[Trace-minimization theorem] \label{thm:trace-min}
Given $\eta>0$, let $\mathcal Z_{\eta}$ be
  defined as 
  \begin{equation}
    \label{Z1}
\mathcal Z_\eta=
\bigg\{
  P\in\mathbb R^{n\times n_c}:  (Pv,Pv)_{\bar R^{-1}}\ge \eta (v,v),\;\;
  v\in \mathbb R^{n_c}\mbox{ and } P\bm 1=\sqrt{n_c\eta}\zeta_1\bigg\}
\end{equation}
Then,  $P\in \argmin_{Q\in \mathcal{Z}_\eta}\operatorname{trace}(Q^TAQ)$ if 
$$
    P\in \mathcal Z_{\eta} \text{ and } \range(P)=\range(P^{\rm opt}).
$$
\end{theorem}

Let $\hat P= \bar R^{-\frac{1}{2}}P$ and define 
\begin{equation}\label{yeta}
    \mathcal Y_{\eta}=\bigg\{P\in \mathbb{R}^{n\times n_c}: (Pv, Pv)\ge \eta(v, v), v\in \mathbb{R}^{n_c} \text{ and } P\bm 1=\sqrt{n_c\eta}\hat \zeta_1 \bigg\},
\end{equation}
where $\hat \zeta_j$ is the eigenvector corresponding to the $j$-th smallest eigenvalue of $\bar R^{\frac{1}{2}}A\bar R^{\frac{1}{2}}$. It is clear that $\bar R^{\frac{1}{2}}A\bar R^{\frac{1}{2}}$ and $\bar RA$ have the same spectrum. Theorem~\ref{thm:trace-min} can be written as

\begin{theorem}\label{thm:trace-min-2}
Given $\eta>0$, let $\mathcal Y_{\eta}$ be
  defined as in \eqref{yeta}. Then, 
$$ 
    P\in \argmin_{Q\in \mathcal{Y}_\eta}\operatorname{trace}(Q^T\bar R^{\frac{1}{2}}A\bar R^{\frac{1}{2}}Q) \text{ if }
    P\in \mathcal Y_{\eta} \text{ and } \range(P)=\range(P^{\rm opt}).
$$
\end{theorem}

Suppose now we have a bilinear form $a(\cdot, \cdot)$ on $V$ which is
symmetric, positive semi-definite, and an inner product
$(\cdot, \cdot)_{\bar R^{-1}}$ on $V$. Here, for example, the operator $\bar{R}$ is
the scaled parallel (resp. successive) subspace correction method
corresponding to the splitting of $V$ as
\[
V = \sum_{i=1}^n\Span\{\phi_i\}. 
\]
In practice, $\bar R$ can be a symmetrization of any $A$-norm
convergent smoother on $V$. Here, for simplicity, we choose $\bar R= D^{-1}$.

We now consider a finite element space $V$ with basis functions $\{\phi_j: j=1:n\}$. 
Let $P=(p_{ij})\in \mathbb{R}^{n\times n_c}$ be such that 
\begin{equation}\label{phicP}
    (\phi_1^c, \dots, \phi_{n_c}^c) = (\hat \phi_1, \dots, \hat \phi_n)P, \quad \hat \phi_j=\phi_j/\|\phi_j\|_A, \quad j=1:n.
\end{equation}

Denote by $\{(\mu_j, \zeta_j)\}$ the eigen-pairs of $\bar RA$, and $\hat \zeta_j= \zeta_j/\|\zeta_j\|_A$.
We then define 
\begin{equation}
    X_{\eta} =\bigg\{ (\phi_1^c, \dots, \phi_{n_c}^c): \text{ it satisfies \eqref{phicP} with } P\in \mathcal Y_{\eta}\bigg\}.
\end{equation}

We consider the minimization problem
\begin{equation}\label{min1}
    \min_{(\phi_1^c,\ldots,\phi_{n_c}^c)\in X_{\eta}} \sum_{j=1}^{n_c}\|\phi_j^c\|_A^2
\end{equation}

We notice that
\begin{eqnarray*}
    \sum_{j=1}^{n_c}\|\phi_j^c\|_A^2 &=& \sum_{j=1}^{n_c}a(\phi_j^c,\phi_j^c)=\sum_{j=1}^{n_c}a(\sum_{k=1}^np_{kj}\hat \phi_k,\sum_{l=1}^np_{lj}\hat \phi_l) \\
    & = & \sum_{j=1}^{n_c}\sum_{k=1}^n\sum_{l=1}^np_{kj}a(\hat \phi_k, \hat \phi_l)p_{lj}= \trace(P^TD^{-\frac{1}{2}}\tilde AD^{-\frac{1}{2}} P).
\end{eqnarray*}
Then Theorem~\ref{thm:trace-min} implies that
\begin{equation}
    \Span\{\zeta_j, 1\le j\le n_c\}=\Span\{\phi_j^0, 1\le j\le n_c\},  
\end{equation}
where
\begin{equation}
    (\phi_1^0, \dots, \phi_{n_c}^0) \in \argmin_{(\phi_1^c, \dots, \phi_{n_c}^c)\in X_{\eta}}\sum_{j=1}^{n_c}\|\phi_j^c\|_A^2.
\end{equation}
In the next section, we use the functional setting and provide details
on the design of energy minimizing basis.

\subsection{Energy minimization basis for AMG and Schwarz
  methods}\label{s:energymin-2}
The discussion above motivates the computation of an energy minimizing
basis as the solution of a global optimization problem with
constraint. This could be of concern regarding the efficiency of the
proposed approach.  As we show later in this section, however, this is
not a concern because the optimization problem is well conditioned and
can be solved efficiently.  We also show below that that the basis
functions solving the energy minimization problem are 
locally harmonic within each coarse grid ``element''.  Such property
of the energy minimizing basis suggests that various ``harmonic
extension'' techniques, used to define coarse spaces in multigrid
method
(see~\cite{chan1998agglomeration,Brezina.M;Cleary.A;Falgout.R;Henson.V;Jones.J;Manteuffel.T;McCormick.S;Ruge.J.2001a,Jones.J;Vassilevski.P.2001a})
are very closely related to the energy minimization algorithms.  This
property also suggests that the energy minimizing basis may also be
used for numerical homogenization for problems having a multiscale
nature (see \cite{efendiev2000convergence,hou1999convergence}).

We start our description with a given set of subdomains 
$\Omega_i$ with the property that none of the subdomains is fully
contained in the union of the rest of them.  More precisely, we have,
\begin{equation}\label{dd}
    \Omega=\bigcup_{i=1}^{n_c}\Omega_i\mbox{ and }
\bar\Omega_i\bigcap\bigg(\bigcup_{j\neq i}\Omega_j\bigg)^c\neq\emptyset,
\end{equation}
where the superscript $c$ is the standard set-complement.
Equivalently, in pure algebraic setting, when there is no function
space in the back ground, we may set up $\Omega_i$ as subset of
vertices of the adjacency graph corresponding to a matrix $A$).
The aim is to construct basis functions $\{\phi_i^H\}_{i=1}^{n_c}$ that are
in $\mathcal X_\eta$ with the following additional restrictions:
$$
{\rm supp}(\phi_i^H)\subset\bar\Omega_i, \quad 1\le i\le n_c.
$$

We want the basis functions to have a total minimal energy among all such functions, namely $\{\phi_i^H\}_{i=1}^{J}$ is
the minimizer of:
\begin{equation}\label{min}
    \min \sum_{i=1}^{n_c}\|\psi_i\|_A^2 \quad
    \mbox{subject to}\quad \psi_i\in V_i \text{ and } (\psi_1, \dots,
    \psi_{n_c})\in \mathcal X_\eta. 
\end{equation}
Here 
\begin{equation}
  \label{vj}
V_i=\{v\in V_h: {\rm supp}(v)\subset \bar\Omega_i\}, \quad 1\le i \le n_c.
\end{equation}
Thanks to \eqref{dd}, the decomposition \eqref{VsumVi} holds, namely
\begin{equation}\label{VsumVi}
    V=\sum_{i=1}^{n_c}V_i.  
\end{equation}

\begin{remark}
In AMG, the minimization problem~\eqref{min} written in terms of the
prolongation matrices is as follows: Find
$P\in \mathbb{R}^{n\times n_c}$ such that
\begin{equation}\label{e:energymin} 
P=\argmin_{Y\in \mathbb{R}^{n\times n_c}}\mathcal{F}(Y), 
\quad Y\bm{1}_{n_c} = \bm{1}_n,\quad \mathcal{F}(Y) = 
\trace(Y^TAY). 
\end{equation} 
In terms of vectors, local support means that few non-zeroes per
column (or per row) are allowed in $P$.  We note that the functions
$\{\phi_i^H\}$ satisfying the properties mentioned above are linearly
independent due to the second assumption in \eqref{dd} and the
constraint in~\eqref{e:energymin}. This linear independence is
equivalent to assuming that $P$ is a full rank matrix (i.e. $\operatorname{rank}(P)=n_c$).
\end{remark}

By the assumption in  \eqref{dd}, for each $j$, there exists $k\in \Omega_j$ such that $k\notin \Omega_i$ for all $i\neq j$. We define 
\begin{equation*}
    \mathcal A_j=\{k\in \Omega_j: k\notin \Omega_i, i\neq j\}, \quad j=1: n_c.
\end{equation*}
Then $\mathcal A_j\cap \mathcal A_i=\emptyset$ if $i\neq j$. We then define 
\begin{equation*}
    \mathcal C = \bigcup_{j=1}^{n_c} \mathcal A_j \text{ and } \mathcal F=\Omega\setminus \mathcal C.
\end{equation*}
We define $N_j\in V'$ as follows
\begin{equation*}
    N_j(v)=\frac{1}{|\mathcal A_j|}\sum_{i\in \mathcal A_j}\psi_i(v)=\frac{1}{|\mathcal A_j|}\sum_{i\in \mathcal A_j}v_i, \quad j= 1:n_c.
\end{equation*}
Clearly, $\{N_j\}_{j=1}^{n_c}$ are linearly independent, and if $\operatorname{supp}(v)\subset \Omega_j$, then $N_i(v)=0$ for all $i\neq j$. 

We define $(V')_c\subset V'$ 
\begin{equation*}
    (V')_c=\operatorname{span}\{N_j: j=1:n_c\},
\end{equation*}
and $V_{hf}\subset V$ 
\begin{equation*}
    V_{hf}=\{v\in V: (g, v)=0, \forall g\in (V')_c\}.
\end{equation*}

If $\{\varphi_j\}_{j=1}^{n_c}$ satisfy
\begin{equation*}
    \operatorname{supp}(\varphi_j)\subset \Omega_j \text{ and } \sum_{j=1}^{n_c}\varphi_j =1,
\end{equation*}
then  we have the ``Gramm" matrix $G=(G_{ij})=(N_j(\varphi_i))=I$. By Lemma~\ref{lemma:direct-sum}, we have
\begin{equation*}
    V=V_{hf}\oplus W_c, \quad W_c=\operatorname{span}\{\varphi_j: j=1:n_c\}.
\end{equation*}

Let us first introduce some notation.  We define the restriction $A_i$
of $A$ on each subspace $V_i$ as
\begin{equation}\label{defAi}
(A_iu_i,v_i)=(A u_i,v_i), \quad \forall \ u_i,v_i\in V_i,  \quad i=1:n_c.
\end{equation}

Let $Q_i: V'\mapsto V_i'$ be a projection defined as the adjoint of the
natural inclusion $V_i\subset V$:
$$
\langle Q_i u', v_i\rangle = \langle u', v_i\rangle \quad \forall v_i\in V_i, \;\; u'\in V'.
$$
We now define the following PSC type preconditioner (c.f. \eqref{PSC})
\begin{equation}\label{T}
    B=\sum_{i=1}^{n_c}A_i^{-1}\imath_i'=
    \sum_{i=1}^{n_c}\imath_iA_i^{-1}\imath_i'.   
\end{equation}
Thanks to \eqref{dd}, it is easy to see that the operator $T:V\mapsto
V$ is an isomorphism.

We are now in a position to state and prove the first result in this
section.
\begin{theorem} \label{thm:theoremA} The minimization problem
  \eqref{min} has a unique solution which is given by 
\begin{equation}\label{phi}
\phi_i^H = A_i^{-1} Q_i B^{-1}1
\end{equation}
satisfying
$$
\operatorname{supp}(\phi_i^H)\subset \Omega_i
$$
 
\end{theorem}
\begin{proof}
  This results actually follows directly from Theorem \ref{thm:add}
  with $v=1$.  Let us give a different proof below. 
then is obtained by finding the critical point of the following quadratic
functional:
$$
    L = \sum_{i=1}^{n_c}\bigg(
\frac{1}{2}\|\phi_i\|_A^2 - \langle\lambda,\phi_i\rangle.
\bigg)
$$ 
Differentiating this functional gives that
$$ [{\partial_{\phi_i}} L] \xi_i = (A\phi_i,\xi_i) -
\langle\lambda,\xi_i\rangle,\quad \xi_i\in V_i.
$$ 
Hence the the $i$-th component of the critical point $(\phi_i^H)$
is given by
\begin{equation}\label{crit1}
(\phi_i,\xi_i)_A = \langle\lambda, \xi_i\rangle, 
\quad\forall \xi_i\in V_i, \quad i=1:n_c.
\end{equation}
From the above equations we obtain that 
$$
\phi_i^H= A_i^{-1} Q_i \lambda. 
$$
Summing up leads to: 
$$
\lambda = B^{-1} 1. 
$$
This gives a derivation of \eqref{phi}.

It is obvious that this unique critical point $(\phi_i^H)$ is indeed
the unique global minimizer of \eqref{min} that has a convex objective
functional and a convex constraint.
\end{proof}
We now show that the constructed basis functions are locally discrete
A-harmonic.  We say that a function $w\in V$ is discrete a-harmonic on
a subdomain $D$ if
$$
(w,v)_A=0,\mbox{ for all } v\in V_{h,0}(D)\equiv\{v\in V_h: {\rm supp}(v)
\subset\bar D\}.
$$
This property requires defining the ``subdomains'' $D$ on which it
holds. Below, we introduce such subdomains in terms of function
spaces. Matrix/vector representations of the considerations below are
easy to write.  To define an analogue of coarse grid elements (an
analogue to a finite element coarse grid), we first consider the
set of all the points in $\Omega$ that are interior to all $\Omega_i$'s:
\begin{equation*}
    \omega_0 = \bigg(\bigcup_{i=1}^{n_c}\partial\Omega_i\bigg)^c \bigcap
\Omega.
\end{equation*}
Given $x\in\omega_0$, define the following function with values in the
subsets of $\{1,\ldots,n_c\}$ which is the set of indices of subdomains
$\Omega_i$ that contain $x$:
\begin{equation}\label{Ix}
  I(x)=\{i: x\in\Omega_i\}.
\end{equation}
To rule out any ambiguity we shall assume that for any $x\in
\omega_0$ the set $I(x)$ is ordered in ascending order. We then define
\begin{equation}\label{Kx}
K_x = \{y\in \omega_0: I(y)=I(x)\}.
\end{equation}
Namely $K_x$ is the intersection of all $\Omega_i$ that contain $x$ (see Figure~\ref{fig:supports}).

The following simple proposition will lead us to an appropriate
definition of coarse grid elements.
\begin{proposition}\label{prop:kx}
For the sets $K_x$ defined in (\ref{Kx}) we have
\begin{itemize}
\item[(a)] $K_x=K_y, \quad \Leftrightarrow\quad I(x)=I(y)$.
\item[(b)] Either $K_x\cap K_y=\emptyset$ or $K_x=K_y$, $x\in
\omega_0$, $y\in \omega_0$.
\item[(c)] There are a finite number $m_H$ of different sets $K_x$, 
$x\in \omega_0$.
\end{itemize}
\end{proposition}
\begin{proof} 
The $(\Rightarrow)$ direction in (a) follows from the fact that $x\in
K_x=K_y$, and hence $I(x)=I(y)$. The other direction follows from the
definition of $K_{x,y}$.

To prove (b) let us assume that there exists $z\in \omega_0$, such
that $z\in K_x$ and $z\in K_y$. The definition of $K_x$ and $K_y$ then
gives that $I(x)=I(y)=I(z)$. By (a), $K_x=K_y$.  This proves (b). 

The conclusion (c) follow directly from (b). 
\end{proof}

Let $\mathcal{T}_H$ denote the finite collection of $m_H$ sets in (c)
from the above Proposition \ref{prop:kx}. We have
$$
\omega_0=\bigcup_{x\in \omega_0} K_x=\bigcup_{K\in \mathcal{T}_H} K. 
$$ As it is obvious that $\bar\omega_0=\bar\Omega$,
\begin{equation}\label{EK}
\bar\Omega=\bar\omega_0=\overline{\bigcup_{K\in {\mathcal T}_H} K}
=\bigcup_{K\in {\mathcal T}_H} \bar K.
\end{equation}
This means that the collection of ${\mathcal T}_H$ forms a non-overlapping
partition of $\Omega$.  Each element in ${\mathcal T}_H$ will be called
a {\it coarse grid element}.

\begin{remark}
It is tempting to show how these macroelements look on an
unstructured grid, and in Fig~\ref{fig:supports}, we have depicted
three such supports together with their intersection. But let us point
out that an essential feature of the technique we present here is that
the coarse elements need not be defined {it explicitly} and they might
have quite complicated shape.
\end{remark}
 
\begin{figure}[!htb]
\begin{tabular}{cc}
\includegraphics*[width=2in]{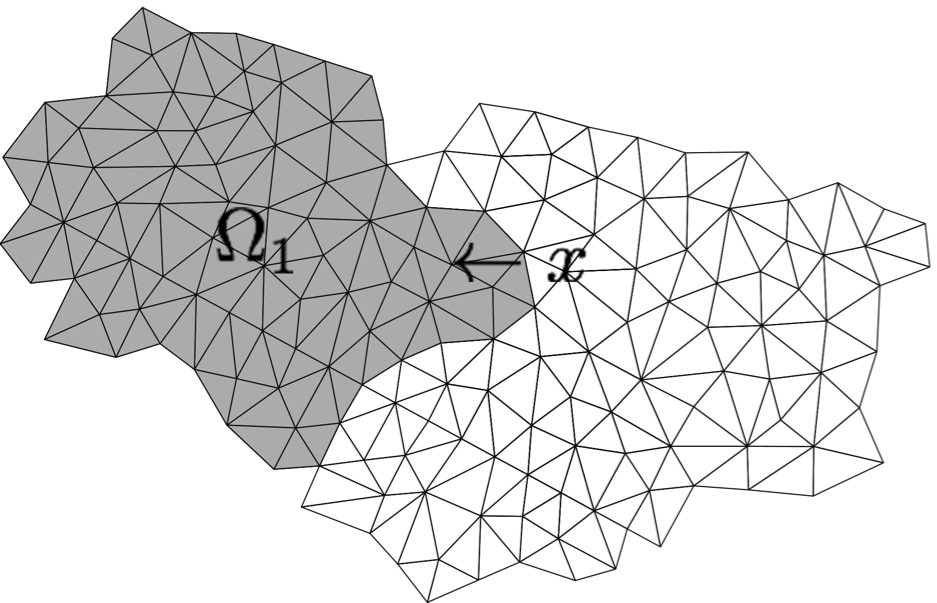}
& \includegraphics*[width=2in]{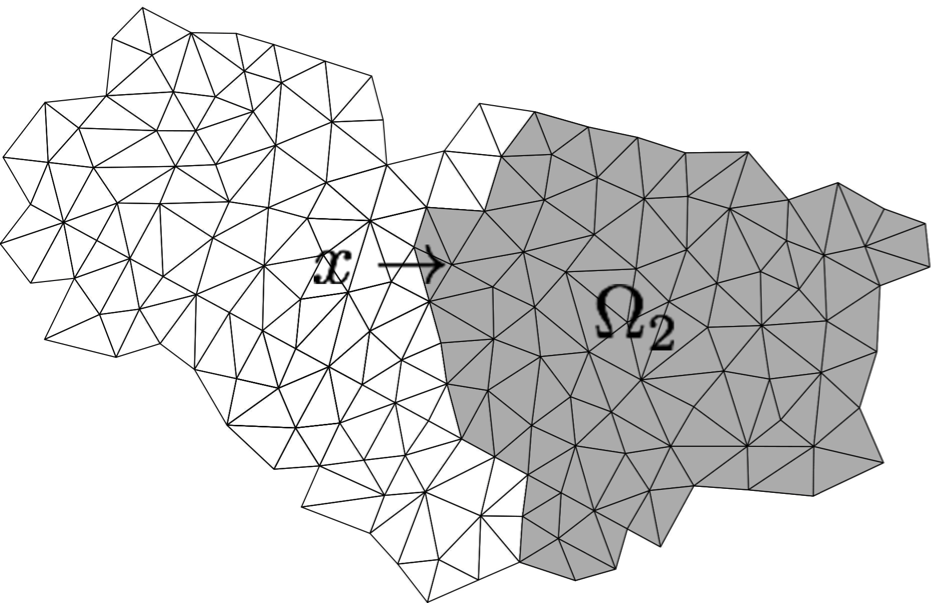}\\
\includegraphics*[width=2in]{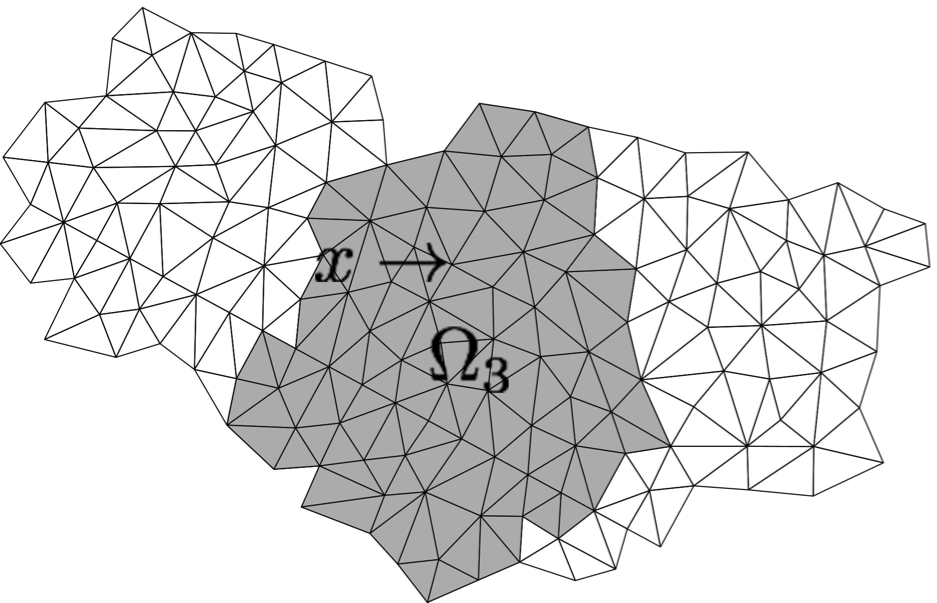} &
\includegraphics*[width=2in]{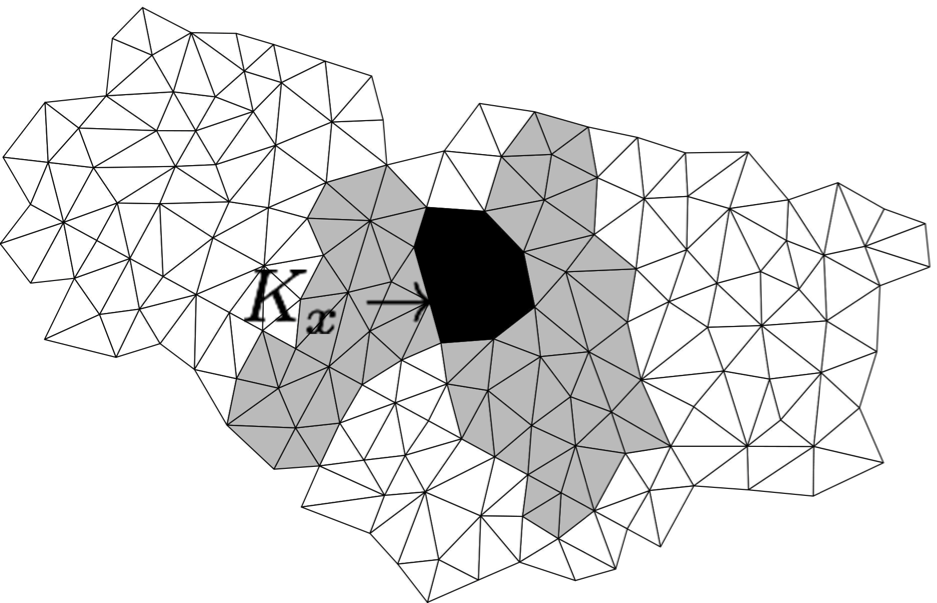}
\end{tabular}
\caption{\label{fig:supports} A piece of triangular grid and supports
of three basis functions. On the right bottom picture, the
intersections are plotted. The darker colored domain correspond to a
coarse element and is intersection of all three supports. The lighter shaded
domains are intersections of two supports and the white area
corresponds to no intersection.}
\end{figure}

\begin{lemma}
Let $\lambda=B^{-1}1$. Assume that for each coarse element $K\in {\mathcal
T}_H$ as defined above, we have $(\bm 1,w_K)_A = 0$ for all $w_K$ supported in $K$. Then
$$
\langle\lambda,\xi\rangle=0,\mbox{ for all } \xi\in V_{h,0}(K).
$$
\end{lemma}
\begin{proof} By definition $K=K_y$ for some $y\in \Omega$.  Thus
$$ 
V_{h,0}(K)=\bigcap_{i\in I(y)} V_i
\mbox{ and } \sum_{i\in I(y)}\phi_i^H(x)=1,
x\in K.
$$
Thus, by (\ref{crit1}), we have 
$$
(\phi_i,\xi)_A = \langle\lambda, \xi\rangle, 
\quad\forall \xi\in V_{h,0}(K)
$$
and 
$$
\langle\lambda, \xi\rangle= 
\sum_{i\in I(y)} (\phi_i,\xi)_A =(\bm{1},\xi)_A=0.
$$ 
The desired result then follows. 
\end{proof}
When the coarsening corresponds to a geometric multigrid and uniform
refinement, the lemma shows that $\lambda=B^{-1}1\in V'$ is a discrete edge
$\delta$-function with respect to the coarse elements (namely
$\lambda$ is supported around $\partial K$). Figure \ref{fig:lambda}
is an example profile of $\lambda$ and a basis function $\Phi_i^H$.

\begin{figure}[!htb]
\centering
\includegraphics*[width=0.4\textwidth]{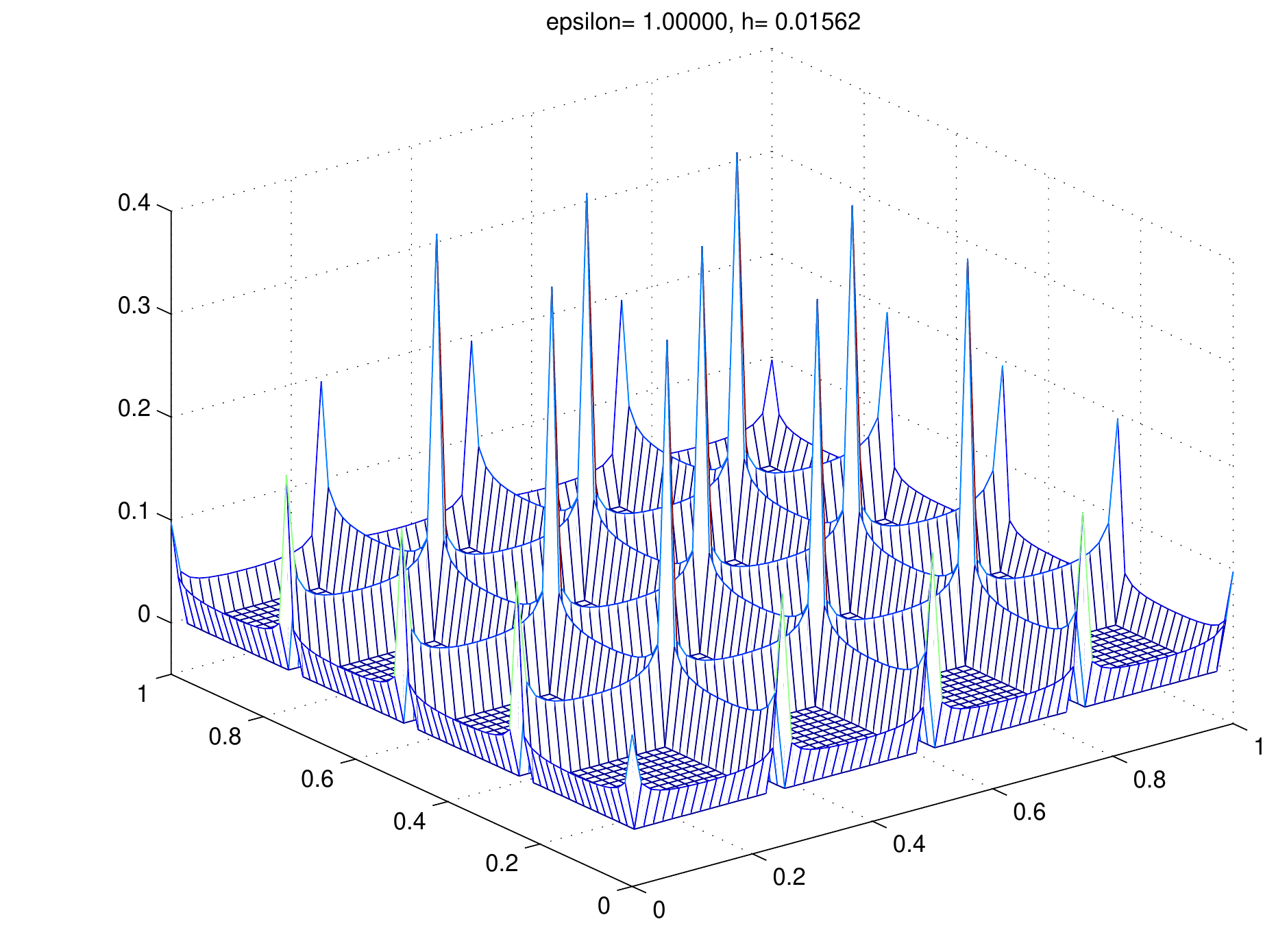}\hfill
\includegraphics*[width=0.4\textwidth]{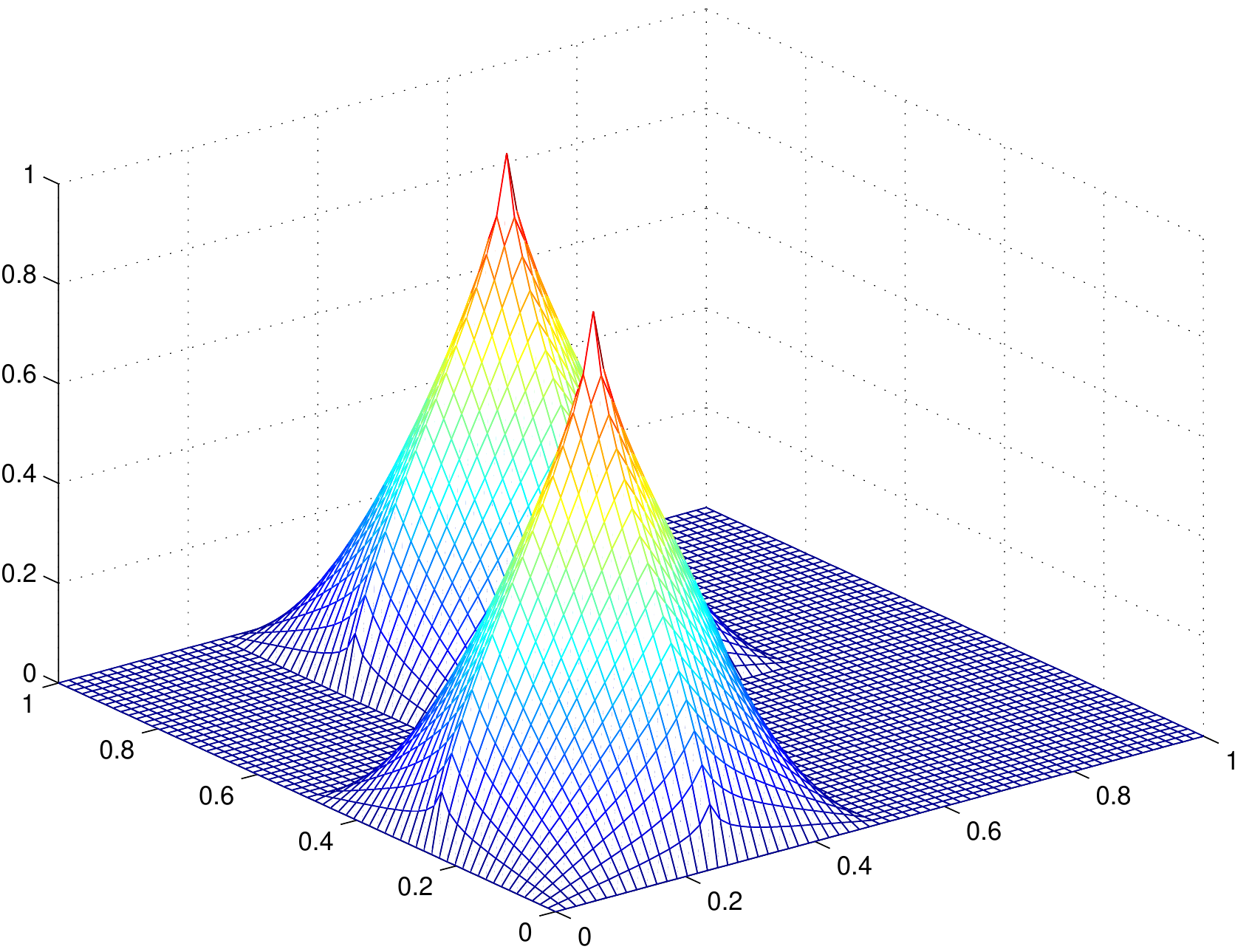}
    \caption{The profile of $\lambda=B^{-1} \mathbf{1}$ (left) and a typical basis function $\phi_i^H$ (right)}
\label{fig:lambda}
\end{figure}
Combining the above result with the identity \eqref{crit1}, we
immediately obtain our second main result in this section. 
\begin{theorem}\label{thm:harmonic}
Each basis function $\phi_i^H$ is discrete a-harmonic on each coarse
element $K\in {\mathcal T}_H$, namely
\begin{equation}\label{harmonic}
(\phi_i^H,v)_A=0,\quad v\in V_{h,0}(K).   
\end{equation}
\end{theorem}
In one space dimension ($d=1$), the above result is rather trivial and
it is in fact already contained in \cite{Wan.W;Chan.T;Smith.B.1999a}.  In
this case, the basis function $(\phi_i^H)$ is analogous to the
generalized finite element basis function in Babu\v{s}ka and
Osborn~\cite{babuvska1983generalized}.

The local harmonic properties in all aforementioned literature are
obtained by {\it construction} from {\it local} element boundaries.
It is interesting to note that the energy minimizing basis studied
here is a result of a more {\it global} construction and the local
harmonic properties is a by-product from the construction.

\subsection{Convergence of energy minimization AMG}
We present a proof for the two-level convergence of AMG based on energy minimization.

We define cut off operators $\chi_j$ on the subdomains $\Omega_j$ introduced in \eqref{dd} as following
\begin{equation*}
    (\chi_j v)(x): = \begin{cases}
        v(x), & \text{if } x\in \bar\Omega_j,\\
        0, & \text{if } x\notin \bar\Omega_j.
    \end{cases}
\end{equation*}
Then we define spaces $W_j$ by
\begin{equation*}
    W_j:= \chi_j V={\chi_jv,: v\in V}.
\end{equation*}
And it is easy to verify that 
\begin{equation}\label{local-basis}
  \{\chi_j\phi_i: \operatorname{supp}\phi_i\cap \Omega_j\neq\emptyset\}
\end{equation}
forms a basis of $W_j$.

The operator $\Pi_j: W_j\mapsto V$ is defined as in \eqref{pi-gamg} with $\phi_j^H$ being the solution of the minimization problem \eqref{min} (\eqref{phi} in Theorem~\ref{thm:theoremA}).

We then have that 
\begin{equation*}
    \sum_{j=1}^{n_c} \Pi_j\chi_j ={\rm Id}
\end{equation*}
In fact, we have 
\begin{equation*}
    \sum_{j=1}^{n_c}\phi_j^H(x)=\sum_{j: x\in\Omega_j}\phi_j(x)=\sum_{j\in I(x)}\phi_j(x)=1,
\end{equation*}
which implies the identity
\begin{equation}
    \sum_{j=1}^{n_c} \Pi_j \chi_j v = I_h\left(\sum_{j=1}^{n_c}\phi_j^H\chi_j v\right)=I_h(v)=v,
    \text{ and hence }    \sum_{j=1}^{n_c} \Pi_j \chi_j =\mathrm{Id}. 
\end{equation}

The operator $A_j: W_j\mapsto W_j'$ is the local restriction of the bilinear form $a(\cdot, \cdot)$ as in \eqref{gmg-Aj}.

\eqref{sum_Aj} is satisfied with decomposition $v=\sum_{j=1}^{n_c}\Pi_jv_j$
where $v_j=\chi_jv\in W_j$:
\begin{equation*}
    \sum_{j=1}^{m_c}\|v_j\|_{A_j}^2  =  \sum_{j=1}^{m_c} a(v, v)_{\Omega_j} \le \co \|v\|_A^2
\end{equation*}
where $\co$ depends on the number of overlaps of $\Omega_j$.

We choose $D: V \mapsto V'$ using the fine grid basis functions as following 
\begin{equation}
    (D \phi_i^h, \phi_j^h): = (A \phi_i^h, \phi_j^h)\delta_{ij}, \quad 1\le i, j\le n.
\end{equation}
Notice that, in the above definition, the matrix representation of $D$ is the diagonal of the matrix representation of $A$. In a similar way, we can define $D_j: W_j\mapsto W_j'$ using basis functions of $W_j$ defined in \eqref{local-basis}, and the matrix representation of $D_j$ is the diagonal of the matrix representation of $A_j$. It is well known that 
\begin{equation*}
    \|v\|_D^2\eqqsim h^{-2}\|v\|_0^2, \forall v\in V, \text{ and } \|v_j\|_{D_j}^2\eqqsim h^{-2}\|v_j\|_0^2, \forall v_j\in W_j.
\end{equation*}

\eqref{assm:D_j} can be verified by the following
\begin{eqnarray*}
    \|\sum_{j=1}^{n_c}\Pi_jv_j\|_D^2 &\lesssim &h^{-2}\int_{\Omega}\left(\sum_{j=1}^{n_c}\phi_j^Hv_j\right)^2= h^{-2}\sum_{i,j=1}^{n_c}\int_{\Omega}\phi_i^Hv_i\phi_j^Hv_j\\
    &\le & h^{-2}\sum_{i,j=1}^{n_c}\|\phi_i\|_{\infty}\|\phi_j\|_{\infty}\int_{\Omega}v_iv_j\\
    &\le & h^{-2}(\max_{1\le j\le n_c}\|\phi_j\|_{\infty})^2\sum_{i, j=1}^{n_c}\int_{\Omega_i\bigcap\Omega_j}v_iv_j\\
    &\le &\co(\max_{1\le j\le n_c}\|\phi_j\|_{\infty})^2\sum_{j=1}^{n_c}h^{-2}\int_{\Omega_j}|v_j|^2\\
    &\lesssim &\co(\max_{1\le j\le n_c}\|\phi_j\|_{\infty})^2\sum_{j=1}^{n_c}\|v_j\|_{D_j}^2.
\end{eqnarray*}

We choose local coarse spaces $W_j^c\subset W_j$ to be the space of constant functions on $\bar\Omega_j$. Then the global coarse space $V_c$ is defined as 
\begin{equation*}
    V_c:=\sum_{j=1}^{n_c}\Pi_jW_j^c.
\end{equation*}
 In fact, for this case, $V_c$ is the subspace spanned by $\{\phi_j^H\}$:
\begin{equation}
    V_c=\operatorname{span}\{\phi_j^H: j=1, \dots, n_c\}.
\end{equation}

We choose the subdomain $\Omega_j$ in a way so that the Poincar\'e inequality holds
\begin{equation*}
    \inf_{v_c\in W_j^c}\|v-v_c\|_0^2\le cd_j^2|v|_1^2,
\end{equation*}
where $d_j$ is the diameter of $\Omega_j$, and $c$ is a constant independent of the mesh size and the size of $\Omega_j$.
Since 
\begin{equation*}
    \|v\|_{D_j}^2\eqqsim h^{-2}\|v\|_0^2, \quad  \|v\|_{A_j}^2 \eqqsim |v|_1^2,
\end{equation*}
We have 
\begin{equation*}
    \inf_{v_c\in W_j^c}\|v-v_c\|_{D_j}^2\le c\left(\frac{d_j}{h}\right)^2\|v\|_{A_j}^2,
\end{equation*}

Combining above discussion with Theorem~\ref{thm:two-level-conv}, we obtain
\begin{theorem}
    The convergence rate of the two-level AMG based on energy minimization is bounded as following
    \begin{equation}
        \|E\|_A\le 1-\mu,
    \end{equation}
    with $0<\mu<1$ depends only on the size and overlaps of the subdomains $\Omega_j$. 
\end{theorem}

\subsection{Bibliographical notes}
The energy minimization seem to encompass several algorithms: the
Lagrange equations for this minimization are solved approximately in
Classical AMG (\S\ref{s:classical-amg}), while the functional is
approximately minimized in the smoothed aggregation (\S\ref{s:agmg}).

For energy minimization approaches we refer to the works by
\cite{Mandel.J;Brezina.M;Vanek.P.1999a,Wan.W;Chan.T;Smith.B.1999a,chan1998agglomeration,Xu.J;Zikatanov.L.2004a,Brannick.J;Zikatanov.L.2006b}.
An interesting fact is that for regularly refined grids, given the
right supports, the trace (energy) minimization prolongation recovers
the coarse basis very closely, although not exactly. The small
discrepancies are due effects from the boundary and the further in
graph distance the coarse grid basis function is, the closer it is to
piece-wise linear.

The extensive numerical experiments reported
in~\cite{Wan.W;Chan.T;Smith.B.1999a,wan1998scalable,Mandel.J;Brezina.M;Vanek.P.1999a,Xu.J;Zikatanov.L.2004a}
show that the energy minimizing basis leads to uniformly convergent
two-grid and multigrid methods for many problems of practical interest
and especially for problems with rough coefficients. These methods
also provide framework for numerical homogenization and are related to
the homogenization methods
in~\cite{1997GrauschopfT_GriebelM_ReglerH-aa} and the M$^3$ techniques
in~\cite{2011LipnikovK_MoultonJ_SvyatskiyD-aa}.

The ``smoothed-aggregation'' approach in the algebraic multigrid
method, as well as theoretical framework is reported in works by
Van\v{e}k, Mandel and
Brezina~\cite{Vanek.P;Mandel.J;Brezina.M.1998a,Mandel.J;Brezina.M;Vanek.P.1999a,Vanek.P;Mandel.J;Brezina.M.1996a} an explicit relation is drawn between the construction
of a base for the coarse space and the ``energy'' of the basis
functions. As pointed out in~\cite{Wan.W;Chan.T;Smith.B.1999a}
and~\cite{Mandel.J;Brezina.M;Vanek.P.1999a}, this can be viewed as one
(or several) steps toward obtaining basis functions minimizing a
quadratic (energy) functional which we consider above. 

Different sparsity patterns including long distance energy minimizing
prolongations are explored in~\cite{2011OlsonL_SchroderJ_TuminaroR-aa},
and for anisotropic problems in~\cite{2012SchroderJ-aa}. 

Constrained energy minimization preserving multiple vectors is
considered in~\cite{2006VassilevskiP_ZikatanovL-aa}. In the FE
settings, when the element stiffness matrices are available, local
energy minimization provides coarse spaces which result in a uniform
two level convergence~\cite{2006KolevT_VassilevskiP-aa}.

\section{Classical AMG}\label{s:classical-amg}
A coarse space in classical AMG is always viewed as a subspace defined
via a prolongation (interpolation) matrix. Its dimension, $n_c$, is a
fraction of the dimension of the finer space. Popular interpolation
schemes in classical AMG are~\emph{direct, standard}, or
\emph{multipass} interpolation. The matrix representations of such
interpolations are of the form $P=\begin{pmatrix}W\\I \end{pmatrix}$,
with $W\in \mathbb{R}^{n_F\times n_C}$, and they can be viewed as
\emph{sparse} approximations to the so called ``ideal'' interpolation
with $W=-A_{FF}^{-1}A_{FC}$, which is, in general, a full
matrix. Here, the matrix $A_{FF}$ is the block of the matrix
corresponding to the $F$-points and $A_{FC}$ is the block
corresponding to the connections between the $C$-points and
$F$-points. The splitting of vertices of the adjacency graph
corresponding to $A$ in subsets $F$ and $C$ is done using one of the
coarsening strategies described in~\S\ref{sc:connection}.

\subsection{Coarse spaces in classical AMG}\label{s:idealHL}

The coarsening algorithm in classical AMG uses 
a splitting of the set of vertices $\{1, 2, \dots, n\}$ of
the graph corresponding to $A$ into two disjoint sets
\begin{equation}\label{e:cf}
    \mathcal C\cup \mathcal F=\{1, 2, \dots, n\}, \quad \mathcal C\cap \mathcal F=\emptyset.
\end{equation}
where $\mathcal C$ is a maximal independent set where the independence is with
respect to the graph of the strength operator $S$ defined
in~\S\ref{s:strength} and the splitting is referred to as a
``$C/F$-splitting''.  A simple greedy $C/F$-splitting algorithm is
introduced in~\S\ref{s:MIS}, Algorithm~\ref{a:MIS}.

We assume that $V$ is equipped with a basis $\{\phi_k\}_{k=1}^n$
and we let $n_f=|\mathcal F|$, $n_c=|\mathcal C|$ to denote the cardinality of the 
sets forming the $C/F$ splitting. Not always, but when convenient we we 
assume that 
\begin{equation}\label{cf-order}
    \mathcal F=\{1,\ldots,n_f\} \text{ and } \mathcal C=\{n_f+1,\ldots,n\}.  
\end{equation}
In this case, the high frequency space $V_{hf}$ as defined in
\eqref{polarVhf} can be written as
\begin{equation}\label{Vhf}
V_{hf}=\operatorname{span}\{\phi_j, j\in \mathcal F\}.
\end{equation}
Following the procedure in \S\ref{sc:connection}, we first proceed to
identify a tentative coarse space $W_c$.  One easiest choice is as
follows:
\begin{equation}
  \label{classicalWc}
W_c=\Span \{\phi_j, j\in \mathcal C\}.
\end{equation}
Obviously the above tentative space satisfies \eqref{VhfWc}, namely
$$
V=V_{hf}\oplus W_c. 
$$
But $W_c$ is hardly a low frequency space.  Given a basis
function in $W_c$, $\phi_{j_k}$, we filter out its high frequency component
to obtain the following coarse basis function,
\begin{equation}
\label{phiHh}
\phi_{k,c}:= (I-\Pi_{hf})\phi_{j_k} = \phi_{j_k} + \sum_{j\notin \mathcal C} p_{jk} \phi_j.  
\end{equation}
Here $P_{hf}$ is the $(\cdot, \cdot)_A$ orthogonal projection onto
$V_{hf}$. By the definition of $\Pi_{hf}$:
$$
\Pi_{hf}\phi_k= \begin{cases} \phi_k, &\text{ if } k\in \mathcal F,\\
-\sum\limits_{j\in \mathcal F}p_{jk}\phi_j, &\text{ if } k\in \mathcal C,\\
\end{cases}
$$
where $p_{jk}$ satisfies
\begin{equation}\label{phf-def}
-\sum\limits_{j\in \mathcal F}p_{jk}(\phi_j, \phi_i)_A = (\phi_k, \phi_i)_A,
\quad \mbox{for all } i\in \mathcal F,\quad k\in \mathcal C.
\end{equation}
In a matrix notation, with the ordering given by \eqref{cf-order},
the matrix form of the equations in~\eqref{phf-def} is
\begin{equation}\label{e:harmonic}
A_{FF} W = A_{FC}, \quad \mbox{where}\quad A = 
\begin{pmatrix}
A_{FF} & A_{FC}\\
A_{FC}^T & A_{CC}
\end{pmatrix}. 
\end{equation}
Here, $W_{jk} = p_{jk},\; j\in \mathcal F, \; k\in \mathcal C$, where $p_{jk}$ are the
coefficients given in~\eqref{phf-def}.

In view of Lemma \ref{lem:basis-in-vc}, we have $S=I-\Pi_{hf}$.  $V_c$
is the span of the functions in \eqref{phiHh}
 \begin{equation*} 
   V_c = \Span\{\phi_{k,c}\}_{k=1}^{n_c}= \operatorname{Range}(I-P_{hf}).
 \end{equation*} 
We note that
$$
(\phi_{1,c},  \ldots, \phi_{n_c,c})=(\phi_1,\ldots, \phi_{n})
\begin{pmatrix}
  W\\
I
\end{pmatrix}.
$$
Thus the corresponding prolongation matrix is
\begin{equation}\label{idealP}
P=\begin{pmatrix}W\\ I\end{pmatrix}, \quad W=-A_{FF}^{-1}A_{FC}.
\end{equation}
The functions $\{\phi_{k, H}\}$ given in \eqref{phiHh} form a basis of
$V_c$ similar to geometric multigrid method for Lagrangian finite
elements.  We denote the prolongation matrix defined in \eqref{idealP}
by $P^{\rm opt}$ and refer to it as the \emph{ideal interpolation}
below.

The following result can be easily established.
\begin{lemma}
  If $A$ satisfies $A \mathbf 1=0$, the solution of~\eqref{e:harmonic}
  also satisfies the constraint in~\eqref{e:energymin}, namely,
  \(W\mathbf 1_{n_c}=\mathbf 1_{n_F}\).
\end{lemma}

Next, we introduce the set of the prolongations $\mathcal P_C$, 
$$
\mathcal P_C=
\left\{P\;\big|\; P=
\begin{pmatrix}
  W\\
I
\end{pmatrix}
\right\}.
$$
and recall from the definition given in \eqref{chi} in \S\ref{sec:energy-min}, that
$$
\mathcal P_C\subset \eta\mathcal X,
$$ 
for some constant $\eta> 0$. For example, $\eta\eqqsim \max_{1\le j\le n}|D_{jj}|$ if $\bar R^{-1} \eqqsim D$, where $D$ is the diagonal of $A$.  
Equivalently, $P\in \mathcal{P}_C$ means that the coarse grid basis is defined as
 \[ \phi_{k,H} = \phi_{j_k} + v_f, \quad j_k\in \mathcal C, \quad v_f\in V_f,
\quad k=1,\ldots,n_c.
 \] 
Clearly, all elements of $\mathcal{P_C}$ are full rank and we also have that 
$P^{\rm opt}\in \mathcal{P}_C$ by definition.

We have the following theorem, showing that the coarse grid matrix corresponding to
 $V^{\rm opt}_c=\range(P^{\rm{opt}})$ has a minimal trace.
\begin{theorem}\label{t:energymin}
If we fix the set of indexes $\mathcal C$
and coarse grid degrees of freedom then for $P^{\rm opt}$ we have
\begin{equation}\label{first-trace-min}
P^{\rm opt}=\argmin_{P\in \mathcal{P}_C}\trace(P^TAP).
\end{equation}
 
Furthermore, if  we denote $A_c=(P^{\rm opt})^TAP^{\rm opt}$, then
\begin{equation}\label{second-trace-min}
    \|v_c\|_{A_c}^2 = \min\left\{\|v\|^2_A \;\;\bigg|\;\;
\langle \phi_{j_k}',v\rangle=v_{c,k}, \; k=1,\ldots,n_c \right\}
\end{equation}
\end{theorem}
\begin{proof} The relation~\eqref{first-trace-min} follows from the following simple
identities.
\begin{eqnarray*} 
\trace(A_c) & = &
\sum_{k=1}^{n_c}\|\phi_{k,H}\|_A^2 =
\sum_{k=1}^{n_c}\|(I-P_{hf})\phi_{j_k}\|_A^2 \\ & = &
\sum_{k=1}^{n_c}\min_{v_{k}\in V_{hf}}\|\phi_{j_k}+v_{k}\|_A^2 \\
& = &
\min\left\{\sum_{k=1}^{n_c}\|\phi_{j_k}+v_{k}\|_A^2\;\;\bigg|\;\;
v_{k}\in V_{hf}, \; k=1,\ldots,n_c\right\}\\ &=&\min_{P\in
\mathcal{P}_C}\trace(P^TAP).
\end{eqnarray*}

To prove \eqref{second-trace-min}, we note that any 
$v\in V$ such that $\langle\phi_{j_k}',v\rangle=v_{c,k}$ can be written as  
\begin{equation*}
    v= w_c+v_f, \text{ where }   w_c=\sum_{k=1}^{n_c}v_{c,k}\phi_{j_k}.
\end{equation*}
By the definition of $P_{hf}$, we have $(I-P_{hf})v_f=0$ and
$v_c=(I-P_{hf})(v-v_f)=(I-P_{hf})v$. We then have
\begin{eqnarray*}
 \|v\|_A^2 &=& \|P_{hf}v\|_A^2+\|(I-P_{hf})v\|_A^2 
=\|P_{hf}v\|_A^2+\|v_c\|_{A_c}^2
\ge \|v_c\|_{A_c}^2
\end{eqnarray*}
\end{proof}

Theorem~\ref{t:energymin} shows that the minimizer $P$ satisfies the
equation~\eqref{e:harmonic}. By Theorem~\ref{t:energymin} it follows
that the minimizer of~\eqref{e:energymin} and the solution
to~\eqref{e:harmonic} are the same in case $A\bm{1}=0$.

\subsection{Quasi-optimality of ideal interpolation}

The following two-level convergence result is well known (see, for example~
\cite{2006MacLachlanS_ManteuffelT_McCormickS-aa}).
\begin{theorem}\label{t:ideal-converges} For $P^{\rm opt}$ defined as
  solution to the minimization problem~\eqref{e:energymin}, or,
  equivalently the solution to~\eqref{e:harmonic}, the two-level AMG
  method with prolongation $P^{\rm opt}$ converges with a rate
  $\|E\|_A\le 1-\delta$, with $\delta$ a constant depending only on
  the maximum degree of the vertices in the graph corresponding to
  $A$ and the threshold $\theta$ in choosing the strong
  connections.
\end{theorem}

\begin{proof}
  According to Corollary~\ref{c:AHF}, we only need to verify that
  $V_{hf}$ consists of $\delta$-algebraic high frequencies as defined
  in Definition~\ref{def:high-frequencies}. Clearly, from the
  discussion in \S\ref{sec:m-matrix}, we can assume $A$ is an
  $M$-matrix with all connections being strong connections.  We
  consider the graph corresponding to $A$ by and recall that the set
  of coarse grid degrees of freedom is a maximal independent set of
  vertices in this graph. By the
  definition of the strength of connection, for any
  $j\in \{1,\ldots,n\}$ and any $i\in N(j)$ we have
    \begin{equation*}
        a_{jj} = \sum_{k\in N_j}-a_{jk}\le |N_j|\max_{k\in N_j}(-a_{jk})\le -\frac{|N_j|}{\theta}a_{ji}
    \end{equation*}

    Next, for any $j\in F$, let $k_j\in \mathcal C$ be such that such that
    $|a_{k_j, j}|>0$. Such $k_j$ exists because the set of
    $C$-vertices is a maximal independent set. This choice is not
    unique, and we just fix one such index $k_j$ for every $j$.  Using
    the notation from \S\ref{sc:connection}, and the fact
    that $v\in V_{hf}$ vanishes at the $C$-vertices, we obtain
    \begin{eqnarray*}
      \|v\|_{\bar R^{-1}}^2 &\le& 
c^D\|v\|_D^2 = c^D\sum_{j\in F}a_{jj}v_j^2 = c^D\sum_{j\in F}a_{jj}(v_j-v_{k_j})^2\\
                            & \le & c^D\sum_{j\in F}-\frac{|N_j|}{\theta}a_{j,k_j}(v_j-v_{k_j})^2\\
                            &\le &\max_{j}(|N_j|) \frac{c^D}{\theta} \sum_{(i,j)\in \mathcal E}-a_{ij}(v_i-v_j)^2\\
                            &\le & \max_{j}(|N_j|) \frac{c^D}{\theta}\|v\|_A^2.
    \end{eqnarray*}
\end{proof}

We note that interpolations like $P^{\rm opt}$ are not (with possible
exception of 1D problems) used in practice, because the prolongation
matrix $P^{\rm opt}$ could have a lot of fill-in and
$(P^{\rm opt})^TAP^{\rm opt}$ is dense. It is, however, also important
to note that sparse approximations to $P^{\rm opt}$ are what is used
in practice. Thus, the energy minimization technique in constructing
coarse space may be viewed as a motivation for the other AMG
techniques used to construct approximations of $V_c^{\rm opt}$.  For
example, most of the known techniques approximate the minimizer of the
functional $\mathcal{F}$ given in~\S\ref{s:energymin-2},
Equation~\eqref{e:energymin}: (1) The prolongation matrices constructed in
the classical AMG framework approximate the solution to equation~\eqref{e:harmonic}
over the subset of $\mathcal{P}_C$ consisting of sparse matrices; (2) The
energy minimization techniques outlined in~\S\ref{s:energymin-2}, and
the smoothed aggregation considered in~\S\ref{s:sa} minimize
(approximately) the trace of the coarse grid matrix $\mathcal{F}(P)$
over a set of sparse matrices $P$.

\subsection{Construction of prolongation matrix $P$}\label{sec:classicP}

\subsubsection{Prolongation}

An intuitive idea to find an approximate solution of the problem
\eqref{e:harmonic} is to use some basic iterative methods such as
Jacobi method and then properly rescale the coefficients of $W$ so
that it satisfies
 \begin{equation}\label{const_preserv}
     W\mathbf 1_{n_c}=\mathbf 1_{n_f},
 \end{equation}
and also fits into a sparsity pattern:
 \begin{equation}\label{sparsity}
     W\in \mathbb{R}_{S}^{n_f\times n_c}.
 \end{equation}
 Following this idea, we construct the interpolations used in
 classical AMG~\cite[Section~A.7]{Trottenberg.U;Oosterlee.C;Schuller.A.2001a}: (1)direct
 interpolation; (2) standard interpolation; and (3) multi-pass
 interpolation in a unified fashion.
\begin{itemize}
\item \emph{Direct interpolation}: Direct interpolation approximates the solution to 
  \eqref{e:harmonic}--\eqref{sparsity} by one Jacobi iteration with
  initial guess $W_0=0$, namely,
        $$W_1=-D_{FF}^{-1}A_{FC}.$$
        In order to satisfy the constraint~\eqref{const_preserv}, we
        rescale $W_1$, to obtain that
        \begin{equation*}
\label{DirectW}
            W=M(-D_{FF}^{-1}A_{FC})=[\operatorname{diag}(A_{FC}\bm 1)]^{-1}A_{FC}.
        \end{equation*}
        where $M$ is a rescaling operator defined as follows
        \begin{equation}\label{rescaling}
            M(Y)= [\operatorname{diag}(Y\bm{1})]^{-1}Y
        \end{equation}
      \item \emph{Standard interpolation}: The construction of
        prolongation matrix via \emph{standard interpolation} can be
        viewed in several different ways. The most natural way probably is
        to view it as a smoothing of the direct interpolation again followed
        by rescaling. 
        Indeed, assuming that the smooth error $(e_F^T,e_C^T)$ satisfies  
        \[
        A_{FF} e_F +A_{FC}e_C = 0,
        \]
        we approximate this equation (using the same notation $e_F$,
        $e_C$ for the approximations) by (see
        \cite{Trottenberg.U;Oosterlee.C;Schuller.A.2001a}):
        \begin{eqnarray}
          e_F&=&-D_{FF}^{-1}A_{FC} e_C + D_F^{-1}A_{FF}(D_{FF}^{-1}A_{FC}) e_C\label{e:ec-ef}\\
          W &= & (I-D_{FF}^{-1}A_{FF}) W^1,\quad W^1=-D_{FF}^{-1}A_{FC}. \label{e:Wclassic}
        \end{eqnarray}
        This is equivalent to a Jacobi smoothing iteration applied to
        the unscaled direct interpolation $W^1$.  Rescaling
        is also needed for the standard interpolation and the final
        formula is
        \begin{equation}\label{i:standard}
          W = 
 [\operatorname{diag}((I-D_{FF}^{-1}A_{FF})W^1\bm 1)]^{-1}(I-D_{FF}^{-1}A_{FF})W^1.
        \end{equation}
        The similarity with the smoothed aggregation discussed in
        \S\ref{s:sa} is obvious, as this is indeed a smoothing applied
        to $W^1$.  

      \item \emph{Multipass interpolation}: The \emph{multipass
          interpolation} is an approximation to the solution
        of~\eqref{e:harmonic} when the $C/F$ splitting has been
        constructed by means of aggressive coarsening, namely, using
        the $(m, l)$-strong connection defined in
        Definition~\ref{def:extend-strong-connection-2}. The set of $F$ is divided into 
$l$ disjoint subsets
        $F_1, F_2, \dots, F_l$ as follows: we first define the distance from a point $j$
        to a subset of points $C$,
        \[ 
        \operatorname{dist}(j, C)=\min \{\operatorname{dist}(j, i): i\in C\}.
        \] 
        Then, we set
        \[ 
        F_k=\{j\in \mathcal F: \operatorname{dist}(j, \mathcal C)=k\}, \quad k=1, 2, \dots, l,\quad 
        l=\max \{\operatorname{dist}(j, \mathcal C): j\in F\}. 
       \] 
        Then $A_{FF}$, $A_{FC}$ can be written as the following block matrices: 
        \begin{equation}
            A_{FF}= \begin{pmatrix}
                    A_{F_1F_1} & A_{F_1F_2} & \cdots & A_{F_1F_l} \\
                    A_{F_2F_1} & A_{F_2F_2} & \cdots & A_{F_2F_l} \\ 
                    \vdots & \vdots & \ddots & \vdots \\
                    A_{F_lF_1} & A_{F_lF_2} & \cdots & A_{F_lF_l} 
            \end{pmatrix},\quad
            A_{FC} = \begin{pmatrix}
                A_{F_1C}\\
                A_{F_2C}\\
                \vdots \\
                A_{F_lC}
                \end{pmatrix}.
        \end{equation}
        We can also write $W$ block-wise as $W^T=(W_{F_1}^T, W_{F_2}^T,\ldots,W_{F_l}^T)$ and 
then \eqref{e:harmonic} takes the form
        \begin{equation}\label{multipass_intp_apprx}
            \left\{ \begin{array}{c}
                A_{F_1F_1}W_{F_1} +A_{F_1F_2}W_{F_2} +\cdots + A_{F_1F_l}W_{F_l}+A_{F_1C}=0\\
                A_{F_2F_1}W_{F_1} +A_{F_2F_2}W_{F_2} +\cdots + A_{F_2F_l}W_{F_l}+A_{F_2C}=0\\
                \vdots\\
                A_{F_lF_1}W_{F_1} +A_{F_lF_2}W_{F_2} +\cdots + A_{F_lF_l}W_{F_l}+A_{F_lC}=0\\
            \end{array}
                \right.
        \end{equation}
        To define the entries of $W_{F_k}$ via  multipass interpolation 
        we use the following steps:
        \begin{enumerate}
        \item For $k=1$, use direct interpolation to approximate
          $W_{F_1}$. More precisely, we write the first equation in
          \eqref{multipass_intp_apprx} as
                $$
                    A_{F_1F_1}W_{F_1} +\hat A_{F_1C} =0,
                $$
                with $\hat A_{F_1C} =A_{F_1F_2}W_{F_2} +\cdots + A_{F_1F_l}W_{F_l}+A_{F_1C}$. Then apply direct interpolation and write $W_{F_1}$ as a function of $W_{F_j}$, $j>1$. 
              \item While $k<l$: 
                \begin{enumerate} 
                \item Write the $(k+1)$-st equation in
                  \eqref{multipass_intp_apprx} substituting the
                  expressions for $W_{F_m}$, $m < (k+1)$ obtained from
                  the previous steps. The
                  $(k+1)$-th equations then has the form
                \begin{equation}\label{e:fk1} \hat A_{F_{k+1}F_{k+1}}W_{F_{k+1}}+
\hat A_{F_{k+1}F_{k+2}}W_{F_{k+2}}+ \cdots \hat A_{F_{k+1}F_l}W_{F_l}+
\hat A_{F_{k+1}C}=0.
                \end{equation} 
              \item  Apply direct interpolation to~\eqref{e:fk1} to write 
                $W_{F_{k+1}}$ as a function of  
                $W_{F_{j}}$, $j>(k+1)$. 
              \item Set $k\leftarrow k+1$.
              \end{enumerate}
            \end{enumerate}
          \end{itemize}

          From the derivation and the definitions above we have the
          following set inclusions describing the sparsity patterns of
          the prolongations defined earlier:
\begin{eqnarray*}
&& \sparse(P) \subset 
\sparse
\begin{pmatrix}
A_{FC}\\I
\end{pmatrix}
    \quad \mbox{(direct interpolation \eqref{DirectW})}\\
&& \sparse(P) = 
\sparse
\begin{pmatrix}
\widehat A_{FC}\\
I
\end{pmatrix}
    \quad \mbox{(standard interpolation~\eqref{i:standard})}\\
&& \sparse(P) = 
\sparse
\begin{pmatrix}
\widehat A_{F_1C}\\
\widehat A_{F_2C}\\
\vdots\\
\widehat A_{F_lC}\\
I
\end{pmatrix}
    \quad \mbox{(multipass interpolation)}.
\end{eqnarray*}

\subsubsection{Interpolation preserving a given vector}

From the previous discussions, it is important to have the
prolongation $P$ preserves some vectors which can represent algebraic
smooth errors. The direct and standard interpolation use a diagonal
scaling to make sure the prolongation preserves constant vectors,
which is the kernel of scalar elliptic operators. To generalize this
idea, here we introduce $\alpha$AMG interpolation  which
constructs prolongation matrix by choosing an initial guess that
preserves the near null component $v^{(1)}$, which may not necessarily
be constant vectors.

We construct  $P\in \mathcal{P}_C$ such that it ``preserves'' 
a given vector $v$, namely,
\begin{equation*}
    v_F=Wv_C.
\end{equation*}
To do this, we first pick an initial guess $W^0$ for $W$ (or $P$), 
\begin{equation}\label{aAMG_W0}
W^0=D_v^{-1}A_{FC},
\end{equation}
where $D_v$ is a diagonal matrix such that the following identity holds
\begin{equation*}
v_F=D_v^{-1}A_{FC}v_C.
\end{equation*}
It is easy to derive that the explicit formula for diagonal entries of $D_v$ is
\begin{equation*}
    d_{kk}=\frac{\sum_{j\in C}a_{kj}v_j^{(1)}}{v_k^{(1)}}.
\end{equation*}
Then the $W$ in such ``vector preserving'' interpolation is obtained
by applying one Jacobi iteration for the linear problem~\eqref{e:harmonic}
with initial guess $W^0$
\begin{equation*}\label{aAMG_W}
    W=D^{-1}A_{FC}+D_{FF}^{-1}(-A_{FC}-A_{FF}D^{-1}A_{FC}).
\end{equation*}
A fully detailed description of the construction a prototype vector
$v$ and coarse space interpolating this vector exactly, using
$\alpha$AMG (the classical AMG version of adaptive AMG) is found
in~\cite{aAMG}. The ideas, however are outlined earlier
in~\cite{1stAMG}.

\subsection{Classical AMG within the abstract AMG framework}\label{sec:CAMG-coarse}
The classical AMG falls in the abstract theory we developed earlier
in~\S\ref{s:2Level} and \S\ref{sec:unifiedAMG}.  To do this, we first
consider an \emph{$M$-matrix relative} of $A$ using the adjacency
graph corresponding to a strength function.  We then use an MIS
algorithm to identify $C$, the set of coarse points, to form a
$C/F$-splitting.  We further split the set of indices $\Omega=\{1, 2,
\dots, n\}$ into subsets $\Omega_1, \Omega_2, \dots, \Omega_J$ so that
\begin{equation}\label{Omegaj}
\Omega=\bigcup_{j=1}^J\Omega_j.
\end{equation} 
Then for each $j\in C$ we define 
\begin{equation}\label{CAMG_Omega}
    \Omega_j:=\{j\}\bigcup F_j^s, \quad j=1, \dots, J.
\end{equation}
where $F_j^s:=\mathcal F\bigcap s_j$, and $s_j$ is the set of strong neighbors
of $j$. This depends on the definition of strength of connection. For example, in the direct interpolation we introduced in the previous section, we simply use the strength connection defined in \eqref{def:strong-connection-2}; in the standard interpolation, we use $(m, l)$-strong connection defined in Definition~\ref{def:extend-strong-connection-2} with $m=1$ and $l=2$.

For each $\Omega_j$ we denote 
\begin{equation}
    \Omega_j=\{m_1, m_2, \dots, m_{n_j}\},
\end{equation}
and let $n_j:=|\Omega_j|$, namely, $n_j$ is the cardinality of $\Omega_j$. 
In accordance with the notation in~\S\ref{sec:unifiedAMG}. We then define 
\begin{equation}
    V_j:=\mathbb{R}^{n_j},
\end{equation}
and the associated operator $\Pi_j: V_j\mapsto V$ 
\begin{equation}\label{pi-camg}
 (\Pi_jv)_i = \begin{cases}
        p_{m_k,k}v_{k}, & \text{ if } i=m_k,\\
        0, & \text{ if } i\notin \Omega_j\\
    \end{cases},
\end{equation}
where $p_{m_k,k}$ are given weights.  As all the constructions below
will be based on the $M$-matrix relative of $A$, and without loss of
generality, we may just use $A$ to denote the $M$-matrix relative.

Following~\S\ref{sec:unifiedAMG}, we introduce the operator $\chi_j: V\mapsto V_j$:
\begin{equation}
    (\chi_jv)_i: = v_{m_i}.
\end{equation}
which takes as argument a vector $v$ and returns only the portion of
it with indices are in $\Omega_j$. Namely, $\chi_j v$ is a vector in
$\mathbb{R}^{n_j}$.  It is immediate to verify that
\begin{equation*}
    \sum_{j=1}^J\Pi_j\chi_j= I.
\end{equation*}
To estimate the convergence rate using the theory in~\S\ref{s:2Level}
We need to verify all the items in Assumptions~\ref{a:2level}. 
To do this, we choose a decomposition $v=\sum_{j=1}^J \Pi_j v_j$ with
$v_j=\chi_jv$. We further define
$C_{p,2}$ as a constant 
depending on the overlaps in the partition $\{\Omega_j\}_{j=1}^J$, 
\begin{equation}
  \co=
  \max_{1\le j \le J} \left|\{l\; : \; \Omega_l\cap \Omega_j \neq \emptyset\}\right|.
    \end{equation}
The local operators $A_j$ on $V_j$ are defined as follows
    \begin{equation}\label{Aj}
        (A_ju, v)=\sum_{\substack{e\in \mathcal E\\e\subset \Omega_j}}\omega_e\delta_{j,e}u\delta_{j,e}v.
    \end{equation}
    Here, $e\subset \Omega_j$ means the two vertices connected by $e$ are in $\Omega_j$, and $\delta_{j, e}u=u_{m_k-m_l}$ for $e=(k, l)$ and $\omega_e$ are the eights determined by the off-diagonal elements in the  $M$-matrix relative of $A$. 
Notice that 
$A_j$ is a symmetric positive semi-definite matrix 
because all the weights $\omega_e$ are non-negative.
Then \eqref{sum_Aj} easily verified:
\begin{equation}\label{verifyA}
    \begin{array}{rcl}
      \sum_{j=1}^{m_c}\|\chi_jv\|_{A_j}^2&=&\sum_{j=1}^{m_c}\sum_{\substack{e\in \mathcal E\\e\subset \Omega_j}}\omega_e(\delta_ev)^2\le \co\sum_{e\in \mathcal E}\omega_e(\delta_e v)^2\\
                                         &=& \co\|v\|_{A_M}^2\le \co C_M\|v\|_A^2,
    \end{array}
\end{equation}
If $D$ is the diagonal of $A$, then we set $D_j$, $j=1:J$ to be the
restriction of $D$ on $\Omega_j$, namely, in
$\mathbb{R}^{n_j\times n_j}$ and
\begin{equation}\label{local_diag}
    (D_j)_{ii}=D_{m_i, m_i}, \quad\mbox{or equivalently}\quad D_j=\chi_jD_j\chi_j'
\end{equation}
We have the following lemma which shows \eqref{assm:D_j}.
\begin{lemma}\label{lem:Dj}
For $D_j$ defined in \eqref{local_diag}, the following inequality holds
    \begin{equation}
        \|\sum_{j=1}^{m_c}\Pi_jw_j\|_D^2\le \co\sum_{j=1}^{m_c}\|w_j\|_{D_j}^2, \quad \forall w_j\in V_j.
    \end{equation}
\end{lemma}
\begin{proof}
    Recall from the definition of $\Pi_j$, we have 
    \begin{equation}
        \|\Pi_jv\|_D\le \|v\|_{D_j}, \quad \forall v\in V_j.
    \end{equation}
Therefore, 
    \begin{eqnarray*}
        \|\sum_{j=1}^{J}\Pi_jv_j\|_D^2 &= & \left(D\sum_{i=1}^{J}\Pi_iv_i, \sum_{j=1}^{J}\Pi_jv_j\right)=  \sum_{i=1}^{J}\sum_{j=1}^{J}(D\Pi_iv_i, \Pi_jv_j)\\
                  &= & \sum_{\substack{1\le i, j\le J\\\Omega_i\cap\Omega_j\ne \emptyset}}(D\Pi_iv_i, \Pi_jv_j)\le  \sum_{\substack{1\le i, j\le J\\\Omega_i\cap\Omega_j\ne \emptyset}}\frac{\|\Pi_jv_i\|_D^2+\|\Pi_jv_j\|_D^2}{2}\\
                  &\le & \co\sum_{j=1}^{J}\|v_j\|_{D_j}^2. 
    \end{eqnarray*}
\end{proof}
We choose the local coarse spaces $V_j^c$ as
\begin{equation}
    V_j^c:=\operatorname{span}\{\bm 1_{n_j}\},
\end{equation}
Then by definition, we have
\begin{equation}
    \mu_j(V_j^c)=\frac{1}{\lambda_j^{(2)}},
\end{equation}
where $\lambda_j^{(2)}$ is the second smallest eigenvalue of the matrix $D_j^{-1}A_j$.
The global coarse space $V_c$ is then obtained by~\eqref{V_c}, and is
\begin{equation}
    V_c=\operatorname{span}\{P_1, P_2, \cdots, P_J\}.
\end{equation}
Finally, by Theorem~\ref{thm:two-level-convergence}, the converges rate of this
two-level geometric multigrid method depends on the
$\min_j(\lambda_j^{(2)})$. If the discrete Poincar{\'e } inequality is
true for each $V_j$, namely ,
\begin{equation}\label{local_poincare}
    \inf_{v_c\in V_j^c} \|v-v_c\|_{D_j}^2\le c_j\|v\|_{A_j}^2, \quad \forall v\in V_j,
\end{equation}
with $c_j$ to be a constant, then the two-level classical AMG method converges uniformly.

\subsection{Bibliographical notes}
The classical coarse space definition in AMG was introduced in
\cite{1stAMG}, and then somewhat improved in~\cite{Stuben.K.1983a,Brandt.A;McCormick.S;Ruge.J.1985a,Ruge.J;Stuben.K.1987a}.

Despite its great success in practical applications, classical AMG
algorithms still lack solid theoretical justifications beyond theory
for two-level methods. It is important to note that a multigrid method
that converges uniformly in the two-level case with an exact coarse
grid solver may not converge uniformly in the multilevel case (see
Brandt~\cite{Brandt.A.1986a} and Ruge and
St\"{u}ben~\cite{Ruge.J;Stuben.K.1987a}).  For classical AMG, the
early theoretical studies of convergence date back to the 1980's (see
\cite{Mandel.J.1988a,McCormick.S.1982a,Brandt.A.1986a,Ruge.J;Stuben.K.1987a}).
A survey of these results by St\"{u}ben is found in the
monograph~\cite{Trottenberg.U;Oosterlee.C;Schuller.A.2001a} where also
the three classical prolongation constructions (direct, standard and
multipass) are given.  In all cases, it is crucial to define
coarsening and interpolation operators so that the interpolation error
is uniformly bounded.  The role of the ideal interpolation as
minimizer of an upper bound for the convergence rate was emphasized
in~\cite{Falgout.R;Vassilevski.P.2004a}.

\section{Aggregation-based AMG}\label{s:agmg}
Aggregation (or agglomeration) refers to an algorithm that splits the
set of vertices of the graph of the filtered matrix as a union of
non-overlapping subsets (aggregates) (each of which forms a connected
sub-graph): 
\begin{equation}
\label{agg1}
\{1, \dots, n\}=\bigcup_{j=1}^{J}\mathcal A_j, \quad \mathcal
A_i\bigcap \mathcal A_j=\emptyset, i\neq j,
\end{equation}
Such a partition can be obtained algorithms described in \S\ref{sc:connection}.

If we are solving a finite element system, the partition \eqref{agg1}
would correspond to a non-overlapping decomposition of $\Omega$ 
 \begin{equation}\label{UA_Omega}
 \Omega=\bigcup_{j=1}^J\Omega_j, 
\quad \Omega_i\bigcap \Omega_j=\emptyset, i\neq j
\end{equation}
such that $\mathcal A_j$ contains the indices associated with the
enumerations of the vertices in the subdomain $\Omega_j$. 

We denote the elements in $\mathcal A_j$ by 
\begin{equation}
    \mathcal A_j=\{m_1, m_2, \dots, m_{n_j}\},
\end{equation}
and let $n_j:=|\mathcal A_j|$, namely, $n_j$ is the number of elements
in $\mathcal A_j$.

\subsection{Unsmoothed aggregation: preserving $1$ kernel vector}
Using the framework we introduced in \S\ref{sec:unifiedAMG}, we define
\begin{equation*}
    V_j:= \mathbb{R}^{n_j},
\end{equation*}
and the associated operator $\Pi_j: \mathbb R^{n_j}\mapsto \mathbb
R^n$ is the trivial extension of $v\in \mathbb R^{n_j}$
\begin{equation}\label{UA_PI}
    (\Pi_jv)= \begin{cases}
        v_k & i=m_k,\\
        0 & i\notin \mathcal A_j.
    \end{cases}
\end{equation}
\begin{equation}
  \label{aggP0}
P=(p_1, p_2, \dots, p_J), \quad     p_j=\Pi_j\mathbf{1}_{n_j}.  
\end{equation}
This prolongation matrix obviously satisfies
\begin{equation}
\label{aggP0-1}
P\mathbf 1_J=\sum_{j-1}^JP_j=\mathbf 1_n.   
\end{equation}
The local coarse space $V_j^c$ is 
\begin{equation*}
    V_j^c:=\operatorname{span}\{\bm{1}_{n_j}\}
\end{equation*}
Then the global coarse space $V_c$ is obtained by \eqref{V_c}. In fact
\begin{equation*}
    V_c=\operatorname{span}\{p_1, p_2, \dots, p_J\}.
\end{equation*}

We note that 
\begin{equation}
(\phi_{1,c},\ldots,\phi_{J,c})=(\phi_1,\ldots,\phi_n)P
\end{equation}
where
\begin{equation}\label{AggBasis0}
    \phi_{j,c}=\sum_{k\in \mathcal A_j}\phi_k, \quad j=1, 2, \dots, J.
\end{equation}
Furthermore, $\Pi_j$ corresponds to matrix representation of the
operator $I_h(\phi_j^H\cdot )$ with the coarse grid basis $\phi_{j,c}$
defined above.   In view of \S\ref{s:gmg-from-amg} (see especially
\eqref{pi-gamg}).  Aggregation based AMG can be also viewed as a
GMG method. 

The above procedure gives a full description of the un-smoothed
aggregation AMG method.  We can use the framework in
S\ref{sec:unifiedAMG} to carry out a two-level convergence analysis.
The local matrices $A_j$ are defined exactly the same as in the case
of classical AMG in \S\ref{sec:CAMG-coarse}, namely, we write $A$ as
in \eqref{sec:CAMG-coarse} and then define $A_j$ by \eqref{Aj}. The
matrices $D_j$ are defined as the restriction of the diagonal of $A$
on $\Omega_j$ as in \eqref{local_diag}. Using the same argument as in
\S\ref{sec:CAMG-coarse}, Assumption~\ref{a:2level} is verified by
\eqref{verifyA} and Lemma~\ref{lem:Dj}.
Theorem~\ref{thm:two-level-convergence} can then be applied to prove
that the two-level unsmoothed aggregation method has a convergence
rate depending only on the local Poincar\'e constants in
\eqref{local_poincare}.

\subsection{Unsmoothed aggregation: preserving 
multiple vectors}\label{s:P-many}
One great advantage of aggregation AMG method is that it can be
easily extended to case that the stiffness matrix $A$ has multiple
kernel or near-kernel vectors. 

To give an illustration how to construct prolongation preserving more
than one vector, we consider the 2D elasticity problems.  In this
case, we have 3 functions, namely the rigid body motion, to preserve:
$$
u_1=
\begin{pmatrix}
  1\\
0
\end{pmatrix}, 
u_2=
\begin{pmatrix}
  0\\
1
\end{pmatrix}, 
u_3=
\begin{pmatrix}
-y\\
x
\end{pmatrix}
$$
The corresponding vectors are
$$
\zeta_1=
\begin{pmatrix}
  1\\
0\\
  1\\
0\\
\vdots\\
1\\
0
\end{pmatrix}, 
\zeta_2=
\begin{pmatrix}
0\\
1\\
0\\
1\\
\vdots\\
0\\
1
\end{pmatrix} \mbox{ and }
\zeta_3=
\begin{pmatrix}
-y_1\\
x_1\\
-y_2\\
x_2\\
\vdots\\
-y_n\\
x_n
\end{pmatrix}\in \mathbb R^{2n}
$$
On each aggregate $\mathcal A_j$, we consider 
$$
\zeta_1^{(j)}=
\begin{pmatrix}
  1\\
0\\
  1\\
0\\
\vdots\\
1\\
0
\end{pmatrix}, 
\zeta_2^{(j)}=
\begin{pmatrix}
0\\
1\\
0\\
1\\
\vdots\\
0\\
1
\end{pmatrix} \mbox{ and }
\zeta_3^{(j)}=
\begin{pmatrix}
-y_{m_1}\\
x_{m_1}\\
-y_{m_2}\\
x_{m_2}\\
\vdots\\
-y_{m_{n_j}}\\
x_{m_{n_j}}
\end{pmatrix}\in \mathbb R^{2n_j}
$$
The prolongation matrix is then given by 
$$
P=(p_1, p_2, \ldots, p_J)\in \mathbb R^{(2n)\times(3J)}\quad
P_j=\Pi_j(\zeta^{(j)}_1,\zeta^{(j)}_2,\zeta^{(j)}_3)\in\mathbb
R^{(2n)\times 3}.
$$
This prolongation matrix satisfies
\begin{equation}
  \label{P1=1}
P (\mathbf{1}_{J}\otimes e_j )=\zeta_j, 1\le j\le 3.  
\end{equation}

The rest of the AMG algorithm based on this prolongation matrix is
similar to the $1$-vector case. 

Extension from 3-vector case as mentioned above to arbitrary
$m$-vector is straightforward.  We denote $m=\dim (N(A))$ and consider
the case when $m\ge 1$.  We now assume that we are given $m$ vectors
$\{\zeta_j\}_{j=1}^m$ which are linearly independent and $m\ll n$ and
we can then proceed to construct a prolongation $P$ such that
$P(\bm{1}_{n_c}\otimes e_j) = \zeta_j$, $j=1:m$.

We would like to point out for the multiple kernel or near-kernel
vector case, we often need some geometric information to describe the
corresponding vectors.  In the finite element case, these vectors
should be obtained by taking the canonical interpolation of the
corresponding kernel functions of the underlying partial differential
operators and we then split these vectors into different aggregates by
using the local degrees of freedom associated with different
aggregates. 

\subsection{Smoothed aggregation}\label{s:sa}
The aggregation procedure described above provides a simple and yet
efficient AMG method.  The resulting method is known as un-smoothed
aggregation AMG (UA-AMG) method.  Such a terminology may be justified
by examining the shape of the basis function as defined by
\eqref{AggBasis0}.  A typical basis in this form, 
is not smooth.  There is a procedure to smooth
out these basis function by using a smoother to smooth these basis
functions, which is equivalent to a few applications of smoothing on
the prolongation matrix as defined \eqref{aggP0} as follows:
\begin{equation}
\label{aggP1}
P_S=(I-R_sA)^\nu P, \quad \mbox{for some } \nu\ge 1.   
\end{equation}
A typical choice is the scaled Jacobi smoother with $R_S=\omega
D^{-1}$.   We note that, in view of \eqref{P1=1},  if $\zeta_j$ is in
the kernel of $A$, we still have 
$$
P_S (\mathbf{1}_{J}\otimes e_j )=\zeta_j, 1\le j\le m.   
$$

Conceivably, a bigger $\nu$ in \eqref{aggP1} would lead to an AMG
algorithm that has a better convergence rate.  But a bigger $\nu$ in
\eqref{aggP1} also means a denser $P_S$ and hence a denser $A_c$, and
hence a more expensive setup for the resulting AMG algorithm.  As an
example, the graphs of the coarse grid matrices corresponding to unsmoothed and
smoothed aggregation are shown in
Figure~\ref{fig:unsmoothed-unsmoothed}. Clearly, the smoothed
aggregation graph is denser than the unsmoothed one. 
\begin{figure}[!htpb]
\begin{center}
\includegraphics*[width=0.35\textwidth]{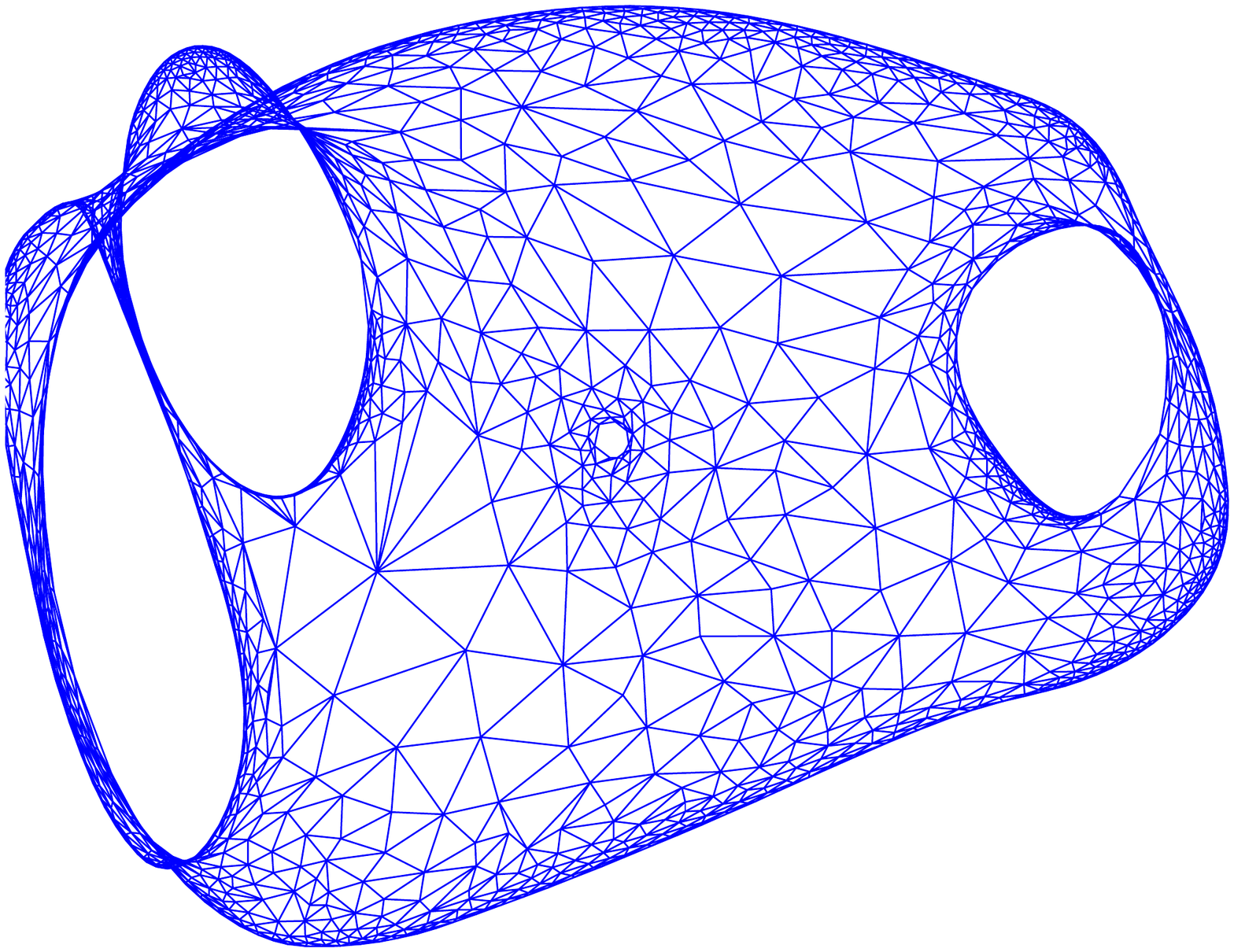}\hspace{0.2\textwidth}
\includegraphics*[width=0.35\textwidth]{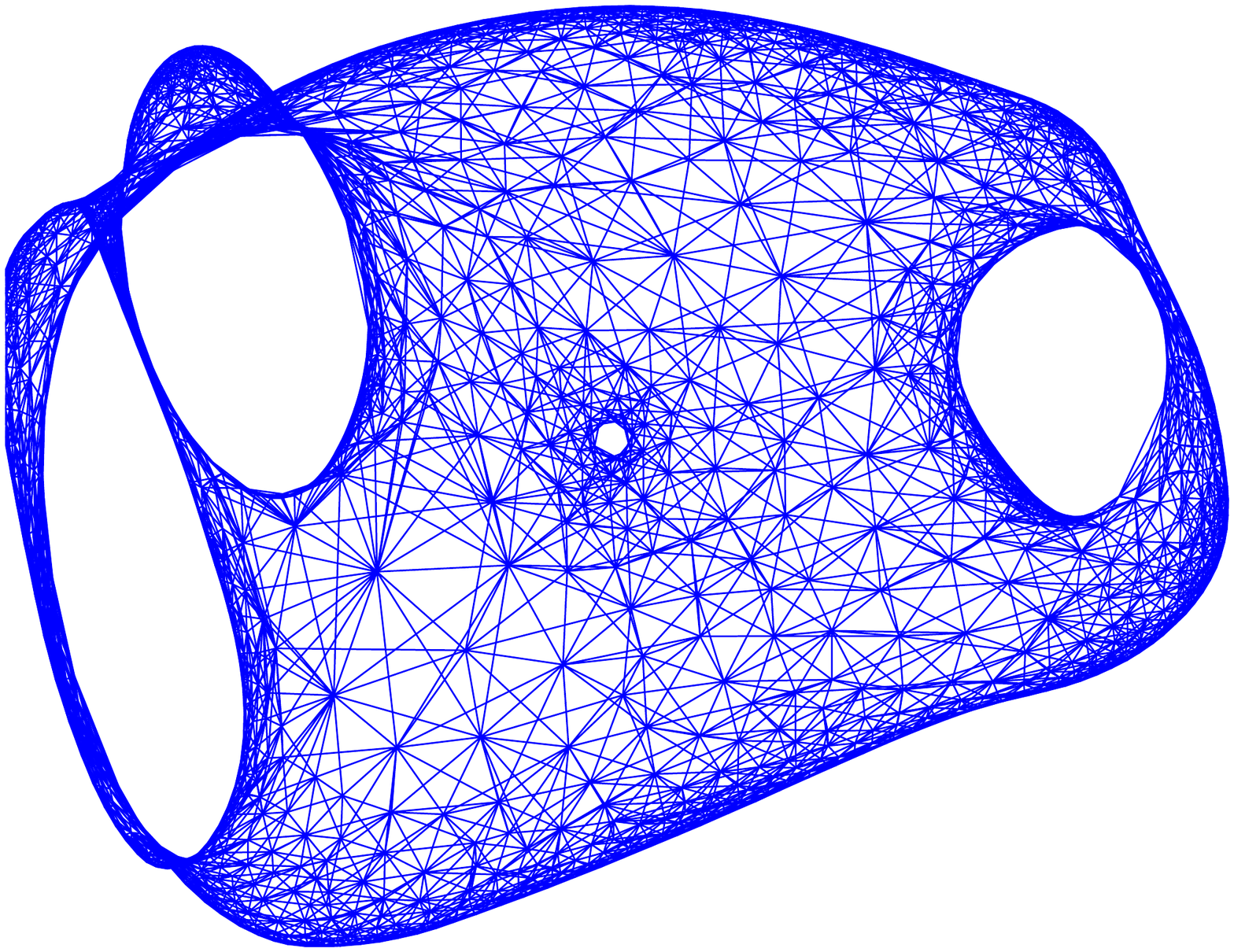}
\end{center}
\caption{Graph of the coarse grid matrix corresponding to the
unsmoothed aggregation (left) and the ``denser'' graph for the coarse
grid matrix obtained by smoothed aggregation (right).
\label{fig:unsmoothed-unsmoothed} }
\end{figure}

In view of Lemma \ref{lem:basis-in-vc}, the basis function
\eqref{AggBasis0} and the resulting UA-AMG corresponds to the use of
tentative coarse space $W_c$.  The smoothed basis functions by means
of \eqref{aggP1} and resulting SAMG corresponds to the use of the
$V_s$ in Lemma~\ref{lem:basis-in-vc} with $S=(I-R_sA)^\nu$.

\subsection{Bibliographical notes}
One of the first aggregation algorithms comes from applications in
economics and is due Vakhutinsky in
1979~\cite{Vakhutinsky.I;Dudkin.L;Ryvkin.A.1979a}.  Later, aggregation
methods have been developed and used for solving discrete elliptic
problems by Blaheta~\cite{Blaheta.R.1986a} and in the calculation of
stationary points of Markov chains by Marek~\cite{Marek.I.1991a}.  

A special class of aggregation coarsening method is based on matching
in graphs (also known as pairwise aggregation) and such methods were
employed in~\cite{Kim.H;Xu.J;Zikatanov.L.2003b} and for nonsymmetric
problems in~\cite{2004KimH_XuJ_ZikatanovL-aa}. Aggregations using
matching in graphs were further used in
\cite{Notay.Y.2010b,Brannick.J;Chen.Y;Zikatanov.L.2012a,Livne.O;Brandt.A.2012a,DAmbra.P;Vassilevski.P.2014a,DAmbra.P;Vassilevski.P.2013a}. To
improve the scalability of the unsmoothed aggregation a number of algorithms were developed:
in \cite{Kim.H;Xu.J;Zikatanov.L.2003b} variable $V$-cycle was used;
\cite{Notay.Y;Vassilevski.P.2008a} propose non-linear Krylov subspace
acceleration;
in~\cite{Olson.L;Schroder.J;Tuminaro.R.2011a} a procedure for
correcting the energy of the coarse-level Galerkin operator is designed.

The aggregation, especially the unsmoothed aggregation, can be used in
conjunction with nonlinear (variable-step/flexible) preconditioning
methods to result in an optimal algorithm. Such nonlinear methods are
called \emph{Algebraic Multilevel Iteration Methods} and were
introduced in~\cite{Axelsson.O;Vassilevski.P1991a}. Nonlinear
multilevel preconditioners were proposed and an additive version of them
was first analyzed in \cite{Axelsson.O;Vassilevski.P1994} (see also
\cite{Golub.G;Ye.Q1999a,Notay.Y2000,Saad.Y2003,Kraus.J2002}. In these
nonlinear multilevel preconditioners, $n$ steps of a preconditioned CG
iterative method provide a polynomial approximation of the inverse of
the coarse grid matrix. The same idea can be used to define the MG
cycles, as shown in~\cite{Vassilevski.P2008,Notay.Y;Vassilevski.P2008}. 
A comprehensive convergence analysis of nonlinear AMLI-cycle
multigrid method for symmetric positive definite problems has been
conducted in~\cite{Hu.X;Vassilevski.P;Xu.J2012a}.  Based on classical
assumptions for approximation and smoothing properties, the nonlinear
AMLI-cycle MG method is shown to be uniformly convergent.

The smoothed aggregation AMG method, first developed by M{\'i}ka and
Van\v{e}k \cite{Mika.S;Vanek.P.1992a,Mika.S;Vanek.P.1992b} and later
extended by Van\v{e}k, Mandel, and Brezina~\cite{Vanek.P;Mandel.J;Brezina.M.1996a}, was motivated by some early work on aggregation-based MG studied by R.~Blaheta in
\cite{Blaheta.R.1986a} and in his dissertation \cite{Blaheta.R.1988a}.

A major work on the theory and SA algorithm is found
in~\cite{Vanek.P;Mandel.J;Brezina.M.1998a}.  The
convergence result proved there is under the assumptions on the sparsity of
the coarse grid matrix and the ratio between the number of coarse and
fine degrees of freedom. For general sparse matrices, and even for
general adapted finite element matrices corresponding to elliptic
equations, verifying such assumptions is difficult and is yet to be done.

\section{Problems with discontinuous and anisotropic
  coefficients}
\label{s:AnisoJump}
A good AMG method should be robust with respect to possible
heterogeneity features such discontinuous jumps and anisotropy in a
given problem.  These heterogeneity should be detected and properly
resolved automatically in an AMG process.  Extensive numerical
experiments have shown that AMG such as classical AMG and SA-AMG are
very robust with respect to these heterogeneities.    One main
technique used to detect and to resolve these heterogeneities is
through the strength of connections (see \S\ref{s:strength}).  In this
section, we shall use the model problem \eqref{Model0} with two special 
set of coefficients, \eqref{coeff12} and \eqref{aniso-a}, to discuss how
classical AMG work for problem with strong heterogeneities. 

\subsection{Jump coefficients}
In this section we consider an the Classical AMG method when applied
to a problem with heterogenous (jump) coefficients, namely
\eqref{Model0} with \eqref{coeff12}.  We begin with a discussion on
how the strength of connection is used to define the sparsity pattern
of the prolongation.

The strength of connection measure was introduced to handle cases such
as jump coefficients and anisotropies in the matrices corresponding to
discretizations of scalar PDEs.  An important observation regarding
the classical AMG is that the prolongation matrix $P$, which defines
the basis in the coarse space, uses only strong connections. We need here the 
\emph{strength operator} $S: V\mapsto V$ defined in \eqref{defS}.
We now focus on the jump coefficient problem defined in \eqref{coeff12}
\[
\alpha(x) = \begin{cases}
\epsilon, \quad x\in \Omega_\epsilon\\
1 \quad x\in \Omega_1\\
\end{cases}
\]
with $\epsilon$ sufficiently small so that the graph
corresponding to the strength operator has at least two connected
components. Directly from the definition of $S_i$ in~\eqref{def:strong-connection-2} we
have the following:

\begin{itemize}
\item The graph corresponding to the strength of connection matrix $S$
  is obtained from the graph corresponding to $A$ by removing all
  entries $a_{ij}$ corresponding to an edge connecting a vertex from
  the interior of $\Omega_\epsilon$ to a vertex in $\overline{\Omega}_1$.

\item In another word, in this setting we have a
  \emph{block-lower-triangular} $S$ with at least two blocks, in which
  the first block corresponds to the vertices interior to $\Omega_\epsilon$
  and the second block corresponds to the rest of the vertices. 

\item The considerations above apply to any configuration of the
  subsets $\Omega_\epsilon$ and $\Omega_1$, in fact, they can even be
  disconnected, form cross points and so on.
\end{itemize}

In Figures~\ref{fig:jump1}--\ref{fig:jump3} we illustrate the strength
of connection graphs and their connected components for specific
coefficients $\alpha(x)$ which is defined as in~\eqref{coeff12} with
$\epsilon=10^{-3}$. The domain $\Omega_\epsilon$ is a union of
elements, and, $T\in \Omega_\epsilon$ if and only if
$(x_T-\frac12)(y_T-\frac12) < 0$, where $(x_T,y_T)$ is the barycenter
of $T$.  The domain $\Omega_1$ is
$\Omega_1 = \Omega\setminus\overline{\Omega}_1$. In
Figure~\ref{fig:jump2}, $\Omega_\epsilon$ and $\Omega_1$ are
disconnected in the graph of the strength matrix $S$. In the graph of
$S$, $\Omega_\epsilon$ is split into two connected components, we
denote them by $\Omega_\epsilon^{(1)}$ and
$\Omega_\epsilon^{(2)}$. $\Omega_1$ is split into two parts, denoted
by $\Omega_1^{(1)}$ and $\Omega_1^{(2)}$, which are connected through
only one point $x_0=(\frac{1}{2}, \frac{1}{2})$.  With a re-ordering
of indices, the strength matrix for this case can be written as
\begin{equation*}
S = \begin{pmatrix}
    S_{\epsilon,1} & & & \\
    & S_{\epsilon,2} & & &\\
    & & S_{11} & &\\
    & & & S_{12} & \\
    & & S_{x1} & S_{x2} &S_{x}\\
\end{pmatrix},
\end{equation*}
where $S_{\epsilon,1}$ and $S_{\epsilon,2}$ are two diagonal blocks
corresponds to the two connected component $\Omega_\epsilon^{(1)}$ and
$\Omega_\epsilon^{(2)}$; $S_{11}$ and $S_{12}$ are two diagonal blocks
corresponds to the two connected component $\Omega_1^{(1)}$ and
$\Omega_1^{(2)}$; $S_0$ is the diagonal entry corresponds to the grid
point $x_0$; $S_{x1}$ and $S_{x2}$ are low rank matrices with only a
few nonzero entries ($2$ nonzero entries in this example) which
contain the connections from $x_0$ to points in $\Omega_1^{(1)}$ and
$\Omega_1^{(2)}$ respectively.
\begin{figure}[!ht]
\centering
{\includegraphics*[width=0.4\textwidth]{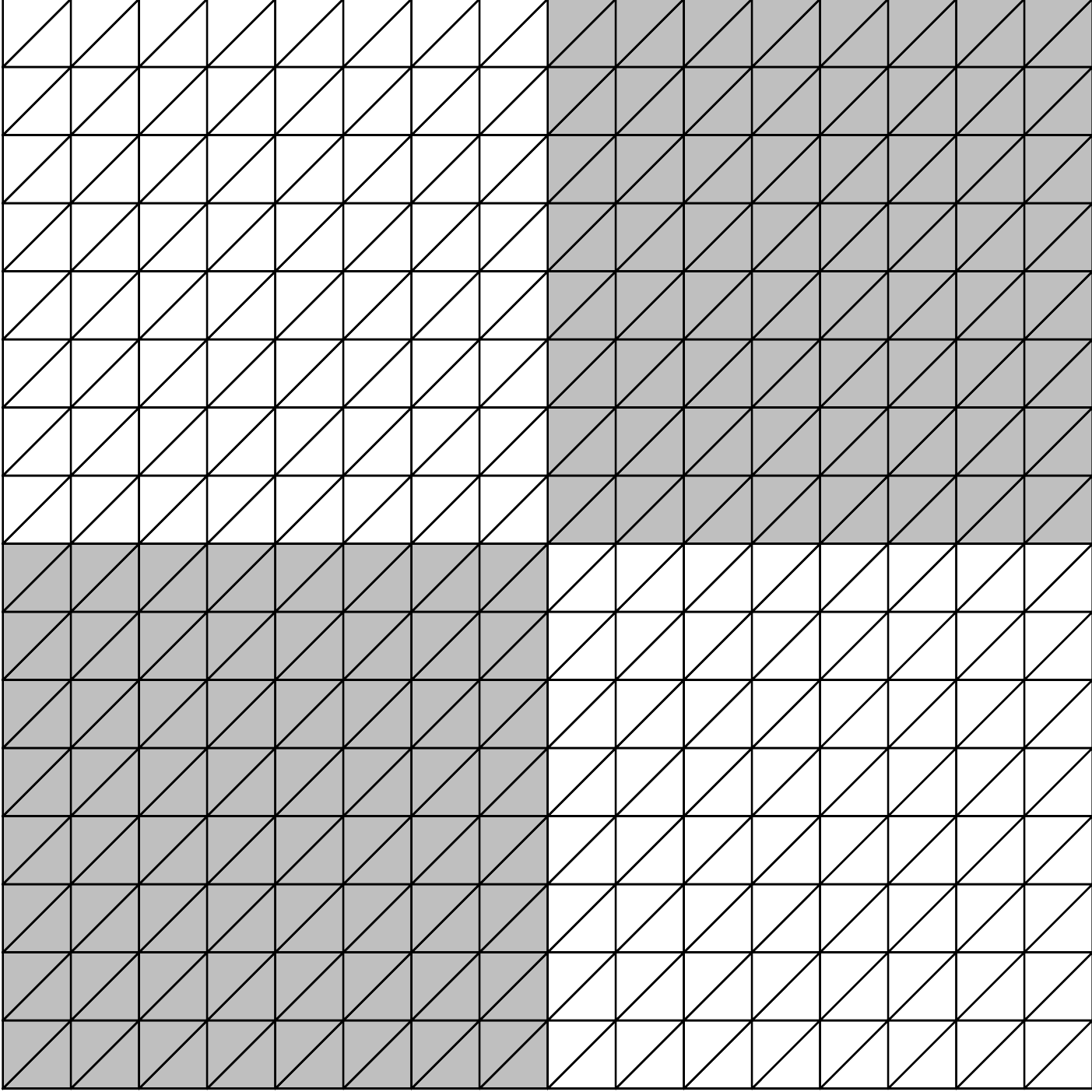}}\hspace*{0.1\textwidth}
{\includegraphics*[width=0.4\textwidth]{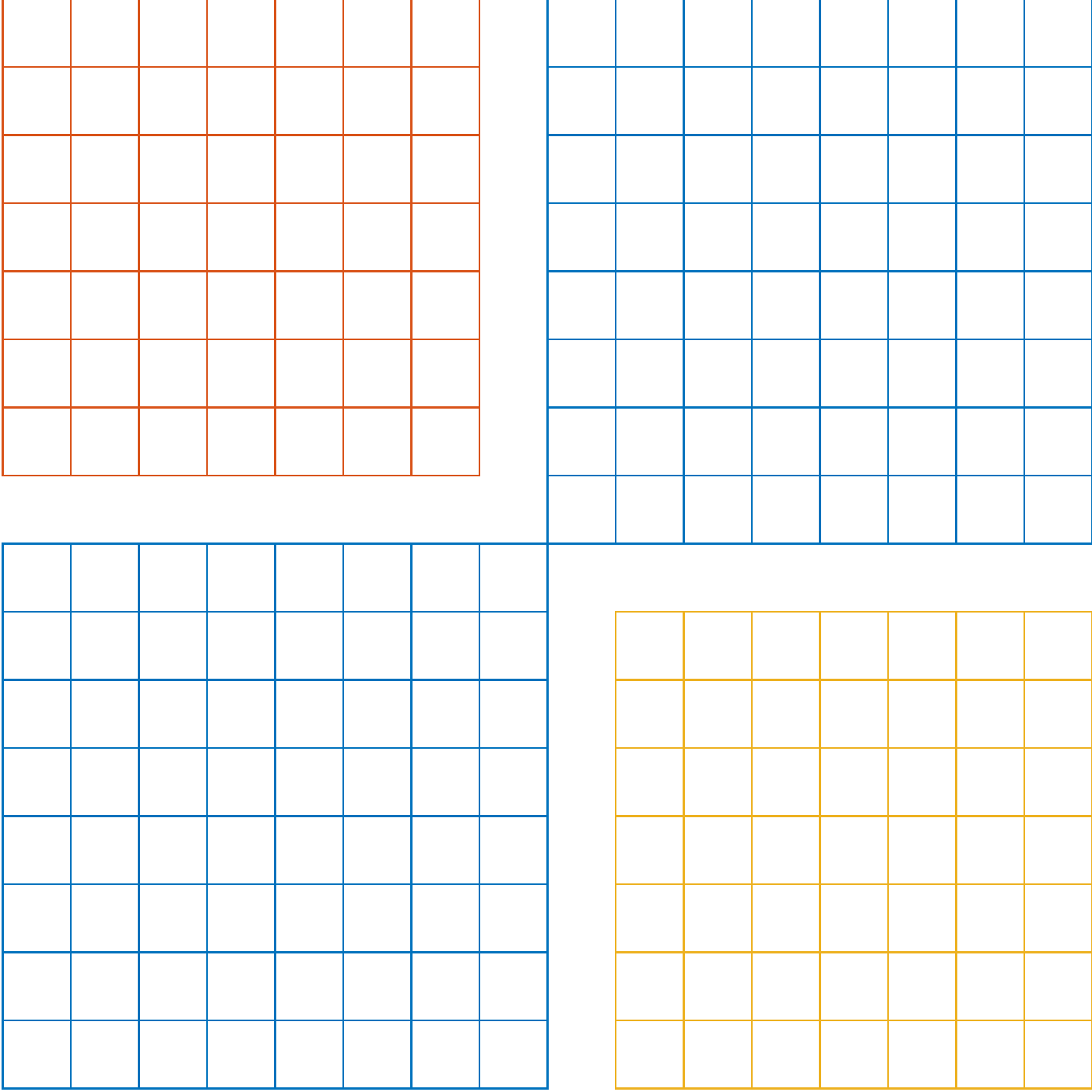}}
\caption{Jump coefficient problem on a uniform mesh. Left: the shaded
  region the coefficient is 1, blank region the coefficient is
  $10^{-3}$.  Right: the strongly connected components in the graph
  corresponding to the strength matrix $S$.\label{fig:jump1}}
\end{figure}

\begin{figure}[!ht]
\centering
{\includegraphics*[width=0.4\textwidth]{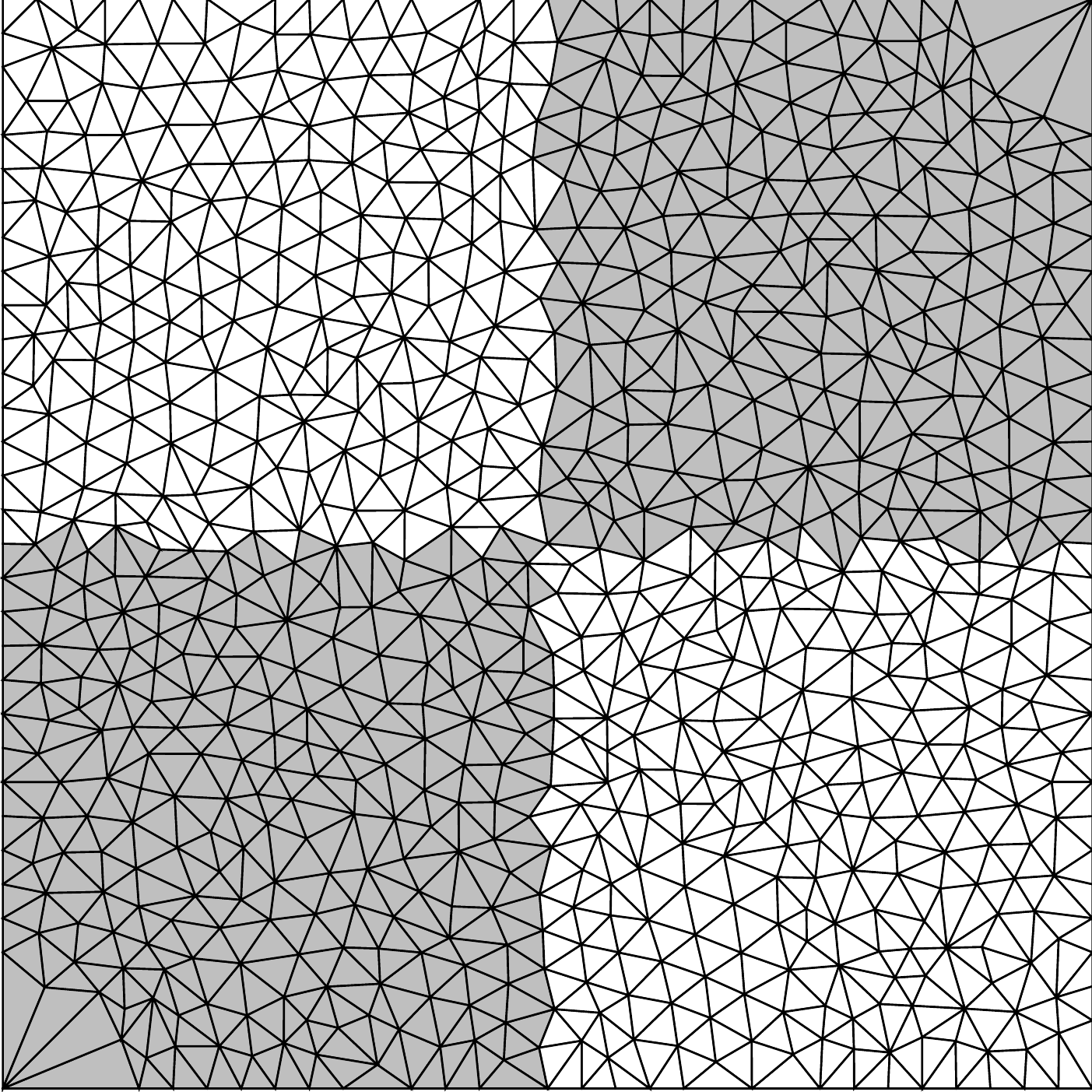}}\hspace*{0.1\textwidth}
{\includegraphics*[width=0.4\textwidth]{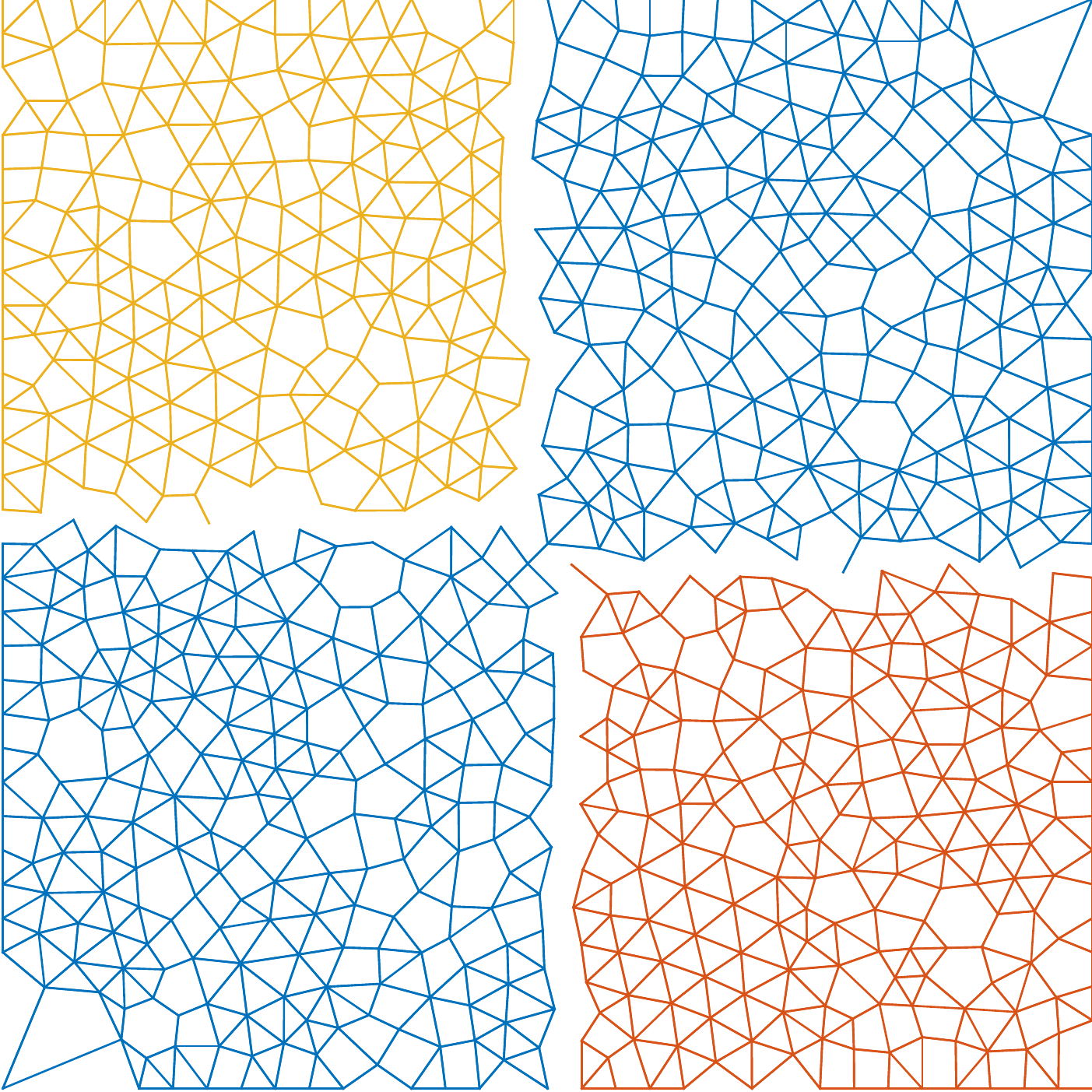}}
\caption{Unstructured mesh: jump coefficient problem on an unstructured
  mesh. Left: the shaded region the coefficient is 1, blank region the
  coefficient is $10^{-3}$.  Right: the strongly connected components in the
  graph corresponding to the strength matrix $S$. \label{fig:jump2}}
\end{figure}

\begin{figure}[!ht]
\centering
{\includegraphics*[width=0.4\textwidth]{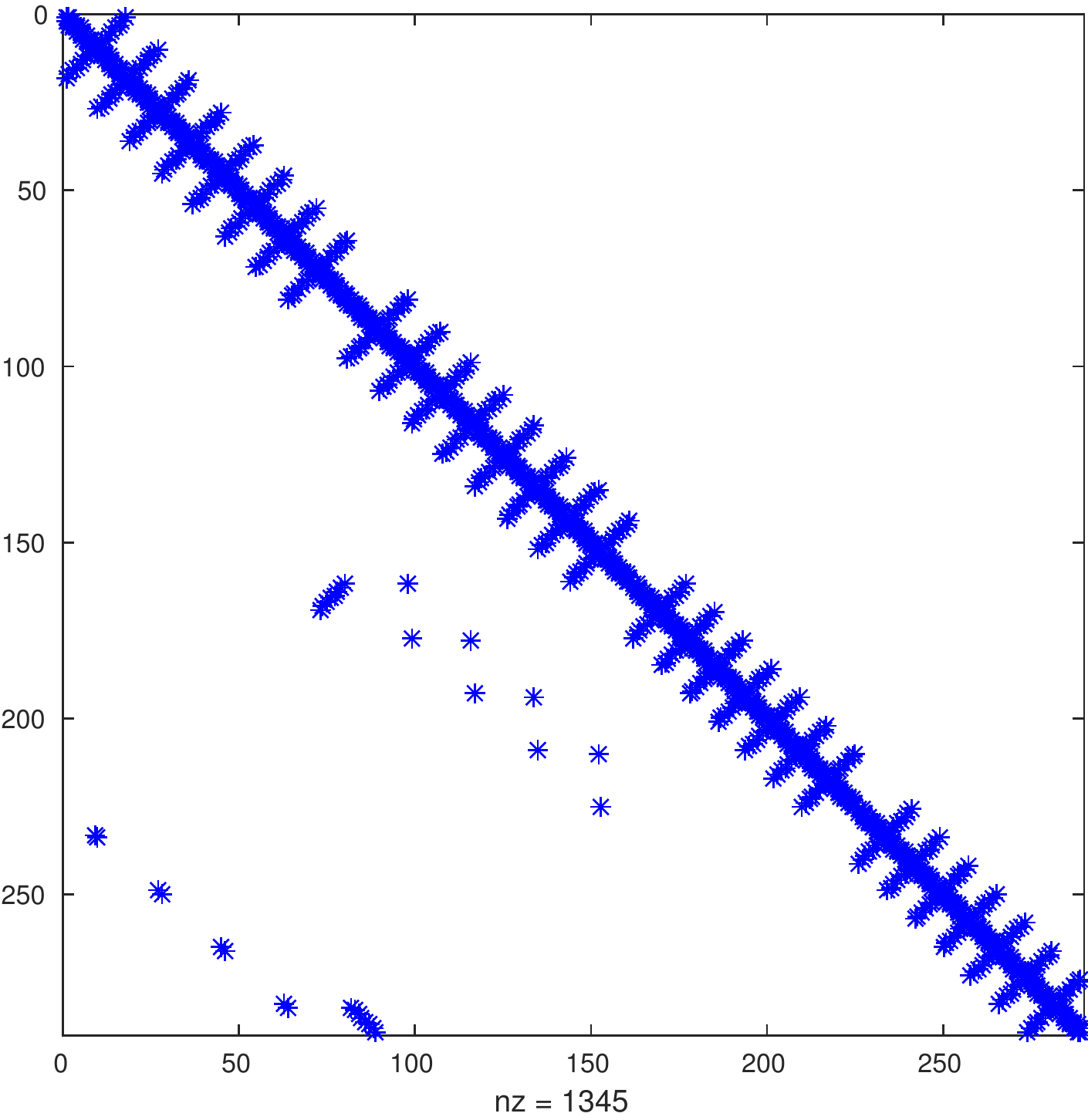}}\hspace*{0.1\textwidth}
{\includegraphics*[width=0.4\textwidth]{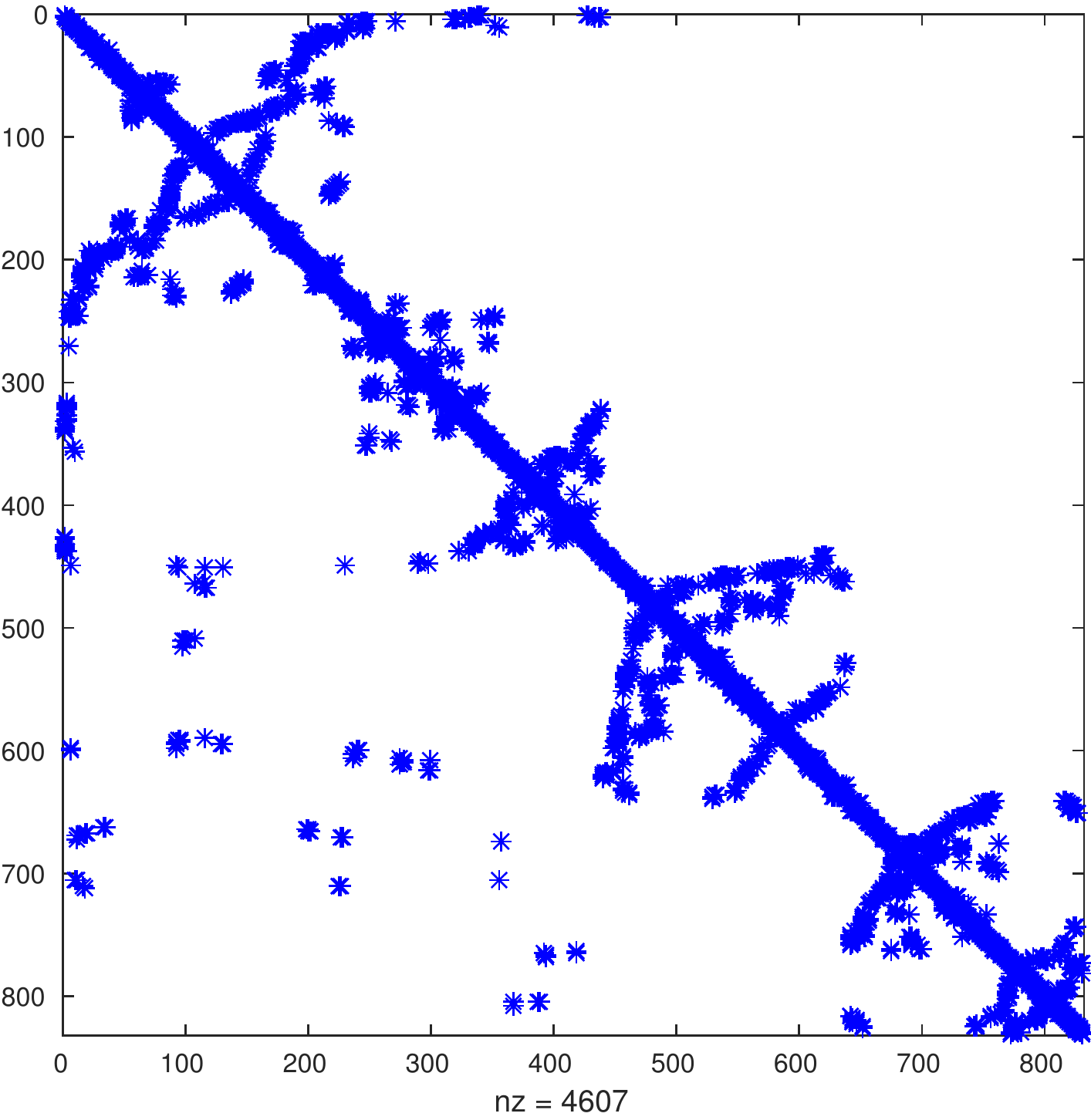}}
\caption{Jump coefficient problem: Nonzero pattern of strength matrix $S$ when reordered in
  block lower triangular form. Structured uniform grid (left) and unstructured mesh
  (right)\label{fig:jump3}}
\end{figure}
We define an operator $A_S: V\mapsto V'$ using $A$ and $S$
\begin{equation*}
    (A_S u, v): = \sum_{S_{ij}\neq 0} \omega_{ij}(u_i -u_j)(v_i-v_j),
\end{equation*}
where $\omega_{ij}=|a_{ij}|$. We have the following lemma.
\begin{lemma}
    For all $v\in V$, we have
    $$(A_Sv, v)\le (A_+v, v) \lesssim (Av, v).$$
\end{lemma}

For the classical AMG, the definition of the prolongation matrix $P$
only uses strong connections, and the coarse space is union of the
coarse spaces in each strongly connected component of the graph of
$A_+$. Note that the classical AMG construction gives the same coarse
space for both $A$ and its $M$-matrix relative $A_+$.  For this
reason, in the considerations that follow we write $A$ instead of
$A_+$ and $\mathcal{E}$ instead of $\mathcal{E}_+$, as the coarse
space is defined using only entries from $A_+$.

We now consider the convergence of classical two-level AMG with
standard interpolation for the jump coefficient problem and we prove
a uniform convergence result for the two level method. 
The same result for direct interpolation can be
obtained by a slight modification of the proof for standard
interpolation case. Before we go through the AMG two-level convergence proof, we first
introduce the following result on a connected graph, which can be viewed
as a discrete version of Poincar\'e inequality.
\begin{lemma}\label{lem:discrete_poincare}
    We consider the following graph Laplacian on a connected undirected graph $\mathcal{G}=(\mathcal{V}, \mathcal{E})$ 
    \begin{equation}\label{graph_lapla}
        \langle Au, v\rangle = \frac{1}{2}\sum_{(i,j)\in \mathcal{E}}(u_i-u_j)(v_i-v_j).
    \end{equation}
    For any $v\in V$, the following estimate is true
    \begin{equation}
        \|v-v_c\|_{\ell^2}^2\le \mu n^2d\langle Av, v\rangle,
    \end{equation}
    where $n=|\mathcal V|$ is the size of the graph, $v_c=\sum_{j=1}^{n}w_jv_j$ is a weighted average of $v$, $\mu=\sum_{j=1}^nw_j^2$, and $d$ is the diameter of the graph.
\end{lemma}

We now consider the Classical AMG coarsening as defined
in~\S\ref{sec:CAMG-coarse} and with an abuse of notation, we use
$\Omega_j$ to denote the set of vertices defined via the $C/F$
splitting in~\eqref{Omegaj}.  
Next Lemma is a spectral equivalence result, showing that the local
operators $A_j$, defined in~\eqref{Aj}, for shape regular mesh, are
spectrally equivalent to a scaling of the graph Laplacian operators
$A_{L,j}$ defined as
\begin{equation}\label{e:ALj}
(A_{L,j} u,v) = \frac12 \sum_{(i,k)\in \Omega_j} (u_i-u_k)(v_i-v_k). 
\end{equation}
\begin{lemma}\label{lem:shape_reg}
  With the assumption we made on the shape regularity of the finite
  element mesh, the following inequalities hold for $A_j$ defined as
  in \eqref{Aj} using the standard interpolation
            \begin{equation}
                c_Lh^{d-2}\langle A_{L,j} v_j, v_j\rangle \le ( A_jv_j, v_j) \le c^Lh^{d-2} \langle A_{L,j} v_j, v_j\rangle,
            \end{equation}
            where $A_{L,j}$ is a graph Laplacian defined in \eqref{graph_lapla} on the graph $\mathcal G_j$, $h$ is the mesh size and $c_L$, $c^L$ are constants depend on the shape regularity constant, and the threshold $\theta$ for the strength of connections. 
    \end{lemma}

\begin{proof}
  By the definition of the strength of connection, we have
    \[
        a_{ii} =\sum_{k\in N_k}-a_{ik} \le -\frac{|N_i|}{\theta}a_{ij},
    \]
    Since $A$ is symmetric, we also have
    \[
        a_{ii}\le -\frac{|N_i|}{\theta}a_{ji},
    \]
    By the definition of $\Omega_j$ in standard interpolation, for any $i\in \Omega_j\setminus\{j\}$, either $i\in F_j^s$ or there exists a $i\in F_j^s$ such $i\in F_{k}^s$. For the latter, $(j, k, i)$ forms a path between $j$ and $i$ going along strong connections. We have then 
    \[
        -a_{ik} \ge  -\frac{\theta}{|N_k|} a_{kk} \ge -\frac{\theta}{|N_k|} a_{kj}\ge -\frac{\theta^2}{|N_k||N_j|} a_{jj}.
    \]
    and 
    \[
        a_{jj}\ge -a_{jk} \ge \frac{|N_k|}{\theta}a_{kk} \ge -\left(\frac{|N_k|}{\theta}\right)^2a_{ik} 
    \]
    Combining the above two inequalities and using the assumption that the mesh is shape regular, for any $l\in \Omega_j$ that is connected with $i$ we have 
    \[
        \sigma_1 a_{ij} \le -a_{il} \le \sigma_2 a_{jj}
        \]
    with constants $\sigma_1$ and $\sigma_2$ which depend on the shape regularity constant and $\theta$.

    Since in the definition of $A_j$ in \eqref{Aj}, $\omega_e=-a_{ij}/2$ for $e=(i, j)$ , we obtain
\begin{equation}
    c_1 a_{jj}\langle A_{L,j} v_j, v_j\rangle \le ( A_jv_j, v_j) \le c_2 a_{jj}\langle A_{L,j} v_j, v_j\rangle.
\end{equation}
    Then by a scaling argument, $a_{jj}\eqqsim h^{d-2}$ and the proof is complete. 
\end{proof}

\begin{theorem}
  The two level method using a coarse space defined as $V_c$ defined via the classical AMG
  is uniformly convergent.
\end{theorem}
\begin{proof}
    By Theorem~\ref{thm:two-level-convergence}, we only need to show that $\mu_c$ is bounded, which can be easily obtained by combining Lemma~\ref{lem:discrete_poincare} -- \ref{lem:shape_reg} with Lemma~\ref{lemma-equiv}.
\end{proof}

\subsection{Anisotropic problem}\label{sec:aniso}

We consider the following problem
\begin{equation}\label{anisotropic}
    \begin{cases}
        -u_{xx}-\varepsilon u_{yy} = f, & \text{in } \Omega,\\
        \frac{\partial u}{\partial n}= 0, & \text{on } \partial \Omega.\\
    \end{cases}
\end{equation}

We discretize the problem using finite element method on an
$n\times n$ uniform triangular grid in $\Omega=(0, 1)\times (0,
1)$.
We order the vertices of the triangulation lexicographically and
denote them by $\{(ih,jh)\}_{i,j=0}^n$.  The stiffness then matrix is
\begin{equation}\label{aniso}
    A=\operatorname{diag}(-\varepsilon I, B, -\varepsilon I) \quad \text{ with } \quad B=\operatorname{diag}(-1, 2(1+\varepsilon), -1).
\end{equation}

It is immediate to see that, for
sufficiently small $0<\varepsilon\ll 1$ the strength operator $S$
has the form:
\begin{equation}\label{anisoS}
    S=\operatorname{diag}(0, S_B, 0) \quad \text{ with } \quad 
S_B=\operatorname{diag}(1, 1, 1).
\end{equation}
and the $M$-matrix relative of $A$ is
\begin{equation}
    A_+=\operatorname{diag}(0, B_+, 0) \quad \text{ with } \quad 
B_+=\operatorname{diag}(-1, 2, -1).
\end{equation}

If we use uniform coarsening to solve above linear problem, it is
proved in \cite{yu2013analysis} (see also \cite{2008ZikatanovL-aa})
that uniform convergence is not achieved when using point relaxation
as a smoother and standard coarsening. A fix for this is to use
strength of connection and coarsen the adjacency graph of the strength
operator $S$. The $C/F$ splitting constructed in this way using the
MIS algorithm from \S\ref{s:MIS} results in semi-coarsening
(coarsening only in one direction).

We now move on to show uniform convergence of the two-level classical
AMG for the anisotropic problem \eqref{anisotropic}. Recall that we consider a uniform grid in
$\mathbb{R}^2$ with lexicographical ordering of the vertices. The
stiffness then matrix is \eqref{aniso}.
We further assume, without loss of generality, that $n=2m+1$ for some
$m$. The coarse grid points using the strength operator defined
in \eqref{anisoS} then are with coordinates $((2i)h,jh)$,
where $i=1:m$ and $j=1:n$. 

Further, as each coarse grid function is uniquely
determined by its values at the coarse points, the function that
corresponds to the point $((2i)h,jh)$ for some $i$ and $j$ is defined
as the unique piece-wise linear function $\phi_{i, j,H}$ satisfying:
\begin{eqnarray*}
&&\phi((2i)h,jh) = 1, \quad 
\phi((2i-1)h,jh) = 1/2, \quad 
\phi((2i+1)h,jh) = 1/2,
\end{eqnarray*}
and $\phi_{i,j, H}(x) = 0$ at all other grid nodes.

Or we can use the basis functions for bilinear element, which can be
written as tensor product.  Let us define first the piece-wise linear
basis in 1D:
\[
\phi_{j,h}(t) = 
\begin{cases}
\frac{(t-(j-2)h)}{h},\quad t\in ((j-2)h,(j-1)h)\\
\frac{(jh-t)}{h},\quad t\in ((j-1)h,jh),\\
0,\quad\mbox{otherwise}
\end{cases}
\]
The basis in $V_h$ then is 
\begin{equation}\label{bilinear_fine}
  \phi_{ij,h}(x,y) = \phi_{i,h}(x)\phi_{j,h}(y)  
\end{equation}
and for the coarse grid basis we have 
\begin{equation}\label{bilinear_coarse}
\phi_{ij,H}(x,y) = \phi_{i,2h}(x)\phi_{j,h}(y).
\end{equation}

The basis functions for the linear elements is the piece-wise linear
interpolation of the bilinear basis.

The subset $\Omega_{ij}$ is the support of this basis function, i.e. the grid points where $\phi_{ij,H}$ is non-zero. More precisely,
\begin{equation}\label{anisotropic_omega}
    \Omega_{ij}=\{((2i-1)h, jh), (2ih, jh), ((2i+1)h, jh)\}.
\end{equation}
$\Omega_{ij}$ consists of the coarse grid point $(2ih, jh)$  and its neighbors on $x$ direction.  Then we define $V_j: =\mathbb{R}^3$.

The operator $\Pi_j: V_j\mapsto V$ is defined by the matrix representation of $I_h(\phi_{ij,H}\cdot)$. $A_{ij}$ is defined as in \eqref{Aj}.
In this case 
\begin{equation}\label{anisotropic_localA}
    A_{ij}= \begin{pmatrix}
        1 & -1 & 0\\
        -1 & 2 & -1\\
        0 & -1 & 1\\
        \end{pmatrix}
\end{equation}
$A_{ij}$ is symmetric positive semi-definite and since  for any $v_{ij}\in V_{ij}$ , we have
\begin{eqnarray*}
    \sum_{i,j}(v_{ij}, v_{ij})_{A_{ij}}&=&\sum_{ij}\sum_{\substack{(k,l)\in \mathcal E\\ k, l\in \Omega_{ij}}}-a_{kl}(v_k-v_l)^2\\
    &\le & 2\sum_{(k, l)\in \mathcal E}-a_{kl}(v_k-v_l)^2\\
    & = & 2(v, v)_A,
\end{eqnarray*}
$A_{ij}$ satisfies \eqref{sum_Aj}.

We define $D_{ij}$ as in \eqref{local_diag}. As
$A_{ij}\bm 1 = \bm 0$, the minimum eigenvalue of $D_{ij}^{-1}A_{ij}$
is $0$ and the corresponding eigenvector is the constant vector, we choose
local coarse space $V_{ij}^c$ to be the space of all constant vectors in
$V_{ij}$. The corresponding global coarse space is then defined as in \eqref{V_c}.

\begin{theorem}
  The two level method with coarse space defined above converges
  uniformly for the anisotropic problem in \eqref{anisotropic}, with
  convergence rate independent of $\varepsilon$ and the mesh size $h$.
\end{theorem}
\begin{proof}
    By Theorem~\ref{thm:two-level-convergence}, the convergence rate depends on the second smallest eigenvalue of $D_{ij}^{-1}A_{ij}$ which is $1$ for all $i$ and $j$.

    Next, Theorem~\ref{thm:two-level-convergence} can be applied to
    this case, and we obtain
    \begin{equation}
        \|E\|_{A} \le 1- \frac{1}{C},
    \end{equation}
    with $C$ independent of $\varepsilon$ and the mesh size, which
    proves the uniform convergence of the AMG method.  
\end{proof}

\subsection{Bibliographical notes}
Fast solvers for problems with heterogeneous and or anisotropic
coefficients have been in the focus of research for the last 3-4
decades. The AMG methods are among the preferred solvers due their
robust behavior with respect to the coefficient variation and
independence of the geometry. Standard multilevel methods for these
problems do have limitations as their convergence may deteriorate as
shown in~\cite{1981AlcouffeR_BrandtA_DendyJ_PainterJ-aa}. The cause
for this in geometric multigrid with standard coarse spaces is
discussed in~\cite{1991BrambleJ_XuJ-aa,1991XuJ-aa} and later in~\cite{1999OswaldP-aa}.  
Attempts to remove the dependence on the size of the coefficient jumps was made by
introducing the matrix dependent prolongations. We refer to
\cite{1982DendyJ-ab,1983DendyJ-aa}, Reusken~\cite{1993ReuskenA-aa},
\cite{1990ZeeuwP-aa}. Several theoretical and numerical results on
geometric and algebraic methods for discontinuous coefficients are
found in the survey paper~\cite{2000ChanT_WanW-aa} and the references
therein. Anisotropic equations and AMG coarsening is considered
in~\cite{Brannick.J;Brezina.M;MacLachlan.S;Manteuffel.T;McCormick.S;Ruge.J.2006a}.

Finally, in addition to the theory presented here, some partial
theoretical results on convergence of an AMG are found in the
classical papers on AMG~\cite{Ruge.J;Stuben.K.1987a}, smoothed
aggregation (SA)~\cite{1996VanekP_MandelJ_BrezinaM-aa}.  Related works
are the frequency filtering and decompositions found in
\cite{1989HackbuschW-aa},
\cite{1992WittumG-aa,1997WagnerC_WittumG-aa,1997WeilerW_WittumG-aa,2008NagelA_FalgoutR_WittumG-aa}.

\section{Bootstrap and adaptive AMG}\label{s:adaptiveAMG}
In all the algorithms studied earlier, we assume that the near-null
spaces are known in advance. Eliminating such an assumption
and extending the range of applicability of optimal multigrid
techniques is attempted in the framework of the Bootstrap/Adaptive AMG
methods which we describe in this section.  In summary, the adaptive
AMG ($\alpha$AMG, $\alpha$SA) and bootstrap AMG (BAMG) algorithms are
aimed at the solution of harder problems for which the standard
variants of the Classical AMG method or the Smoothed Aggregation
method may converge slow.
The bootstrap/adaptive algorithms make special choices coarse spaces. Based on a given a smoother. self- improve until achieving the desired convergence rate. 

\subsection{Sparsity of  prolongation matrices}\label{s:sparse-summary}

We give here a very short summary on the sparsity patterns of the
prolongation matrices that we have defined for energy minimization AMG
in~\S\ref{sec:energy-min}, Classical AMG in~\S\ref{s:classical-amg},
and Aggregation AMG in~\S\ref{s:agmg}. First, for the energy
minimizing AMG we have that the sparsity pattern of the prolongation
could be prescribed in advance, and in some sense this approach is
more general as it can also use the sparsity patterns given below for
the Classical AMG and Aggregation AMG approaches. For specific
restrictions related to the number of vectors interpolated exactly by
the energy minimizing prolongation we refer to
\cite{Xu.J;Zikatanov.L.2004a} and
\cite{2006VassilevskiP_ZikatanovL-aa}.
There are also ways to define the prolongation with changing pattern. Such algorithms are useful because they provide a mechanism to control the sparsity pattern of the prolongation. 
The Algorithm~\ref{a:P-LS} is first described in~\cite{original_BS} and later included in the bootstrap adaptive AMG algorithm designed in~\cite{BS_2011}.
In the algorithm description, we refer to the graph of the $M$-matrix
relative as defined in~\S\ref{sc:connection}.

\begin{algorithm}[h]\caption{Prolongation via least squares minimization\label{a:P-LS}}
\begin{algorithmic}[1]
  \State {\bf Input:} Matrices $\Psi\in \mathbb{R}^{n\times m}$, 
$\Psi_c\in \mathbb{R}^{n_c\times n_c}$, a norm $\|\cdot\|$, 
  initial sparsity pattern $\sparse(P)$ 
for $P$; bound on maximum nonzeroes per row $s_{\max{}}$;  
\State {\bf Output:} Prolongation matrix $P$ such that $P\Psi_c \approx \Psi$.  

\State Set $i=1$.

\While{($i \leq n$)}

\State  Find 
$\widetilde p_i=\argmin_{\widetilde p_i\in \mathbb R^{n_c}_{S_i}}\|\Psi_c^T\widetilde p_i-\widetilde \Psi_i\|$.

\If{($\|\Psi_c^T \widetilde p_i-\widetilde \Psi_i\| > \varepsilon$) {\bf and}
  ($|S(i)|\leq s_{max}$)} 

\State $S(i) \leftarrow S(i) \cup_{j_c} \{j_c\}$, where $j_c$ are close to
$i$ in graph (algebraic) distance.

\Else 

\State $i\leftarrow i+1$

\EndIf

\EndWhile
\For{$i=1:n$, $j_c\in S(i)$} 
\IIf{($p_{i,j_c}\le \varepsilon_p$)} Set $p_{i,j_c}=0$ \EndIIf
\EndFor
\State \Return $P$
\end{algorithmic}
\end{algorithm}

\begin{algorithm}[H]\caption{Prolongation via aggregation\label{a:P-UASA}}
\begin{algorithmic}[1]
\State {\bf Input:} 
Matrix $\Psi\in \mathbb{R}^{n\times m}$, an aggregation 
\(
\cup_{i=1}^{n_c} \mathcal A_i = \{1,\ldots,n\}, 
\)
and a prolongator smoother $S:\mathbb{R}^n\mapsto \mathbb{R}^n$ 
(in case of Smoothed Aggregation (SA)).   

\State {\bf Output:} Prolongation matrix $P$ such that $\range(P\Psi_c) 
\approx \range(\Psi)$.  
\State Set 
$(\Psi_{\mathcal{A}_i})_{kj} = 
\begin{cases} 
\Psi_{kj}, \quad k\in \mathcal{A}_i, j=1:m\\
0, \quad k\notin \mathcal{A}_i, j=1:m,
\end{cases}$.
\State Set $P = (\Psi_{\mathcal A_1 },\ldots,\Psi_{\mathcal A_{n_c} })$
\IIf{(SA)} $P\leftarrow (I-\rs A) P$  \EndIIf
\State \Return $P$
\end{algorithmic}
\end{algorithm}

\subsection{Notation} 
Given a matrix $A\in \mathbb R^{n\times n}$ and a relaxation
(smoother) $R$ for this matrix, adaptive procedure to construct a
sequence of coarse spaces which are characterized by the sequence of
prolongation matrices
$$
P_j^{m}: \mathbb R^{n_j^m}\mapsto \mathbb R^{n_{j+1}^m}.
$$
The corresponding V-cycle matrix is denoted $B^m$. We introduce the following notation, used throughout this section 

\begin{itemize}
 \item For a matrix $Y\in \mathbb{R}^{n\times m}$ we set
\[
Y= (y_1, \ldots, y_m)=
\begin{pmatrix}
\widetilde y_1^T\\
\vdots\\
\widetilde y_n^T
\end{pmatrix}
\]
Clearly, here $\{y_i\}_{i=1}^m$ are the columns of $Y$ and $\{\widetilde y_j^T\}_{j=1}^n$ are the rows of $Y$.

\item $\mathcal{V}=(V_1,\ldots V_J)$ is a multilevel hierarchy of
  spaces $V_{j}\subset V_{j+1}$, $j=1:(J-1)$ and the fine grid spaces
  is denoted by $V_J$.

\item By $\{P_{j-1}^{j}\}_{j=2}^{J}$ we denote the prolongation matrix
  from a coarser level $(j-1)$ to finer level $j$, and
  $P_j=P_{J-1}^{J}P_{J-2}^{J-1}\ldots P_{j}^{j+1}$, $j=1,\ldots (J-1)$ is the
  prolongation from level from level $j$ to level $J$  

\item The set of all prolongations up to level $j$ we denote by
  $\mathcal{P}_j=\{P_{j-1}^j\}_{j=2}^{j}$, and 
  $\mathcal{P}=\mathcal{P}_J$. As we have mentioned on several
  occasions, the set of prolongation matrices $\mathcal{P}$ completely
  determines the multilevel hierarchy of spaces $\mathcal{V}$.

\item Set of test vectors on every level
\[  
\mathcal C=\{\Psi_1, \cdots, \Psi_J\}, \quad \Psi_j\in \mathbb{R}^{n_j\times m}.
\]
In an adaptive method, the hierarchy of coarse spaces is constructed
so that $\Psi_j \approx P_{j-1}^j\Psi_{j-1}$, or, 
$P_{j-1}^j$ is constructed so that $\Psi_j$ can be well approximated by elements from 
$\range{P_{j-1}^j}$. 

\item We need the standard $V$-cycle preconditioner $B(\mathcal{P}_j)$
  with hierarchy of spaces given by $\mathcal{P}_j$ starting at level
  $j$.

\item $A_j$ is the restriction of the fine grid matrix $A$ on level
  $j$ and $M_j$ is the restriction of the ``mass'' matrix, defined as:
\(A_j =P_j^T A P_j\), \(M_j = P_j^TP_j\).
\end{itemize}

A generic adaptive AMG algorithm changes the set $\mathcal{P}$,
adjusting the hierarchy of spaces $\mathcal{V}$. In general, in
adaptive procedure the number of levels is not known and in some of
the algorithms below we use $V_1$ as the finest space by mapping the
indexes in the notations above as
\[
j\leftarrow J-j+1, \quad j=1,\ldots J. 
\]

\subsection{A basic adaptive algorithm}
We describe now a generic Adaptive algorithm for constructing coarse
spaces. We use an approach slightly different than what is in the
literature and fit into one framework both BAMG and $\alpha$SA.

\begin{description}
\item[Step 0] Given $A\in \mathbb R^{n\times n}$ and the associated graph
  $\mathcal G(A)=(\mathcal V, \mathcal E)$.
\item[Step 1] Initialization
\begin{enumerate}[1.]
\item Given $m_0\ge 1$, $q\ge 1$, $1\le n_0< n$ and $\delta_0\in (0, 1)$.
\item $\mathcal P\leftarrow \emptyset$.
\item $V_c\leftarrow V=\mathbb{R}^n$; $\mathcal V \leftarrow \{V_c \}$.
\item $n_c\leftarrow n$; $m\leftarrow m_0$.
\item $B\leftarrow R$.
\item Randomly pick $m$ test-vectors 
$\bm{\Psi}_0=(\psi_1, \ldots \psi_m)$, $\bm{\Psi}_0\in \mathbb{R}^{n\times m}$. 
\item $\bm{\Psi}\leftarrow \bm{\Psi}_0$; $\mathcal C\leftarrow \{\bm{\Psi}_0\}$.
\end{enumerate}

\item[Step 2]  If $n_c\le n_0$, go to {\bf Step 3}, else do:
  \begin{enumerate}
      \item  Make a copy of $\bm{\Psi}$:  $ \hat{\bm{\Psi}}\leftarrow\bm{\Psi}. $
Then compute
\begin{equation}\label{rate}
          \bm{\Psi}\leftarrow (I-BA)^q \bm{\Psi}, \quad
          \delta =\max_{1\le i\le m}\frac{\|\Psi_i\|_A}{\|\hat \Psi_i\|_A}.
\end{equation}

\item 
    If $\delta \le \delta_0$ then {\bf Stop}. 
\item Use a coarsening strategy (c.f. \S\ref{sc:connection}) to update $n_c$ and find a
    set of coarse grid DOFs $\{\dof_i(\cdot)\}_{i=1}^{n_c}$. Then set
          $$V_c\leftarrow \mathbb{R}^{n_c}, \quad \mathcal V \leftarrow \mathcal V\cup \{V_c\}. $$
\item Form the ``restriction'' of $\bm\Psi$ to the coarse space:
  \begin{equation*}
        \bm{\Psi}_c\leftarrow \operatorname{Restrict}(\bm{\Psi}, \{\dof_i(\cdot)\}_{i=1}^{n_c})
    \end{equation*}
Then set  $\mathcal C\leftarrow \mathcal C\cup \{\bm{\Psi}_c\}$.
\item Identify a sparsity pattern $\sparse(P)$ (c.f. \S\ref{s:sparse-summary} 
also \S\ref{sec:classicP},  and \S\ref{s:P-many}).
\item Find a prolongation matrix $P$ using $\bm{\Psi}_c$ and $\bm{\Psi}$ by applying Algorithm~\ref{a:P-LS} or Algorithm~\ref{a:P-UASA}. Then set 
    $$\mathcal P\leftarrow \mathcal P\cup \{P\}, \quad A\leftarrow P^TAP, \quad \bm{\Psi}\leftarrow \bm{\Psi}_c, B\leftarrow R_c, $$
  where $R_c$ is the relaxation on the coarse grid $V_c$.
\end{enumerate}
\item[Step 3]  
    Order the spaces in $\mathcal V$ increasing with respect to their dimension as
        \begin{equation*}
            \mathcal V=\{V_1, V_2, \dots,V_J\} \text{ with }  V_J=\mathbb{R}^n,
        \end{equation*}
        and the corresponding prolongation matrices 
        \begin{equation*}
            \mathcal P=\{P_j^{j+1}\}_{j=1}^{J-1}, \quad P_j^{j+1} : V_j\mapsto V_{j+1}, \quad j=1, 2, \dots, J-1.
        \end{equation*}
    Set $B$ to be the $V$-cycle AMG method based on above coarse spaces and prolongation matrices.
  \item[Step 4] Compute $\delta$ in \eqref{rate} using current $B$. If
    $\delta \le \delta_0$, {\bf Stop}. Else, update
    $\mathcal C$, $\mathcal P$ and $\mathcal V$ by modifying, removing
    from, or adding to the set of test-vectors (Bootstrap AMG uses
        Algorithm~\ref{a:MGE}, adaptive SA uses Algorithm~\ref{a:SA}). Then
    go to {\bf Step 3}.
\end{description}

The restrict operator in Step 2.4 is following:
\begin{eqnarray*}
\bm\Psi_c=\bigg(\dof_i(\psi_j)\bigg)_{1\le i\le n_c, 1\le j\le m}, \quad \mbox{bootstrap AMG}
\\
\bm\Psi_c=\bigg(\dof_i(\psi_j)\bigg)_{1\le i\le n_c}\otimes e_j\quad j=1:m, \quad\mbox{adaptive aggregation}
\end{eqnarray*}

\begin{algorithm}[h]\caption{MGE algorithm: approximating of $l_e$ eigenpairs of $A$\label{a:MGE}}
  \begin{algorithmic}
    \For{$j=J,\dots,1$}
      \If{$j=J$ (coarsest grid)}
         \State Find $\{(\varphi_J^{(k)},\lambda_J^{(k)})\}_{k=1}^{l_e}$-- the solutions to the eigenvalue problems:
            \State      $A_J\varphi_J^{(k)}=\lambda_J^{(k)}M_J\varphi_J^{(k)}$, $k=1,\dots,l_e$; 
        \Else
          \For{$k = 1,\ldots,l_e$}
            \State Set $v_{j}^{(k)} = (P_{j-1}^{j})^T\varphi_{j-1}^{(k)}$, 
             $\mu_{j}^{(k)}=\lambda_{j-1}^{(k)}$, and $C_{j}^{(k)} = A_{j}-\mu_{j}^{(k)}M_{j}$
           \State Relax on $C_{j}^{(k)} w = 0$, i.e., $\varphi_{j}^{(k)} = (I - S_{j}^{(k)}C_{j}^{(k)})\; v_{j}^{(k)}$.
           \State Set $\lambda_j^{(k)} = \dfrac{\langle A_{j}\varphi_{j}^{(k)},\varphi_j^{(k)}\rangle}{\langle M_j\varphi_j^{(k)},\varphi_j^{(k)}\rangle}$
        \EndFor
      \EndIf
    \EndFor
  \end{algorithmic}
\end{algorithm}

\begin{algorithm}\caption{$\alpha$SA: adding a test vector $\psi$ 
to the current set of test vectors $\Psi_J$\label{a:SA}}
\begin{description}
\item[Step 1] Update $\bm{\Psi}_J$ by adding $\psi$ as a new column:
\(\bm{\Psi}_J\leftarrow [\bm{\Psi}_J, \psi]\). 
\item[Step 2]  For $l=J, \dots, 3$:
  \begin{enumerate}[(1)]
  \item Update $\Psi_{l-1}$ by evaluating
    the coarse grid degrees of freedom $\dof_j(\psi)$, 
    $j=1:n_{l-1}$. This process adds $n_{l-1}$ rows to $\Psi_{l-1}$. 
  \item Use Algorithm~\ref{a:P-UASA} (SA) to define a prolongation
    $P_{l-1}^l$ matrix and coarse grid operator $A_{l-1}$.
  \item Update $\mathcal{P}_{l} \leftarrow \{P_{l-1}^l\}\cup \mathcal{P}_{l-1}$. 
  \item As the number of rows in $\Psi_{l-1}$ was increased in
    Step~2(1), we need to change $P_{l-2}^{l-1}$ in
    $\mathcal{P}_{l-1}$ in order to keep the set $\mathcal{P}$
    consistent.  We define $\widetilde{P}_{l-2}^{l-1}$ as
    the ``bridge'' prolongation via algorithm~\ref{a:bridge} and we
    set
\[
\widetilde{\mathcal{P}}_{l-1} = 
\leftarrow \{\widetilde{P}_{l-2}^{l-1}\}\cup \mathcal{P}_{l-2} 
\]
  \item Test the convergence of $B(\widetilde{\mathcal{P}}_{l-1})$ on the newly added test vector $\psi_{l-1}$: 
\begin{equation*}
\hat \psi_{l-1}\leftarrow \psi_{l-1},\quad
\psi_{l-1}\leftarrow (I - B_{l-1}(\widetilde{\mathcal{P}}_{l-1})A)^q(\psi_{l-1}).
\end{equation*}
If
$\left(\frac{\|P_{l-1}\psi_{l-1}\|^2_A}{\|P_{l-1}\widehat\psi_{l-1}\|^2_A}\right)^{1/q}\le
\delta$ then {\bf Stop}.
 \item Update the coarse representation of $\bm{\Psi}_{l-1}$:
    \begin{equation*}
      \bm{\Psi}_{l-1}\leftarrow [\hat{\bm{\Psi}}^{l-1}, \psi_{l-1}].
    \end{equation*}
  \end{enumerate}
\end{description}
\end{algorithm}

\begin{algorithm}\caption{$\alpha$SA setup: construction of a
    ``bridging'' prolongator \label{a:bridge}}
    \begin{enumerate}
    \item Denote the last column of $\Psi_{l-1}$ by $\psi_{l-1}$,
      and let $\hat{\Psi}_{l-1}$ consist of all other columns of
      $\Psi_{l-1}$.
        \item Create a prolongation $P_{l-2}^{l-1}$ using Algorithm~\ref{a:P-UASA} with the 
$\Psi_{l-2}$ by fitting all the vectors in $\hat{\Psi}_{l-1}$.
    \end{enumerate}
\end{algorithm}

\subsection{Bibliographical notes}\label{s:biblio-adaptive}
The AMG approaches we have considered are aimed at adaptive choice of
coarse spaces and multilevel hierarchies in an AMG algorithm. The
majority of known to date adaptive AMG methods use the operator $A$
and aim to capture the worst case errors. All particular details
regarding \emph{bootstrap AMG} (BAMG) are found in~\cite{original_BS},
\cite{BS_2011}, and \cite{BS_2015}. The Adaptive Classical AMG is
described in \emph{adaptive AMG or $\alpha$AMG}~\cite{aAMG}, and the
\emph{adaptive Smoothed Aggregation or $\alpha$SA} is discussed in
detail in~\cite{aSA}, \cite{aSAM}.  While the adaptive and bootstrap
MG processes have been successful in several settings, they are still
only a heuristic attempt to overcome serious barriers in achieving
good performance from the classical AMG point of view.  Indeed, the
costs of achieving added robustness using a bootstrap or adaptive MG
algorithm are significant.

Central to the adaptive MG framework is an important distinction
between the role of the underlying multigrid algorithm (aggregation or
classical AMG) and what the additional, adaptive elements it should
provide.  If the idea of self-improving the coarse spaces is poorly
implemented, the adaptive and bootstrap multigrid algorithms can
easily degenerate into an algorithm with no better convergence
properties than a classical Krylov method preconditioned by the MG
smoother \cite{Falgout.R.2004a}.

The basic ideas on adaptive AMG are outlined in the early works on
classical AMG~\cite{1stAMG}. In fact, the adaptive process of
constructing interpolation based on fitting a set of test vectors for
badly scaled matrices was introduced as early as
in~\cite{ruge1983algebraic}. Some of the basic ideas of adaptive AMG
are also found
in~\cite{ruge5110final,ruge1986amg,mccormick1989algebraic,Brandt.A;McCormick.S;Ruge.J.1985a}. Further
advancement in the adaptive AMG methods, and new ideas for using
eigenvector approximations to guide the adaptive process were
introduced in the bootstrap AMG (BAMG) framework
in~\cite{Brandt.A.2001a}. The BAMG is a self-learning multigrid
algorithm that automatically determines the algebraically smooth
errors in a given problem and was further developed in~\cite{BS_2011},
and \cite{BS_2015}.  Adaptive AMG algorithms were further developed
and new were introduced in some more recent works: ($\alpha$AMG) in
the framework of Classical AMG
\cite{Brezina.M;Falgout.R;MacLachlan.S;Manteuffel.T;McCormick.S;Ruge.J.2006b};
and ($\alpha$SA) in the framework of Smoothed Aggregation
AMG~\cite{Brezina.M;Falgout.R;MacLachlan.S;Manteuffel.T;McCormick.S;Ruge.J.2004a,Brezina.M;Falgout.R;Maclachlan.S;Manteuffel.T;Mccormick.S;Ruge.J.2005a}. Other
adaptive multilevel algorithms are the adaptive filtering and the
filtering decompositions~\cite{1992WittumG-aa,1997WagnerC_WittumG-aa}
and the multilevel multigraph algorithms~\cite{2002BankR_SmithR-aa}.

These references have more specific details on the implementation of
BAMG/$\alpha$AMG/$\alpha$SA.  One improvement on the original BAMG method
is given in~\cite{Manteuffel.T;McCormick.S;Park.M;Ruge.J.2010b}, where
an \textit{indirect} BAMG (iBAMG) method is introduced. Compared to
BAMG, the iBAMG follows more closely the spirit of
classical AMG, which attempts to collapse unwanted connections based
on the assumption that the smooth error is locally constant.

In the work~\cite{Brandt.A;Brannick.J;Kahl.K;Livshits.I.2011a}, BAMG
is paired with an adaptive relaxation
\cite{Brandt.A.2000a,Kahl.K.2009a} and a multigrid eigensolver
\cite{Brezina.M;Manteuffel.T;McCormick.S;Ruge.J;Sanders.G;Vassilevski.P.2008a}.
A combination of the bootstrap cycling scheme
\cite{Livshits.I.2008a,Kahl.K.2009a} and then the multigrid
eigensolver is used to compute sufficiently accurate sets of test
vectors and adaptive relaxation is used to improve the AMG setup
cycle.

In the recent years there have been introduced also other adaptive
approaches to constructing hierarchy of spaces.  Classical AMG based
approach in defining the hierarchy of spaces is considered in the
framework of adaptive reduction algorithms in
\cite{2010BrannickJ_FrommerA_KahlK_MacLachlanS_ZikatanovL-aa} and
\cite{2006MacLachlanS_ManteuffelT_McCormickS-aa}. Adaptive BoxMG was
considered in~\cite{2012MacLachlanS_MoultonJ_ChartierT-aa}.
Specialized adaptive approach for Markov Chains is presented
in~\cite{2011De-SterckH_MillerK_TreisterE_YavnehI-aa}.

\section{Concluding remarks}\label{s:conclusion}
In this paper, we try to give a coherent presentation of a number
algebraic multigrid (AMG) methods.  But this presentation, limited by
our current theoretical understanding of AMG algorithms in general, is
by no means complete.  We choose to include those AMG algorithms that
can fit into the theoretical frameworks that are presented in this
paper.  One notable exception is the bootstrap and adaptive AMG
presented in \S\ref{s:adaptiveAMG}.  This type of algorithms still
lack a good theoretical understanding, but they provide an algorithmic
framework to leverage and extend those AMG algorithms presented in
\S\ref{s:GMG}-\ref{s:agmg} for a wider range of applications.

There are still many AMG algorithms in the literature that we are not
able to include in this article for two main reasons.  One is that
those algorithms can not be easily cast within our theoretical
frameworks; another is that there are algorithms that the authors are
yet to comprehend on a reasonable theoretical level.  Examples of
algorithms and results that need further investigation and analysis
include adaptive filtering, multilevel ILU methods, the multilevel
convergence properties of SA-AMG, BAMG, Adaptive AMG and many others.

The AMG methods studied in this paper are obtained by optimizing
the choice of coarse spaces based on a given smoother.  Indeed, almost
all the existing AMG method follow this strategy.   It is possible,
however, that an AMG method can also be designed by optimizing the choice
of smoother based on a given coarsening strategy.  This, in our
view, is a subject worthy further investigation.   Theoretically, it
is also conceivable to optimize both coarsening and smoother
simultaneously. 

Finally, we would like to note that several AMG software packages have
been developed and are currently in use, most noticeably: \\
Hypre: \url{http://acts.nersc.gov/hypre/},\\
 Trilinos: \url{https://trilinos.org/},\\
 Multigraph: \url{http://ccom.ucsd.edu/~reb/software.html},\\
AGMG: \url{http://homepages.ulb.ac.be/~ynotay/AGMG/},\\
 PyAMG: \url{http://pyamg.org/},\\
 FASP: \url{http://fasp.sourceforge.net/}.

\section*{Acknowledgments}
The authors wish to thank Stephen~F.~McCormick, Xiaozhe Hu, James
Brannick, Scott MacLachlan for many helpful discussions and
suggestions and especially to thank Hongxuan Zhang for his help
throughout the preparation of this article.

The work of the first author was supported by the DOE Grant
DE-SC0009249 as part of the Collaboratory on Mathematics for
Mesoscopic Modeling of Materials and by NSF grants DMS-1412005
and DMS-1522615.  The work of the second author was supported
by NSF grants DMS-1418843 and DMS-1522615.

\bibliographystyle{actaagsm}
\bibliography{AMG_AN_only}

\end{document}